\documentclass[11pt]{article}

\usepackage{amsfonts,amssymb, latexsym, amsmath, amsthm,mathrsfs, verbatim,geometry}

\usepackage{graphics}
\usepackage{slashed}
\usepackage{url,color,xcolor}
\usepackage{enumerate}
\usepackage{esint}
\usepackage{pifont}
\usepackage{upgreek}
\usepackage{times}
\usepackage{calligra}
\usepackage{graphicx}
\usepackage{caption}
\usepackage{listings}

\allowdisplaybreaks

\usepackage{hyperref}
\hypersetup{
    colorlinks,
    citecolor=red,
    filecolor=green,
    linkcolor=blue,
    urlcolor=black
}
\definecolor{backcolour}{rgb}{0.95,0.95,0.92}
\lstdefinestyle{mystyle}{
    backgroundcolor=\color{backcolour},   
}

\def\R{\mathbb R}

\def\N{\mathbb N}
\def\Z{\mathbb Z}

\def\cala{\mathcal{A}}

\def\be{\begin{equation}}
\def\ee{\end{equation}}
\def\bea{\begin{eqnarray}}
\def\eea{\end{eqnarray}}
\def\beas{\begin{eqnarray*}}
\def\eeas{\end{eqnarray*}}

\def\d{\partial}

\def\l{\lambda}
\def\pa{\partial }

\def\l{\lambda}

\def\lv{\left\vert}
\def\rv{\right\vert}

\def\bcr{\begin{color}{red}}
\def\bcb{\begin{color}{blue}}
\def\ec{\end{color}}

\def\lv{\left\vert}
\def\rv{\right\vert}

\newcommand{\ga}{\gamma}
\newcommand{\de}{\delta}
\newcommand{\al}{\alpha}

\newcommand{\om}{\omega}
\newcommand{\la}{\lambda}

\newcommand{\ka}{\kappa}
\newcommand{\eps}{\varepsilon}
\newcommand{\dif}{\operatorname{d}\!}

\newcommand{\beq}{\begin{equation}}
\newcommand{\eeq}{\end{equation}}
\newcommand{\beqs}{\begin{equation*}}
\newcommand{\eeqs}{\end{equation*}}
\newcommand{\beqa}{\begin{equation}\begin{aligned}}
\newcommand{\eeqa}{\end{aligned}\end{equation}}
\newcommand{\beqas}{\begin{equation*}\begin{aligned}}
\newcommand{\eeqas}{\end{aligned}\end{equation*}}

\newtheorem{theorem}{Theorem}[section]
\newtheorem{definition}[theorem]{Definition}
\newtheorem{proposition}[theorem]{Proposition}
\newtheorem{prop}[theorem]{Proposition}

\newtheorem{corollary}[theorem]{Corollary}
\newtheorem{lemma}[theorem]{Lemma}
\newtheorem{remark}[theorem]{Remark}

\renewcommand{\theequation}{\arabic{section}.\arabic{equation}}

\title{Gravitational Collapse for Polytropic Gaseous Stars: Self-similar Solutions}

\author{Yan Guo\thanks{Division of Applied Mathematics, Brown University, Providence, RI 02912, USA, Email: Yan\_Guo@brown.edu.}, \ Mahir Had\v zi\'c\thanks{Department of Mathematics, University College London, London WC1E 6XA, UK. Email: m.hadzic@ucl.ac.uk.}, \  Juhi Jang\thanks{Department of Mathematics, University of Southern California, Los Angeles, CA 90089, USA, and Korea Institute for Advanced Study, Seoul, Korea.  Email: juhijang@usc.edu.}, \ and Matthew Schrecker\thanks{Department of Mathematics, University College London, London WC1E 6XA, UK. Email: m.schrecker@ucl.ac.uk.},}

\date{}

\begin{document}

\maketitle

\abstract{In the supercritical range of the polytropic indices $\gamma\in(1,\frac43)$ we show the existence of smooth radially symmetric self-similar solutions to the gravitational Euler-Poisson system. These solutions exhibit gravitational collapse in the sense that the density blows-up in finite time. Some of these solutions were numerically found by Yahil in 1983 and they can be thought of as polytropic analogues of the Larson-Penston collapsing solutions in the isothermal case $\ga=1$. They each contain a sonic point, which leads to numerous mathematical difficulties in the existence proof.}

\tableofcontents
\section{Introduction and the main result}

The rigorous description of stellar collapse in the context of Newtonian gravity is a fundamental mathematical problem. It is believed, at least for some classes of initial data, that on approach to singularity a self-gravitating gaseous star will enter an approximately self-similar regime~\cite{Larson69, Penston69, Shu77, Yahil83, Harada03}, which will intertwine the spatial and the time scales in a universal manner dictated by the scaling symmetries of the problem. The purpose of this paper is to construct radially symmetric examples of {\em exactly} self-similar imploding solutions for the full range of the supercritical polytropic pressure laws.

A self-gravitating Newtonian star is described using the gravitational Euler-Poisson equations, coupling the isentropic compressible Euler equations to a gravitational potential. In three spatial dimensions, under the assumption of radial symmetry, these equations take the form 
\begin{align}
\d_t\rho +\d_r(\rho u)+\frac{2}{r}\rho u &=0, \label{eq:EPCONT}\\
\rho\big(\d_t u+u\d_ru\big)+\d_rp+\frac{1}{r^2}\rho m& =0, \label{eq:EPMOM}
\end{align}
where the principal unknowns $\rho(t,r)$ and $u(t,r)$ are the density and radial velocity of the star, respectively, and depend only on time $t$ and the radial coordinate $r=|x|$. 
Equation~\eqref{eq:EPCONT} gives the conservation of mass and~\eqref{eq:EPMOM} is the conservation of momentum with the given pressure law and gravitational force. We will assume throughout that the pressure  $p=P(\rho)$ satisfies the polytropic equation of state
\begin{equation}\label{E:EOS}
P(\rho)=\ka\rho^\ga,\quad \ga\in(1,\frac43),\:\ka>0,
\end{equation}
and the mass function $m(t,r)$ is defined by
\beq
m(t,r)=4\pi\int_0^r\sigma^2\rho(t,\sigma)\,\dif\sigma.
\eeq
Notice that the term $\frac{m}{r^2}$ appearing in the momentum equation~\eqref{eq:EPMOM} corresponds to the radial component of the gravitational force field $\nabla\phi$ generated by the 
gravitational potential $\phi$, which by definition solves the Poisson equation
$$\Delta\phi=4\pi\rho,\quad \lim_{|x|\to\infty}\phi(t,x)=0.$$
This is easily checked under the assumptions of radial symmetry.

A natural criticality scale is introduced in the problem by varying the polytropic index $\ga$ in the pressure law~\eqref{E:EOS}. It is easily checked that the nonlinear flow associated with~\eqref{eq:EPCONT}--\eqref{E:EOS} is invariant under the scaling transformation
\beqa\label{E:SCALING}
\rho(t,r)\mapsto \la^{-\frac{2}{2-\ga}}\rho\Big(\frac{t}{\la^{\frac{1}{2-\ga}}},\frac{r}{\la}\Big),\\
u(t,r)\mapsto \la^{-\frac{\ga-1}{2-\ga}}u\Big(\frac{t}{\la^{\frac{1}{2-\ga}}},\frac{r}{\la}\Big).
\eeqa
This scaling is in fact the only invariant scaling for the compressible Euler-Poisson system, by contrast to the compressible Euler equations which allow for a 2-parameter family of invariant scalings, see for example~\cite{Merle19}.

When $\gamma>\frac43$ the problem is {\em mass-subcritical} with respect to the scaling~\eqref{E:SCALING}, see~\cite{Hadzic18}. In this case, under the assumption of finite total mass and energy, it is known that no collapsing solutions can exist, see~\cite{Deng02}. In the {\em mass-critical} case, there is a well-known finite-dimensional family of collapsing stars discovered by Goldreich and Weber~\cite{Goldreich80}, see also~\cite{Makino92,FuLin98,Deng03}. The goal of this paper is to prove the existence of 
self-similar solutions describing gravitational collapse in the {\em mass-supercritical} regime $\ga\in(1,\frac43)$.

Motivated by~\eqref{E:SCALING}, we define the self-similar variable
\beq
y=\frac{r}{\sqrt{\ka}(-t)^{2-\ga}}
\eeq
and formally look for solutions to~\eqref{eq:EPCONT}--\eqref{E:EOS} of the form
\beqa
\rho(t,r)=(-t)^{-2}\tilde\rho(y),\\
u(t,r)=\sqrt{\ka}(-t)^{1-\ga}\tilde u(y).
\eeqa
Substituting this ansatz into the continuity equation~\eqref{eq:EPCONT} and dropping the tilde-s, we derive
\beqa
\rho'\big(u+(2-\ga)y)+\rho u'+2\rho+\frac{2}{y}\rho u=0.
\eeqa
Multiplying through by $y^2$, we simplify to find
$$(4-3\ga)y^2\rho=\big(y^2\rho(u+(2-\ga)y)\big)'$$
which we integrate to get a representation for the self-similar local mass as
\beq
\int_0^yz^2\rho(z)\,\dif z=\frac{y^2\rho(u+(2-\ga)y)}{4-3\ga}.
\eeq
Thus we derive from the momentum equation~\eqref{eq:EPMOM} the second self-similar equation
\beq
\rho(u+(2-\ga)y)u'+\ga\rho^{\ga-1}\rho'+(\ga-1)\rho u+\frac{4\pi}{4-3\ga}\rho^2(u+(2-\ga)y)=0.
\eeq
It will be convenient in what follows to work with the re-scaled relative velocity, rather than working directly with the velocity $u$. The new relative velocity is defined as
\beq\label{E:OMEGADEF}
\om=\frac{u+(2-\ga)y}{y}.
\eeq
We therefore arrive at the self-similar ODE system
\beqa
y\om\rho'+y\rho\om'=&\,(4-3\ga)\rho-3\rho\om,\\
y\rho\om\big(y\om'+\om-(2-\ga))+\ga\rho^{\ga-1}\rho'=&\,-(\ga-1)y\rho\om-(\ga-1)(\ga-2)y\rho-\frac{4\pi}{4-3\ga}\rho^2y\om.
\eeqa
Equivalently we may rewrite the system in the form
\beqa\label{eq:EPSS}
\rho'=& \frac{y\rho\left(2\om^2+(\ga-1)\om-\frac{4\pi\rho\om}{4-3\ga}+(\ga-1)(2-\ga)\right)}{\ga\rho^{\ga-1}-y^2\om^2},\\
\om'=&\,\frac{4-3\ga-3\om}{y} -  \frac{y\om\left(2\om^2+(\ga-1)\om-\frac{4\pi\rho\om}{4-3\ga}+(\ga-1)(2-\ga)\right)}{\ga\rho^{\ga-1}-y^2\om^2}.
\eeqa

We refer to~\eqref{eq:EPSS} as the self-similar Euler-Poisson system.
Clearly, this system has a singularity at $y=0$. However, there is a further singularity which occurs whenever $\ga\rho^{\ga-1}-y^2\om^2=0$. This is of fundamental importance and the presence of such singularities, as we shall explain below, is unavoidable in the study of {\em smooth} self-similar solutions to~\eqref{eq:EPSS} satisfying physically reasonable boundary conditions.  
This motivates the following definition.

%%%%%%%%%%%%%%%%%%%%%%%
%%%%%%%%%%%%%%%%%%%%%%%

\begin{definition}[Sonic point]
Let $(\rho(\cdot),\om(\cdot))$ be a $C^1$-solution to the self-similar Euler-Poisson system~\eqref{eq:EPSS} on the interval $(0,\infty)$. A point $y_*\in(0,\infty)$ such that 
\[
\ga\rho^{\ga-1}(y_*)-y_*^2\om^2(y_*)=0
\] 
is called a \textit{sonic point}.
\end{definition}

%%%%%%%%%%%%%%%%%%%%%%%
%%%%%%%%%%%%%%%%%%%%%%%

If $y_\ast$ is a sonic point, then the hypersurface defined by the relation $r(t)=\sqrt\kappa y_\ast (-t)^{2-\ga}$ corresponds to the backward acoustic cone emanating from the
origin $(0,0)$ (\cite{Brenner98,Merle19}). It satisfies the relation $\dot r(t) = u(t,r(t))-c_s(t,r(t))$, where $c_s$ is the sound speed $c_s=\frac{dP}{d\rho}= \sqrt{\kappa\ga \rho^{\ga-1}}$.

%%%%%%%%%%%%%%%%%%%%%%%
%%%%%%%%%%%%%%%%%%%%%%%

We are looking for smooth solutions which are both regular at the (scaling) origin $y=0$ and satisfy suitable decay conditions as $y\to\infty$. Taking Taylor expansions at the origin and in the far-field (as $y\to\infty$), we see that we require the initial and asymptotic boundary conditions
\begin{align}
\rho(0)>0, \quad&\om(0)=\frac{4-3\ga}{3}, \label{eq:IC}\\
\rho(y)\sim y^{-\frac{2}{2-\ga}}\text{ as }y\to\infty, \quad& \lim_{y\to\infty}\om(y)=2-\ga. \label{eq:BC}
\end{align}
From these conditions, it is clear that any continuous solution of \eqref{eq:EPSS} and \eqref{eq:IC}--\eqref{eq:BC} must have at least one sonic point $y_*>0$.

In the {\em isothermal} case $\gamma=1$, the existence of global solutions satisfying \eqref{eq:EPSS} and \eqref{eq:IC}--\eqref{eq:BC} has a long history in the astrophysics literature, primarily relying on numerical methods. In their seminal works, Larson \cite{Larson69} and Penston \cite{Penston69} independently showed (numerically) the existence of a globally defined solution satisfying~\eqref{eq:IC}--\eqref{eq:BC} and with only a single sonic point present. Subsequently, Hunter \cite{Hunter77} numerically constructed a full sequence of further solutions, commonly referred to as Hunter-type solutions, each of which also has a single sonic point (see also the work of Shu \cite{Shu77} and the later work of Whitworth and Summers~\cite{Whitworth85}). Despite the physically simplifying assumption $\gamma=1$,  these families of solutions attracted a lot of attention in the physics literature as possible prototype models for the behaviour of the core in late stages of gravitational collapse. In fact, the Larson-Penston (henceforth, LP) solutions were judged to be the more stable solutions under subsequent numerical analysis \cite{Brenner98,Maeda01,Ori88}. They also play an important role in the Newtonian criticality theory and the resolution of the so-called {\em self-similarity hypothesis}, see~\cite{Harada03}.

However, the assumption that the flow is isothermal received criticism on physical grounds, for example by Yahil \cite{Yahil83}, who pointed out that the physical condition of finite energy is violated unless $\ga>\frac65$. The value $\ga=\frac65$ plays the role of the {\em energy-critical} exponent with respect to the scaling~\eqref{E:SCALING}, see~\cite{Hadzic18}.
More importantly, different values of $\gamma>1$ allow us to encode stars with different thermodynamic properties and it is therefore important to understand the space of self-similar flows in the range $\ga\in(1,\frac43)$. In the range $\ga\in(\frac65,\frac43)$, Yahil~\cite{Yahil83} constructed a family of numerical self-similar solutions to~\eqref{eq:EPSS} with finite energy. These solutions share certain characteristics with the isothermal LP solution. For example, the physical radial velocity remains strictly negative (except at the origin, where it vanishes) up to the collapse time in both the Yahil solutions and the LP solutions. This property does not hold for Hunter solutions and has been tied to the possible dynamic instabilities of such solutions by Maeda--Harada \cite{Maeda01}. This leads us to the following definition.
\begin{definition}[Yahil-type solution]\label{D:YAHIL}
Let $\ga\in(1,\frac43)$. A pair of $C^1$ functions $(\rho,\om)$ defined on a connected interval $I\subset[0,\infty)$ satisfying the self-similar Euler-Poisson system \eqref{eq:EPSS} is said to be of Yahil-type if
\begin{itemize}
\item[(i)] There exists a unique sonic point $y_*\in I$;
\item[(ii)] For all $y\in I$, $\rho(y)>0$ and for all $y\in I\setminus\{0\}$, $u(y)<0$.
\end{itemize}
\end{definition}

Recently, the first three authors were able to construct LP solutions in the case $\ga=1$ in \cite{Guo20}. 
The main result of this paper is to show that Yahil solutions exist for the full physical range $\ga\in(1,\frac43)$, including the finite energy range ($\ga>\frac65$).

%%%%%%%%%%%%%%%%%%%%%%%%%
%%%%%%%%%%%%%%%%%%%%%%%%%

\begin{theorem}\label{T:MAIN}
For each $\ga\in(1,\frac43)$, there exists a global, real-analytic, Yahil-type solution $(\rho,\om)$ of \eqref{eq:EPSS}, \eqref{eq:IC}--\eqref{eq:BC} with a single sonic point $y_*$ and satisfying the natural, physical  conditions
\beqa\label{E:RHOOMBOUNDS}
&\rho(y)>0 \text{ for all }y\in[0,\infty), \quad -\frac23 y< u(y)<0 \text{ for all }y\in(0,\infty).
\eeqa
In addition, both $\rho$ and $\om$ are strictly monotone on their domain of definition:
\beqa \label{E:MONOTONICITY}
&\rho'(y)<0 \text{ for all }y\in(0,\infty),\quad \om'(y)>0\text{ for all }y\in(0,\infty).
\eeqa
\end{theorem}

The proof of this theorem is a consequence of a delicate analysis of the nonautonomous dynamical system~\eqref{eq:EPSS} in the regions separated by the sonic point $y_\ast$, presented in Sections~\ref{S:RIGHT} and~\ref{S:LEFT}. The combination of results derived in these two sections gives Theorem~\ref{T:MAIN} and the short argument is given in Section~\ref{S:PROOF}.

%%%%%%%%%%%%%%%%%%%%%%%%%
%%%%%%%%%%%%%%%%%%%%%%%%%

The most famous class of special solutions to the radially symmetric Euler-Poisson system are the Lane-Emden steady stars~\cite{Chandrasekhar38}, known to be of finite mass and energy 
if $\gamma\in[\frac65,2)$. Their dynamic stability is a classical subject, and in the case $\gamma>\frac 43$ they are known to be linearly stable and conditionally nonlinearly stable~\cite{Rein03}. By contrast, when $\gamma\in[\frac65,\frac43)$ the Lane-Emden stars are unstable~\cite{Jang08,Jang14}. In the critical case $\gamma=\frac43$, the Lane-Emden stars are spectrally stable, but nonlinearly unstable. The latter statement follows by observing that the above mentioned Goldreich-Weber (henceforth GW) collapsing stars can be chosen initially to be arbitrarily close to the corresponding steady Lane-Emden stars. In fact, due to the mass-critical nature of the problem, the GW collapse is a consequence of an effective separation of variables in the problem, where the solution corresponds to a time-modulated spatial profile, which satisfies a Lane-Emden-like equation. By time-reversal, there also exist global-in-time expanding GW-solutions, whose nonlinear stability was shown in~\cite{Hadzic18}.

The solutions constructed in Theorem~\ref{T:MAIN} ($1<\gamma<\frac43$) are very different from the GW solutions ($\ga=\frac43$), and owe their existence to a subtle balancing of the three dominant forces in the problem: inertia, pressure, and gravity.  A completely different portion of the phase-space is populated by the so-called dust-like collapsing stars, which have been shown to exist in~\cite{Guo20b}. The solutions constructed in~\cite{Guo20b} do not honour the scaling invariance implied by~\eqref{E:SCALING}, but are instead to a leading order approximated by the so-called dust solutions, which solve~\eqref{eq:EPCONT}--\eqref{eq:EPMOM} without the pressure term $p$.

%The system of ODE~\eqref{eq:EPSS} is hard since, as explained above, the flow must pass through a sonic point. The requirement of smoothness at such a point then leads to a number of mathematical difficulties. Generically, if we stipulate that some $y_\ast\in(0,\infty)$ be a sonic point, then the flow to the left and the right of that point will not be global. It is only for special values of $y_\ast$ where the corresponding solution is in fact globally defined on $[0,\infty)$.
%In a recent pioneering study~\cite{Merle19} of self-similar solutions for the compressible Euler system with the equation of state $P(\rho)=\rho^\ga$ ($\ga>1$), the authors systematically developed the existence theory for $C^\infty$-self-similar solutions of the Euler flow; the underlying $2\times2$ system of ODE is in this case autonomous (in contrast to~\eqref{eq:EPSS}). The smoothness of the self-similar solutions across the sonic point is in fact a crucial ingredient in the proof of their (finite codimension) nonlinear stability~\cite{Merle19b}. 

As explained above, the most exciting physical feature of the self-similar solutions that we construct is their behaviour in the inner core region, as a possible model of typical stellar collapse scenario.   
Nevertheless, for completeness we also discuss some {\em global} properties of the solution, in particular the size of the total mass and total energy.
The solutions constructed in Theorem~\ref{T:MAIN} have infinite total mass
\[
M[\rho] = 4\pi \int_0^\infty \rho(t,r) r^2\dif r,
\]
as can easily be seen from the asymptotic behaviour~\eqref{eq:BC}. A short calculation shows that for any fixed $t<0$, asymptotically as $r\to\infty$
\begin{align}\label{E:RHOMPHIASYMP}
\rho(t,r)\sim_{r\to\infty} r^{-\frac2{2-\ga}}, \ \ m(t,r)\sim_{r\to\infty} r^{\frac{4-3\ga}{2-\ga}}, \ \ \phi(t,r)\sim_{r\to\infty} r^{\frac{2(1-\ga)}{2-\ga}},
\end{align}
where $m(t,r) := 4\pi \int_0^r \rho(t,s) s^2\dif s$ is the mass contained in a ball of radius $r$. On the other hand, the total energy 
\[
E[\rho,u]=4\pi \int_0^\infty \left(\frac12 \rho u^2 + \frac1{\ga-1}\rho^\ga - \frac1{8\pi}|\pa_r\phi|^2\right) r^2\dif r
\]
of the solutions constructed in Theorem~\ref{T:MAIN} is finite when $\ga\in(\frac65,\frac43)$ and infinite for $\ga\in(1,\frac65]$. This can be easily seen 
from~\eqref{E:RHOMPHIASYMP} and the asymptotic behaviour $u(t,r)_{r\to\infty}\sim r^{\frac{1-\ga}{2-\ga}}$ for any fixed $t<0$, which is established later in Lemma~\ref{lemma:asymptotics}.

A further surprising outcome of our work is the provision of a new context within which to consider the above mentioned distinction between the LP- and Hunter-type solution.
In the context of the isothermal problem ($\ga=1$), the demand that the solution be regular produces two possible algebraic ``branches" for the 
Taylor expansion coefficients at the sonic point. The LP-solution constructed in~\cite{Guo20} belongs to one of them, all the Hunter solutions to the other, and the branches intersect at exactly one point.
When $\ga>1$, we will show that there are two analogous branches. Remarkably, in the formal $\gamma\to1$ limit one of them converges to two portions of the two isothermal branches that together form a continuous curve
containing both the LP- and Hunter solutions in the isothermal case. We thus term the solutions coming off this $\gamma>1$-branch the {\em Larson-Penston-Hunter- (LPH-) type} solutions. We comment on this further in
Section~\ref{SS:METHODS}, while the detailed analysis can be found in
Section~\ref{S:SONICPOINT}.

%%%%%%%%%%%%%%%%%%%%%
%%%%%%%%%%%%%%%%%%%%%

\subsection{Methodology}\label{SS:METHODS}

%%%%%%%%%%%%%%%%%%%%%
%%%%%%%%%%%%%%%%%%%%%

Due to the importance that the sonic condition will play throughout all of the subsequent analysis, we define here a function 
\beq\label{def:G}
G(y;\rho,\om)=\ga\rho^{\ga-1}-y^2\om^2.
\eeq

%%%%%%%%%%%%%%%%%%%%%
%%%%%%%%%%%%%%%%%%%%%

\begin{definition}[Sonic, supersonic, and subsonic]\label{def:subsonicitysupersonicity}
We say that the flow is subsonic whenever $G(y;\rho,\om)>0$, supersonic when $G(y;\rho,\om)<0$, and sonic when $G(y;\rho,\om)=0$.
\end{definition}

%%%%%%%%%%%%%%%%%%%%%
%%%%%%%%%%%%%%%%%%%%%

For convenience, we denote by $h(\rho,\om)$ the function
\beq\label{def:h}
h(\rho,\om)=2\om^2+(\ga-1)\om-\frac{4\pi\rho\om}{4-3\ga}+(\ga-1)(2-\ga).
\eeq
The system \eqref{eq:EPSS} may then be written concisely as
\beqa\label{eq:rhoom}
\rho'=&\,\frac{y\rho h(\rho,\om)}{G(y;\rho,\om)},\\
\om'=&\,\frac{4-3\ga-3\om}{y}-\frac{y\om h(\rho,\om)}{G(y;\rho,\om)}.
\eeqa
There are two known, explicit solutions to the system \eqref{eq:rhoom}, the {\em Friedman} solution 
\beq
\om_F=\frac{4-3\ga}{3},\quad \rho_F=\frac{1}{6\pi}
\eeq
which satisfies the initial condition \eqref{eq:IC} at the origin, but fails the asymptotic boundary condition \eqref{eq:BC}, and the {\em far-field} solution,
\beq\label{def:far-field}
\om_f=2-\ga,\quad \rho_f=ky^{-\frac{2}{2-\ga}},\text{ where }k=\Big(\frac{\ga(4-3\ga)}{2\pi(2-\ga)^2}\Big)^{\frac{1}{2-\ga}},
\eeq
which satisfies the asymptotic boundary condition \eqref{eq:BC}  but fails the initial condition \eqref{eq:IC} .
Note that the constant $k>0$ is well-defined due to $\ga<4/3$. 

The Friedman and far-field solutions have sonic points at $y_F(\ga)$, $y_f(\ga)$, respectively, with
\beqa\label{def:yfyF}
y_F(\ga)=\frac{3}{4-3\ga}\sqrt{\frac{\ga}{(6\pi)^{(\ga-1)}}},\quad y_f(\ga)=\frac{\sqrt{\ga}}{2-\ga}\Big(\frac{4-3\ga}{2\pi}\Big)^{\frac{\ga-1}{2}}.
\eeqa
For all $\ga\in(1,\frac43)$, we see that $0<y_f(\ga)<y_F(\ga)<\infty$. Henceforth, we will drop the explicit dependence on $\ga$ for $y_f$ and $y_F$, emphasising that for each $\ga\in(1,\frac43)$, $[y_f,y_F]$ is a compact interval.

The system of ODE~\eqref{eq:EPSS} is challenging since, as explained above, the flow must pass through a sonic point. The requirement of smoothness at such a point then leads to a number of mathematical difficulties. Generically, if we stipulate that some $y_\ast\in(0,\infty)$ be a sonic point, then the flow around that point will not be global. It is only for special values of $y_\ast$ where the corresponding solution is in fact globally defined on $[0,\infty)$.
In a recent pioneering study of self-similar solutions for the compressible Euler system with the equation of state $P(\rho)=\rho^\ga$ ($\ga>1$), Merle, Rapha\"el, Rodnianski, and Szeftel~\cite{Merle19} systematically developed the existence theory for $C^\infty$-self-similar solutions of the Euler flow; the underlying $2\times2$ system of ODE is in this case autonomous (in contrast to~\eqref{eq:EPSS}). The smoothness of the self-similar solutions across the sonic point is in fact a crucial ingredient in the proof of their (finite codimension) nonlinear stability~\cite{Merle19b}. 

We will seek a solution with sonic point at some $y_*\in(y_f,y_F)$. 
Making the formal Taylor expansion around the sonic point $y_*$, we set
\beq
\rho(y)=\sum_{N=0}^\infty \rho_N(y-y_*)^N,\quad \om(y)=\sum_{N=0}^\infty \om_N (y-y_*)^N.
\eeq
In order to have a smooth solution through $y_*$, we require that the values $\rho_0=\rho(y_*)$ and $\om_0=\om(y_*)$ are constrained by the two identities
\beq
G(y_*;\rho_0,\om_0)=0,\quad h(\rho_0,\om_0)=0.
\eeq
For all $y_*\in[y_f,y_F]$, we will show below that there is a unique pair $(\rho_0,\om_0)$ satisfying these two conditions. When we come to solve for the first order coefficients $(\rho_1,\om_1)$, however, we see that the picture becomes more complicated. In fact, there are again two possible branches from which the coefficients may be chosen. 
However, as we next explain, it is natural to view the $\gamma=1$-case as a degenerate case. Namely, the possible pairs lie on graphs as shown in Figure \ref{fig:gamma=1}, parametrised by $\om_0$ (equivalently, by $y_*$). In this case, the LP solution constructed in \cite{Guo20} lies in the region of the LP branch for which $\om_0<\frac12$ (equivalently $y_*>2$) while the numerically constructed Hunter solutions all lie in the region $\om_0>\frac12$ (equivalently $y_*<2$), compare also \cite[Fig. 2]{Hunter77}.

\begin{figure}[h]
\centering
\begin{minipage}{0.75\textwidth}
\centering
  \includegraphics[width=.8\linewidth]{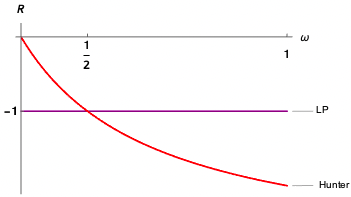}
\caption{{\footnotesize Plot of possible values $R=y_*\rho_1/\rho_0$ as a function of $\om_0\in[\frac13,1]$ in case $\ga=1$.}}
\label{fig:gamma=1}
\end{minipage}
\end{figure}

However, as soon as we increase $\ga>1$, a bifurcation occurs. The picture then looks like one of the cases in Figure \ref{fig:2}. The Hunter and LP solutions in the case $\ga=1$ actually live on the same branch of the solutions, a feature that is concealed in the isothermal case by the degeneracy that makes the branches coalesce at this value. For $\ga>1$, the analogue of the isothermal LP solution is the global solution with a unique sonic point $y_*$ such that the first order coefficient $\rho_1$ lies on the joint LP-Hunter (henceforth LPH) branch, and with maximal $y_*$ (equivalently minimal $\om_0$) - this is the lower (blue) branch in Figure~\ref{fig:2}. Such a solution which will be shown to correspond to the Yahil solution that we are looking to construct, see Definition~\ref{D:YAHIL}. 
%We therefore look for the Yahil solution for $\ga>1$ as a solution with sonic point $y_*$ and coefficients on the LPH branch with maximal $y_*$.

\begin{figure}[h]
\centering
\begin{minipage}{0.9\textwidth}
\centering
  \includegraphics[width=.48\linewidth]{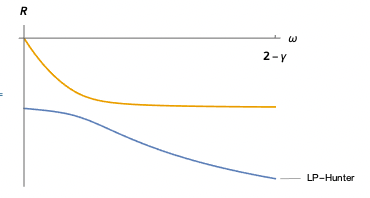}\hspace{.04\linewidth}\includegraphics[width=.48\linewidth]{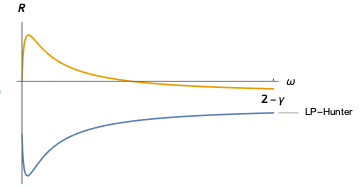}
\caption{{\footnotesize Plot of possible values $R=y_*\rho_1/\rho_0$ as functions of $\om_0\in[\frac{4-3\ga}{3},2-\ga]$ for $\ga=1.01$ and $1.3$}}
\label{fig:2}
\end{minipage}
\end{figure}

Once we have correctly identified the branch of solutions on which the LPH-type solution should lie, we seek the globally defined Yahil-type solution whose Taylor expansion at the sonic point is of LPH-type (see Definition \ref{def:LPH} for the precise meaning). We then proceed in four key steps, as in the earlier work of the first three authors, \cite{Guo20}.

\textbf{Step 1:} The first step is to complete the Taylor expansion at each potential sonic point $y_*\in[y_f,y_F]$ to obtain a local, analytic solution around $y_*$, denoted by
\[
(\rho(y;y_\ast),\om(y;y_\ast)).
\]
When clear from the context, we shall occasionally drop the dependence on $y_\ast$ in the notation above.
In comparison to \cite{Guo20}, the convergence of the Taylor series is significantly complicated by the presence of the term $\rho^{\ga-1}$ with its non-integer power. Various technical tricks are employed, using the Fa\`a di Bruno formula, to control the size of the coefficients arising in the expansion, while interval arithmetic is employed to control rigorously the sign of three key quantities (see \eqref{ineq:A2}--\eqref{ineq:dNdetA0} and Appendix \ref{A:IADETERMINANTS} below).

\textbf{Step 2:} Second, we show that the local solution arising from each $y_*\in[y_f,y_F]$ may be extended globally to the right, remains supersonic, and satisfies the correct asymptotic boundary condition \eqref{eq:BC}. This is based on the identification of several invariant regions to the right. Compared with the isothermal case, the key property to show is that the flow remains supersonic, a fact that is no longer trivially true. The asymptotics follow easily from the global existence and bounds obtained.

\textbf{Step 3:} The third, and key, step, is to show that there exists a critical value $\bar y_*$ for which the local analytic solution extends smoothly up to the singular point at the origin with limit $\om(y)\to\frac{4-3\ga}{3}$ as $y\to0$. Similarly to \cite{Guo20}, this $\bar y_*$ is found as the infimum of a fundamental set 
$$Y=\Big\{y_*\in(y_f,y_F)\,|\,\text{ there exists }y\text{ such that }\om(y;\tilde y_*)=\frac{4-3\ga}{3}\text{ for all }\tilde y_*\in[y_*,y_F)\Big\}.$$
It is here that many of the additional complications arising from the choice $P(\rho)=\rho^\ga$, $\ga>1$ make themselves felt. Many of the invariances that were easily available in the case $\ga=1$ are either significantly harder to prove or fail altogether. For example, we no longer have that the region $\{\om>\frac12\}$ is invariant as $y$ decreases. These losses are due to the non-linear structure of the quantities $h(\rho,\om)$ and $G(y;\rho,\om)$. Whereas, for $\ga=1$, the sets in phase space in which $h$ or $G$ have a constant sign are simply half-spaces (parametrised by $y$ due to the non-autonomous nature of the system), for $\ga>1$, they have a much more complicated structure, with a change in the geometry of the set $\{h(\rho,\om)=0\}$ especially at $\ga=\frac{10}{9}$, see Lemma \ref{lemma:f1structure}. This makes itself felt at a number of levels. For example, the sets $\{\om'(y)=0\}$ and $\{\om''(y)=0\}$ in the $(\rho,\om)$ have an intersection in the region $\{h<0,G>0\}$, something which cannot happen for $\ga=1$, while there are no obvious invariant regions for $G$ either.

To resolve the difficulties caused by these features, we prove a new and stronger property for the relative velocity $\om(y;y_*)$ for all $y_*\in Y$: {\em monotonicity} with respect to $y$. By a careful analysis of the phase plane and a continuity argument, we are able to show that for all $y_*\in Y$, the function $\om(y;y_*)$ is strictly monotone as long as it remains above the Friedman solution. This property, which is proved in the key Proposition \ref{prop:S=Y} below, allows us to propagate a lower bound for the quantity $G$ to the left, preventing the formation of additional sonic points and allowing us to extend the solution as far as the origin, $y=0$.

\textbf{Step 4:} The final step in the scheme is to show that the solution $(\rho(y;\bar y_*),\om(y;\bar y_*))$ connects smoothly to the origin. More precisely, we show that the solution is analytic on $(0,\infty)$ and $C^1$ at the scaling origin $y=0$. 
With a little extra work, one can show that the solution is in fact smooth at $y=0$, but the argument is not included in this paper, as it is not central to our proof of the existence of Yahil-type solutions.
This is achieved by exploiting again the monotonicity proved for $\omega$ to demonstrate that $\om(\cdot;\bar y_*)$ attains the boundary condition $\om(0;\bar y_*)=\frac{4-3\ga}{3}$ and that the density remains both bounded and monotone. This greatly simplifies the proof of the equivalent step in \cite{Guo20} and removes the need for a topological upper- and lower solution argument of the kind used in~\cite{Guo20}.

At three points throughout the proof (twice in the Taylor expansion at the sonic point in Propositions \ref{prop:R1W1} and \ref{P:AN} and then once more in extending the solution to the right in the technical Lemma \ref{lemma:Qmintervalarithmetic}), we require an understanding of the sign of key quantities depending polynomially on $\om$ and $\ga$. As the quantities are significantly too complicated to control by hand, we employ rigorous interval arithmetic, a means of computer-assisted proof that has been used several times recently to resolve open questions in the theory of PDE, see for example \cite{Castro17,Cohen21,GS14}. A useful overview of the method and its applications, along with a wealth of references to recent applications, can be found in \cite{GS18}. Unlike in these works, our use of interval arithmetic is elementary, as we perform most of the analysis directly, only employing interval arithmetic to find bounds for the maxima and minima of certain explicit polynomial quantities.

The paper is organised as follows. 
Details of the sonic point expansion, the definition of the LPH-type solutions, and the local existence of real-analytic solutions in the vicinity of the sonic point are presented in Section~\ref{S:SONICPOINT}. In Section~\ref{S:RIGHT} we show that for {\em any} $y_\ast\in[y_f,y_F]$ there exists an LPH-type solution on $[y_\ast,\infty)$
and provide a detailed asymptotic description of the solution as $y\to\infty$. Section~\ref{S:LEFT} is devoted to the existence problem to the left of the sonic point, and contains some of the key conceptual insights of the paper. In particular, Proposition~\ref{P:LEFTGLOBAL} shows that there exists a $\bar y_\ast\in[y_f,y_F]$ such that the associated local LPH-type solution extends to the whole
interval $(0,\bar y_\ast]$. The main theorem is then easily obtained by gluing together the constructed left- and right solutions, and the proof is presented in Section~\ref{S:PROOF}.

Several technical lemmas are stated and proved in Appendices~\ref{A:EU} and~\ref{A:COMB}. Appendix~\ref{A:EU} contains the standard existence and uniqueness argument {\em away} from the sonic points, while Appendix~\ref{A:COMB} contains the details of an involved combinatorial argument used to prove the existence of real-analytic solution in a neighbourhood of a sonic point.
Several of our arguments in Sections~\ref{S:RIGHT} and~\ref{S:LEFT} involve complicated multinomial expressions depending on $\ga$, $\om_0$, and $y_\ast$. Their signs play a crucial role in the proofs and we resort to rigorous, computer-assisted proofs by way of interval arithmetic to check the relevant signs. Appendix~\ref{A:IA} contains all the details of such arguments including the associated interval arithmetic code. Finally, Appendix~\ref{app:continuity} contains a detailed proof of some of the key continuity properties of the LPH-type solutions, used heavily in Section~\ref{S:LEFT}. Such a proof is not standard in the literature, but is quite similar to a related proof in~\cite{Guo20}, and the details are therefore moved to an appendix.

\bigskip 

{\bf Acknowledgments.}
Y. Guo's research is supported in part by NSF DMS-grant 2106650.
M. Had\v zi\'c's and M. Schrecker's  research is supported by the EPSRC Early Career Fellowship EP/S02218X/1.
J. Jang's research is supported by the NSF DMS-grant 2009458
and the Simons Fellowship (grant number 616364). 

%%%%%%%%%%%%%%%%%%%%%%%%%%%%%%%
%%%%%%%%%%%%%%%%%%%%%%%%%%%%%%%

\section{The sonic point} \label{S:SONICPOINT}

%%%%%%%%%%%%%%%%%%%%%%%%%%%%%%%
%%%%%%%%%%%%%%%%%%%%%%%%%%%%%%%

As discussed in the introduction, our strategy for constructing a solution to the system \eqref{eq:EPSS} is to begin from a sonic point $y_*$, obtain a solution locally around this point, and then to extend to both the left and to the right. The purpose of this section is to provide the solution locally around the sonic point. This is a difficult endeavour, as it requires us to first clarify how the condition of smoothness (in fact analyticity) at the sonic point affects our definition of the solution we seek after. This will lead us to the notion of the Larson-Penston-Hunter (LPH) branch. The next step involves a combinatorial argument that shows that locally around the sonic point there indeed exist analytic solutions of the LPH-type.

%%%%%%%%%%%%%%%%%%%%%%%%%%%%%%%
%%%%%%%%%%%%%%%%%%%%%%%%%%%%%%%

\subsection{The formal Taylor expansion}

Any smooth solution to the flow \eqref{eq:EPSS} must satisfy that, at any sonic point, $y_*$, the values $\rho_0=\rho(y_*)$, $\om_0=\om(y_*)$ satisfy the constraint
\beq\label{eq:sonicconstraint}
h(\rho_0,\om_0)=2\om_0^2+\big(\ga-1-\frac{4\pi\rho_0}{4-3\ga}\big)\om_0+(\ga-1)(2-\ga)=0.
\eeq
For notational reasons, we define
\beq\label{def:f1}
f_1(\om)=\frac{4-3\ga}{4\pi\om}\big(2\om^2+(\ga-1)\om+(\ga-1)(2-\ga)\big),
\eeq
so that $h(\rho,\om)=0$ corresponds to $\rho=f_1(\om)$. The structure of the level set $h(\rho,\om)=0$, equivalently $\rho=f_1(\om)$, will play an important role, both in solving for the Taylor coefficients at the sonic point (see Lemma \ref{lemma:rho0om0} below), but also in demonstrating certain crucial invariances along the flow in Section \ref{sec:invariants}.
\begin{lemma}\label{lemma:f1structure}
Let $\ga\in(1,\frac43)$ and consider the function $f_1(\om)$ on the domain $\om\in(0,2-\ga)$. On this domain,  $f_1$ is uniformly convex with a global minimum at 
\beq
\om_*=\sqrt{\frac{(\ga-1)(2-\ga)}{2}}.
\eeq
For $\ga\in(1,\frac{10}{9})$, the inequality $\om_*<\frac{4-3\ga}{3}$ holds while the inequality is reversed if $\ga\in(\frac{10}{9},\frac43)$ and equality holds at $\ga=\frac{10}{9}$.

In particular, $f_1'(\om_*)=0$ and, if $\ga\in(1,\frac{10}{9}]$, we have $f_1'(\om)\geq0$ for all $\om\in[\frac{4-3\ga}{3},2-\ga]$ (with strict inequality if at least one of $\ga<\frac{10}{9}$ or $\om>\frac{4-3\ga}{3}$ holds).
If $\ga\in(\frac{10}{9},\frac43)$, then for $\om\in[\frac{4-3\ga}{3},\om_*)$, $f_1'(\om)<0$ and for $\om\in(\om_*,2-\ga]$, $f_1'(\om)>0$.
\end{lemma}
The proof is by a simple, direct calculation, and so we omit it.
\begin{figure}[h]
\centering
\begin{minipage}{0.9\textwidth}
\centering
  \includegraphics[width=.4\linewidth]{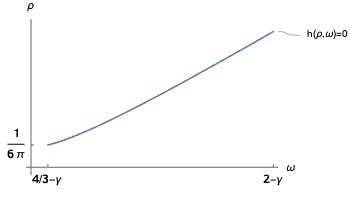}\hspace{.2\linewidth}\includegraphics[width=.4\linewidth]{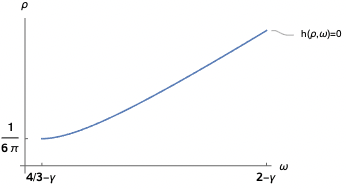}\\
  \includegraphics[width=.4\linewidth]{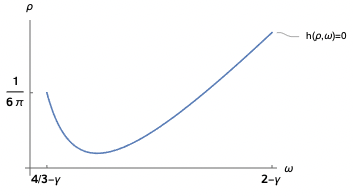}
\caption{{\footnotesize$\{h(\rho,\om)=0\}$ level sets for $\ga=1.08$, $\frac{10}{9}$, and $1.2$}}
\label{fig:hlevelset}
\end{minipage}
\end{figure}

Figure \ref{fig:hlevelset} plots the level set $h(\rho,\om)=0$ in the $(\rho,\om)$ plane for the cases $\ga=1.08,\frac{10}{9},1.2$ and $\om\in\big[\frac{4-3\ga}{3},2-\ga\big]$. The minimum for $\ga\geq\frac{10}{9}$ occurs at $\om=\sqrt{\frac{(\ga-1)(2-\ga)}{2}}\geq\frac{4-3\ga}{3}$ as stated in Lemma \ref{lemma:f1structure}.
\begin{lemma}\label{lemma:rho0om0}
For any $y_*\in[y_f,y_F]$, where $y_f$, $y_F$ are defined as in \eqref{def:yfyF}, there exists a unique pair $(\rho_0(y_*),\om_0(y_*))$ with $\rho_0(y_*)\geq\underline{\rho}>0$, where $\bar\rho$ depends only on $\ga$, satisfying 
\beq
G(y_*,\rho_0,\om_0)=0,\quad h(\rho_0,\om_0)=0.
\eeq
Moreover, the mapping $y_*\mapsto\om_0(y_*)$ is a strictly decreasing function for $y_*\in[y_f,y_F]$ with $$\om_0(y_f)=2-\ga,\quad \om_0(y_F)=\frac{4-3\ga}{3}.$$
\end{lemma}
\begin{proof}
We begin by recalling the definition of the function $f_1(\om)$ from \eqref{def:f1} and define also a function $f_2(\om;y_*)$ so that
\beqa\label{def:f1f2}
f_1(\om)=&\,\frac{4-3\ga}{4\pi\om}\big(2\om^2+(\ga-1)\om+(\ga-1)(2-\ga)\big),\\
f_2(\om;y_*)=&\,\Big(\frac{y_*^2\om^2}{\ga}\Big)^{\frac{1}{\ga-1}}.
\eeqa
As discussed above, the constraint $h(\rho_0,\om_0)=0$ is equivalent to $\rho_0=f_1(\om_0)$, while we see easily that $G(y_*,\rho_0,\om_0)=0$ if and only if $\rho_0=f_2(\om_0;y_*)$. So we seek $\om_0(y_*)$ such that $f_1(\om_0(y_*))=f_2(\om_0(y_*);y_*)$. This value is then defined to be $\rho_0(y_*)$. We easily check that 
$$f_1(\frac{4-3\ga}{3})=\frac{1}{6\pi},\quad f_1(2-\ga)=\frac{4-3\ga}{2\pi}.$$
Moreover, as $y_f$, $y_F$ are the sonic points corresponding to the far-field and Friedman solutions, respectively, we also know that
$$f_2(\frac{4-3\ga}{3};y_F)=f_1(\frac{4-3\ga}{3})=\frac{1}{6\pi},\quad f_2(2-\ga;y_f)=f_1(2-\ga)=\frac{4-3\ga}{2\pi}.$$
Noting then that 
\beq\label{ineq:f2derivs}
\d_\om f_2(\om;y_*),\d_{y_*}f_2(\om;y_*)>0,
\eeq we see that for any $y_*\in(y_f,y_F)$, we have
\beqas
f_2(\frac{4-3\ga}{3};y_*)<f_2(\frac{4-3\ga}{3};y_F)=\frac{1}{6\pi},\quad f_2(2-\ga;y_*)>f_2(2-\ga;y_f)=\frac{4-3\ga}{2\pi},
\eeqas
and so 
\beqas
(f_2(\cdot;y_*)-f_1)(\frac{4-3\ga}{3})<0<(f_2(\cdot;y_*)-f_1)(2-\ga).
\eeqas
Thus, by the intermediate value theorem, we see that $\om_0(y_*)$ exists as required, and hence so does $\rho_0(y_*)$. The uniqueness follows from the following observations:
\beqas
f_1'(\om)=&\,-\frac{f_1(\om)}{\om}+\frac{4-3\ga}{4\pi\om}\big(4\om+(\ga-1)\big),\\
\d_\om f_2(\om;y_*)=&\,\frac{2\om y_*^2}{(\ga-1)\ga}\Big(\frac{y_*^2\om^2}{\ga}\Big)^{\frac{1}{\ga-1}-1}=\frac{2f_2(\om;y_*)}{(\ga-1)\om}.
\eeqas
Thus at any point $\om>0$ such that $f_2(\om;y_*)\geq f_1(\om)$, we have
\beqas
\d_\om\big(f_2(\om;y_*)-f_1(\om)\big)=&\,\frac{2f_2(\om;y_*)}{(\ga-1)\om}+\frac{f_1(\om)}{\om}-\frac{4-3\ga}{4\pi\om}\big(4\om+(\ga-1)\big)\\
\geq &\,\frac{1}{\om}\Big(\frac{\ga+1}{\ga-1}f_1(\om)-\frac{4-3\ga}{4\pi}\big(4\om+(\ga-1)\big)\Big)\\
\geq &\,\frac{4-3\ga}{4\pi\om}\Big(\frac{\ga+1}{\ga-1}\big(2\om+(\ga-1)+\frac{(\ga-1)(2-\ga)}{\om}\big)-\big(4\om+(\ga-1)\big)\Big)\\
\geq&\,\frac{4-3\ga}{4\pi\om}\Big(2\om\frac{3-\ga}{\ga-1}+2+\frac{(\ga+1)(2-\ga)}{\om}\Big)>0
\eeqas
and so the uniqueness follows easily.
The monotonicity properties of $y_*\mapsto\om_0(y_*)$ then follow directly from \eqref{ineq:f2derivs} and $f_{2,\om}(\om_0(y_*),y_*)-f_{1,\om}(\om_0(y_*))>0$ as 
\beqas
0=&\,\frac{\dif}{\dif{y_*}}(f_2(\om_0(y_*);y_*)-f_1(\om_0(y_*))\\
=&\,\big(f_{2,\om}(\om_0(y_*),y_*)-f_{1,\om}(\om_0(y_*))\big)\om_0'(y_*)+f_{2,y_*}(\om_0(y_*),y_*).
\eeqas
To find the uniform lower bound $\rho_0\geq\underline{\rho}>0$, we note that $\rho_0=f_1(\om_0)$. As $\om_0\in[\frac{4-3\ga}{3},2-\ga]$, we easily obtain $f_1(\om_0)\geq \frac{(4-3\ga)(\ga-1)}{4\pi}>0$ as required.
\end{proof}

We seek a local solution around a sonic point $y_*\in[y_f,y_F]$ via a Taylor expansion. To that end, we now derive the necessary relations for the coefficients of the expansion. Suppose we have an analytic solution of system \eqref{eq:EPSS}. Then, after rearranging, we have
\begin{align}
\big(\ga\rho^{\ga-1}-y^2\om^2\big)\rho'=&\,{(\ga-1)y\rho(\om+2-\ga)}-{2y\rho\om}\big(\frac{2\pi\rho}{4-3\ga}-\om\big),\label{eq:EPSSSonic1}\\
\big(\ga\rho^{\ga-1}-y^2\om^2\big)\om'=&\,\frac{4-3\ga-3\om}{y}\big(\ga\rho^{\ga-1}-y^2\om^2\big)-{(\ga-1)y\om(\om+2-\ga)}+{2y\om^2}\big(\frac{2\pi\rho}{4-3\ga}-\om\big).\label{eq:EPSSSonic2}
\end{align}
We write the formal power series
\beq\label{E:FORMAL}
\rho(y)=\sum_{N=0}^\infty\rho_N(y-y_*)^N,\quad \om(y)=\sum_{N=0}^\infty\om_N(y-y_*)^N.
\eeq
By Lemma \ref{lemma:rho0om0}, we see that a choice of $y_*\in[y_f,y_F]$ defines a unique pair of values $(\rho_0,\om_0)$ for the Taylor series. We define the obvious notation
\beqa
(\om^2)_N=\sum_{k+j=N}\om_k\om_j,&\quad (\rho\om)_N=\sum_{k+j=N}\rho_k\om_j,\\
(\rho^2\om)_N=\sum_{k+j+l=N}\rho_k\rho_j\om_l,&\quad(\rho\om^2)_N=\sum_{k+j+l=N}\rho_k\om_j\om_l.
\eeqa
We recall the Fa\`a di Bruno formula for the $N$-th derivative of a composition,
\beq
\frac{\dif^N}{\dif y^N}\big(f(g(y))\big)=\sum_{(\la_1,\ldots,\la_N)\in M_N}\frac{N!}{\la_1!\cdots \la_N!}f^{(\la_1+\cdots+\la_N)}(g(y))\prod_{j=1}^N\Big(\frac{g^{(j)}(y)}{j!}\Big)^{\la_j},
\eeq
where $$M_N=\{(\la_1,\ldots,\la_N)\in(\Z_{\geq0})^N\,|\,\sum_{j=1}^Nj\la_j=N\}.$$
Taking $f(x)=x^{\ga-1}$, $g(y)=\rho(y)$ in this formula, we therefore obtain
\beqa
&\frac{\frac{\dif^N}{\dif y^N}\big(\rho^{\ga-1}(y)\big)\big|_{y=y_*}}{N!}\\
&=\sum_{(\la_1,\ldots,\la_N)\in M_N}\frac{(\ga-1)\cdots(\ga-(\la_1+\cdots+\la_N))\rho_0^{\ga-(\la_1+\cdots+\la_N)-1}}{\la_1!\cdots \la_N!}\prod_{j=1}^N\rho_j^{\la_j}=:P_N
\eeqa
and thus we have the power series
\beqa\label{eq:seriesPN}
\rho^{\ga-1}=&\,\sum_{N=0}^\infty\frac{\frac{\dif^N}{\dif y^N}\big(\rho^{\ga-1}(y)\big)\big|_{y=y_*}}{N!}(y-y_*)^N=\sum_{N=0}^\infty P_N(y-y_*)^N.
\eeqa
Throughout this section, for $N<0$, we set coefficients $\rho_{N},\om_N$ etc to be zero.
\begin{lemma}
For each $N\geq 1$, the power series coefficients satisfy the relations
\begin{align}
\sum_{k+j=N}&(k+1)\rho_{k+1}\big(\ga P_j-y_*^2(\om^2)_j-2y_*(\om^2)_{j-1}-(\om^2)_{j-2}\big)\nonumber\\
=&(\ga-1)(2-\ga)\big(y_*\rho_N+\rho_{N-1}\big)+(\ga-1)\big(y_*(\rho\om)_N+(\rho\om)_{N-1}\big)\label{eq:coefficients1}\\
&-\frac{4\pi}{4-3\ga}\big(y_*(\rho^2\om)_N+(\rho^2\om)_{N-1}\big)+2\big(y_*(\rho\om^2)_N+(\rho\om^2)_{N-1}\big),\nonumber\\
\sum_{k+j=N}&(k+1)\om_{k+1}\big(\ga P_j-y_*^2(\om^2)_j-2y_*(\om^2)_{j-1}-(\om^2)_{j-2}\big)\nonumber\\
=&\frac{4-3\ga}{y_*}\sum_{k+j=N}\frac{(-1)^k}{y_*^k}\big(\ga P_j-y_*^2(\om^2)_j-2y_*(\om^2)_{j-1}-(\om^2)_{j-2}\big)\label{eq:coefficients2}\\
&-\frac{3}{y_*}\sum_{k+j+l=N}\om_l\frac{(-1)^k}{y_*^k}\big(\ga P_j-y_*^2(\om^2)_j-2y_*(\om^2)_{j-1}-(\om^2)_{j-2}\big)\nonumber\\
&-(\ga-1)(2-\ga)\big(y_*\om_N+\om_{N-1}\big)-(\ga-1)\big(y_*(\om^2)_N+(\om^2)_{N-1}\big)\nonumber\\
&+\frac{4\pi}{4-3\ga}\big(y_*(\rho\om^2)_N+(\rho\om^2)_{N-1}\big)-2\big(y_*(\om^3)_N+(\om^3)_{N-1}\big).\nonumber
\end{align}
\end{lemma}
\begin{proof}
We begin the proof by noting the identities, for a general power series,
\beqas
y\sum_{N=0}^\infty b_N(y-y_*)^N=&\,\sum_{N=0}^\infty\big(y_*b_N+b_{N-1}\big)(y-y_*)^N,\\
y^2\sum_{N=0}^\infty b_N(y-y_*)^N=&\,\sum_{N=0}^\infty\big(y_*^2b_N+2y_*b_{N-1}+b_{N-2}\big)(y-y_*)^N,
\eeqas
where we define $b_{N}=0$ for any $N<0$.

\textit{Step 1: Derive \eqref{eq:coefficients1}.}\\
We begin by substituting the power series into \eqref{eq:EPSSSonic1}. The left hand side of this equation then becomes
\beqa\label{eq:powerseries1}
&\Big(\sum_{N=0}^\infty\big(\ga P_N-y_*^2(\om^2)_N-2y_*(\om^2)_{N-1}-(\om^2)_{N-2}\big)(y-y_*)^N\Big)\Big(\sum_{N=0}^\infty(N+1)\rho_{N+1}(y-y_*)^N\Big)\\
&=\sum_{N=0}^\infty\Big(\sum_{k+j=N}(k+1)\rho_{k+1}\big(\ga P_j-y_*^2(\om^2)_j-2y_*(\om^2)_{j-1}-(\om^2)_{j-2}\big)\Big)(y-y_*)^N.
\eeqa
The right hand side of \eqref{eq:EPSSSonic1} becomes
\beqa\label{eq:powerseries2}
&(\ga-1)\sum_{N=0}^\infty\Big((2-\ga)\big(y_*\rho_N+\rho_{N-1}\big)+y_*(\rho\om)_N+(\rho\om)_{N-1}\Big)(y-y_*)^N\\
&-2\sum_{N=0}^\infty\Big(\frac{2\pi}{4-3\ga}\big(y_*(\rho^2\om)_N+(\rho^2\om)_{N-1}\big)-\big(y_*(\rho\om^2)_N+(\rho\om^2)_{N-1}\big)\Big)(y-y_*)^N.
\eeqa
Equating the $N$-th order terms of \eqref{eq:powerseries1} and \eqref{eq:powerseries2}, we have the claimed relation \eqref{eq:coefficients1}, that is, for all $N\in\N\cup\{0\}$,
\beqa
\sum_{k+j=N}&(k+1)\rho_{k+1}\big(\ga P_j-y_*^2(\om^2)_j-2y_*(\om^2)_{j-1}-(\om^2)_{j-2}\big)\\
=&(\ga-1)(2-\ga)\big(y_*\rho_N+\rho_{N-1}\big)+(\ga-1)\big(y_*(\rho\om)_N+(\rho\om)_{N-1}\big)\\
&-\frac{4\pi}{4-3\ga}\big(y_*(\rho^2\om)_N+(\rho^2\om)_{N-1}\big)+2\big(y_*(\rho\om^2)_N+(\rho\om^2)_{N-1}\big).
\eeqa
\textit{Step 2: Derive \eqref{eq:coefficients2}.}\\
 To prove  from \eqref{eq:EPSSSonic2}, we begin by expanding the term $\frac{4-3\ga-3\om}{y}\big(\ga\rho^{\ga-1}-y^2\om^2\big)$ by noting first that
$$\frac{1}{y}=\frac{1}{y_*}\sum_{N=0}^\infty\frac{(-1)^N}{y_*^N}(y-y_*)^N.$$
Then we find
\beqas
\frac{4-3\ga-3\om}{y}&\big(\ga\rho^{\ga-1}-y^2\om^2\big)\\
=&\Big(\frac{1}{y_*}\sum_{N=0}^\infty\frac{(-1)^N}{y_*^N}(y-y_*)^N\Big)\Big(4-3\ga-3\sum_{N=0}^\infty\om_N(y-y_*)^N\Big)\\
&\times\Big(\sum_{N=0}^\infty\big(\ga P_N-y_*^2(\om^2)_N-2y_*(\om^2)_{N-1}-(\om^2)_{N-2}\big)(y-y_*)^N\Big)\\
=&\frac{4-3\ga}{y_*}\sum_{N=0}^\infty\Big(\sum_{k+j=N}\frac{(-1)^k}{y_*^k}\big(\ga P_j-y_*^2(\om^2)_j-2y_*(\om^2)_{j-1}-(\om^2)_{j-2}\big)\Big)(y-y_*)^N\\
&-\frac{3}{y_*}\sum_{N=0}^\infty\Big(\sum_{k+j+l=N}\om_l\frac{(-1)^k}{y_*^k}\big(\ga P_j-y_*^2(\om^2)_j-2y_*(\om^2)_{j-1}-(\om^2)_{j-2}\big)\Big)(y-y_*)^N.
\eeqas
Thus, expanding \eqref{eq:EPSSSonic2} and equating terms of the same order, we find 
\beqa
\sum_{k+j=N}&(k+1)\om_{k+1}\big(\ga P_j-y_*^2(\om^2)_j-2y_*(\om^2)_{j-1}-(\om^2)_{j-2}\big)\\
=&\frac{4-3\ga}{y_*}\sum_{k+j=N}\frac{(-1)^k}{y_*^k}\big(\ga P_j-y_*^2(\om^2)_j-2y_*(\om^2)_{j-1}-(\om^2)_{j-2}\big)\\
&-\frac{3}{y_*}\sum_{k+j+l=N}\om_l\frac{(-1)^k}{y_*^k}\big(\ga P_j-y_*^2(\om^2)_j-2y_*(\om^2)_{j-1}-(\om^2)_{j-2}\big)\\
&-(\ga-1)(2-\ga)\big(y_*\om_N+\om_{N-1}\big)-(\ga-1)\big(y_*(\om^2)_N+(\om^2)_{N-1}\big)\\
&+\frac{4\pi}{4-3\ga}\big(y_*(\rho\om^2)_N+(\rho\om^2)_{N-1}\big)-2\big(y_*(\om^3)_N+(\om^3)_{N-1}\big).
\eeqa
This is \eqref{eq:coefficients2}.
\end{proof}

Before studying the solvability of this system for the higher order coefficients, we first collect a pair of identities satisfied by the first order coefficients, $(\rho_1,\om_1$).

%%%%%%%%%%%%%%%%%%%%%%%%%%%
%%%%%%%%%%%%%%%%%%%%%%%%%%%

\begin{lemma}[First order Taylor coefficients]\label{lemma:rho1om1}
Let $\ga\in(1,\frac43)$ and consider the formal Taylor expansion \eqref{E:FORMAL}. Let
\be\label{E:RWDEF}
R:=\frac{y_*\rho_1}{\rho_0}\text{ and } \ W:=y_*\om_1.
\ee 
Then the pair $(R,W)$ satisfies the following system of algebraic equations:
\begin{align}
&(\ga-1)\om_0^2 R^2-2\om_0 W R+(\ga-1)(\om_0+2-\ga)R-2\om_0 W+(\ga-1)(2-\ga)\frac{W}{\om_0}=0,\label{eq:rho1quad}\\
&2\om_0 W^2-(\ga-1)\om_0^2RW+W\big(-2(4-3\ga-3\om_0)\om_0+(\ga-1)(2-\ga)\big)\nonumber\\
&+\big((5-3\ga)\om_0^2+(5-3\ga)(\ga-1)\om_0+(\ga-1)(2-\ga)\big)\om_0R-2(4-3\ga-3\om_0)\om_0^2=0,\label{eq:om1quad}
\end{align}
with the additional constraint 
\beq\label{eq:RWlinear}
R\om_0+W=4-3\ga-3\om_0.
\eeq
\end{lemma}

%%%%%%%%%%%%%%%%%%%%%%%%%%%

\begin{proof}
In the case $N=1$, we note that $P_1=(\ga-1)\rho_0^{\ga-2}\rho_1$, $(\om^2)_1=2\om_1\om_0$ etc to find from \eqref{eq:coefficients1}
\beqa\label{eq:rho1quad1}
&\ga(\ga-1)\rho_0^{\ga-2}\rho_1^2-2y_*^2\rho_1\om_1\om_0-2y_*\rho_1\om_0^2\\
&=(\ga-1)(2-\ga)(y_*\rho_1+\rho_0)+(\ga-1)(y_*\rho_0\om_1+y_*\rho_1\om_0+\rho_0\om_0)\\
&-\frac{4\pi}{4-3\ga}(y_*\rho_0^2\om_1+2y_*\rho_0\om_0\rho_1+\rho_0^2\om_0)+2(y_*\rho_1\om_0^2+2y_*\rho_0\om_0\om_1+\rho_0\om_0^2)\\
&=(\ga-1)y_*\rho_0\om_1-\frac{4\pi}{4-3\ga}(y_*\rho_0^2\om_1+y_*\rho_0\om_0\rho_1)+4y_*\rho_0\om_0\om_1,
\eeqa
where we have used \eqref{eq:sonicconstraint} twice.

From \eqref{eq:coefficients2} we get
\beqa\label{eq:om1quad1}
\om_1&\ga(\ga-1)\rho_0^{\ga-2}\rho_1-2y_*^2\om_1^2\om_0-2y_*\om_1\om_0^2\\
=&\,\frac{4-3\ga-3\om_0}{y_*}\big(\ga(\ga-1)\rho_0^{\ga-2}\rho_1-2y_*^2\om_1\om_0-2y_*\om_0^2\big)-(\ga-1)(2-\ga)(y_*\om_1+\om_0)\\
&-(\ga-1)(2y_*\om_0\om_1+\om_0^2)+\frac{4\pi}{4-3\ga}(y_*\om_0^2\rho_1+2y_*\rho_0\om_0\om_1+\rho_0\om_0^2)-2(3y_*\om_0^2\om_1+\om_0^3)\\
=&\,\frac{4-3\ga-3\om_0}{y_*}\big(\ga(\ga-1)\rho_0^{\ga-2}\rho_1-2y_*^2\om_1\om_0-2y_*\om_0^2\big)+\frac{4\pi}{4-3\ga}y_*\rho_1\om_0^2-2y_*\om_0^2\om_1\\
&+(\ga-1)(2-\ga)y_*\om_1,
\eeqa
where we have used \eqref{eq:sonicconstraint} again.
Rearranging \eqref{eq:rho1quad1}, we can use \eqref{eq:sonicconstraint} further to write
\beqas
0=&\,\ga(\ga-1)\rho_0^{\ga-2}\rho_1^2-2y_*^2\om_0\om_1\rho_1-2y_*\om_0^2\rho_1-(\ga-1)y_*\rho_0\om_1\\
&+\frac{4\pi}{4-3\ga}y_*\rho_0^2\om_1+\frac{4\pi}{4-3\ga}y_*\rho_0\om_0\rho_1-4y_*\rho_0\om_0\om_1\\
=&\,\ga(\ga-1)\rho_0^{\ga-2}\rho_1^2-2y_*^2\om_0\om_1\rho_1+(\ga-1)(\om_0+2-\ga)y_*\rho_1-2y_*\rho_0\om_0\om_1\\
&+(\ga-1)(2-\ga)y_*\frac{\rho_0}{\om_0}\om_1.
\eeqas
Thus, using also the sonic condition to replace $\ga\rho_0^{\ga-2}=\frac{y_*^2\om_0^2}{\rho_0}$ and dividing through by $\rho_0$, we recall the definitions of $R$, $W$ and arrive at
\beq
(\ga-1)\om_0^2 R^2-2\om_0RW+(\ga-1)(\om_0+2-\ga)R-2\om_0W+(\ga-1)(2-\ga)\frac{W}{\om_0}=0,
\eeq
that is, we have \eqref{eq:rho1quad}.

Working now from \eqref{eq:om1quad1}, we rearrange to find
\beqas
0=&\,2\om_0W^2-(\ga-1)\om_0^2RW+2\om_0^2W+(4-3\ga-3\om_0)\big((\ga-1)\om_0^2R-2\om_0W-2\om_0^2\big)\\
&+\frac{4\pi}{4-3\ga}\rho_0\om_0^2R-2\om_0^2W+(\ga-1)(2-\ga)W\\
=&\,2\om_0W^2+W\big(-(\ga-1)\om_0^2R-2(4-3\ga-3\om_0)\om_0+(\ga-1)(2-\ga)\big)\\
&+(4-3\ga-3\om_0)(\ga-1)\om_0^2R+\om_0\big(2\om_0^2+(\ga-1)\om_0+(\ga-1)(2-\ga)\big)R\\
&-2(4-3\ga-3\om_0)\om_0^2,
\eeqas
which is exactly \eqref{eq:om1quad}.

To show \eqref{eq:RWlinear}, we work from \eqref{eq:rhoom}. Multiplying the first equation by $\rho$, the second by $\om$ and summing, we obtain
\beqs
(\rho\om)'=\frac{4-3\ga-3\om}{y}\rho.
\eeqs
Substituting in the formal Taylor expansion and grouping the terms at order zero, we find
\beqs
\rho_1\om_0+\rho_0\om_1=\frac{4-3\ga-3\om_0}{y_*}\rho_0.
\eeqs
Multiplying through by $\frac{y_*}{\rho_0}$ and recalling~\eqref{E:RWDEF} we arrive at~\eqref{eq:RWlinear}.
\end{proof}

%%%%%%%%%%%%%%%%%%%%%%%%%%%
%%%%%%%%%%%%%%%%%%%%%%%%%%%

\begin{remark}
The coefficients of the quadratics in~\eqref{eq:rho1quad}--\eqref{eq:om1quad}
depend only on $\ga$ and on $\om_0$ (hence also on $y_*$). 
\end{remark}

%%%%%%%%%%%%%%%%%%%%%%%%%%%
%%%%%%%%%%%%%%%%%%%%%%%%%%%

Our next lemma establishes the key recursive relation that will allow us to compute the $N$-th order Taylor coefficients in terms of $(\rho_k,\om_k)$, $0\le k\le N-1$.

%%%%%%%%%%%%%%%%%%%%%%%%%%%
%%%%%%%%%%%%%%%%%%%%%%%%%%%

\begin{lemma}\label{Lem:N}
Let $N\geq2$ and define the matrix $\cala_N$ by
\beq\label{def:calA}
\cala_N= \begin{pmatrix} \cala^N_{11} & \cala^N_{12} \\ \cala^N_{21} & \cala^N_{22},
\end{pmatrix},
\eeq
where the matrix coefficients $\cala^N_{ij}$, $i,j\in\{1,2\}$ depend on $N$, $\gamma$, $\om_0$, $\rho_1$ and $\om_1$ and are given explicitly by~\eqref{E:A11}--\eqref{E:A22} below.
Then the coefficients $(\rho_N,\om_N)$ in the formal series expansion~\eqref{E:FORMAL} satisfy the algebraic equation
\be\label{E:RECURSIVE}
\cala_N\begin{pmatrix}
\rho_N\\
\om_N
\end{pmatrix}=\begin{pmatrix}
\mathcal{F}_N\\
\mathcal{G}_N
\end{pmatrix},
\ee
where the polynomials $\mathcal{F}_N$ and $\mathcal{G}_N$ are given by~\eqref{E:FNDEF} and~\eqref{E:GNDEF} below.
\end{lemma}

%%%%%%%%%%%%%%%%%%%%%%

\begin{proof}
We begin from \eqref{eq:coefficients1} and group the terms on the left hand side as follows.
\beqa\label{eq:rhoNeq1}
&N\rho_N\big(\ga P_1-2y_*^2\om_1\om_0-2y_*\om_0^2\big)+\rho_1\big(\ga P_N-2y_*^2\om_N\om_0\big)-\rho_1 y_*^2\sum_{\substack{j+k=N,\\j,k\neq N}}\om_j\om_k\\
&-\rho_1\big(2y_*(\om^2)_{N-1}-(\om^2)_{N-2}\big)+\sum_{\substack{k+j=N\\j\neq0,1,N}}(k+1)\rho_{k+1}\big(\ga P_j-y_*^2(\om^2)_j-2y_*(\om^2)_{j-1}-(\om^2)_{j-2}\big)\\
&=\rho_N\Big(N\big(\ga(\ga-1)\rho_0^{\ga-2}\rho_1-2y_*^2\om_1\om_0-2y_*\om_0^2\big)+\ga(\ga-1)\rho_0^{\ga-2}\rho_1\Big)-2y_*^2\rho_1\om_0\om_N\\
&\quad+\mathcal{F}_N^I,
\eeqa
where
\beqas
\mathcal{F}_N^I=&-\rho_1 y_*^2\sum_{\substack{j+k=N,\\j,k\neq N}}\om_j\om_k-\rho_1\big(2y_*(\om^2)_{N-1}-(\om^2)_{N-2}\big)\\
&+\sum_{\substack{k+j=N\\j\neq0,1,N}}(k+1)\rho_{k+1}\big(\ga P_j-y_*^2(\om^2)_j-2y_*(\om^2)_{j-1}-(\om^2)_{j-2}\big)\\
&+\ga\rho_1\sum_{\substack{(\la_1,\ldots,\la_N)\in M_N\\\la_N=0}}\frac{(\ga-1)\cdots(\ga-(\la_1+\cdots+\la_N))\rho_0^{\ga-1-(\la_1+\cdots+\la_N)}}{\la_1!\cdots \la_N!}\prod_{j=1}^N\rho_j^{\la_j},
\eeqas
and we have applied the definition of $P_j$ to isolate the term with a $\rho_N$ contribution as
\beqas
P_N=&\,(\ga-1)\rho_0^{\ga-2}\rho_N+\hspace{-2mm}\sum_{\substack{(\la_1,\ldots,\la_N)\in M_N\\\la_N=0}}\hspace{-1mm}\frac{(\ga-1)\cdots(\ga-(\la_1+\cdots+\la_N))\rho_0^{\ga-1-(\la_1+\cdots+\la_N)}}{\la_1!\cdots \la_N!}\prod_{j=1}^N\rho_j^{\la_j}
\eeqas
and also recalled 
$$P_1=(\ga-1)\rho_0^{\ga-2}\rho_1.$$
Studying the right hand side of \eqref{eq:coefficients1}, we find expand to isolate terms at order $N$ and then apply \eqref{eq:sonicconstraint} to eliminate terms with factors of $\frac{4\pi}{4-3\ga}$ as follows:
\beqa\label{eq:rhoNeq2}
&(\ga-1)(2-\ga)y_*\rho_N+(\ga-1)y_*\rho_N\om_0+(\ga-1)y_*\rho_0\om_N-2\frac{4\pi}{4-3\ga}y_*\rho_N\om_0\rho_0-\frac{4\pi}{4-3\ga}y_*\rho_0^2\om_N\\
&+2y_*\rho_N\om_0^2+4y_*\rho_0\om_0\om_N+(\ga-1)(2-\ga)\rho_{N-1}+(\ga-1)\Big(y_*\sum_{\substack{k+j=N\\k\neq0,N}}\rho_k\om_j+(\rho\om)_{N-1}\Big)\\
&-\frac{4\pi}{4-3\ga}\Big(y_*\sum_{\substack{k+j+l=N\\k,j,l\neq N}}\rho_k\rho_j\om_l+(\rho^2\om)_{N-1}\Big)+2\Big(y_*\sum_{\substack{k+j+l=N\\ k,j,l\neq N}}\rho_k\om_j\om_l+(\rho\om^2)_{N-1}\Big)\\
&=-\big(2\om_0^2+(\ga-1)\om_0+(\ga-1)(2-\ga)\big)y_*\rho_N+2y_*\rho_0\om_0\om_N-(\ga-1)(2-\ga)y_*\frac{\rho_0}{\om_0}\om_N+\mathcal{F}_N^{II},
\eeqa
where
\beqas
\mathcal{F}_N^{II}=&\,(\ga-1)(2-\ga)\rho_{N-1}+(\ga-1)\Big(y_*\sum_{\substack{k+j=N\\k\neq0,N}}\rho_k\om_j+(\rho\om)_{N-1}\Big)\\
&-\frac{4\pi}{4-3\ga}\Big(y_*\sum_{\substack{k+j+l=N\\k,j,l\neq N}}\rho_k\rho_j\om_l+(\rho^2\om)_{N-1}\Big)+2\Big(y_*\sum_{\substack{k+j+l=N\\ k,j,l\neq N}}\rho_k\om_j\om_l+(\rho\om^2)_{N-1}\Big),
\eeqas
where we have applied \eqref{eq:sonicconstraint}. Thus, as  \eqref{eq:rhoNeq1}  is equal to \eqref{eq:rhoNeq2}, we rearrange to arrive at
\begin{align}
&\rho_N\Big(N\big(\ga (\ga-1)\rho_0^{\ga-2}\rho_1-2y_*^2\om_1\om_0-2y_*\om_0^2\big)\notag\\
&\qquad+\ga(\ga-1)\rho_0^{\ga-2}\rho_1+y_*\big(2\om_0^2+(\ga-1)\om_0+(\ga-1)(2-\ga)\big)\Big) \notag \\
&+\om_N\Big(-2y_*^2\rho_1\om_0-2y_*\rho_0\om_0+(\ga-1)(2-\ga)y_*\frac{\rho_0}{\om_0}\Big) \notag \\
&=\mathcal{F}_N^{II}-\mathcal{F}_N^I \\
&=:\mathcal{F}_N(\rho_0,\ldots,\rho_{N-1},\om_0,\ldots,\om_{N-1}) \label{E:FNDEF}
\end{align}
Thus we have found
\beqa
&\rho_N\Big((N+1)(\ga-1)\frac{y_*^2\om_0^2}{\rho_0}\rho_1-2Ny_*^2\om_1\om_0-2(N-1)y_*\om_0^2+y_*\big((\ga-1)\om_0+(\ga-1)(2-\ga)\big)\Big)\\
&+\om_N\Big(-2y_*^2\rho_1\om_0-2y_*\rho_0\om_0+(\ga-1)(2-\ga)y_*\frac{\rho_0}{\om_0}\Big)\\
&=\mathcal{F}_N.
\eeqa
Considering now \eqref{eq:coefficients2}, we expand the left hand side as above as
\beqa\label{eq:omNeq1}
&N\om_N\big(\ga P_1-2y_*^2\om_1\om_0-2y_*\om_0^2\big)+\om_1\big(\ga P_N-2y_*^2\om_N\om_0\big)-\om_1 y_*^2\sum_{\substack{j+k=N,\\j,k\neq N}}\om_j\om_k\\
&-\om_1\big(2y_*(\om^2)_{N-1}-(\om^2)_{N-2}\big)+\sum_{\substack{k+j=N\\j\neq0,1,N}}(k+1)\om_{k+1}\big(\ga P_j-y_*^2(\om^2)_j-2y_*(\om^2)_{j-1}-(\om^2)_{j-2}\big)\\
&=\om_N\Big(N\big(\ga(\ga-1)\rho_0^{\ga-2}\rho_1-2y_*^2\om_1\om_0-2y_*\om_0^2\big)-2y_*^2\om_1\om_0\Big)+\rho_N\ga(\ga-1)\rho_0^{\ga-2}\om_1\\
&\quad+\mathcal{G}_N^I,
\eeqa
where
\beqa
&\mathcal{G}_N^I=-\om_1\big(2y_*(\om^2)_{N-1}-(\om^2)_{N-2}\big)-\om_1 y_*^2\sum_{\substack{j+k=N,\\j,k\neq N}}\om_j\om_k\\
&+\sum_{\substack{k+j=N\\j\neq0,1,N}}(k+1)\om_{k+1}\big(\ga P_j-y_*^2(\om^2)_j-2y_*(\om^2)_{j-1}-(\om^2)_{j-2}\big)\\
&+\ga\om_1\sum_{\substack{(\la_1,\ldots,\la_N)\in M_N\\\la_N=0}}\frac{(\ga-1)\cdots(\ga-(\la_1+\cdots+\la_N))\rho_0^{\ga-1-(\la_1+\cdots+\la_N)}}{\la_1!\cdots \la_N!}\prod_{j=1}^N\rho_j^{\la_j},
\eeqa
where we have applied the definition of $P_j$ to isolate the term with a $\rho_N$ contribution.

Working with the right hand side of \eqref{eq:coefficients2}, we have
\beqa\label{eq:omNeq2}
&\frac{4-3\ga-3\om_0}{y_*}\big(\ga P_N-y_*^2(\om^2)_N\big)-(\ga-1)(2-\ga)y_*\om_N-2(\ga-1)y_*\om_N\om_0\\
&+\frac{4\pi}{4-3\ga}y_*\rho_N\om_0^2+2\frac{4\pi}{4-3\ga}y_*\rho_0\om_0\om_N-6y_*\om_0^2\om_N\\
&+\widetilde{\mathcal{G}}_N^{II},
\eeqa
where
\beqa
&\widetilde{\mathcal{G}}_N^{II}=\frac{4-3\ga-3\om_0}{y_*}\big(-2y_*(\om^2)_{N-1}-(\om^2)_{N-2}\big)\\
&+\frac{4-3\ga}{y_*}\sum_{\substack{k+j=N\\j\neq N}}\frac{(-1)^k}{y_*^k}\big(\ga P_j-y_*^2(\om^2)_j-2y_*(\om^2)_{j-1}-(\om^2)_{j-2}\big)\\
&-\frac{3}{y_*}\sum_{\substack{k+j+l=N\\j\neq N}}\om_l\frac{(-1)^k}{y_*^k}\big(\ga P_j-y_*^2(\om^2)_j-2y_*(\om^2)_{j-1}-(\om^2)_{j-2}\big)-(\ga-1)(2-\ga)\om_{N-1}\\
&-(\ga-1)\Big(y_*\sum_{\substack{k+j=N\\k\neq0,N}}\om_k\om_j+(\om^2)_{N-1}\Big)+\frac{4\pi}{4-3\ga}\Big(y_*\sum_{\substack{k+j+l=N\\k,j,l\neq N}}(\rho_k\om_j\om_l)+(\rho\om^2)_{N-1}\Big)\\
&-2\Big(y_*\sum_{\substack{k+j+l=N\\k,j,l\neq N}}(\om_k\om_j\om_l)+(\om^3)_{N-1}\Big).
\eeqa
Grouping the terms on the first two lines here, we again expand $P_N$ to find the contribution
\beqa\label{eq:omNeq3}
&\frac{4-3\ga-3\om_0}{y_*}\big(\rho_N\ga(\ga-1)\rho_0^{\ga-2}-2y_*^2\om_0\om_N\big)+(\ga-1)(2-\ga)y_*\om_N-2y_*\om_0^2\om_N\\
&+\frac{\om_0}{\rho_0}\big(2\om_0^2+(\ga-1)\om_0+(\ga-1)(2-\ga)\big)y_*\rho_N-\frac{4-3\ga-3\om_0}{y_*}y_*^2\sum_{\substack{j+k=N\\j,k\neq N}}\om_j\om_k\\
&+\frac{4-3\ga-3\om_0}{y_*}\ga\sum_{\substack{(\la_1,\ldots,\la_N)\in M_N\\\la_N=0}}\frac{(\ga-1)\cdots(\ga-(\la_1+\cdots+\la_N))\rho_0^{\ga-1-(\la_1+\cdots+\la_N)}}{\la_1!\cdots \la_N!}\prod_{j=1}^N\rho_j^{\la_j},
\eeqa
where we have again applied \eqref{eq:sonicconstraint}. Setting 
\beqas
\mathcal{G}_N^{II}=&\,\widetilde{\mathcal{G}}_N^{II}-\frac{4-3\ga-3\om_0}{y_*}y_*^2\sum_{\substack{j+k=N\\j,k\neq N}}\om_j\om_k\\
&+\frac{4-3\ga-3\om_0}{y_*}\ga\hspace{-2mm}\sum_{\substack{(\la_1,\ldots,\la_N)\in M_N\\\la_N=0}}\hspace{-1mm}\frac{(\ga-1)\cdots(\ga-(\la_1+\cdots+\la_N))\rho_0^{\ga-1-(\la_1+\cdots+\la_N)}}{\la_1!\cdots \la_N!}\prod_{j=1}^N\rho_j^{\la_j},
\eeqas
we substitute \eqref{eq:omNeq3} back into \eqref{eq:omNeq2} and equate with \eqref{eq:omNeq1} to arrive at
\begin{align}
&\om_N\Big(N\big(\ga(\ga-1)\rho_0^{\ga-2}\rho_1-2y_*^2\om_1\om_0-2y_*\om_0^2\big)-2y_*^2\om_1\om_0\Big)+\rho_N\ga(\ga-1)\rho_0^{\ga-2}\om_1\notag\\
&-\frac{4-3\ga-3\om_0}{y_*}\big(\rho_N\ga(\ga-1)\rho_0^{\ga-2}-2y_*^2\om_0\om_N\big)-(\ga-1)(2-\ga)y_*\om_N+2y_*\om_0^2\om_N \notag \\
&-\frac{\om_0}{\rho_0}\big(2\om_0^2+(\ga-1)\om_0+(\ga-1)(2-\ga)\big)y_*\rho_N \notag \\
&=\mathcal{G}_N^{II}-\mathcal{G}_N^I \notag \\
&=:\mathcal{G}_N.\label{E:GNDEF}
\end{align}
Thus we have 
\beqa
&\om_N\Big(N(\ga-1)\frac{y_*^2\om_0^2}{\rho_0}\rho_1-2(N+1)y_*^2\om_1\om_0-2(N+2)y_*\om_0^2+2(4-3\ga)y_*\om_0-(\ga-1)(2-\ga)y_*\Big)\\
&+\rho_N\Big(\frac{y_*^2\om_0^2}{\rho_0}(\ga-1)\big(\om_1-\frac{4-3\ga-3\om_0}{y_*}\big)-y_*\frac{\om_0}{\rho_0}\big(2\om_0^2+(\ga-1)\om_0+(\ga-1)(2-\ga)\big)\Big)\\
&=\mathcal{G}_N.
\eeqa
So we have found the claimed identity with 
\begin{align}
&\mathcal{A}_{11}=y_*\Big((N+1)(\ga-1)\om_0^2\frac{y_*\rho_1}{\rho_0}-2Ny_*\om_1\om_0-2(N-1)\om_0^2+\big((\ga-1)\om_0+(\ga-1)(2-\ga)\big)\Big),\label{E:A11}\\
&\mathcal{A}_{12}=y_*\rho_0\Big(-2\frac{y_*\rho_1}{\rho_0}\om_0-2\om_0+\frac{(\ga-1)(2-\ga)}{\om_0}\Big),\\
&\cala_{21}=\frac{y_*}{\rho_0}\Big(\om_0^2(\ga-1)\big(y_*\om_1-(4-3\ga-3\om_0)\big)-\om_0\big(2\om_0^2+(\ga-1)\om_0+(\ga-1)(2-\ga)\big)\Big),\\
&\cala_{22}=y_*\Big(N(\ga-1)\om_0^2\frac{y_*\rho_1}{\rho_0}-2(N+1)y_*\om_1\om_0-2(N+2)\om_0^2+2(4-3\ga)\om_0-(\ga-1)(2-\ga)\Big).\label{E:A22}
\end{align}
\end{proof}

%%%%%%%%%%%%%%%%%%%%%%%%%%%
%%%%%%%%%%%%%%%%%%%%%%%%%%%

\begin{lemma}\label{lemma:detAN}
Consider the formal series expansion~\eqref{E:FORMAL} and recall the definitions
\be
R=\frac{y_*\rho_1}{\rho_0}\text{ and } W=y_*\om_1.
\ee 
Then the map $N\mapsto \det(\mathcal{A}_N)$ is a quadratic polynomial of the form
\be\label{E:DETAQUADRATIC}
\det\mathcal{A}_N=\sum_{j=0}^2A_jN^j,
\ee
where $A_0$, $A_1$, and $A_2$
are $(\gamma,\om_0,R,W)$-dependent functions given by the formulas:
\begin{align}
A_2=&\,\big(-2(3-\ga)\om_0^2+\om_0(\ga-1)(5\ga-9)-(\ga-1)(2-\ga)(\ga+1)\big)\om_0^2 R\notag\\
&+8\om_0^3 W+4\om_0^4,\label{E:ATWODEF}\\
A_1=&\,-\big(2(3-\ga)\om_0^2+2\om_0(\ga-1)+(\ga-1)(2-\ga)(\ga+1)\big)\om_0^2 R \notag \\
&+\big(8\om_0^2-4(4-3\ga)\om_0-2(\ga-1)\om_0\big)\om_0 W+\big(4\om_0^4-14\om_0^3+10\ga\om_0^3\big),\label{E:AONEDEF}\\
A_0=&\,2\big(\om_0^2(\ga-1)-\om_0(\ga-1)-\ga(\ga-1)(2-\ga)\big)\om_0^2R\notag \\
&+\big(-16\om_0^2+4(\ga-1)\om_0^2+4(4-3\ga)\om_0-2(\ga-1)\om_0-2(\ga-1)(2-\ga)(\ga+1)\big)\om_0W\notag \\
&+(6\ga-30)\om_0^4+(6\ga^2-44\ga+46)\om_0^3+(3\ga^3-12\ga^2+11\ga-2)\om_0^2\notag\\
&+(3\ga^4-10\ga^3+5\ga^2+10\ga-8)\om_0 \label{E:AZERODEF}
\end{align}
\end{lemma}

%%%%%%%%%%%%%%%%%%%%%%%%%%

\begin{proof}
We begin with the following identity. Multiplying \eqref{eq:rho1quad} by $(\ga-1)\om_0^2$ and \eqref{eq:om1quad} by $2\om_0$ and summing, we get
\beqa\label{eq:quadsubstitute}
&(\ga-1)^2\om_0^4R^2-4(\ga-1)\om_0^3RW+4\om_0^2W^2\\
&=-(\ga-1)^2\om_0^2(\om_0+2-\ga)R+2(\ga-1)\om_0^3W-(\ga-1)^2(2-\ga)W\om_0\\
&-2\om_0W\big(-2(4-3\ga-3\om_0)\om_0+(\ga-1)(2-\ga)\big)\\
&-2\big((5-3\ga)\om_0^2+(5-3\ga)(\ga-1)\om_0+(\ga-1)(2-\ga)\big)\om_0^2R+4(4-3\ga-3\om_0)\om_0^3\\
&=\om_0^2R\Big(-(\ga-1)^2(\om_0+2-\ga)-2\big((5-3\ga)\om_0^2+(5-3\ga)(\ga-1)\om_0+(\ga-1)(2-\ga)\big)\Big)\\
&+\om_0W\Big(2(\ga-1)\om_0^2-(\ga-1)^2(2-\ga)-2\big(-2(4-3\ga-3\om_0)\om_0+(\ga-1)(2-\ga)\big)\Big)\\
&+4(4-3\ga-3\om_0)\om_0^3.
\eeqa
Now we expand the determinant as
\begin{align}
y_*^{-2}&\det\mathcal{A}\notag\\
=&\,\Big((N+1)(\ga-1)\om_0^2R-2NW\om_0-2(N-1)\om_0^2+\big((\ga-1)\om_0+(\ga-1)(2-\ga)\big)\Big)\notag\\
&\:\times\Big(N(\ga-1)\om_0^2R-2(N+1)W\om_0-2(N+2)\om_0^2+2(4-3\ga)\om_0-(\ga-1)(2-\ga)\Big)\notag\\
&-\Big(-2R\om_0-2\om_0+\frac{(\ga-1)(2-\ga)}{\om_0}\Big)\notag\\
&\:\times\Big(\om_0^2(\ga-1)\big(W-(4-3\ga-3\om_0)\big)-\om_0\big(2\om_0^2+(\ga-1)\om_0+(\ga-1)(2-\ga)\big)\Big)\notag\\
=&\,N(N+1)(\ga-1)^2\om_0^4R^2-\big(2N^2+2(N+1)^2\big)(\ga-1)\om_0^3RW+4N(N+1)\om_0^2W^2\notag\\
&+R\Big((N+1)(\ga-1)\om_0^2\big(-2(N+2)\om_0^2+2(4-3\ga)\om_0-(\ga-1)(2-\ga)\big)\notag\\
&\hspace{9mm}+N(\ga-1)\om_0^2\big(-2(N-1)\om_0^2+\big((\ga-1)\om_0+(\ga-1)(2-\ga)\big)\big)\Big)\notag\\
&+W\Big(-2N\om_0\big(-2(N+2)\om_0^2+2(4-3\ga)\om_0-(\ga-1)(2-\ga)\big)\notag\\
&\hspace{9mm}-2(N+1)\om_0\big(-2(N-1)\om_0^2+\big((\ga-1)\om_0+(\ga-1)(2-\ga)\big)\big)\Big)\label{eq:detA1}\\
&+\Big(-2(N-1)\om_0^2+\big((\ga-1)\om_0+(\ga-1)(2-\ga)\big)\Big)\notag\\
&\quad\times\Big(-2(N+2)\om_0^2+2(4-3\ga)\om_0-(\ga-1)(2-\ga)\Big)\notag\\
&+2\om_0^3(\ga-1)RW\notag\\
&+R\Big(-2\om_0^3(\ga-1)(4-3\ga-3\om_0)-\om_0^2\big(2\om_0^2+(\ga-1)\om_0+(\ga-1)(2-\ga)\big)\Big)\notag\\
&+W\Big(-\om_0^2(\ga-1)\big(-2\om_0+(\ga-1)(2-\ga)\frac{1}{\om_0}\big)\Big)\notag\\
&+\big(2\om_0-(\ga-1)(2-\ga)\frac{1}{\om_0}\big)\notag\\
&\quad\times\Big(-\om_0^2(\ga-1)(4-3\ga-3\om_0)-\om_0\big(2\om_0^2+(\ga-1)\om_0+(\ga-1)(2-\ga)\big)\Big).\notag
\end{align}
We first re-group the quadratic terms in $(R,W)$ and substitute \eqref{eq:quadsubstitute} to get
\beqas
N(&N+1)(\ga-1)^2\om_0^4R^2-\big(2N^2+2(N+1)^2-2\big)(\ga-1)\om_0^3RW+4N(N+1)\om_0^2W^2\\
=&N(N+1)\big((\ga-1)^2\om_0^4R^2-4(\ga-1)\om_0^3RW+4\om_0^2W^2\big)\\
=&N(N+1)\\
&\times\Big(\om_0^2R\Big(-(\ga-1)^2(\om_0+2-\ga)-2\big((5-3\ga)\om_0^2+(5-3\ga)(\ga-1)\om_0+(\ga-1)(2-\ga)\big)\Big)\\
&\quad+\om_0W\Big(2(\ga-1)\om_0^2-(\ga-1)^2(2-\ga)-2\big(-2(4-3\ga-3\om_0)\om_0+(\ga-1)(2-\ga)\big)\Big)\\
&\quad+4(4-3\ga-3\om_0)\om_0^3\Big).
\eeqas
Substituting this into \eqref{eq:detA1}, we group the terms  by order in $N$ as
\beqa
\det&\mathcal{A}_N=y_*^2\Big(A_2N^2+A_1N+A_0\Big),
\eeqa
where
\beqas
A_2=&\,\om_0^2R\Big(-(\ga-1)^2(\om_0+2-\ga)-2\big((5-3\ga)\om_0^2+(5-3\ga)(\ga-1)\om_0+(\ga-1)(2-\ga)\big)\Big)\\
&+\om_0W\Big(2(\ga-1)\om_0^2-(\ga-1)^2(2-\ga)-2\big(-2(4-3\ga-3\om_0)\om_0+(\ga-1)(2-\ga)\big)\Big)\\
&+4(4-3\ga-3\om_0)\om_0^3-4(\ga-1)\om_0^4R+8\om_0^3W+4\om_0^4,
\eeqas
\beqas
A_1=&\,\om_0^2R\Big(-(\ga-1)^2(\om_0+2-\ga)-2\big((5-3\ga)\om_0^2+(5-3\ga)(\ga-1)\om_0+(\ga-1)(2-\ga)\big)\Big)\\
&+\om_0W\Big(2(\ga-1)\om_0^2-(\ga-1)^2(2-\ga)-2\big(-2(4-3\ga-3\om_0)\om_0+(\ga-1)(2-\ga)\big)\Big)\\
&+4(4-3\ga-3\om_0)\om_0^3-4(\ga-1)\om_0^4R+(\ga-1)\om_0^2R\big(2(4-3\ga)\om_0+(\ga-1)\om_0\big)\\
&-8\om_0^3W-2\om_0W\big(2(4-3\ga)\om_0+(\ga-1)\om_0\big)\\
&+2\om_0^2\Big(4\om_0^2-2(4-3\ga)\om_0+(\ga-1)(2-\ga)\Big)-2\om_0^2\Big(2\om_0^2+(\ga-1)\om_0+(\ga-1)(2-\ga)\Big)
\eeqas
and $A_0$ is the remainder. Simplifying these expressions and that for $A_0$ results in \eqref{E:ATWODEF}--\eqref{E:AZERODEF} to conclude the proof.
\end{proof}

%%%%%%%%%%%%%%%%%%%%%%%%%%
%%%%%%%%%%%%%%%%%%%%%%%%%

\subsection{Branch selection}

%%%%%%%%%%%%%%%%%%%%%%%%%%
%%%%%%%%%%%%%%%%%%%%%%%%%

To find solutions that are smooth through the sonic point, we must first calculate the first order Taylor coefficients $(\rho_1,\om_1)$ as functions of the 
parameters $\gamma$ and $y_\ast$. 

%In particular, the coefficients $(\rho_1,\om_1)$ satisfy the following lemma.
%%%%%%%%%%%%%%%%%%%%%%%%%
%%%%%%%%%%%%%%%%%%%%%%%%%

\begin{lemma}[The two solution branches]\label{L:BRANCHES}
Let $\gamma\in(1,\frac43)$ be given and let $y_\ast\in[y_f(\gamma),y_F(\gamma)]$. There exist exactly two pairs
$(R_i,W_i)$, $i=1,2$ solving
the system of algebraic equations~\eqref{eq:rho1quad}, \eqref{eq:om1quad}, \eqref{eq:RWlinear}. The functions $R_i$ are given by
\begin{align}
&R_1=\frac{(9-7\ga) \om_0^2 - 
   8 \om_0^3-\sqrt{\om_0^3s(\om_0)}}{2\om_0^3(\ga+1)},\label{E:R1} \\
&R_2=\frac{(9-7\ga) \om_0^2 - 
   8 \om_0^3+\sqrt{\om_0^3s(\om_0)}}{2\om_0^3(\ga+1)},\label{E:R2}
\end{align}
where 
\beqa
s(\om_0)=&\,-4(4-3\ga)(\ga+1)(\ga-1)(2-\ga)+\big((\ga-1)(\ga^2-5\ga+5)+\ga^2+6\ga-3\big)\om_0\\
&-8(3 \ga^2- 15 \ga +14)\om_0^2+8(5-3\ga)\om_0^3
\eeqa
is strictly positive for all $\om_0\in[\frac{4-3\ga}{3},2-\ga]$, $\ga\in(1,\frac43)$.\\
For any $i=1,2$, $W_i$ is determined by
$R_i$ through the formula
\beq\label{eq:W1(R1)}
W_i=4-3\ga-3\om_0-\om_0R_i.
\eeq
\end{lemma}

%%%%%%%%%%%%%%%%%%%%%%%%%

\begin{proof}
By rearranging \eqref{eq:rho1quad}, we see
\beq\label{eq:WRidentity}
\Big(2\om_0(R+1)-\frac{(\ga-1)(2-\ga)}{\om_0}\Big)W=(\ga-1)\om_0^2 R^2+(\ga-1)(\om_0+2-\ga)R.
\eeq
Rearranging \eqref{eq:RWlinear} to solve for $W$ as 
\beq
W=4-3\ga-3\om_0-\om_0R,
\eeq
we obtain the claimed relation \eqref{eq:W1(R1)}. We then substitute this into \eqref{eq:WRidentity} to obtain the following quadratic for $R$:
\beqa\label{eq:Rquadratic}
\Big(2\om_0(R+1)-\frac{(\ga-1)(2-\ga)}{\om_0}\Big)\Big(4-3\ga-3\om_0-\om_0R\Big)=(\ga-1)\om_0^2 R^2+(\ga-1)(\om_0+2-\ga)R
\eeqa
 with roots $R_1$, $R_2$ as claimed in~\eqref{E:R1}--\eqref{E:R2} from the quadratic formula.
We postpone the verification that $s(\om_0)>0$ to Appendix \ref{A:TOTHERIGHT}. 

One can check that equation \eqref{eq:om1quad} is also satisfied by these roots by similarly substituting \eqref{eq:W1(R1)} into \eqref{eq:om1quad} then simplifying. This again yields a quadratic in $R$ which, on inspection, turns out to be exactly \eqref{eq:Rquadratic} up to a factor of $\om_0$, and hence has the same roots.
\end{proof}

We will see in the following Subsection \ref{subsec:LPH} that the physically relevant solution branch is that given by $(R_1,W_1)$. We therefore collect some useful estimates on the coefficients derived from this branch.
\begin{prop}\label{prop:R1W1}
Let $\gamma\in(1,\frac43)$ be given and let $y_\ast\in[y_f(\ga), y_F(\ga)]$
and consider the branch $(R_1,W_1)$ defined in Lemma~\ref{L:BRANCHES}.
Then 
\beq
-\frac{4}{(4-3\ga)(2-\ga)}<R_1<-\frac{1}{2-\ga}.
\eeq
Moreover, if $\ga\in[\frac{10}{9},\frac{4}{3})$, then the upper bound on $R_1$ may be taken to satisfy
\beq
R_1\leq-\frac{2\ga}{(2-\ga)(\ga+1)},
\eeq 
where the  inequality is strict provided either $\ga>\frac{10}{9}$ or $y_*<y_F$.

Finally,
\beq
W_1>0\text{ for }y_*>y_f\text{ and }W_1|_{y_*=y_f}=0.
\eeq
\end{prop}

%%%%%%%%%%%%%%%%%%%%%%%%%%%

\begin{proof}
The proof relies in part on interval arithmetic and it is presented in detail in Appendix~\ref{A:TOTHERIGHT}.
\end{proof}

%%%%%%%%%%%%%%%%%%%%%%%%%
%%%%%%%%%%%%%%%%%%%%%%%%%

%%%%%%%%%%%%%%%%%%%%%
%%%%%%%%%%%%%%%%%%%%%
\begin{proposition}[Positivity of $\det\mathcal{A}_N$]\label{P:AN}
Let $\gamma\in(1,\frac43)$ be given and let $y_\ast\in[y_f(\ga), y_F(\ga)]$.
Let $A_0,A_1,A_2$ be functions of $\rho_0,\om_0, R,W$ given by~\eqref{E:ATWODEF}--\eqref{E:AZERODEF}
and assume that $R=R_1$ and $W=W_1$, where the branch $(R_1,W_1)$ is defined in Lemma~\ref{L:BRANCHES}. 
\begin{enumerate}
\item[{\em (i)}]
The following inequalities hold:
\begin{align}
A_2& >0,\label{ineq:A2} \\
4A_2+A_1&>0,\label{ineq:detA20} \\
4A_2+2A_1+A_0 &>0 \label{ineq:dNdetA0}
\end{align}
\item[{\em (ii)}]
There exist constants $c_1,c_2>0$, depending only on $\ga$,
so that
\be\label{E:ANBOUNDS}
c_1N^2\leq\det\mathcal{A}_N\leq c_2N^2, \ \ N\ge 2.
\ee
In particular, the matrix $\mathcal A_N$ is invertible for all $N\ge2$ and the formal Taylor coefficients $(\rho_N,\om_N)$ are well-defined
through the formula
\be\label{E:RHONOMNDEF}
\begin{pmatrix}
\rho_N\\
\om_N
\end{pmatrix}=\cala_N^{-1}\begin{pmatrix}
\mathcal{F}_N\\
\mathcal{G}_N
\end{pmatrix}, \ \ N\ge 2,
\ee
where the source terms $\mathcal F_N$, $\mathcal G_N$ are defined in Lemma~\ref{Lem:N}. 
\item[(iii)]
There exists a constant $\beta_0=\beta_0(y_\ast, \gamma)>0$ such that  
\begin{align}
| \rho_N|\leq  \frac{\beta_0}{N} \left(  |\mathcal F_N| + \frac{1}{N} |\mathcal G_N| \right)\label{recR1} \\
|\om_N| \leq \frac{\beta_0}{N} \left(  |\mathcal G_N| + \frac{1}{N} |\mathcal F_N| \right)\label{recW1}.
\end{align} 
\end{enumerate}
\end{proposition}

%%%%%%%%%%%%%%%%%%%%%%%%%

\begin{proof}
{\em Proof of part {\em (i)}.}
The proof of~\eqref{ineq:A2}--\eqref{ineq:dNdetA0} relies on interval arithmetic and it is presented in detail in Appendix~\ref{A:IADETERMINANTS}.

\noindent
{\em Proof of part {\em (ii)}.}
Since $\det\mathcal{A}_2=y_*^2\big(4A_2+2A_1+A_0\big)>0$ by~\eqref{ineq:dNdetA0} and, for $N\ge 2$,
$\frac{\dif}{\dif N}\det\mathcal{A}_N=y_*^2\big(2NA_2+A_1\big)\geq y_*^2\big(4A_2+A_1\big)$, it follows from~\eqref{ineq:detA20}--\eqref{ineq:dNdetA0}
that 
\begin{align}
&\det\mathcal{A}_2=y_*^2\big(4A_2+2A_1+A_0\big)>0,\\
&\frac{\dif}{\dif N}\det\mathcal{A}_N=y_*^2\big(2NA_2+A_1\big)\geq y_*^2\big(4A_2+A_1\big)>0.\label{ineq:dNdetA}
\end{align}
These estimates then easily imply~\eqref{E:ANBOUNDS}.
Claim~\eqref{E:RHONOMNDEF} is an obvious consequence of the invertibility of $\cala_N$ and Lemma~\ref{Lem:N}. 

\noindent
{\em Proof of part {\em (iii)}.}
From~\eqref{E:RHONOMNDEF} it follows that
\begin{align}
\rho_N = \frac{\mathcal{A}_{22}}{\det \mathcal{A}_N} \mathcal F_N - \frac{\mathcal{A}_{12}}{\det \mathcal{A}_N } \mathcal G_N, \label{recR}  \\
\om_N = \frac{ \mathcal{A}_{11}}{\det \mathcal{A}_N} \mathcal G_N- \frac{\mathcal{A}_{21}}{\det \mathcal{A}_N} \mathcal F_N, \label{recW}
\end{align} 
and thus~\eqref{recR1}--\eqref{recW1} follow directly from~\eqref{E:ANBOUNDS} and~\eqref{E:A11}--\eqref{E:A22}.

\end{proof}

%%%%%%%%%%%%%%%%%%%%%%%%%
%%%%%%%%%%%%%%%%%%%%%%%%%

\subsection{Larson-Penston-Hunter- (LPH-) type solutions}\label{subsec:LPH}

%%%%%%%%%%%%%%%%%%%%%%%%%
%%%%%%%%%%%%%%%%%%%%%%%%%
In order to distinguish the relevant solution branch for the first order Taylor coefficients, we compare directly to the situation in the case $\ga=1$. 
\begin{lemma}
Let $\ga\in(1,\frac43)$, $y_*\in[y_f(\ga),y_F(\ga)]$, and consider the functions $R_i$, $i=1,2$ as in Lemma \ref{L:BRANCHES} as functions of both $\om_0(y_*)$ and $\ga$. As $\ga\to1$, these coefficients satisfy the limits
\beqa
R_1(\om_0)\to&\, \frac{1-4\om_0-|1-2\om_0|}{2\om_0},\\
R_2(\om_0)\to&\, \frac{1-4\om_0+|1-2\om_0|}{2\om_0}.
\eeqa
\end{lemma}

\begin{proof}
The identities for the limit as $\ga\to1^+$ for $R_i$ follow directly from the identities \eqref{E:R1}--\eqref{E:R2}.
\end{proof}

Thus, to maintain compatibility with the LP solution in the case $\ga=1$, we note that, in that case, the sonic point lies in the interval $(2,3)$ with $\frac{\rho'(y_*)y_*}{\rho(y_*)}=-1$ (compare \cite{Guo20}), and hence the LP-type branch, for $\ga>1$, should be chosen to be the $1$-branch. In this case, we find that the limit of $W_1(\om_0)$ as $\ga\to1$ is $1-2\om_0$, again in compatibility with the $\ga=1$ case resolved  in \cite{Guo20}.

This motivates the following definition.

\begin{definition}[Larson-Penston-Hunter (LPH) type solutions]\label{def:LPH}
Let $\gamma\in(1,\frac43)$ be given and let $y_\ast\in[y_f(\ga),y_F(\ga)]$.
We say that a sequence $(\rho_N,\om_N)$, $N\in\mathbb N$ associated with a 
formal
power series expansion
\begin{align}\label{E:FORMAL2}
\rho(y)=\sum_{N=0}^\infty\rho_N(y-y_*)^N,\quad \om(y)=\sum_{N=0}^\infty\om_N(y-y_*)^N,
\end{align}
is of {\em Larson-Penston-Hunter (LPH) type}
if the following properties are satisfied
\begin{enumerate}
\item[{\em (i)}] 
\beq
G(y_*,\rho_0,\om_0)=0,\quad h(\rho_0,\om_0)=0.
\eeq
\item[{\em (ii)}] 
\be\label{E:RWCOND}
\rho_1:=\frac{\rho_0R_1}{y_\ast}, \ \ \ \  \om_1:=\frac{W_1}{y_*},
\ee 
where the pair $(R_1,W_1)$ corresponds to the branch defined by~\eqref{E:R1} and~\eqref{eq:W1(R1)} from Lemma~\ref{L:BRANCHES}.
\item[{\em (iii)}] 
For any $N\ge2$, the coefficients $(\rho_N,\om_N)$ satisfy the recursive relation~\eqref{E:RHONOMNDEF}.
\end{enumerate}
If the series~\eqref{E:FORMAL2} converge, we say that the functions $\rho$ and $\om$ are of LPH-type.
\end{definition}

%%%%%%%%%%%%%%%%%%%%%%%%%
%%%%%%%%%%%%%%%%%%%%%%%%%

\begin{remark}
As shown in Proposition~\ref{P:AN}, the matrix $\cala_N$ defined in Lemma~\ref{Lem:N} is indeed invertible for all $N\ge2$ and therefore 
for any LPH-type sequence the
coefficients $(\rho_N,\om_N)$, $N\ge2$ are therefore uniquely determined as functions of $\rho_0,\om_0,\rho_1,\om_1$. 
\end{remark}

%%%%%%%%%%%%%%%%%%%%%%%%%%
%%%%%%%%%%%%%%%%%%%%%%%%%

\subsection{The induction argument and the series convergence}

%%%%%%%%%%%%%%%%%%%%%%%%%%%
%%%%%%%%%%%%%%%%%%%%%%%%%%%

In order to prove the convergence of the formal power series~\eqref{E:FORMAL} we prove the crucial lemma, which establishes favourable growth bounds for the coefficients $(\rho_N,\om_N)$. 
The proof is based on involved combinatorial arguments that are presented in Appendix~\ref{A:COMB}, culminating in Lemma~\ref{lem:FGbound}.

%%%%%%%%%%%%%%%%%%%%%%%%%%
%%%%%%%%%%%%%%%%%%%%%%%%%

\begin{lemma}\label{L:INDUCTION}
Let $\gamma\in(1,\frac43)$ and $\alpha\in(1,2)$ be given. Let $(\rho_N,\om_N)$, $N\in\mathbb N$ be the coefficients in the formal Taylor expansion
of $\rho,\om$ about $y=y_\ast$ given by Proposition~\ref{P:AN}. Then there exists a constant $C>1$ such that for any $y_\ast\in[y_f(\gamma),y_F(\gamma)]$ the bounds
\begin{align}
\lv \rho_N\rv \le \frac{C^{N-\alpha}}{N^3}, \label{E:ASS1}\\
\lv \om_N\rv \le \frac{C^{N-\alpha}}{N^3}, \label{E:ASS2}
\end{align}
hold for all $N\ge2$.
\end{lemma}

%%%%%%%%%%%%%%%%%%%%%%%%%%%%

\begin{proof}
We use mathematical induction to prove the lemma. When $N=2$ clearly there exists a constant $\bar C=\bar C(y_\ast,\alpha)>0$ such that the claimed bounds hold true as the recursive relation~\eqref{E:RHONOMNDEF}  defining $(\rho_2,\om_2)$ involves only products of continuous functions composed with $(y_*,\rho_0,\om_0,\rho_1,\om_1)$, all of which are bounded. 

Suppose now that for some $N\geq3$, \eqref{E:ASS1}--\eqref{E:ASS2} hold for all $2\le m\le N-1$. This implies that the assumptions~\eqref{assumptionRm}--\eqref{assumptionWm} hold true and thus by Lemma~\ref{lem:FGbound} we conclude that~\eqref{FNbound}--\eqref{GNbound} hold.
Therefore, from Proposition~\ref{P:AN} and~\eqref{FNbound}--\eqref{GNbound} we obtain
\begin{align}
\lv \rho_N\rv \le \frac{c \beta_0\beta C^{N-\alpha}}{N^3} \left(\frac1{C^{\alpha-1}}+\frac1{C^{2-\alpha}}+\frac{1}{CN}\right),
\end{align}
for some universal constant $c>0$.
Similarly, 
\begin{align}
\lv \om_N\rv \le \frac{c \beta_0\beta C^{N-\alpha}}{N^3} \left(\frac1{C^{\alpha-1}}+\frac1{C^{2-\alpha}}+\frac{1}{CN}\right).
\end{align}
It is now clear that we can choose $C=C(\gamma,y_\ast)$ sufficiently large so that the claimed estimates~\eqref{E:ASS1}--\eqref{E:ASS2} hold at $N$. Since $y_\ast$ ranges over a compact interval and all the constants involved vary continuously in $y_\ast$, we may choose the constant $C$ above uniformly in $y_\ast\in[y_f(\gamma),y_F(\gamma)]$. We conclude by induction on $N$.
\end{proof}

%%%%%%%%%%%%%%%%%%%%%%%%%%%%
%%%%%%%%%%%%%%%%%%%%%%%%%%%%

\begin{theorem}\label{thm:Taylor}
Let $\gamma\in(1,\frac43)$ be given and for any $y_\ast\in[y_f(\ga),y_F(\ga)]$ consider the sequence $(\rho_N,\om_N)$, $N\in\mathbb N$ which corresponds to the 
formal Taylor coefficients associated with an LPH-type solution.
Then there exists a $\nu>0$ independent of $y_\ast$ such that the series
\[
\rho(y;y_\ast): = \sum_{N=0}^\infty \rho_N (y-y_\ast)^N, \ \
\om(y;y_\ast): = \sum_{N=0}^\infty \om_N (y-y_\ast)^N
\]
converge absolutely and the functions 
$(\rho(\cdot;y_*),\om(\cdot;y_*))$ are real analytic solutions to \eqref{eq:EPSS} on the interval $(y_*-\nu,y_*+\nu)$. Moreover, $y_*$ is a sonic point for the flow, there are no other sonic points on the interval, and the solutions are continuous with respect to $y_*\in[y_f,y_F]$.
\end{theorem}

%%%%%%%%%%%%%%%%%%%%%%%%%%%%

\begin{proof}
Let $\alpha\in(1,2)$ be fixed. By Lemma~\ref{L:INDUCTION} there exists a constant $C=C(\gamma,\alpha)$ such that 
\begin{align}
\lv \sum_{N=2}^\infty \rho_N (y-y_\ast)^N\rv \le  \sum_{N=2}^\infty |\rho_N| |y-y_\ast|^N
\le \sum_{N=2}^\infty \frac{|C(y-y_\ast)|^N}{C^\alpha N^3} <\infty,
\end{align}
and therefore the formal power series $\sum_{N=0}^\infty \rho_N (y-y_\ast)^N$ converges absolutely as long as $|y-y_\ast|<\nu$, for any $0<\nu<\frac1C$. Similarly, the power series $\sum_{N=0}^\infty \om_N (y-y_\ast)^N$ also converges absolutely as long as $|y-y_\ast|<\nu$. The real analyticity is clear. Recalling~\eqref{def:G} we have
\begin{align}
G(y;\rho,\om)& =\ga\rho(y)^{\ga-1}-y^2\om(y)^2\notag \\
&  = \ga \left(\sum_{N=0}^\infty \rho_N (y-y_\ast)^N\right)^{\ga-1} - (y_\ast+(y-y_\ast))^2\left(\sum_{N=0}^\infty \om_N (y-y_\ast)^N\right)^2 \notag\\
& = \left(\ga(\ga-1)\rho_0^{\ga-2}\rho_1 -2y_\ast\om_0(1+y_\ast\om_1)\right)(y-y_\ast) + O(|y-y_\ast|^2) \notag \\
& = \left((\ga-1) y_\ast^2\om_0^2\frac{\rho_1}{\rho_0} -2y_\ast\om_0(1+y_\ast\om_1)\right)(y-y_\ast) + O(|y-y_\ast|^2) \notag \\
& = y_\ast\om_0 \left((\ga-1)y_\ast\om_0\frac{\rho_1}{\rho_0} - 2 - 2y_\ast\om_1 \right)(y-y_\ast) + O(|y-y_\ast|^2) \notag\\
& =  y_\ast\om_0 \left((\ga-1)\om_0 R_1 - 2 - 2W_1 \right)(y-y_\ast) + O(|y-y_\ast|^2),
\end{align}
where we have used the sonic condition $G(y_\ast,\rho,\om)=0$ in the second and the third line, and the notation $(R_1,W_1)$, see Lemma~\ref{L:BRANCHES}.
Now observe that $\om_0>0$ by Lemma~\ref{lemma:rho0om0}, and $R_1<0$, $W_1\geq0$ by Proposition~\ref{prop:R1W1}. Therefore 
$(\ga-1)\om_0 R_1 - 2 - 2W_1 <0$ and therefore, upon possibly choosing a smaller $\nu>0$, it follows that $G(y;\rho,\om)$ is strictly positive for $y\in(y_\ast-\nu,y_\ast)$ and strictly negative
for $y\in(y_\ast,y_\ast+\nu)$. In particular, the right-hand side of~\eqref{eq:EPSS} is well-defined and it is straightforward to verify that $(\rho,\om)$ is a solution to~\eqref{eq:EPSS}.
\end{proof}

In the final proposition of this section, we collect some remaining facts concerning the LPH Taylor expansions.
\begin{prop}\label{prop:far-field}
Let $\ga\in(1,\frac43)$. For $y_*\in[y_f(\ga),y_F(\ga)]$, the following properties hold at the sonic point:
\begin{itemize}
\item[(i)] The branch $(R_1,W_1)$ that we take for the re-scaled first derivatives at the sonic point $y_*$ satisfies $(R_1,W_1)(y_f)=(-\frac{2}{2-\ga},0)$, $W_1(y_*)>0$ for all $y_*\in(y_f,y_F]$.
\item[(ii)] The local LPH-type solution obtained by Theorem \ref{thm:Taylor} with $y_*=y_f$ is exactly the far-field solution $$(\rho(y;y_f),\om(y;y_f))\equiv(\rho_f(y),\om_f(y))=(ky^{-\frac{2}{2-\ga}},2-\ga).$$
\item[(iii)] The local LPH-type solution obtained by Theorem \ref{thm:Taylor} with $y_*=y_F$ is not the Friedman solution: $(\rho(\cdot;y_F),\om(\cdot;y_F))\neq(\rho_F,\om_F)$.
\end{itemize}
\end{prop}

\begin{proof}
(i) By Lemma \ref{lemma:rho0om0}, we know  $\om_0(y_*)\in[\frac{4-3\ga}{3},2-\ga]$. Then, by Proposition \ref{prop:R1W1}, we have $W_1(\om_0)\geq 0$ for all $\om_0\in[\frac{4-3\ga}{3},2-\ga]$ with equality if and only if $\om_0=2-\ga$. In addition, $R_1(2-\ga)=-\frac{2}{2-\ga}$ by direct computation from \eqref{E:R1}.\\
(ii) To see that the solution obtained at $y_f$ is the far-field solution, it is enough to note that $\rho_0$ is uniquely determined by $y_*$ also through the relation $\rho_0=f_1(\om_0(y_*))$, and hence we have that $\om_0(y_f)=2-\ga=\om_f(y_f)$ and $\rho_0(y_f)=\rho_f({y}_f)$. Thus the solution locally around the sonic point is determined entirely by the choice of the branch $(R_1,W_1)$ for the first order terms in the Taylor expansion. As $W_1=0$, $R_1=-\frac{2}{2-\ga}$ are equal to the corresponding values for the far-field solution, the Taylor expansions of the solution derived from the choice $y_*=y_f$ and the far-field solution are equal. Thus the solutions are locally equal (as both are analytic functions) and, by uniqueness theory for the ODE system away from the sonic point and $y=0$, therefore globally equal on all of $(0,\infty)$.\\
(iii) As in item (i), we know that $W_1(y_F)>0$ by Proposition \ref{prop:R1W1}, hence $\om_1(y_F)>0$ also. As the Friedman solution satisfies $\om_F'(y)\equiv0$ for all $y$, the two solutions are not equal.
\end{proof}

%%%%%%%%%%%%%%%%%%%%%%%%%%%
%%%%%%%%%%%%%%%%%%%%%%%%%%%

%%%%%%%%%%%%%%%%%%%%%%%%%%%
%%%%%%%%%%%%%%%%%%%%%%%%%%%

\section{Solution to the right of the sonic point}\label{S:RIGHT}

%%%%%%%%%%%%%%%%%%%%%%%%%%%
%%%%%%%%%%%%%%%%%%%%%%%%%%%
Now that we have established the existence of a local solution to \eqref{eq:EPSS} around each choice of sonic point $y_*\in[y_f,y_F]$, we show in this section that the local solution can be extended to the right on the whole interval $(y_*,\infty)$ while remaining strictly supersonic and satisfying suitable asymptotics. For $y_*=y_f$, we know from Proposition \ref{prop:far-field} the obtained solution is simply the far-field solution $(\rho_f,\om_f)$ which is globally defined and supersonic for all $y>y_f$. We will therefore restrict in the sequel to the case $y_*\in(y_f,y_F]$.

The strategy of the section is to identify certain inequalities that propagate along the flow to the right and provide qualitative control on the solutions. Because the system \eqref{eq:EPSS} is non-autonomous, we cannot argue simply from a fixed phase plane analysis, but instead we make use of dynamical arguments that prevent the crossing of certain critical values by particular quantities fundamental to the flow. After a number of technical lemmas, we prove the key continuation estimates in Proposition \ref{prop:continuationright}. We then demonstrate that the flow remains strictly supersonic to the right and so deduce that it exists globally on $(y_*,\infty)$ in Lemma \ref{lemma:supersonic}. Finally, in Lemmas \ref{lemma:asymptotics} and \ref{lemma:rightmonotonicity}, we study the asymptotics and monotonicity of the solution.

For each $y_*\in[y_f(\ga),y_F(\ga)]$, let $(\rho,\om)=(\rho(\cdot;y_*),\om(\cdot;y_*))$ be the local LPH-type solution of Theorem \ref{thm:Taylor}. We define the maximal extension time to the right as 
\beq
y_{\max}(y_*):=\sup\{y>y_*\,|\,(\rho,\om) \text{ extends as a strictly supersonic solution of \eqref{eq:EPSS} on }(y_*,y)\},
\eeq
where we recall the definition of supersonicity from Definition \ref{def:subsonicitysupersonicity}.

The first lemma in this section states and proves the basic estimates that we will use to propagate the solution and verifies that they hold in a small neighbourhood of the sonic point.
\begin{lemma}[Initial inequalities]\label{lemma:rightinitial}
Let $\ga\in(1,\frac43)$, $y_*\in(y_f,y_F]$ (recall that we suppress the dependence of $y_f$, $y_F$ on $\ga$ where clear) and let $(\rho,\om)$ be the unique LPH-type solution to \eqref{eq:EPSS} to the right of $y_*$ given by Theorem \ref{thm:Taylor}. Then there exists $\bar\nu>0$ (depending on $y_*$) such that for $y\in(y_*,y_*+\bar\nu)$, the strictly supersonic flow satisfies also the inequalities
\beq\label{ineq:rightinitial}
\frac{4-3\ga}{3}<\om(y)<2-\ga,\quad \frac{4\pi y^2\rho\om}{4-3\ga}-\frac{2}{2-\ga}\ga\rho^{\ga-1}>0,\quad -\frac{4}{(4-3\ga)(2-\ga)}<\frac{\rho'y}{\rho}<-\frac{1}{2-\ga}.
\eeq
\end{lemma}

\begin{proof}
By Theorem \ref{thm:Taylor}, the existence of $\nu>0$ such that the solution remains supersonic on $(y_*,y_*+\nu)$ is clear. Moreover, by Lemma \ref{lemma:rho0om0}, we know that if $y_*\in(y_f,y_F)$, we have $\frac{4-3\ga}{3}<\om(y_*)<2-\ga$, and hence, as $\om$ is continuous on $[y_*,y_*+\nu]$, there exists $\bar\nu\in(0,\nu)$ such that 
$$\frac{4-3\ga}{3}<\om(y)<2-\ga\text{ for }y\in(y_*,y_*+\bar\nu).$$
On the other hand, if $y_*=y_F$, then $\om(y_*)=\frac{4-3\ga}{3}$ and, by Proposition \ref{prop:R1W1}, $\om'(y_*)>0$, hence by possibly shrinking $\bar\nu$, we again have the claimed estimate.

Similarly, by Proposition \ref{prop:R1W1} and smoothness of the flow, by possibly shrinking $\bar\nu$, we retain the final inequality of \eqref{ineq:rightinitial}
$$-\frac{4}{(4-3\ga)(2-\ga)}<\frac{\rho'y}{\rho}<-\frac{1}{2-\ga}.$$
Finally, we check the second condition in \eqref{ineq:rightinitial} through the following observation:
\beqa\label{ineq:f(y*)}
\frac{4\pi y_*^2\rho_0\om_0}{4-3\ga}-\frac{2}{2-\ga}\ga\rho_0^{\ga-1}=&\,\frac{4\pi y_*^2\rho_0\om_0}{4-3\ga}-\frac{2}{2-\ga}y_*^2\om_0^2\\
=&\,y_*^2\Big(\frac{2-2\ga}{2-\ga}\om_0^2+(\ga-1)\om_0+(\ga-1)(2-\ga)\Big)>0
\eeqa
for $\om_0\in[\frac{4-3\ga}{3},2-\ga)$, where we have used $\rho_0=f_1(\om_0)$ (compare \eqref{def:f1}) in the second line to eliminate $\rho_0$, and observe that the quadratic function of $\om_0$ in the parentheses factorises as
$$\frac{2-2\ga}{2-\ga}\om_0^2+(\ga-1)\om_0+(\ga-1)(2-\ga)=-(\ga-1)(2-\ga-\om)\big(\frac{2\om}{2-\ga}+1\big)$$
to deduce the sign. By again exploiting continuity of the flow and possibly shrinking $\bar\nu$, we conclude.
\end{proof}
We will also need the following two  lemmas.

%%%%%%%%%%%%%%%%%%%%%
%%%%%%%%%%%%%%%%%%%%%
\begin{lemma}
Let $\ga\in(1,\frac43)$. For any $C^1$ solution $(\rho,\om)$ of \eqref{eq:EPSS}, the following identities hold along the flow at any point $y>0$ such that $y$ is not a sonic point:
\begin{align}
\om'&=\frac{4-3\ga-3\om}{y}-\frac{\om}{\rho}\rho', \label{eq:invariance1} \\
\big(\rho\om y^{\frac{2}{2-\ga}}\big)'=&\,y^{\frac{2}{2-\ga}}\frac{\rho}{y}(4-3\ga)\big(1-\frac{\om}{2-\ga}\big),  \label{eq:weightedrho'}\\
\big(\om y^{\frac{2}{2-\ga}}\big)'=&\,y^{\frac{2}{2-\ga}}\Big(\frac{(4-3\ga)(1-\frac{\om}{2-\ga})}{y}-\frac{\om}{\rho}\rho'\Big),   \label{eq:weightedomega'}\\
\big(\om^2 y^{\frac{2}{2-\ga}}\big)'=&\,y^{\frac{2}{2-\ga}}\Big(2\om\frac{(4-3\ga)(1-\frac{\om}{2-\ga})}{y}-\frac{2}{2-\ga}\frac{\om^2}{y}-2\frac{\om^2}{\rho}\rho'\Big)  \label{eq:weightedomega2'}.
\end{align}
\end{lemma}

\begin{proof}
Identity~\eqref{eq:invariance1} is a trivial consequence of~\eqref{eq:rhoom}.
Identity \eqref{eq:weightedrho'} follows from using \eqref{eq:invariance1} in the following:
\beqas
\big(\rho\om y^{\frac{2}{2-\ga}}\big)'=&\,(\rho\om)'y^{\frac{2}{2-\ga}}+\frac{2}{2-\ga}y^{\frac{2}{2-\ga}}\frac{\rho}{y}\om\\
=&\,y^{\frac{2}{2-\ga}}\frac{4-3\ga-3\om}{y}\rho+\frac{2}{2-\ga}y^{\frac{2}{2-\ga}}\frac{\rho}{y}\om
\eeqas
and grouping the $\om$ terms.

To obtain \eqref{eq:weightedomega'}, we again apply \eqref{eq:invariance1} to find
\beqas
\big(\om y^{\frac{2}{2-\ga}}\big)'=&\,y^{\frac{2}{2-\ga}}\big(\frac{4-3\ga-3\om}{y}-\frac{\om}{\rho}\rho'+\frac{2}{2-\ga}\frac{\om}{y}\big)
\eeqas
and group terms. The proof of \eqref{eq:weightedomega2'} is similar.
\end{proof}

\begin{lemma}\label{L:PMQM}
Let $\ga\in(1,\frac43)$, let $(\rho,\om)$ be a $C^1$ solution  of \eqref{eq:EPSS} and suppose that $y>0$ is not a sonic point of the flow.
\begin{itemize}
\item[(i)] For any $m\geq 0$, the derivative of $\rho$ may be expressed through the following relation: 
\beq\label{eq:rho'identitym}
\frac{\rho'y}{\rho}+\frac{m}{2-\ga}=y^{-\frac{2(\ga-1)}{2-\ga}}\frac{P_m(y,\rho,\om)}{y^2\om^2-\ga\rho^{\ga-1}}, 
\eeq
where
\beqa\label{def:Pm}
P_m(y,\rho,\om)=&-\frac{4-m-2\ga}{2-\ga}y^{\frac{2}{2-\ga}}\om^2-(\ga-1)y^{\frac{2}{2-\ga}}\big(\om+(2-\ga)\big)\\
&-\frac{m}{2-\ga}\ga\big(y^{\frac{2}{2-\ga}}\rho\big)^{\ga-1}+\frac{4\pi y^{\frac{2}{2-\ga}}\om\rho}{4-3\ga}. 
\eeqa
We usually suppress the explicit dependence of $P_m$ on $(\rho,\om)$, writing instead $P_m(y)=P_m(y,\rho(y),\om(y))$ where clear.
\item[(ii)] At any point $y_1$ at which the flow is smooth and not sonic and where $P_m(y_1)=0$, the derivative of $P_m$ satisfies the identity
\beqa\label{eq:Pm'}
{P}_m'(y_1)=\frac{y_1^{\frac{2}{2-\ga}}}{y_1\om(y_1)}Q_m\Big(\om(y_1),\frac{\ga\rho(y_1)^{\ga-1}}{y_1^2}\Big),
\eeqa
where
\beqa\label{def:Qm}
Q_m(\om, \mathcal{R})=\bigg(&\big(1-\frac{\om}{2-\ga}\big)\Big(-\frac{4-m-2\ga}{2-\ga}(4-3\ga)\om^2+(\ga-1)(4-3\ga)(2-\ga)\Big)\\
&-\frac{2(4-m-2\ga)(m-1)\om^3}{(2-\ga)^2}-\frac{m(\ga-1)}{2-\ga}\om^2-2(\ga-1)\om\\
&+\mathcal{R}\frac{m}{(2-\ga)^2}\Big((4-3\ga)(2-\ga)-\om\big(4-3\ga+(\ga-1)(2-m)\big)\Big)\bigg)\Big|_{y_1}.
\eeqa
\end{itemize}
\end{lemma}

\begin{proof}
(i) To show \eqref{eq:rho'identitym}, we let $m\geq 0$. Then, rearranging the first equation of \eqref{eq:EPSS}, we find
\beqa
\frac{\rho'y}{\rho}=&\,\frac{-2y^2\om^2-(\ga-1)y^2(\om+2-\ga)+\frac{4\pi y^2\om\rho}{4-3\ga}}{y^2\om^2-\ga\rho^{\ga-1}}\\
=&\,\frac{-\frac{m}{2-\ga}\big(y^2\om^2-\ga\rho^{\ga-1}\big)+\big(\frac{m}{2-\ga}-2\big)y^2\om^2-\frac{m}{2-\ga}\ga\rho^{\ga-1}-(\ga-1)y^2(\om+2-\ga)+\frac{4\pi y^2\om\rho}{4-3\ga}}{y^2\om^2-\ga\rho^{\ga-1}}\\
=&\,-\frac{m}{2-\ga}+\frac{-\frac{4-m-2\ga}{2-\ga}y^2\om^2-(\ga-1)y^2\big(\om+(2-\ga)\big)-\frac{m}{2-\ga}\ga\rho^{\ga-1}+\frac{4\pi y^2\om\rho}{4-3\ga}}{y^2\om^2-\ga\rho^{\ga-1}},
\eeqa
and pulling out a factor of $y^{-\frac{2(\ga-1)}{2-\ga}}$ leaves us with the claimed identity.

(ii) By \eqref{eq:weightedrho'}--\eqref{eq:weightedomega2'}, as the flow is smooth at $y_1$, 
\beqa\label{eq:P'equation}
{P}_m'(y)=y^{\frac{2}{2-\ga}}\bigg(&-\frac{4-m-2\ga}{2-\ga}\Big(2\om\frac{(4-3\ga)(1-\frac{\om}{2-\ga})}{y}-\frac{2}{2-\ga}\frac{\om^2}{y}-\frac{2\om^2}{\rho}\rho'\Big)\\
&-(\ga-1)\Big(\frac{(4-3\ga)(1-\frac{\om}{2-\ga})}{y}-\frac{\om}{\rho}\rho'\Big)-\frac{2(\ga-1)}{y}\\
&+4\pi\frac{\rho}{y}\big(1-\frac{\om}{2-\ga})-\frac{m(\ga-1)}{2-\ga}\ga\big(y^{\frac{2}{2-\ga}}\rho\big)^{\ga-2}\big(\rho'+\frac{2}{2-\ga}\frac{\rho}{y}\big)\bigg).
\eeqa
From the identity $P_m(y_1)=0$, we rearrange to find
\beq\label{eq:Pm=0}
4\pi\rho=\frac{(4-m-2\ga)(4-3\ga)}{2-\ga}\om+(\ga-1)(4-3\ga)\big(1+\frac{2-\ga}{\om}\big)+\frac{m(4-3\ga)}{2-\ga}\ga\frac{\rho^{\ga-1}}{\om y_1^2},
\eeq
where all functions are evaluated at $y_1$. In addition, by \eqref{eq:rho'identitym}, as $P_m(y_1)=0$ we also have $\frac{\rho'y}{\rho}=-\frac{m}{2-\ga}$.

Substituting \eqref{eq:Pm=0} and $\frac{\rho'y}{\rho}=-\frac{m}{2-\ga}$ into \eqref{eq:P'equation}, we have
\begin{align*}
{P}_m'(y_1)=y^{\frac{2}{2-\ga}}\bigg(&\frac{1-\frac{\om}{2-\ga}}{y}\Big(-\frac{4-m-2\ga}{2-\ga}(4-3\ga)\om+(\ga-1)(4-3\ga)\frac{2-\ga}{\om}\Big)\\
&-\frac{2(4-m-2\ga)(m-1)\om^2}{(2-\ga)^2 y}-\frac{m(\ga-1)}{2-\ga}\frac{\om}{y}-\frac{2(\ga-1)}{y}\\
&+\frac{m(4-3\ga)}{2-\ga}\big(1-\frac{\om}{2-\ga}\big)\frac{\ga\rho^{\ga-1}}{\om y^3}-\frac{m(\ga-1)(2-m)}{(2-\ga)^2}\ga\rho^{\ga-1}\frac{1}{y^3}\bigg)\Big|_{y_1}\\
=y^{\frac{2}{2-\ga}-1}\bigg(&\big(1-\frac{\om}{2-\ga}\big)\Big(-\frac{4-m-2\ga}{2-\ga}(4-3\ga)\om+(\ga-1)(4-3\ga)\frac{2-\ga}{\om}\Big)\\
&-\frac{2(4-m-2\ga)(m-1)\om^2}{(2-\ga)^2}-\frac{m(\ga-1)}{2-\ga}{\om}-2(\ga-1)\\
&+\frac{\ga\rho^{\ga-1}}{y^2}\frac{m}{(2-\ga)^2\om}\Big((4-3\ga)(2-\ga)-\om\big(4-3\ga+(\ga-1)(2-m)\big)\Big)\bigg)\Big|_{y_1},
\end{align*}
which yields the required inequality after factoring out $\om^{-1}$.
\end{proof}

With these identities, we will show that as long as the flow remains strictly supersonic, the inequalities of Lemma \ref{lemma:rightinitial} above also hold strictly. For the proof, we will require also the following technical lemma containing properties of the functions $Q_m$.

%%%%%%%%%%%%%%%%%%%%%%%%%%%%
%%%%%%%%%%%%%%%%%%%%%%%%%%

\begin{lemma}\label{lemma:Qmintervalarithmetic}
Define the functions
\beq\label{def:Qmpm}
Q_m^+(\om)=Q_m(\om,0),\quad Q_m^-(\om)=Q_m(\om,\om^2),
\eeq
where we recall the definition of $Q_m$ from Lemma~\ref{L:PMQM}.
Then, for any $\ga\in(1,\frac43)$, there exists $\de_0>0$ such that for all $\om\in[\frac{4-3\ga}{3},2-\ga]$, we have
\begin{align}
&Q_m^\pm(\om)<0 &&\text{ for all }m\in\Big[1,\frac{2\ga}{\ga+1}+\de_0\Big],\label{ineq:Qmmax}\\
&Q_{\frac{4}{4-3\ga}}^\pm(\om)>0. && \label{ineq:Qmmin}
\end{align}
\end{lemma}
The proof is deferred to Appendix \ref{app:Qm}.

We are now able to state and prove the continuation estimates for the extension of the LPH-type solutions on their maximal supersonic interval of existence, $(y_*,y_{\max}(y_*))$.
\begin{proposition}\label{prop:continuationright}
Let $\ga\in(1,\frac43)$,  $y_*\in[y_f,y_F]$, and let $(\rho,\om)$ be the extension of the unique LPH-type solution  obtained from Theorem \ref{thm:Taylor} to $(y_*,y_{\max}(y_*))$. Then the following strict inequalities hold on the whole interval $(y_*,y_{\max}(y_*))$:
\beq\label{ineq:rightinvariance}
\frac{4-3\ga}{3}<\om<2-\ga,\quad \frac{4\pi y^2\rho\om}{4-3\ga}-\frac{2}{2-\ga}\ga\rho^{\ga-1}>0,\quad -\frac{4}{(4-3\ga)(2-\ga)}\frac{\rho}{y}<\rho'<-\frac{1}{2-\ga}\frac{\rho}{y}.
\eeq
Moreover, on this interval, we retain $\rho>0$.
\end{proposition}

\begin{proof}
We begin the proof by observing that the upper and lower bounds on $\rho'$ of \eqref{ineq:rightinvariance} guarantee that as long as the flow lives to the right of $y_*$ and satisfies the weak forms of these inequalities, we always retain $|(\log\rho)'|\leq C$, and hence $\rho>0$. Thus we assume this throughout the following.

By \eqref{ineq:rightinitial}, we know that all of the inequalities \eqref{ineq:rightinvariance} hold on the interval $(y_*,y_*+\bar\nu)$. By the smoothness and extendability of the flow guaranteed by Proposition \ref{prop:Picard} and Theorem \ref{thm:Taylor}, the set  
$$\mathfrak{Y}:=\big\{y_1\in(y_*,y_{\max})\,|\, \text{\eqref{ineq:rightinvariance} holds on }(y_*,y_1]\big\}$$
 is clearly relatively open in $(y_*,y_{\max})$.

We therefore work to show that $\mathfrak{Y}$ is also relatively closed. We therefore suppose $(y_*,y_1)\subset\mathfrak{Y}$, i.e., we assume that \eqref{ineq:rightinvariance} holds on the interval $(y_*,y_1)$ with $y_1< y_{\max}$. Showing that \eqref{ineq:rightinvariance} holds strictly at $y_1$ also is then sufficient to conclude the proof. Clearly the weak versions of \eqref{ineq:rightinvariance} hold on $(y_*,y_1]$ and the flow is strictly supersonic on this whole interval. As we have guaranteed already that $\rho_0>\rho(y_1)>0$ and $\om(y_1)$ is bounded, we may apply again the local existence theorem, Proposition \ref{prop:Picard}, to deduce that the flow can be smoothly extended past $y_1$, and hence is smooth at $y_1$ itself.

From \eqref{eq:invariance1}, we see that as $\rho'\leq0$ and $\om\geq\frac{4-3\ga}{3}>0$, then
$$\om'\geq\frac{4-3\ga-3\om}{y},$$
and hence
$$(y^3\om)'\geq(4-3\ga)y^2,$$
leading to
$$y^3\om(y)\geq\frac{4-3\ga}{3}y^3+y_*^3\big(\om(y_*+\frac{\bar\nu}{2})-\frac{4-3\ga}{3}\big)>\frac{4-3\ga}{3}y^3, $$
for all $y\in[y_*,y_1]$. Clearly then $\om(y_1)>\frac{4-3\ga}{3}$ also, as required.

To close the upper bound on $\om$, we first rearrange the first equation of \eqref{eq:EPSS} as
\beqa\label{eq:rho'identity}
\frac{\rho'y}{\rho}=&\,\frac{-2y^2\om^2-(\ga-1)y^2(\om+2-\ga)+\frac{4\pi y^2\om\rho}{4-3\ga}}{y^2\om^2-\ga\rho^{\ga-1}}\\
=&\,-\frac{2}{2-\ga}+\frac{(\ga-1)y^2\big(\frac{2\om^2}{2-\ga}-\om-(2-\ga)\big)-\frac{2}{2-\ga}\ga\rho^{\ga-1}+\frac{4\pi y^2\om\rho}{4-3\ga}}{y^2\om^2-\ga\rho^{\ga-1}}.
\eeqa
Note that, by assumption, on $[y_*+\bar\nu,y_1]$, $y^2\om^2-\ga\rho^{\ga-1}>0$. We apply also \eqref{eq:invariance1} to calculate
\beqa\label{eq:2-ga-om}
\big(2&-\ga-\om\big)'=-\frac{\frac{4-3\ga}{2-\ga}\big(2-\ga-\om\big)}{y}+\frac{\om}{y}\Big(\frac{\rho'y}{\rho}+\frac{2}{2-\ga}\Big)\\
=&\,-\frac{\frac{4-3\ga}{2-\ga}\big(2-\ga-\om\big)}{y}-\frac{\om}{y}\frac{(\ga-1)y^2\big(\frac{2\om}{2-\ga}+1\big)}{y^2\om^2-\ga\rho^{\ga-1}}(2-\ga-\om)+\frac{\om}{y}\frac{\frac{4\pi y^2\rho\om}{4-3\ga}-\frac{2}{2-\ga}\ga\rho^{\ga-1}}{y^2\om^2-\ga\rho^{\ga-1}}. 
\eeqa
Defining $$W(y)=\exp\Big(\int^y_{y_*+\bar\nu}\frac{1}{\tilde{y}}\Big(\frac{4-3\ga}{2-\ga}+\om\frac{(\ga-1)\tilde y^2\big(\frac{2\om}{2-\ga}+1\big)}{\tilde y^2\om^2-\ga\rho^{\ga-1}}\Big)\dif\tilde{y}\Big),$$
we have 
$$\big(W(2-\ga-\om)\big)'=W\frac{\om}{y}\frac{\frac{4\pi y^2\rho\om}{4-3\ga}-\frac{2}{2-\ga}\ga\rho^{\ga-1}}{y^2\om^2-\ga\rho^{\ga-1}}.$$
As $\frac{4\pi y^2\rho\om}{4-3\ga}-\frac{2}{2-\ga}\ga\rho^{\ga-1}\geq0$ on $[y_*+\bar\nu,y_1]$, we have
\beqs
W(2-\ga-\om)\geq W(2-\ga-\om)\big|_{y_*+\bar\nu}>0,
\eeqs
and hence 
\beq
\om<2-\ga\text{ on }[y_*,y_1].
\eeq
Turning now to $\frac{4\pi y^2\rho\om}{4-3\ga}-\frac{2}{2-\ga}\ga\rho^{\ga-1}$, we suppose for a contradiction that 
$$\frac{4\pi y_1^2\rho(y_1)\om(y_1)}{4-3\ga}-\frac{2}{2-\ga}\ga\rho(y_1)^{\ga-1}=0.$$
From \eqref{eq:rho'identity}, at $y_1$, we therefore have
\beq\label{ineq:rho'(y_1)}
\frac{\rho'(y_1)y_1}{\rho(y_1)}+\frac{2}{2-\ga}=\frac{(\ga-1)y_1^2\big(\frac{2\om(y_1)}{2-\ga}+1\big)\big(\om(y_1)-(2-\ga)\big)}{y_1^2\om(y_1)^2-\ga\rho(y_1)^{\ga-1}}<0
\eeq
due to $\om<2-\ga$.
Note now the simple scaled identity
$$\frac{4\pi y^2\rho\om}{4-3\ga}-\frac{2}{2-\ga}\ga\rho^{\ga-1}=y^{-\frac{2(\ga-1)}{2-\ga}}\Big(\frac{4\pi y^{\frac{2}{2-\ga}}\rho\om}{4-3\ga}-\frac{2}{2-\ga}\ga\big(y^{\frac{2}{2-\ga}}\rho\big)^{\ga-1}\Big).$$
Differentiating the term in the bracket, we use \eqref{eq:weightedrho'} to see
\beqas
\Big(&\frac{4\pi y^{\frac{2}{2-\ga}}\rho\om}{4-3\ga}-\frac{2}{2-\ga}\ga\big(y^{\frac{2}{2-\ga}}\rho\big)^{\ga-1}\Big)'\Big|_{y=y_1}\\
&=4\pi y^{\frac{2}{2-\ga}}\frac{\rho}{y}\big(1-\frac{\om}{2-\ga}\big)-\frac{2\ga(\ga-1)}{2-\ga}\big(y^{\frac{2}{2-\ga}}\rho\big)^{\ga-2}y^{\frac{2}{2-\ga}}\big(\rho'+\frac{2}{2-\ga}\frac{\rho}{y}\big)\Big|_{y=y_1}\\
&\geq4\pi y^{\frac{2}{2-\ga}}\frac{\rho}{y}\big(1-\frac{\om}{2-\ga}\big)\Big|_{y=y_1}>0,
\eeqas
where we have used \eqref{ineq:rho'(y_1)} in the first inequality on the last line and $\om(y_1)<2-\ga$ in the second. But this contradicts the assumption that $y_1$ is the first point at which $\frac{4\pi y^2\rho\om}{4-3\ga}-\frac{2}{2-\ga}\ga\rho^{\ga-1}=0$, hence the derivative must be non-positive. So
$$\Big(\frac{4\pi y^2\rho\om}{4-3\ga}-\frac{2}{2-\ga}\ga\rho^{\ga-1}\Big)\Big|_{y_1}>0. $$
Next, we consider the quantity
$$\frac{\rho'y}{\rho}+\frac{1}{2-\ga}.$$
Applying \eqref{eq:rho'identitym} in the case $m=1$, we get
\beqa
\frac{\rho'y}{\rho}+\frac{1}{2-\ga}=y^{-\frac{2(\ga-1)}{2-\ga}}\frac{P_1(y,\rho,\om)}{y^2\om^2-\ga\rho^{\ga-1}},
\eeqa
where we recall from \eqref{def:Pm} that
$$P_1=-\frac{(3-2\ga)}{2-\ga}y^{\frac{2}{2-\ga}}\om^2-(\ga-1)y^{\frac{2}{2-\ga}}\big(\om+(2-\ga)\big)-\frac{1}{2-\ga}\ga\big(y^{\frac{2}{2-\ga}}\rho\big)^{\ga-1}+\frac{4\pi y^{\frac{2}{2-\ga}}\om\rho}{4-3\ga}.$$
By assumption, we have $P_1(y)<0$ for all $y\in(y_*,y_1)$. By Proposition \ref{prop:continuationright} and Proposition \ref{prop:Picard}, as the flow is assumed supersonic, the flow may be extended smoothly to the right of $y_1$, and hence is smooth at $y_1$. Suppose now that at $y_1$, $\frac{\rho'y}{\rho}=-\frac{1}{2-\ga}$ for the first time (otherwise we are done). Then we must also have that $P_1'(y_1)\geq 0$, $P_1(y_1)=0$, and hence, at $y_1$, by \eqref{eq:Pm'},
\beqas
P_1'(y_1)=\frac{y^{\frac{2}{2-\ga}}}{y\om}Q_1&\Big(\om(y_1),\frac{\ga\rho(y_1)^{\ga-1}}{y_1^2}\Big).
\eeqas
Note that $Q_m(\om,\mathcal{R})$ is linear in $\mathcal{R}$ and that, as the flow is supersonic, we have always $0\leq\frac{\ga\rho^{\ga-1}}{y^2}\leq\om^2$. Thus, 
$$Q_1\Big(\om(y_1),\frac{\ga\rho(y_1)^{\ga-1}}{y_1^2}\Big)\leq \max\big\{Q_1(\om(y_1),0),Q_1(\om(y_1),\om(y_1)^2)\big\}<0$$
by Lemma \ref{lemma:Qmintervalarithmetic}.
Thus $P'(y_1)<0$, contradicting $P'(y_1)\geq 0$. So we obtain
$$\frac{\rho'y}{\rho}+\frac{1}{2-\ga}<0,\text{ for }y\in[y_*,y_1].$$
To conclude the final inequality, the lower bound for $\frac{\rho'y}{\rho}$, we let $m=\frac{4}{4-3\ga}$ and apply again \eqref{eq:rho'identitym} to find
\beqa
\frac{\rho'y}{\rho}+\frac{4}{(4-3\ga)(2-\ga)}=y^{-\frac{2(\ga-1)}{2-\ga}}\frac{P_{m}(y,\rho,\om)}{y^2\om^2-\ga\rho^{\ga-1}}.
\eeqa
If $y_1$ is the first point where $\frac{\rho'y}{\rho}=-\frac{4}{(4-3\ga)(2-\ga)}$, then $P_m(y_1)=0$, $P_m'(y_1)\leq0$ and so, at $y_1$, by \eqref{eq:Pm'}, we have
\beqas
P_m'(y_1)=\frac{y^{\frac{2}{2-\ga}}}{y\om}Q_m\Big(\om(y_1),\frac{\ga\rho(y_1)^{\ga-1}}{y_1^2}\Big).
\eeqas
Again, as $Q_m(\om,\mathcal{R})$ is linear in $\mathcal{R}$ and  $0\leq\frac{\ga\rho^{\ga-1}}{y^2}\leq\om^2$, we have
$$Q_m\Big(\om(y_1),\frac{\ga\rho(y_1)^{\ga-1}}{y_1^2}\Big)\geq \min\big\{Q_m(\om(y_1),0),Q_m(\om(y_1),\om(y_1)^2)\big\}<0$$
by Lemma \ref{lemma:Qmintervalarithmetic}. This contradicts the assumption $P_m(y_1)=0$, and hence we have
\beq
\frac{\rho'y}{\rho}+\frac{4}{(4-3\ga)(2-\ga)}>0.
\eeq
\end{proof}
To show that the flow remains supersonic to the right, and hence the global existence to the right, we need a slightly sharper upper bound on the derivative of the density, provided by the following lemma.

\begin{lemma}\label{lemma:rho'improved}
Let $\ga\in(1,\frac{4}{3})$, $y_*\in[y_f,y_F]$ and define $R_1=\frac{\rho_1 y_*}{\rho_0}$ as in Proposition \ref{prop:R1W1}. Let $(\rho,\om)$ be the extension of the unique LPH-type solution  obtained from Theorem \ref{thm:Taylor} to $(y_*,y_{\max}(y_*))$.  Then there exists $\de>0$ such that, for any $R>\max\{R_1,-\frac{2\ga}{(2-\ga)(\ga+1)}-\de\}$, we retain the inequality $\frac{\rho'y}{\rho}<R$ on the whole of $(y_*,y_{\max})$.
\end{lemma}

\begin{remark}
In effect, this says that if $\frac{\rho_1y_*}{\rho_0}<-\frac{2\ga}{(2-\ga)(\ga+1)}$, then we retain $\frac{\rho'y}{\rho}<-\frac{2\ga}{(2-\ga)(\ga+1)}$ as long as the flow stays supersonic. If, on the other hand, we only have $\frac{\rho_1y_*}{\rho_0}\geq-\frac{2\ga}{(2-\ga)(\ga+1)}$, then we will at least keep $\frac{\rho' y}{\rho}\leq \frac{\rho_1y_*}{\rho_0}$ as long as the flow stays supersonic.
\end{remark}
\begin{proof}
Choose $\de>0$ such that $\de(2-\ga)<\de_0$ with $\de_0$ the constant defined in Lemma \ref{lemma:Qmintervalarithmetic} and let $m\in (1,\frac{2\ga}{\ga+1}+\de(2-\ga))$ be such that $R_1<-\frac{m}{2-\ga}$. Applying again \eqref{eq:rho'identitym}, we find
\beqa
\frac{\rho'y}{\rho}+\frac{m}{2-\ga}=y^{-\frac{2(\ga-1)}{2-\ga}}\frac{{P_m}}{y^2\om^2-\ga\rho^{\ga-1}},
\eeqa
where
$$P_m=-\frac{(4-m-2\ga)}{2-\ga}y^{\frac{2}{2-\ga}}\om^2-(\ga-1)y^{\frac{2}{2-\ga}}\big(\om+(2-\ga)\big)-\frac{m}{2-\ga}\ga\big(y^{\frac{2}{2-\ga}}\rho\big)^{\ga-1}+\frac{4\pi y^{\frac{2}{2-\ga}}\om\rho}{4-3\ga}.$$
Suppose now that at $y_1$, $\frac{\rho'y}{\rho}=-\frac{m}{2-\ga}$ for the first time, so that ${P_m}(y)<0$ for all $y\in(y_*,y_1)$. By Proposition \ref{prop:continuationright} and Proposition \ref{prop:Picard}, as the flow is assumed supersonic, the flow may be extended smoothly to the right of $y_1$, and hence is smooth at $y_1$. Suppose now that at $y_1$, $\frac{\rho'y}{\rho}=-\frac{m}{2-\ga}$ for the first time (otherwise we are done). Then we must also have that $P_m'(y_1)\geq 0$, $P_m(y_1)=0$, and hence, at $y_1$, by \eqref{eq:Pm'},
\beqas
P_m'(y_1)=\frac{y^{\frac{2}{2-\ga}}}{y\om}Q_m&\Big(\om(y_1),\frac{\ga\rho(y_1)^{\ga-1}}{y_1^2}\Big).
\eeqas
Note that $Q_m(\om,\mathcal{R})$ is linear in $\mathcal{R}$ and that, as the flow is supersonic, we have always $0\leq\frac{\ga\rho^{\ga-1}}{y^2}\leq\om^2$. Then,
$$Q_m\Big(\om(y_1),\frac{\ga\rho(y_1)^{\ga-1}}{y_1^2}\Big)\leq \max\big\{Q_m(\om(y_1),0),Q_m(\om(y_1),\om(y_1)^2)\big\}.$$
Applying Lemma \ref{lemma:Qmintervalarithmetic}, for $m\in[1,\frac{2\ga}{\ga+1}+\de(2-\ga)]$, $\om\in[\frac{4-3\ga}{3},2-\ga]$ this is strictly negative, leading to the desired contradiction.
\end{proof}
With this, we may prove that the flow remains supersonic to the right for all $y>y_*$, concluding the proof of existence to the right.

%%%%%%%%%%%%%%%%%%%%%%%%%%%%%
%%%%%%%%%%%%%%%%%%%%%%%%%%%%%

\begin{lemma}\label{lemma:supersonic}
Let $\ga\in(1,\frac43)$,  $y_*\in[y_f,y_F]$. Then $y_{\max}(y_*)=\infty$, i.e.~the unique LPH-type solution $(\rho,\om)$ to the right of $y_*$ obtained from Theorem \ref{thm:Taylor} extends smoothly as a strictly supersonic solution of \eqref{eq:EPSS} to the whole of $(y_*,\infty)$.
\end{lemma}

%%%%%%%%%%%%%%%%%%%%%%%%%%%%%

\begin{proof}
Let now
$$S=y^\frac{2}{2-\ga}\om^2-\ga\big(y^{\frac{2}{2-\ga}}\rho\big)^{\ga-1}.$$
By Theorem \ref{thm:Taylor}, there exists $\de>0$ such that $S>0$ on $(y_*,y_*+\de]$.

By Proposition \ref{prop:continuationright} and the local existence and uniqueness Proposition \ref{prop:Picard}, the only obstruction to continuing the solution to the right is if strict supersonicity fails.

Suppose for a contradiction that $y_{\max}(y_*)<\infty$. Then there exists $y_0\in(y_*,y_{\max}]$ such that $\liminf_{y\to y_0^-}S(y)=0$ where $S(y)>0$ on $(y_*,y_0)$. The flow is then smooth on $(y_*,y_0)$, but may not extend smoothly up to $y_0$.

 A simple calculation using \eqref{eq:weightedrho'}--\eqref{eq:weightedomega2'} shows that, for all $y\in(y_*,y_0)$,
\beqas
S'(y)=&\,y^{\frac{2}{2-\ga}}\Big(2\om\frac{(4-3\ga)(1-\frac{\om}{2-\ga})}{y}-\frac{2}{2-\ga}\frac{\om^2}{y}-2\frac{\om^2\rho'}{\rho}\Big)\\
&-(\ga-1)\ga\big(y^{\frac{2}{2-\ga}}\rho\big)^{\ga-2}y^{\frac{2}{2-\ga}}\big(\rho'+\frac{2}{2-\ga}\frac{\rho}{y}\big)\\
=&\,y^{\frac{2}{2-\ga}}\Big(2\om\frac{(4-3\ga)(1-\frac{\om}{2-\ga})}{y}-\frac{2}{2-\ga}\frac{\om^2}{y}-2\frac{\om^2\rho'}{\rho}\Big)\\
&+(\ga-1)(S-y^{\frac{2}{2-\ga}}\om^2)\big(\frac{\rho'}{\rho}+\frac{2}{2-\ga}\frac{1}{y}\big). 
\eeqas
Rearranging this identity, we obtain
\beqa\label{eq:S'(y)}
S'(y)=&\,y^{\frac{2}{2-\ga}}\Big(2\om\frac{(4-3\ga)(1-\frac{\om}{2-\ga})}{y}-\frac{2}{2-\ga}\frac{\om^2}{y}-2\om^2\frac{\rho'}{\rho}-(\ga-1)\om^2\frac{\rho'}{\rho}-2\frac{(\ga-1)}{2-\ga}\frac{\om^2}{y}\Big)\\
&+S(y)(\ga-1)\big(\frac{\rho'}{\rho}+\frac{2}{2-\ga}\frac{1}{y}\big) \\
=&\,y^{\frac{2}{2-\ga}-1}\om F(\om,\frac{\rho'y}{\rho})+S(y)(\ga-1)\big(\frac{\rho'}{\rho}+\frac{2}{2-\ga}\frac{1}{y}\big),
\eeqa
where
$$F(\om,R):=2(4-3\ga)\big(1-\frac{\om}{2-\ga}\big)-\om(\ga+1)\big(R+\frac{2\ga}{(2-\ga)(\ga+1)}\big). $$
As the flow is smooth (analytic) through $y_*$ by construction, then this identity also holds at $y_*$, where $S(y_*)=0$. In particular, this gives us the inequality
\beq\label{ineq:F(y_*)}
F(\om_0,R_1)=\de^*>0,
\eeq
where we have defined, as usual, $R_1=\frac{\rho_1 y_*}{\rho_0}$. We distinguish now two cases: $R_1<-\frac{2\ga}{(2-\ga)(\ga+1)}$ and $R_1\geq-\frac{2\ga}{(2-\ga)(\ga+1)}$.

\textit{Case 1:} Suppose that $R_1<-\frac{2\ga}{(2-\ga)(\ga+1)}$. Then, by Lemma \ref{lemma:rho'improved}, there exists $\de>0$ such that 
$$\frac{\rho' y}{\rho}\leq-\frac{2\ga}{(2-\ga)(\ga+1)}-\de\text{ for all }y\in(y_*,y_0).$$
Thus as we have also $\frac{4-3\ga}{3}<\om<2-\ga$, we obtain $$F(\om,\frac{\rho'y}{\rho})\geq \om(\ga+1)\de\geq\de\frac{(\ga+1)(4-3\ga)}{3}=:\tilde\de>0.$$
By the estimates of Proposition \ref{prop:continuationright}, there exists $M>0$, depending only on $y_*$, $y_0$ and $\ga$, such that, for all $y\in(y_*,y_0)$,
$$y^{1-\frac{2}{2-\ga}}\om^{-1}\Big|(\ga-1)\big(\frac{\rho'}{\rho}+\frac{2}{2-\ga}\frac{1}{y}\big)\Big|\leq M.$$
Thus, if $S(y)\leq\frac{\tilde\de}{2M}$, we obtain from \eqref{eq:S'(y)} $S'(y)>0$, contradicting $\liminf_{y\to y_0}S(y)=0$.

\textit{Case 2:} Suppose now that $R_1\geq-\frac{2\ga}{(2-\ga)(\ga+1)}$. By Proposition \ref{prop:R1W1}, this forces $\ga\leq\frac{10}{9}$. As $\rho'<0$ by Proposition \ref{prop:continuationright}, we know that on $(y_*+\bar\nu,y_0)$ ($\bar\nu$ taken as in Lemma \ref{lemma:rightinitial}), we have $\rho<\rho_0-\de$ for some small $\de>0$. By Lemma \ref{lemma:f1structure} (compare also Figure \ref{fig:hlevelset}), there exists $\bar\eps>0$, depending  on $\de$, $\rho_0$ and $\ga\leq\frac{10}{9}$, such that if $0\leq h(\rho,\om)<\bar\eps$, $\om>\frac{4-3\ga}{3}$ and $\rho<\rho_0-\de$, then $\om<\om_0$. Here $h(\rho,\om)$ is as defined above in \eqref{def:h}. 

By Proposition \ref{prop:continuationright}, we have a bound $M>0$, depending only on $y_*$, $y_0$ and $\ga$, such that 
\beq\label{ineq:Sbounds}
\big|\frac{\rho'}{y\rho}\big|+y^{1-\frac{2}{2-\ga}}\om^{-1}\Big|(\ga-1)\big(\frac{\rho'}{\rho}+\frac{2}{2-\ga}\frac{1}{y}\big)\Big|\leq M\text{ on }(y_*,y_0). 
\eeq
Let $\eps>0$ be such that $\eps M<\min\{\de^*,\bar\eps\}$. As $S$ is differentiable on $(y_*,y_0)$, there exists $y_1\in(y_*,y_0)$ such that
$$S'(y_1)\leq0\text{  and }S(y_1)=\eps.$$
From the first equation of \eqref{eq:rhoom}, we obtain
$$\big|h(\rho(y_1),\om(y_1))\big|=\Big|-S(y_1)\frac{\rho'(y_1)}{y_1\rho(y_1)}\Big|\leq \eps M<\bar\eps.$$
Thus, by construction of $\bar\eps$, we also obtain $\om(y_1)<\om_0$. 

We use Lemma \ref{lemma:rho'improved} to see that $\frac{\rho' y}{\rho}\leq R_1$ on $(y_*,y_0)$ and so, noting that $\frac{\d}{\d R}F(\om,R)<0$,  we have $F(\om(y_1),\frac{\rho'(y_1)y_1}{\rho(y_1)})\geq F(\om(y_1),R_1)$. 
Now as $R_1\geq-\frac{2\ga}{(2-\ga)(\ga+1)}$, it is clear from the definition of $F$ that $\frac{\d}{\d\om}F(\om,R_1)<0$, and so, as $\om(y_1)<\om_0$, we obtain
$$F(\om(y_1),\frac{\rho'(y_1)y_1}{\rho(y_1)})\geq F(\om(y_1),R_1)>F(\om_0,R_1)=\de^*>0,$$
and so, using \eqref{ineq:Sbounds} and $\eps M<\de^*$ in \eqref{eq:S'(y)}, we find $S'(y_1)>0$, a contradiction to the definition of $y_1$.
\end{proof}

%%%%%%%%%%%%%%%%%%%%%%%%%%%%%
%%%%%%%%%%%%%%%%%%%%%%%%%%%%%

\begin{lemma}[Asymptotics]\label{lemma:asymptotics}
Let $\ga\in(1,\frac43)$,  $y_*\in[y_f,y_F]$. Then the local LPH-type solution $(\rho,\om)$ obtained from Theorem \ref{thm:Taylor} may be extended to the right as a Yahil-type solution of \eqref{eq:EPSS} on the whole interval $[y_*,\infty)$. 

Moreover, as $y\to\infty$, the asymptotics of $(\rho,\om)$ are as follows. There exist constants $\bar k_1>0$ and $\bar k_2>0$ such that
$$y^{\frac{1}{2-\ga}}(2-\ga-\om(y))\to\bar k_1,\quad y^{\frac{2}{2-\ga}}\rho(y)\to\bar k_2\quad\text{ as }y\to\infty.$$
\end{lemma}

%%%%%%%%%%%%%%%%%%%%%%%%%%%%%

\begin{proof}
The global existence to the right follows from Proposition \ref{prop:continuationright} and Lemma \ref{lemma:supersonic}, while the negativity of $u(y)$ follows directly from the bounds $\frac{4-3\ga}{3}<\om(y)<2-\ga$.

We begin by showing the asymptotics for $\om$. Recall from \eqref{eq:2-ga-om} the identity
\beqa\label{eq:om'extraidentity}
\big(2&-\ga-\om\big)'=\\
&-\frac{\frac{4-3\ga}{2-\ga}\big(2-\ga-\om\big)}{y}-\frac{\om}{y}\frac{(\ga-1)y^2\big(\frac{2\om}{2-\ga}+1\big)}{y^2\om^2-\ga\rho^{\ga-1}}(2-\ga-\om)+\frac{\om}{y}\frac{\frac{4\pi y^2\rho\om}{4-3\ga}-\frac{2}{2-\ga}\ga\rho^{\ga-1}}{y^2\om^2-\ga\rho^{\ga-1}}.
\eeqa
From Lemma \ref{lemma:rho'improved} and the initial estimate $\frac{y_*\rho_1}{\rho_0}<-\frac{1}{2-\ga}$, we see that there exists $\epsilon>0$ such that $\rho'\leq\big(-\frac{1}{2-\ga}-\epsilon\big)\frac{\rho}{y}$. As also $\rho>0$, we easily see that
$$0<\rho(y)\leq Cy^{-\frac{1}{2-\ga}-\epsilon},$$
and so, for $y$ large, we may estimate
$$\Big|\frac{\om}{y}\frac{\frac{4\pi y^2\rho\om}{4-3\ga}-\frac{2}{2-\ga}\ga\rho^{\ga-1}}{y^2\om^2-\ga\rho^{\ga-1}}\Big|\leq Cy^{-1-\frac{1}{2-\ga}-\epsilon}.$$
We re-write the middle term of \eqref{eq:om'extraidentity} as 
\beqas
-\frac{\om}{y}&\frac{(\ga-1)y^2\big(\frac{2\om}{2-\ga}+1\big)}{y^2\om^2-\ga\rho^{\ga-1}}(2-\ga-\om)\\
=&\,-\frac{1}{y}(\ga-1)\big(\frac{2}{2-\ga}+\frac{1}{\om}\big)(2-\ga-\om)+\frac{\om}{y}\frac{\ga\rho^{\ga-1}(\ga-1)y^2\big(\frac{2\om}{2-\ga}+1\big)}{y^2\om^2(y^2\om^2-\ga\rho^{\ga-1})}(2-\ga-\om)\\
=&\,-\frac{\frac{3(\ga-1)}{2-\ga}(2-\ga-\om)}{y}-\frac{(\ga-1)(2-\ga-\om)^2}{y\om(2-\ga)}+O\big(y^{-3-\frac{1}{2-\ga}}\big).
\eeqas
Thus, we find
\beqas
\big(2&-\ga-\om\big)'\leq-\frac{\frac{1}{2-\ga}(2-\ga-\om)}{y}+Cy^{-1-\frac{1}{2-\ga}-\epsilon},
\eeqas
leading to the desired estimate
$$0<2-\ga-\om(y)\leq Cy^{-\frac{1}{2-\ga}},$$
as claimed. With this quantitative decay established, it is easier to see that this decay is also sharp by using this estimate to treat the quadratic term in $(2-\ga-\om)$ as higher order and so obtain a lower bound of the same form: $2-\ga-\om\geq cy^{-\frac{1}{2-\ga}}$. Indeed, we easily see that the quantity $\big(y^{\frac{1}{2-\ga}}(2-\ga-\om)\big)'$ is integrable as $y\to\infty$, giving the existence of $\bar k_1$ as claimed.

Treating now $\rho$, we see from \eqref{eq:rho'identity} that
\beqas
\big(y^{\frac{2}{2-\ga}}\rho\big)'=&\,\frac{y^{\frac{2}{2-\ga}}\rho}{y}\big(\frac{\rho'y}{\rho}+\frac{2}{2-\ga}\big)\\
=&\,\frac{y^{\frac{2}{2-\ga}}\rho}{y}\Big(\frac{-(\ga-1)y^2\big(\frac{2\om}{2-\ga}+1\big)(2-\ga-\om)+\frac{4\pi y^2\om\rho}{4-3\ga}-\frac{2}{2-\ga}\ga\rho^{\ga-1}}{y^2\om^2-\ga\rho^{\ga-1}}\Big)
\eeqas
and the asymptotics just obtained for $2-\ga-\om$ and $\rho$ immediately yield that $y^{\frac{2}{2-\ga}}\rho$ remains bounded as $y\to\infty$. In particular, the right hand side of this identity is integrable as $y\to\infty$, giving the claimed convergence of $y^{\frac{2}{2-\ga}}\rho$.
\end{proof}

%%%%%%%%%%%%%%%%%%%%%%%%%%%%%
%%%%%%%%%%%%%%%%%%%%%%%%%%%%%

\begin{lemma}\label{lemma:rightmonotonicity}
Let $\ga\in(1,\frac43)$, $y_*\in[y_f,y_F]$, and let $(\rho,\om)$ be the global Yahil-type solution to the right of \eqref{eq:EPSS} obtained as the extension of the LPH-type solution from Theorem \ref{thm:Taylor}. Then the solution remains monotone (strictly monotone for $y_*>y_f$) in both $\rho$ and $\om$. 
\end{lemma}

%%%%%%%%%%%%%%%%%%%%%%%%%%%%%

\begin{proof}
In the case $y_*=y_f$, we know that the solution to the right is exactly the far-field solution $(\rho_f,\om_f)=(ky^{-\frac{2}{2-\ga}},2-\ga)$. We therefore need only to consider the case $y_*>y_f$ for which $\om'(y_*)>0$. Moreover, by the estimate $\rho'<-\frac{1}{2-\ga}\frac{y}{\rho}$ of Proposition \ref{prop:continuationright} above, we have $\rho'<0$ for all $y>y_*$. It remains only to show that we retain also $\om'(y)>0$.

Suppose now that there exists a point $y_0>y_*$ such that $\om'(y_0)=0$. Then, from \eqref{eq:EPSS}, 
we have
\beq\label{eq:om(y0)1}
\frac{4-3\ga-3\om(y_0)}{y_0}=\frac{y_0\om(y_0)h(y_0)}{G(y_0)}.
\eeq
Differentiating $h(\rho,\om)$, we obtain
\beqa\label{eq:h'}
\frac{\dif}{\dif y}h(\rho,\om)=&\,4\om\om'+(\ga-1)\om'-\frac{4\pi}{4-3\ga}\om\rho'-\frac{4\pi}{4-3\ga}\rho\om' \\
=&\,2\om\om'-\frac{4\pi}{4-3\ga}\om\rho'+\frac{h(\rho,\om)-(\ga-1)(2-\ga)}{\om}\om' \\
=&\,\big(2\om^2-(\ga-1)(2-\ga)+h(\rho,\om)\big)\frac{4-3\ga-3\om}{y\om}-y\frac{h(4\om^2+(\ga-1)\om)}{G(\rho,\om,y)}.
\eeqa
Thus, at $y_0$, 
\beqa\label{eq:h'(y0)1}
h'(y_0)=&\,-\frac{4\pi}{4-3\ga}\om\rho'=-\frac{4\pi}{4-3\ga}\frac{y\rho\om h}{G}.
\eeqa
Arguing directly, we differentiate $G$ to obtain
\beqas
G'=&\,(\ga-1)\ga\rho^{\ga-2}\rho'-2y\om^2-2y^2\om\om'\\
=&\,(\ga-1)\ga\rho^{\ga-1}\frac{yh}{G}-2y\om^2-2y^2\om\om'.
\eeqas
Thus, at $y_0$,
\beqa\label{eq:G'(y0)1}
G'(y_0)=&\,(\ga-1)\ga\rho^{\ga-1}\frac{4-3\ga-3\om}{y\om}-2y\om^2.
\eeqa
We now further differentiate the second equation of \eqref{eq:rhoom} to obtain
\beqas
\om''=&\,-\frac{3\om'}{y}-\frac{4-3\ga-3\om}{y^2}-\frac{\om h}{G}-\frac{y\om' h}{G}-\frac{y\om h'}{G}+\frac{y\om h G'}{G^2}.
\eeqas
Hence, at $y_0$, we find
\beqas
\om''(y_0)=-\frac{4-3\ga-3\om}{y^2}-\frac{\om h}{G}-\frac{y\om h'}{G}+\frac{y\om h G'}{G^2}=-2\frac{\om h}{G}-\frac{y\om h'}{G}+\frac{y\om h G'}{G^2},
\eeqas
where we have used \eqref{eq:om(y0)1} in the second equality. Substituting \eqref{eq:h'(y0)1} into the second term and \eqref{eq:G'(y0)1} into the third term, we get
\beqas
\om''(y_0)=&\,-\frac{2\om h G}{G^2}+\frac{\frac{4\pi}{4-3\ga}y^2\om^2\rho h}{G^2}+\frac{y\om h\big((\ga-1)\ga\rho^{\ga-1}\frac{4-3\ga-3\om}{y\om}-2y\om^2\big)}{G^2}\\
\geq&\,\frac{\om h}{G^2}\Big(2(y^2\om^2-\ga\rho^{\ga-1})+y^2\big(2\om^2+(\ga-1)\om+(\ga-1)(2-\ga)\big)\\
&\quad+(\ga-1)\ga\rho^{\ga-1}\frac{4-3\ga-3\om}{\om}-2y^2\om^2\Big),
\eeqas
where we have used that $h>0$ (from $\rho'<0$) to obtain $\frac{4\pi}{4-3\ga}\rho\om>2\om^2+(\ga-1)\om+(\ga-1)(2-\ga)$.
Grouping terms, we then find
\beqas
\om''(y_0)\geq&\,\frac{\om h}{G^2}\Big(y^2\om^2\big(2+\frac{\ga-1}{\om}+\frac{(\ga-1)(2-\ga)}{\om^2}\big)+\ga\rho^{\ga-1}\big(-2+(\ga-1)\frac{4-3\ga-3\om}{\om}\big)\Big)\\
\geq&\,\frac{\om h}{G^2}y^2\om^2\Big(2+\frac{\ga-1}{\om}+\frac{(\ga-1)(2-\ga)}{\om^2}-2+(\ga-1)\frac{4-3\ga-3\om}{\om}\Big),
\eeqas
where we have used that $\ga\rho^{\ga-1}<y^2\om^2$ and $-2+(\ga-1)\frac{4-3\ga-3\om}{\om}<0$ for $\om\in(\frac{4-3\ga}{3},2-\ga)$. Thus,
\beqas
\om''(y_0)\geq&\,\frac{y^2\om h}{G^2}\Big((\ga-1)\om+(\ga-1)(2-\ga)+(\ga-1)(4-3\ga-3\om)\om\Big)>0
\eeqas
for all $\om\in(\frac{4-3\ga}{3},2-\ga)$ (indeed, one easily checks that the roots of the quadratic on the right are $-\frac13$ and $2-\ga$ while the coefficient of the quadratic term is negative), a contradiction to $\om'(y_0)=0$.
\end{proof}

%%%%%%%%%%%%%%%%%%%%%%%%%%
%%%%%%%%%%%%%%%%%%%%%%%%%%

\section{Solution to the left of the sonic point}\label{S:LEFT}

%%%%%%%%%%%%%%%%%%%%%%%%%%
%%%%%%%%%%%%%%%%%%%%%%%%%%

To construct a global solution to \eqref{eq:EPSS}, we now need to solve to the left of the sonic point. This is the core of the construction of the global self-similar solution and is the most challenging part of the proof analytically. We develop an ad hoc shooting method, varying the sonic time $y_*$ as our shooting parameter, to find a critical $\bar y_*$ for which the associated, local, LPH-type solution given by Theorem \ref{thm:Taylor} can be extended smoothly up to the origin without meeting a second sonic point.

To proceed with this shooting argument, we partition the set of sonic times into three parts, defined by the relation of the associated $\om(y;y_*)$ to the Friedman solution $\om_F\equiv\frac{4-3\ga}{3}$. The key set of values $y_*$ is those for which $\om(\cdot;y_*)$ intersects $\om_F$ before a second sonic point occurs, which we call $\mathcal{Y}$ (see definition below). As we expect a global solution to agree with the Friedman solution only at the origin, we find the critical $\bar y_*$ which leads to the global solution as the infimum of a connected component of $\mathcal{Y}$.

Throughout the section, the functions $(\rho(\cdot;y_*),\om(\cdot;y_*))$ will be taken to refer to the extension of the unique LPH-type solution obtained from Theorem \ref{thm:Taylor} as a solution of \eqref{eq:EPSS}.

Following the strategy of \cite{Guo20}, we can first define the sonic time and then partition the set $[y_f,y_F]$ as follows.
\begin{definition}[Sonic time, $\mathcal X, \mathcal Y, \mathcal Z$]\label{def:sonictime}
\beq
s(y_*)=\inf\{y\in(0,y_*)\,|\,(\rho(\cdot;y_*),\om(\cdot;y_*))\text{ extends onto }(y,y_*] \text{ and } \ga\rho(y;y_*)^{\ga-1}-y^2\om(y;y_*)^2>0\}
\eeq
and then the following sets:
\beqa
\mathcal{X}=&\,\{y_*\in[y_f,y_F)\,|\,\inf_{y\in(s(y_*),y_*)}\om(y;y_*)>\frac{4-3\ga}{3}\},\\
\mathcal{Y}=&\,\{y_*\in[y_f,y_F)\,|\,\text{ there exists }y\in(s(y_*),y_*)\text{ such that }\om(y;y_*)=\frac{4-3\ga}{3}\},\\
\mathcal{Z}=&\,\{y_*\in[y_f,y_F)\,|\,\om(y;y_*)>\frac{4-3\ga}{3}\text{ for all }y\in(s(y_*),y_*)\text{ and }\inf_{y\in(s(y_*),y_*)}\om(y;y_*)\leq\frac{4-3\ga}{3}\},
\eeqa
as well as the fundamental set
\beq
Y=\{y_*\in[y_f,y_F)\,|\,\text{for all }\tilde{y}_*\in[y_*,y_F),\text{ there exists }y\in(s(\tilde y_*),\tilde y_*)\text{ such that }\om(y;y_*)=\frac{4-3\ga}{3}\}.
\eeq
Finally, we define the value
\beq
\bar y_*=\inf Y.
\eeq
\end{definition}
Note that $y_f\in\mathcal{X}$ as $(\rho(\cdot;y_f),\om(\cdot;y_f))=(\rho_f,\om_f)$.

\begin{remark}
The unique extension of the local, unique LPH-type solution onto $(s(y_*),y_*)$ can be thought of as a maximal extension of the solution obtained by Theorem \ref{thm:Taylor}, and for the rest of this section, we will take the solution $(\rho(\cdot;y_*),\om(\cdot;y_*))$ of \eqref{eq:EPSS} to be defined on this maximal interval.
\end{remark}

To show that the solution associated to $\bar y_*$ can be extended to the origin to give a global solution, we require a number of further properties. First, we will show various continuity properties along the flow, a priori bounds away from the sonic time, upper semi-continuity of the sonic time and the openness of $\mathcal{Y}$. Next, we will demonstrate some basic invariant regions that hold as $y$ decreases. The key insight that will allow us to show the global existence of the solution is that, for $y_*\in Y$, the solution $\om(\cdot;y_*)$ must remain monotone as $y$ decreases until $\om$ meets the Friedman value $\frac{4-3\ga}{3}$. By propagating this property along ${Y}$ to $\bar y_*$ in the key Proposition \ref{prop:S=Y}, we are able to show that no second sonic point forms in the solution from $\bar y_*$, and hence the solution may be extended to the origin. In the final part of this section, we also conclude that the global solution indeed takes the value $\om(0)=\frac{4-3\ga}{3}$ at the origin and that the density remains bounded globally.

%%%%%%%%%%%%%%%%%%%%%%%%%%%%%%%%%%%%
%%%%%%%%%%%%%%%%%%%%%%%%%%%%%%%%%%%%

\subsection{Continuity properties}

We first show the simple positivity of the density to the left of the sonic point.
\begin{lemma}\label{L:RHOPOSITIVE}
Let $\ga\in(1,\frac43)$, $y_*\in[y_f,y_F]$ and let $(\rho,\om)$ be the associated unique LPH-type solution on $(s(y_*),y_*)$. Then $\rho(y)>0$ for all $y\in(s(y_*),y_*)$.
\end{lemma}

\begin{proof}
From the first equation of \eqref{eq:rhoom}, we rearrange to find
$$\big(\log\rho\big)'=\frac{yh(\rho,\om)}{G(y;\rho,\om)}.$$
For any $y_1\in(s(y_*),y_*-\nu)$, where $\nu$ is as in Theorem \ref{thm:Taylor}, we know that as the solution exists, is continuous, and $G>0$ on the closed interval $[y_1,y_*-\nu]$, we have a bound
$$\Big|\frac{yh(\rho,\om)}{G(y;\rho,\om)}\Big|\leq C,$$
where $C$ may depend on $y_1$, $y_*$ etc., and so, integrating, we see that on $[y_1,y_*-\nu]$, $\log\rho$ remains bounded, and hence $\rho>0$. As $y_1\in(s(y_*),y_*-\nu)$ was arbitrary, we conclude that $\rho>0$ holds on the whole interval $(s(y_*),y_*)$.
\end{proof}

\begin{lemma}\label{lemma:trivialinvariance}
Let $\ga\in(1,\frac43)$, $y_*\in[y_f,y_F]$ and let $(\rho,\om)$ be the associated unique LPH-type solution on $(s(y_*),y_*)$
Then, if there exists $y_0\in(s(y_*),y_*)$ such that $\om(y_0)=0$, we have that 
\beq
\om(y)<0 \text{ and }\rho(y)<\rho(y_0) \text{ for all }y\in(s(y_*),y_0).\eeq
\end{lemma}

\begin{proof}
For any $y_0\in(s(y_*),y_*)$ such that $\om(y_0)=0$, the second equation of \eqref{eq:rhoom} gives $\om'(y_0)=\frac{4-3\ga}{y}>0$, which is only possible if $\om>0$ on an interval to the right of $y_0$. On the other hand, if there exists $y_0\in(s(y_*),y_*)$ such that $\om(y_0)=0$, then as $\om(y)<0$ for all $y\in(s(y_*),y_0)$, we obtain for all such $y$ that 
$$h(\rho,\om)=2\om^2+(\ga-1)\om+(\ga-1)(2-\ga)-\frac{4\pi}{4-3\ga}\rho\om>0,$$
where we have used that the quadratic function $2\om^2+(\ga-1)\om+(\ga-1)(2-\ga)>0$ for all $\om\in\R$ and $\rho>0$. Thus, from the first equation of \eqref{eq:rhoom}, we have $\rho'>0$ on $(s(y_*),y_0)$ and so $\rho(y)<\rho(y_0)$ on the whole interval.
\end{proof}

We begin by establishing some \textit{a priori} estimates on the solution to the left as long as it remains subsonic, i.e., as long as we remain on the interval $(s(y_*),y_*)$.
\begin{lemma}\label{lemma:apriori}
Let $\ga\in(1,\frac43)$, $y_*\in[y_f,y_F]$ and let $(\rho,\om)$ be the associated unique LPH-type solution on $(s(y_*),y_*)$. Let $\al>\frac{4-3\ga}{\ga-1}>0$. Then there exists $C>0$, depending on $\ga$ and $\al$ but independent of $y_*\in[y_f,y_F]$, such that the solution $(\rho,\om)$ satisfies the \textit{a priori} bounds
\begin{align}
&\rho(y)<\frac{C}{y^{3+\al}},\label{ineq:rhoapriori}\\
&|\om(y)|\leq\frac{1}{y} \sqrt{\ga}\frac{C^{\frac{\ga-1}{2}}}{y^{\frac{(3+\al)(\ga-1)}{2}}}.\label{ineq:omapriori}
\end{align}
\end{lemma}
\begin{proof}
Throughout the proof, constants will appear depending continuously on $\rho_0$, $\om_0$ and $y_*$. We will use the continuous dependence with respect to $y_*\in[y_f,y_F]$ of these parameters to make the dependence on $\ga$ only. 

\textit{Step 1: Prove \eqref{ineq:omapriori} assuming \eqref{ineq:rhoapriori}.}\\
To prove the \textit{a priori} bounds on $\om$, we observe that it suffices to prove the claimed upper bound \eqref{ineq:rhoapriori} for $\rho$ on $(s(y_*),y_*)$ as the condition $G(y;\rho,\om)>0$ then yields the simple bound
$$|\om(y)|\leq\frac{\sqrt{\ga}\rho^{\frac{\ga-1}{2}}}{y}$$
which gives the claimed bound for $\om$ of \eqref{ineq:omapriori}.

\textit{Step 2: Prove \eqref{ineq:rhoapriori} in the region $\{\om\leq0\}$.}\\
 We first note that
\beq\label{eq:momentum}
(\rho\om)'=\frac{4-3\ga-3\om}{y}\rho,
\eeq
which follows easily from~\eqref{eq:invariance1}.
From here we see that for any $y\in(s(y_*),y_*)$, we have
$$(\rho\om)'+\frac{3}{y}\rho\om=\frac{4-3\ga}{y}\rho>0,$$
where we have used Lemma~\ref{L:RHOPOSITIVE}.
We easily deduce $(y^3\rho\om)'>0$ and thus
\beq\label{ineq:rhoupper1}
(\rho\om)(y)<\frac{\rho_0\om_0y_*^3}{y^3}.
\eeq
By Lemma \ref{lemma:trivialinvariance}, the region $\{\om\leq0\}$ is invariant under the dynamics of the flow to the left and, if there exists $y_1\in(s(y_*),y_*)$ such that $\om(y_1)=0$, then $\rho(y)<\rho(y_1)$ on the whole interval $(s(y_*),y_1)$. It is therefore sufficient to prove that \eqref{ineq:rhoapriori} holds on the interval $[y_0,y_*-\nu]$, where $$y_0=\inf\{y\in(s(y_*),y_*)\,|\,\om(y)>0\text{ on }(y,y_*)\}.$$

\textit{Step 3: Conclude the bound \eqref{ineq:rhoapriori} for $\rho$ on the remaining region, $\{\om>0\}$.}\\
 Let $\de\in(0,\frac{4-3\ga}{3})$ be fixed (and small). Then on the set $\{\om\geq\de y^\al\}$, where $\de$ and $\al>0$ are to be chosen later, we have from \eqref{ineq:rhoupper1}
\beq\label{ineq:rhoupper2}
\rho(y)<\frac{\rho_0\om_0y_*^3}{\de y^{3+\al}}=:C_*\frac{1}{\de}y^{-(3+\al)}.
\eeq
By continuity of the flow away from the sonic point, the set $$A=\{y\in(s(y_*),y_*)\,|\,\om(y)\in(0,\de y^\al),\:\rho(y)> \frac12\frac{C_*}{\de}y^{-(3+\al)}\}$$ is an open subset of $(y_0,y_*)$. If $A$ is empty, we are done. Suppose $A$ is not empty. It may therefore be written as a (possibly countable) union of disjoint, non-empty, open intervals. Taking such an interval, $(y_1,y_2)$, note that by the invariance of the set $\{\om\leq0\}$, we cannot have $\om(y_2)=0$. We must therefore have either $\om(y_2)=\de y_2^\al$ (in which case \eqref{ineq:rhoupper2} applies) or $\rho(y_2)=\frac12 \frac{C_*}{\de}y_2^{-(3+\al)}$ and hence, in either case, 
\be\label{E:RHOSQUISHED}
 \frac{C_*}{2\de}y_2^{-(3+\al)}\leq \rho(y_2)< \frac{C_*}{\de}y_2^{-(3+\al)}.
 \ee
For $\de$ sufficiently small, depending only on $\ga$ and $\al$, on $(y_1,y_2)$, we have $$\frac{\ga}{2}\rho^{\ga-1}\leq G(y;\rho,\om)\leq\ga\rho^{\ga-1}\text{ and }h(\rho,\om)\geq-\frac{4\pi}{4-3\ga}\rho\om\geq -\frac{4\pi}{4-3\ga}\de y^\al\rho.$$
Therefore, from the first equation of \eqref{eq:rhoom}, we have the lower bound
$$\rho'\geq -\frac{\de}{\ga}\frac{8\pi}{4-3\ga}y^{\al+1}\rho^{3-\ga}=:-c_*\de y^{\al+1}\rho^{3-\ga}\text{ on }(y_1,y_2).$$
Rearranging and integrating this differential inequality leads to
\beq\label{ineq:rhoupper3}
\rho(y)^{-(2-\ga)}\geq\rho(y_2)^{-(2-\ga)}-\frac{\de c_*(2-\ga)}{2+\al}y_2^{2+\al}+\frac{\de c_*(2-\ga)}{2+\al}y^{2+\al}\text{ for all }y\in(y_1,y_2).
\eeq
Note now that, by~\eqref{E:RHOSQUISHED}, $\rho(y_2)$ satisfies
$$C_1\de^{2-\ga}y_2^{(3+\al)(2-\ga)}\leq\rho(y_2)^{-(2-\ga)}\leq C_2\de^{2-\ga}y_2^{(3+\al)(2-\ga)},$$
for some constants $C_1$, $C_2>0$ depending only on $\ga$ (where we have used that the constant $C_*$ depends continuously on $y_*\in[y_f,y_F]$ to remove dependence on $y_*$). Noting that $2-\ga<1$ so that $\de\ll\de^{2-\ga}$ for $\de\ll1$, we now choose $\al>\frac{4-3\ga}{\ga-1}>0$ so that $(3+\al)(2-\ga)<2+\al$, and hence, provided $\de$ was chosen small (depending on $\ga$, $\al$), we have
$$\rho(y_2)^{-(2-\ga)}-\frac{\de c_*(2-\ga)}{2+\al}y_2^{2+\al}\geq \frac{1}{2}\rho(y_2)^{-(2-\ga)}.$$
Thus, returning to \eqref{ineq:rhoupper3}, we obtain, for all $y\in(y_1,y_2)$,
$$\rho(y)\leq \Big(\frac12 \rho(y_2)^{-(2-\ga)}\Big)^{-\frac{1}{2-\ga}}\leq C\de^{-1}y_2^{-(3+\al)}\leq C\de^{-1}y^{-(3+\al)},$$
which yields the desired claim as the obtained estimate is independent of the choice of component $(y_1,y_2)$ and $C$ depends on $\ga$ and $\al$ only.
\end{proof}
The following lemma allows us to extend any solution further to the left from a point $y_0\in(0,y_*)$ provided the solution is uniformly subsonic, i.e., $G(y;\rho,\om)\geq\eta>0$. Moreover, the time that we may extend by depends only on $y_0$ and $\eta$.

%%%%%%%%%%%%%%%%%%%%%%%%%

\begin{lemma}\label{lemma:extension}
Let $\ga\in(1,\frac43)$, $y_*\in[y_f,y_F]$ and let $(\rho,\om)$ be the associated unique LPH-type solution on $(s(y_*),y_*)$. Suppose that, for some $y_0\in(s(y_*),y_*-\nu)$, we have $G(y;\rho,\om)\geq\eta>0$ for all $y\in[y_0,y_*-\nu]$. Then there exists $\tau>0$, depending only on $\ga$, $y_0$ and $\eta$, such that the solution may be extended onto the interval $[y_0-\tau,y_*]$ while remaining subsonic, i.e., $s(y_*)\leq y_0-\tau$. Moreover, on the extended region, $[y_0-\tau,y_0]$, we retain the inequalities
\beqa
C_\eta\leq\rho\leq M,\quad |\om|\leq M,\quad G(y;\rho,\om)\geq\frac12 \eta,
\eeqa
where $C_\eta$ and $M$ depend only on $\ga$, $y_0$ and $\eta$.
\end{lemma}

%%%%%%%%%%%%%%%%%%%%%%%%%

\begin{proof}
By Lemma \ref{lemma:apriori}, we have a constant $M>0$, depending only on $\ga$ and $y_0$ such that
$$0<\rho\leq \frac12 M,\quad |\om|\leq \frac12 M\text{ on }[y_0,y_*-\nu].$$
As $G(y;\rho,\om)\geq\eta$ on the whole interval, we make the trivial estimate
$$\rho\geq\frac{1}{\ga}\Big(\eta+y^2\om^2\Big)^{\frac{1}{\ga-1}}\geq 2C_\eta,$$
where $C_\eta$ depends only on $\ga$ and $\eta$.

We are therefore in the situation of Proposition \ref{prop:Picard} with constant $M$ having only the dependence claimed in the statement of the lemma. 
All of the estimates then follow from Proposition \ref{prop:Picard}.
\end{proof}

%%%%%%%%%%%%%%%%%%%%%%%%%%

Before we can continue, we need some continuity properties both of the sonic time, and of the flow with respect to $y_*$ away from sonic points. 

%%%%%%%%%%%%%%%%%%%%%%%%%%%%%%%%%%%
%%%%%%%%%%%%%%%%%%%%%%%%%%%%%%%%%%%

\begin{prop}\label{prop:sonictime}
Let $\ga\in(1,\frac43)$ and $y_*\in[y_f,y_F]$. Then the following hold.
\begin{itemize}
\item[(i)] The sonic time is upper semi-continuous:
$$\limsup_{\tilde y_*\to y_*}s(\tilde y_*)\leq s(y_*).$$
\item[(ii)] Suppose $(y_*^n)_{n=1}^\infty\subset [y_f,y_F]$ converge $y_*^n\to y_*$. Suppose further that there exist $y_0\in(0,y_*-\nu)$ and $\eta>0$ such that $s(y_*^n)<y_0$ for all $n$, $\rho(y;y_*^n)$ and $\om(y;y_*^n)$ are uniformly bounded on $[y_0,y_*]$, and
$$G(y;\rho(y;y_*^n),\om(y;y_*^n))\geq\eta\quad\text{ for all }n\in\N,\:y\in[y_0,y_*-\nu].$$
Then there exists $\tau=\tau(\eta,y_0)>0$ such that
$$s(y_*)<y_0-\tau,\quad s(y_*^n)<y_0-\tau\text{ for all }n\in\N.$$
\item[(iii)] Suppose that $s(y_*)<y_0$. Then for any $\eps>0$, there exist $\de>0$ and $\tau>0$ such that for all $\tilde y_*\in[y_f,y_F]$ satisfying $|\tilde y_*-y_*|<\de$, the estimate
$$\big|(\rho(y;\tilde y_*),\om(y;\tilde y_*))-(\rho(y;y_*),\om(y;y_*))\big|<\eps$$
holds uniformly in $y$ on $[y_0-\tau,y_*-\nu]$.
\end{itemize}
\end{prop}
\begin{proof}
As the proof of this Proposition is substantially similar to the proof of \cite[Proposition 4.5]{Guo20}, we defer the details to Appendix \ref{app:continuity}.
\end{proof}
%%%%%%%%%%%%%%%%%%%%%%%%%%%%%%%%%%%
%%%%%%%%%%%%%%%%%%%%%%%%%%%%%%%%%%%
\subsection{Invariant structures}\label{sec:invariants}
\begin{definition}
We define the critical time
\beq
y_c(y_*)=\inf\{y\in(s(y_*),y_*)\,|\,\om(\tilde y,y_*)>\frac{4-3\ga}{3}\text{ for all }\tilde y\in(y,y_*)\}.
\eeq
\end{definition}

\begin{lemma}\label{lemma:leftinvariants}
Let $\ga\in(1,\frac43)$, $y_*\in[y_f,y_F]$ and let $(\rho,\om)$ be the associated unique LPH-type solution on $(s(y_*),y_*)$. Suppose that $y_0\in(s(y_*),y_*)$ is such that on $(y_0,y_*)$, we have $h(\rho,\om)<0$ and $\om>\frac{4-3\ga}{3}$. Then the following  hold.
\begin{itemize}
\item[(i)] At most one of the conditions $h(\rho,\om)=0$ and $\om=\frac{4-3\ga}{3}$ can occur at $y_0$.
\item[(ii)] If $h(\rho,\om)=0$ at $y_0$, then $\inf_{y\in(s(y_*),y_*)}\om>\frac{4-3\ga}{3}$.
\item[(iii)] If there exists $y_1\in[y_c(y_*),y_*]$ such that $y_1>0$ and $\lim_{y\searrow y_1}\om(y)=\frac{4-3\ga}{3}$, then we must have $y_1>s(y_*)$.
\end{itemize}
Thus if $\inf_{y\in(s(y_*),y_*)}\om\leq\frac{4-3\ga}{3}$, we must have that $h<0$ on $(y_c(y_*),y_*)$.
\end{lemma}
\begin{remark}
Recalling the notation of Definition \ref{def:sonictime}, (i) if $y_*\in\mathcal{Y}\cup\mathcal{Z}$ then $h<0$ on $(y_c(y_*),y_*)$;\\
(ii) if $y_*\in[y_f,y_F]\setminus\mathcal{Y}$ and $y_c(y_*)=s(y_*)>0$, then
$$\limsup_{y\searrow s(y_*)+}\om(y)>\frac{4-3\ga}{3}.$$ 
\end{remark}

\begin{proof}
(i) Suppose that at $y_0$ both $h=0$ and $\om=\frac{4-3\ga}{3}$. Solving the condition $h(\rho,\frac{4-3\ga}{3})=0$ leads directly to $\rho=\frac{1}{6\pi}$. Using the local existence and uniqueness of the ODE system around a non-sonic (and non-zero) point $y_0$ from Proposition \ref{prop:Picard}, we therefore get that the solution is locally, and hence also globally, the Friedman solution, $\om_F\equiv\frac{4-3\ga}{3}$, $\rho_F\equiv\frac{1}{6\pi}$. In particular, at the sonic point $y_*$ we must also have $(\rho_0,\om_0)=(\frac{1}{6\pi},\frac{4-3\ga}{3})$ and hence $y_*=y_F$. But this is a contradiction as the Friedman solution is not of LPH-type by Proposition \ref{prop:far-field}(iii).

\noindent (ii) Suppose that $h(\rho,\om)=0$ at $y_0$ (for short, we will write $h(y_0)=0$). As $h<0$ on $(y_0,y_*)$, we must have $h'(y_0)\leq 0$. Note also that, by part (i), we have $\om(y_0)>\frac{4-3\ga}{3}$. As $y_0$ is not a sonic point and $h(y_0)=0$, we have that $\rho'(y_0)=0$ also. Thus, at $y_0$, from \eqref{eq:h'}, we have
\beqa
\frac{\dif}{\dif y}h(\rho,\om)\big|_{y=y_0}=&\,\big(2\om-\frac{(\ga-1)(2-\ga)}{\om}\big)\om'\\
=&\,\frac{1}{\om}\big(2\om^2-(\ga-1)(2-\ga)\big)\big(\frac{4-3\ga-3\om}{y}\big),
\eeqa
where we have again used that $h=0$ in the $\om'$ equation of \eqref{eq:rhoom}. Clearly as $\om(y_0)>\frac{4-3\ga}{3}$, the second bracket is strictly negative (and $\om>0$). The first bracket satisfies
$$2\om^2-(\ga-1)(2-\ga)\begin{cases}<0&\text{ if }|\om|<\om_*,\\
>0 &\text{ if }|\om|>\om_*,
\end{cases}$$
where we recall $\om_*=\sqrt{\frac{(\ga-1)(2-\ga)}{2}}$ from Lemma \ref{lemma:f1structure}.\\
\textit{Case 1:}  $\om(y_0)\in(\frac{4-3\ga}{3},\om_*)$. In this case, we arrive at a contradiction to $h'(y_0)\leq0$.\\
(recall from Lemma \ref{lemma:f1structure} that $\om_*\geq\frac{4-3\ga}{3}$ only for $\ga>\frac{10}{9}$ with equality at $\ga=\frac{10}{9}$.)\\
\textit{Case 2:}  $\om(y_0)\geq \om_*$. In this case, we break the proof into several steps.

\textit{Step 1: We first show that there exists $\de>0$ such that for $y\in(y_0-\de,y_0)$, we have $h>0$.}\\
We treat this in two sub-cases. First, suppose that $\om(y_0)>\om_*$. Then we have $h'(y_0)<0$, and hence the existence of such a $\de>0$ is clear.

If, on the other hand, $\om(y_0)=\om_*$, we have that $h'(y_0)=h(y_0)=0$. By part (i), we must have $\om_*>\frac{4-3\ga}{3}$ (and hence we have $\ga>\frac{10}{9}$). Recall from \eqref{eq:rhoom} that
$$\rho'=\frac{y\rho h}{\ga\rho^{\ga-1}-y^2\om^2}.$$
So $$\rho''=\Big(\frac{y\rho}{\ga\rho^{\ga-1}-y^2\om^2}\Big)h'+\Big(\frac{y\rho}{\ga\rho^{\ga-1}-y^2\om^2}\Big)'h.$$
Thus also $\rho''(y_0)=\rho'(y_0)=0$.

Differentiating the middle line of \eqref{eq:h'} further, we use again $h(y_0)=h'(y_0)=0$ to see
\beqas
h''(y_0)=&\,\Big(2(\om')^2+2\om\om''-\frac{(\ga-1)(2-\ga)}{\om}\om''+\frac{(\ga-1)(2-\ga)}{\om^2}(\om')^2\Big)\Big|_{y=y_0}\\
=&\,\Big(2+\frac{(\ga-1)(2-\ga)}{\om(y_0)^2}\Big)(\om'(y_0))^2>0,
\eeqas
where we have used $\om(y_0)=\om_*$ in the second line and $\om'(y_0)<0$. But this forces $h$ to have a minimum at $y_0$, contradicting $h(y_0)=0$ and $h(y)<0$ for $y>y_0$. 

\textit{Step 2: Conclude the invariance of the region $\{h>0\}$.}\\
 Now for $y\in(y_0-\de,y_0)$, as $h>0$, we must have $\rho'>0$ and $\om'<0$ as we also have $\om>\frac{4-3\ga}{3}$. Thus, as we decrease $y$, we are in an invariant region as $\rho$ decreases and $\om$ increases, taking us further away from the level set $\{h=0\}$. Compare Figure \ref{fig:hlevelset}. Thus as long as the flow exists, we will retain in particular for $y\in(s(y_*),y_0)$ the inequality $\om(y)>\om_*>\frac{4-3\ga}{3}$.

\noindent (iii) Suppose $\limsup_{y\searrow y_1}\om(y_1)=\frac{4-3\ga}{3}$. Then we must have, for $y$ close to $y_1$,
$$y^2\om(y)^2<y_*^2\om_0^2=\ga\rho_0^{\ga-1}<\ga\rho(y)^{\ga-1},$$
where we have used that $\om_0>\frac{4-3\ga}{3}$ and also $\rho'<0$ on $(y_c(y_*),y_*)$ by part (ii). Thus the flow is still uniformly subsonic at $y_1$ and hence either $y_1=0$ or $s(y_*)<y_1$.
\end{proof}

\subsection{Properties of the fundamental set $Y$}
We begin by proving a basic topological property of $\mathcal{Y}$, and hence of $Y$: that the set is open.
\begin{lemma}
Let $\ga\in(1,\frac43)$. The set  $\mathcal{Y}$ is open. Therefore also $Y$ is the open interval $(\bar y_*,y_F)$.
\end{lemma}
\begin{proof}
Let $y_*\in\mathcal{Y}$. As $h(\cdot;y_*)<0$ on $[y_c(y_*),y_*-\nu]$ by Lemma \ref{lemma:leftinvariants}, we must have at $y_c(y_*)$ that $\om'(y_c(y_*);y_*)>0$, and so there exists $\tau>0$ such that $\om(y;y_*)<\frac{4-3\ga}{3}$ for $y\in(y_c(y_*)-2\tau,y_c(y_*))$ and $s(y_*)<y_c(y_*)-2\tau$ (by definition of $\mathcal{Y}$, $y_c(y_*)>s(y_*)$, so this may be achieved by taking $\tau$ smaller if necessary).  Let $\eps>0$ be sufficiently small. By parts (i) and (iii) of Proposition \ref{prop:sonictime}, there exists $\de>0$ such that for all $\tilde y_*\in[y_f,y_F]$ satisfying $|\tilde y_*-y_*|<\de$, we have $s(\tilde y_*)<y_c(y_*)-\frac32\tau$ and 
$$\big|(\rho(y;\tilde y_*),\om(y;\tilde y_*))-(\rho(y;y_*),\om(y;y_*))\big|<\eps$$
for all $y\in[y_c(y_*)-\tau,y_*-\nu]$. By taking $\eps<\frac{4-3\ga}{3}-\om(y_c(y_*)-\tau;y_*)$, we get that for all $\tilde y$ satisfying $|\tilde y_*-y_*|<\de$, 
$$\om(y_c(y_*)-\tau;\tilde y_*)<\frac{4-3\ga}{3},$$
and hence $\tilde y_*\in\mathcal{Y}$ also. We have shown that $\mathcal{Y}$ is open.

To show the claim for $Y$, we note that clearly $Y$ is a connected component of $\mathcal{Y}$. As $\mathcal{Y}$ is open, $Y$ is therefore an open interval. Moreover, by Lemma \ref{L:BRANCHES} and the continuity with respect to both $y$ and $y_*$ of Theorem \ref{thm:Taylor}, we have that there exist $\de_1>0$ and $\epsilon>0$ such that, for $y_F-\de_1\leq y_*\leq y_F$ and $y\in[y_*-\nu,y_*]$ we have $\om'(y)\geq\epsilon$. A simple continuity argument then reveals, as $\om(y_F;y_F)=\frac{4-3\ga}{3}$, there exists $\de>0$ such that $(y_F-\de,y_F)\subset\mathcal{Y}$. Thus $Y$ is non-empty and we have
$$Y=(\bar y_*,y_F).$$
\end{proof}
We wish to prove that the LPH-type solution associated to $\bar y_*$ exists on all of $(0,\bar y_*)$, i.e., that $s(\bar y_*)=0$. To prove this, we show the stronger property that, for all $y_*\in Y$, the function $\om(\cdot;y_*)$ remains strictly monotone on the set $[y_c(y_*),y_*]$. This is not simply a technical observation but is a key stage in constructing a globally defined LPH-type solution. In providing the additional qualitative information of monotonicity for $\om$, this represents a significant advance over earlier work in the isothermal case. We therefore make the following definition.
\begin{definition}
Let $\ga\in(1,\frac43)$. The set of $y_*\in Y$ for which the relative velocity $\om$ remains strictly monotone to the right of the critical time $y_c(y_*)$ is defined to be
\beq
\mathcal{S}:=\{y_*\in Y\,|\,\text{ for all }\tilde y_*\in[y_*,y_F),\:\om'(y;\tilde{y}_*)>0\text{ for all }y\in[y_c(\tilde{y}_*),\tilde{y}_*]\}.
\eeq
\end{definition}
Note that if $y_*$ is close to $y_F$, then the monotonicity holds on $[y_c(y_*),y_*]$ and $y_*\in\mathcal{S}$. 

The key property that we will now prove is that $\mathcal{S}=Y$. In addition to giving the monotonicity of $\om(\cdot;y_*)$ for all $y_*\in Y$, this also guarantees a uniform lower bound on the function $G$, and hence ensures that the flow remains strictly subsonic. Before stating and proving this result, we first note a technical lemma that will be essential for the proof.

\begin{lemma}\label{lemma:omega3}
Let $\ga\in(1,\frac43)$, $y_*\in[y_f,y_F]$ and let $(\rho,\om)$ be the associated unique LPH-type solution on $(s(y_*),y_*)$. Suppose that at a point $y_0\in(s(y_*),y_*)$ such that $\om(y_0)\in(\frac{4-3\ga}{3},2-\ga)$, we have that $\om'(y_0)=\om''(y_0)=0$. Then $\om^{(3)}(y_0)<0$.
\end{lemma}
The proof of this lemma is delayed until after Corollary~\ref{cor:Glowerbound} and the proof of Proposition \ref{prop:S=Y} further below.

%%%%%%%%%%%%%%%%%%%%%%%%%%%%
%%%%%%%%%%%%%%%%%%%%%%%%%%%%

\begin{prop}\label{prop:S=Y}
Let $\ga\in(1,\frac43)$. Then, for all $y_*\in Y$, the solution $(\rho(\cdot;y_*),\om(\cdot;y_*))$ defined by Theorem \ref{thm:Taylor} and extended to the interval $(s(y_*),y_*)$ satisfies $\om'(y;{y}_*)>0\text{ for all }y\in[y_c({y}_*),{y}_*]$, and so $$\mathcal{S}=Y.$$
\end{prop}

We note the following important corollary.

%%%%%%%%%%%%%%%%%%%%%%%
%%%%%%%%%%%%%%%%%%%%%%%

\begin{corollary}\label{cor:Glowerbound}
Let $\ga\in(1,\frac43)$, $y_*\in[y_f,y_F]$ and let $(\rho,\om)$ be the associated unique LPH-type solution on $(s(y_*),y_*)$. There exists $\eta>0$ such that, for all $y_*\in Y$, 
$$G(y;\rho(y;y_*),\om(y;y_*))\geq\eta>0\quad\text{ for all }y\in[y_c(y_*),y_*-\nu].$$
\end{corollary}

%%%%%%%%%%%%%%%%%%%%%%%

\begin{proof}
By continuity properties at the sonic point $y_*$ (from Theorem \ref{thm:Taylor}), there exist $\nu>0$ and $\eta>0$ (independent of $y_*$) such that $G(y_*-\nu;y_*)\geq\eta>0$ for all $y_*\in[y_f,y_F]$. Then, for any $y_*\in\mathcal{S}$, as $\om'>0$ and $\rho'<0$ on $[y_c(y_*),y_*]$, we retain $G(y;y_*)\geq\eta$ on $[y_c(y_*),y_*-\nu]$ as
$$\frac{\dif}{\dif y}G(y;\rho,\om)=\ga(\ga-1)\rho^{\ga-2}\rho'-2y\om^2-2y^2\om\om'<0.$$
Thus we have a uniform lower bound on $G$ for $y_*\in\mathcal{S}$ and, as $\mathcal{S}=Y$ by Proposition \ref{prop:S=Y}, we conclude.
\end{proof}

%%%%%%%%%%%%%%%%%%%%%%%%%%%%

\begin{proof}[Proof of Proposition \ref{prop:S=Y}]
We note by the proof of Corollary \ref{cor:Glowerbound} above that for $y_*\in\mathcal{S}$ we have a uniform lower bound $G(y;y_*)\geq\eta$ on $[y_c(y_*),y_*-\nu]$ for $y_*\in\mathcal{S}$.

Note in addition that \eqref{eq:rhoom} gives that $\om'=\mathcal{W}(y,\om,\rho)$ for some continuous function $\mathcal{W}$ away from sonic points. Continuity (respectively uniform continuity) of $\om$, $\rho$ etc with respect to $y$ or $y_*$ then leads to continuity (respectively uniform continuity) of $\om'$.

To conclude the proof of the Proposition, we will proceed in several steps to show that $\mathcal{S}$ is both open and relatively closed in $Y$.

\textit{Step 1: We first show that $\mathcal{S}$ is open.}\\
 Take $y_*\in\mathcal{S}$. Then we have the lower bounds $G\geq\eta$, $\om'\geq c_1$, $-h\geq c_2$ on $[y_c(y_*),y_*-\nu]$ for some $c_1,c_2>0$. By Lemma \ref{lemma:extension}, we can therefore extend the solution onto an interval $[y_c-\tau,y_*]$, where $\tau=\tau(\eta,y_c)>0$, and retain the inequality $\om'\geq \frac12 c_1>0$. By upper semi-continuity of the sonic time, there exists $\de>0$ such that if $|\tilde{y}_*-y_*|<\de$, we have 
$$s(\tilde{y}_*)<s(y_*)+\frac{\tau}{2}< y_c(y_*)-\frac{\tau}{2}.$$
Using that $\mathcal{S}\subset Y$ and open-ness of $Y$, by possibly shrinking $\de>0$, we may assume that if $|\tilde{y}_*-y_*|<\de$, then $\tilde{y}_*\in Y$ and that, by the uniform continuity property of Proposition \ref{prop:sonictime}(iii), $y_c(\tilde{y}_*)\geq y_c(y_*)-\frac{\tau}{4}$ and, as $\om'$ is a continuous function of $(y,\rho,\om)$, also $\om'(\cdot;y_*)>\frac{c_1}{4}$ on $[y_c(y_*)-\frac{\tau}{4},\tilde{y}_*]$, in particular, $\tilde{y}_*\in\mathcal{S}$.

\textit{Step 2: We collect properties associated to a sequence of $y_*^n\in \mathcal{S}$ with $y_*^n\to y_*\in Y$.}\\
To show $\mathcal{S}$ is relatively closed in $Y$, first suppose $y_*^n\in\mathcal{S}$ are such that $y_*^n\to y_*\in Y$. Clearly if any of the $y_*^n\leq y_*$, then also $y_*\in \mathcal{S}$. It therefore suffices to suppose that $y_*^n$ decreases monotonically to $y_*$. Suppose for a contradiction that there exists $y_0\in[y_c(y_*),y_*]$ such that $\om'(y_0;y_*)=0$. Clearly, as $h<0$ on $[y_c(y_*),y_*-\nu]$ and $\om'(y_*-\nu;y_*)>0$, we must have $y_0\in(y_c(y_*),y_*-\nu)$ (and we suppose without loss of generality that we are working with the largest such $y_0$). Moreover, as each of the $y_*^n\in\mathcal{S}$, we have the uniform lower bound $G(y;y_*^n)\geq \eta$ on $[y_c(y_*^n),y_*-\nu]$ (we have used the monotonicity of $y_*^n$ to replace the upper limit on the interval with $y_*-\nu$ rather than $y_*^n-\nu$).

Note that, by assumption, $y_*\in Y$. Therefore $y_c(y_*)>s(y_*)$.

\textit{Step 3: We show that there exists $\tau>0$ such that $\om'(y;y_*)<0$ on $(y_0-\tau,y_0)$ and $y_0-\tau>y_c(y_*)+\tau$.}\\
  By definition of $y_0$, we must have $\om''(y_0;y_*)\geq0$. If $\om''(y_0;y_*)>0$, the claim easily follows. On the other hand, as $y_*\in(y_f,y_F)$, then $\om(y_*;y_*)\in(\frac{4-3\ga}{3},2-\ga)$ and, by definition of $y_0\in(y_c(y_*),y_*)$, we see that $\om'(y;y_*)>0$ on $(y_0,y_*)$, leading to $\frac{4-3\ga}{3}<\om(y_0;y_*)<\om(y_*;y_*)<2-\ga$. Thus, by Lemma \ref{lemma:omega3}, if $\om'(y_0;y_*)=\om''(y_0;y_*)=0$, we have $\om^{(3)}(y_0;y_*)<0$. This then forces $\om'(\cdot;y_*)<0$ on a punctured interval centred at $y_0$, a contradiction. The existence of the claimed $\tau$ is proved.
  
\textit{Step 4: Apply uniform convergence to obtain a contradiction and deduce $\mathcal{S}$ is relatively closed.}\\
  Upper semi-continuity of the sonic time from Proposition \ref{prop:sonictime} again gives that, for $n$ sufficiently large, $s(y_*^n)<s(y_*)+\frac{\tau}{2}<y_c(y_*)+\frac{\tau}{2}<y_0-\frac{3\tau}{2}$. \\
Suppose for a contradiction that $\limsup_{n\to\infty}y_c(y_*^n)=\bar y_c> y_0-\tau$. Without loss of generality, we take a further subsequence $y_*^n$ such that $y_c(y_*^n)\to\bar y_c$. By Lemma \ref{lemma:extension}, there exists $T=T(\eta,\bar y_c)\in(0,\tau)$ such that 
  $$G(y;\rho(y;y_*^n),\om(y;y_*^n))\geq\frac12\eta\quad\text{ for }y\in[\bar y_c-T,y_*-\nu],\text{ all }n\in\N.$$
  Therefore, applying the uniform convergence of Proposition \ref{prop:sonictime}(iii), we obtain
  $$\om(\bar y_c-T;y_*)=\lim_{n\to\infty}\om(\bar y_c-T;y_*^n)\leq \frac{4-3\ga}{3},$$
  a contradiction to $y_c(y_*)<y_0-2\tau$ as $\bar y_c>y_0-\tau$ and $T<\tau$.
  
 Thus, for $n$ sufficiently large, we obtain that $y_c(y_*^n)\leq y_0-\frac{\tau}{2}$ and hence $\om'(y;y_*^n)>0$ on $(y_0-\frac{\tau}{2},y_0)$ as well as $G(y;\rho(y;y_*^n),\om(y;y_*^n))\geq\eta$ on $[y_0-\frac{\tau}{2},y_*-\nu]$. But this gives a contradiction to the convergence $$\om'(y_0-\frac{\tau}{4};y_*^n)\to\om'(y_0-\frac{\tau}{4};y_*)<0\text{ as }n\to\infty.$$
 Thus $y_*\in\mathcal{S}$ and so $\mathcal{S}$ is relatively closed in $Y$.
 
  As $\mathcal{S}$ is relatively open and closed in $Y$ and $Y$ is connected, we must therefore have $\mathcal{S}=Y$.
\end{proof}

%%%%%%%%%%%%%%%%%%%%%%%%%%%%
%%%%%%%%%%%%%%%%%%%%%%%%%%%%

\begin{proof}[Proof of Lemma \ref{lemma:omega3}]
\textit{Step 1: Derive identities for $\om(y_0)$, $h'(y_0)$ and $G'(y_0)$.}\\
 We begin by recalling from \eqref{eq:om(y0)1} and \eqref{eq:h'(y0)1} the identities
\begin{align}
\frac{4-3\ga-3\om(y_0)}{y_0}=&\,\frac{y_0\om(y_0)h(y_0)}{G(y_0)},\label{eq:om(y0)}\\
\frac{h'}{h}(y_0)=&\,-\frac{4\pi}{4-3\ga}\frac{y_0\rho(y_0)\om(y_0)}{G(y_0)}\\
=&\,\frac{yh}{G}-\frac{y(2\om^2+(\ga-1)\om+(\ga-1)(2-\ga))}{G}\notag\\
=&\,\frac{4-3\ga-3\om}{y\om}-\frac{y(2\om^2+(\ga-1)\om+(\ga-1)(2-\ga))}{G}.\label{eq:h'(y0)}
\end{align}
Also, from \eqref{eq:h'}, we recall that
\beqs
h'(y)=2\om\om'-\frac{4\pi}{4-3\ga}\om\rho'+\frac{h-(\ga-1)(2-\ga)}{\om}\om'.
\eeqs
Arguing directly, we differentiate $G$ to obtain
\beqas
G'=&\,(\ga-1)\ga\rho^{\ga-2}\rho'-2y\om^2-2y^2\om\om'\\
=&\,(\ga-1)\ga\rho^{\ga-1}\frac{yh}{G}-2y\om^2-2y^2\om\om'\\
=&\,(\ga-1)(G+y^2\om^2)\frac{yh}{G}-2y\om^2-2y^2\om\om'.
\eeqas
Thus, at $y_0$,
\beqa\label{eq:G'(y0)}
G'(y_0)=&\,(\ga-1)(G+y^2\om^2)\frac{4-3\ga-3\om}{y\om}-2y\om^2.
\eeqa
\textit{Step 2: Derive identities for $\om''(y_0)$ and solve for $\rho(y_0)$, $G(y_0)$ and $h(y_0)$.}\\ We now further differentiate the ODE for $\om$ to obtain
\beqas
\om''=&\,-\frac{3\om'}{y}-\frac{4-3\ga-3\om}{y^2}-\frac{\om h}{G}-\frac{y\om' h}{G}-\frac{y\om h'}{G}+\frac{y\om h G'}{G^2}.
\eeqas
Hence, at $y_0$, we find
\beqa
\om''(y_0)=-\frac{4-3\ga-3\om}{y^2}-\frac{\om h}{G}-\frac{y\om h'}{G}+\frac{y\om h G'}{G^2}=-2\frac{4-3\ga-3\om}{y^2}-\frac{y\om h'}{G}+\frac{y\om h G'}{G^2},
\eeqa
where we have used \eqref{eq:om(y0)} in the second equality. Recalling that at $y_0$ we have $\om''(y_0)=0$, this gives the identity
\beq\label{eq:om''}
\frac{y\om h G'}{G^2}-\frac{y\om h'}{G}=2\frac{4-3\ga-3\om}{y^2}.
\eeq
Applying \eqref{eq:om(y0)}, \eqref{eq:h'(y0)} and \eqref{eq:G'(y0)} to expand the left hand side, we find at $y_0$
\beqas
2\frac{4-3\ga-3\om}{y^2}=&\,\frac{y\om h}{G}\big(\frac{G'}{G}-\frac{h'}{h}\big)\\
=&\,\frac{4-3\ga-3\om}{y}\Big((\ga-1)(G+y^2\om^2)\frac{4-3\ga-3\om}{y\om G}-\frac{2y\om^2}{G}\\
& \hspace{23mm}-\frac{4-3\ga-3\om}{y\om}+\frac{y(2\om^2+(\ga-1)\om+(\ga-1)(2-\ga))}{G}\Big).
\eeqas
Simplifying, we find
\beqas
\frac{2}{y}=&\,(\ga-2)\frac{4-3\ga-3\om}{y\om}+\frac{(\ga-1)y\om(4-3\ga-3\om)+y((\ga-1)\om+(\ga-1)(2-\ga))}{G}\\
=&\,(\ga-2)\frac{4-3\ga-3\om}{y\om}+(\ga-1)y\frac{-3\om^2 +(5-3\ga)\om+2-\ga}{G},
\eeqas
which we rearrange to solve for $G(y_0)$ as
\beqs
G(y_0)\frac{(4-3\ga)(2-\ga-\om)}{y\om}=(\ga-1)y\big(-3\om^2 +(5-3\ga)\om+2-\ga\big)=y(\ga-1)(2-\ga-\om)(3\om+1),
\eeqs
so that
\beq\label{eq:G(y0)}
G(y_0)=y^2\om\frac{(\ga-1)(3\om+1)}{4-3\ga}.
\eeq
Note therefore that
\beq\label{eq:rho(y0)}
\ga\rho^{\ga-1}(y_0)=G(y_0)+y_0^2\om(y_0)^2=y^2\om(\frac{(\ga-1)(3\om+1)}{4-3\ga}+\om)=y^2\om\frac{\om+\ga-1}{4-3\ga}
\eeq
and, from \eqref{eq:om(y0)},
\beq\label{eq:h(y0)}
h(y_0)=\frac{4-3\ga-3\om}{y^2\om}G=\frac{(\ga-1)(3\om+1)(4-3\ga-3\om)}{4-3\ga}.
\eeq
Therefore also
\beqa\label{eq:rho(y0)2}
\rho(y_0)=&\,-\frac{(4-3\ga)\big(h-2\om^2-(\ga-1)\om-(\ga-1)(2-\ga)\big)}{4\pi\om}\\
=&\,\frac{\om^2(3\ga-1)+\om(6\ga-5)(\ga-1)-(\ga-1)^2(4-3\ga)}{4\pi\om}.
\eeqa
\textit{Step 3: Collect necessary identities for $h''(y_0)$ and $G''(y_0)$.}\\
 To compute $\om^{(3)}(y_0)$, we first need $h''(y_0)$ and $G''(y_0)$. Clearly, from \eqref{eq:h'}, we have
\beqas
h''(y_0)=&\,-\frac{4\pi}{4-3\ga}\om\rho''=-\frac{4\pi}{4-3\ga}\om\Big(\frac{\rho h+y\rho' h}{G}+\frac{y\rho h'}{G}-\frac{y\rho h G'}{G^2}\Big)\\
=&\,-\frac{4\pi}{4-3\ga}\om\Big(\frac{\rho(4-3\ga-3\om)}{y^2\om}+\frac{4-3\ga-3\om}{y\om}\frac{y\rho h}{G}-2\frac{\rho(4-3\ga-3\om)}{y^2\om}\Big)\\
=&\,-\frac{4\pi}{4-3\ga}\rho\Big(-\frac{4-3\ga-3\om}{y^2}+\frac{(4-3\ga-3\om)^2}{y^2\om}\Big),
\eeqas
where we have used \eqref{eq:om''} in the middle line and \eqref{eq:om(y0)} repeatedly.

Similarly, we compute $G''(y_0)$ as
\beqas
G''&\,(y_0)=\ga(\ga-1)^2\rho^{\ga-1}\frac{y^2h^2}{G^2}+\ga(\ga-1)\rho^{\ga-1}\Big(\frac{h}{G}+\frac{yh'}{G}-\frac{yhG'}{G^2}\Big)-2\om^2\\
=&\,(G+y^2\om^2)\Big((\ga-1)^2\frac{(4-3\ga-3\om)^2}{y^2\om^2}+(\ga-1)\Big(\frac{4-3\ga-3\om}{y^2\om}-2\frac{4-3\ga-3\om}{y^2\om}\Big)\Big)-2\om^2\\
=&\,\frac{(\om+\ga-1)(4-3\ga-3\om)}{(4-3\ga)\om}\big((\ga-1)^2(4-3\ga-3\om)-(\ga-1)\om\big)-2\om^2,
\eeqas
where we have again used \eqref{eq:om''} in the middle line and \eqref{eq:rho(y0)} in the last line.

\textit{Step 4: Conclude an identity for $\om^{(3)}(y_0)$ and prove the sign condition.}\\
 Finally, we compute $\om^{(3)}(y_0)$:
\beqas
\om^{(3)}(y_0)=&\,2\frac{4-3\ga-3\om}{y^3}-2\frac{\om h'}{G}+2\frac{\om hG'}{G^2}-\frac{y\om h''}{G}+2\frac{y\om h'G'}{G^2}+\frac{y\om h G''}{G^2}-2\frac{y\om h (G')^2}{G^3}\\
=&\,6\frac{4-3\ga-3\om}{y^3}-4\frac{G'}{G}\frac{4-3\ga-3\om}{y^2}-\frac{h''}{h}\frac{4-3\ga-3\om}{y}+\frac{G''}{G}\frac{4-3\ga-3\om}{y},
\eeqas
by using again \eqref{eq:om''}. Substituting in the identities for $h''(y_0)$, $G''(y_0)$, we get
\beqa\label{eq:om3bar}
&\,\frac{\om^{(3)}(y_0)y_0^3}{4-3\ga-3\om}\\
&\,=6-4y\big((\ga-1)(1+\frac{y^2\om^2}{G})\frac{4-3\ga-3\om}{y\om}-\frac{2y\om^2}{G}\big)\\
&\quad+\frac{4\pi}{4-3\ga}\frac{y^2\rho}{h}\Big(-\frac{4-3\ga-3\om}{y^2}+\frac{(4-3\ga-3\om)^2}{y^2\om}\Big)\\
&\quad+\frac{y^2}{G}\Big(\frac{(\om+\ga-1)(4-3\ga-3\om)}{(4-3\ga)\om}\big((\ga-1)^2(4-3\ga-3\om)-(\ga-1)\om\big)-2\om^2\Big).
\eeqa
By inserting \eqref{eq:G(y0)} for $G(y_0)$, \eqref{eq:h(y0)} for $h(y_0)$, and \eqref{eq:rho(y0)2} for $\rho(y_0)$, this becomes a polynomial in $\om$ with coefficients depending on $y_0$. Taking it term-by-term, we substitute \eqref{eq:G(y0)} into the second term to find
\beqa
-&4y\big((\ga-1)(1+\frac{y^2\om^2}{G})\frac{4-3\ga-3\om}{y\om}-\frac{2y\om^2}{G}\big)\\
&=-4y\bigg((\ga-1)\Big(1+\frac{(4-3\ga)\om}{(\ga-1)(3\om+1)}\Big)\frac{4-3\ga-3\om}{y\om}-\frac{2\om(4-3\ga)}{y(\ga-1)(3\om+1)}\bigg)\\
&=-4\frac{\om^3(3\ga-5)-\om^2(6\ga-7)(\ga-1)+\om(\ga-1)^2(4-3\ga)}{(\ga-1)\om^2(3\om+1)}.
\eeqa
For the third term, we use \eqref{eq:h(y0)} and \eqref{eq:rho(y0)2} to get
\beqas
&\frac{4\pi}{4-3\ga}\frac{y^2\rho}{h}\Big(-\frac{4-3\ga-3\om}{y^2}+\frac{(4-3\ga-3\om)^2}{y^2\om}\Big)\\
&\:=(4-3\ga-4\om)\frac{\om^2(3\ga-1)+\om(6\ga-5)(\ga-1)-(\ga-1)^2(4-3\ga)}{(\ga-1)\om^2(3\om+1)}\\
&\:=\frac{-4(3\ga-1)\om^3-(33\ga^2-59\ga+24)\om^2+(\ga-1)(4-3\ga)(10\ga-9)\om-(\ga-1)^2(4-3\ga)^2}{(\ga-1)\om^2(3\om+1)}.
\eeqas
For the last term, we again substitute \eqref{eq:G(y0)} to get
\beqas
\frac{y^2}{G}&\Big(\frac{(\om+\ga-1)(4-3\ga-3\om)}{(4-3\ga)\om}\big((\ga-1)^2(4-3\ga-3\om)-(\ga-1)\om\big)-2\om^2\Big)\\
=&\,\frac{4-3\ga}{(\ga-1)\om(3\om+1)}\Big(\frac{(\om+\ga-1)(4-3\ga-3\om)}{(4-3\ga)\om}\big((\ga-1)^2(4-3\ga-3\om)-(\ga-1)\om\big)-2\om^2\Big)\\
=&\,\frac{(4-3\ga)^2(\ga-1)^3-9(\ga-1)^3(4-3\ga)\om+\big(27(\ga-1)^3-(\ga-1)\big)\om^2+\big(9\ga^2-9\ga-2\big)\om^3}{(\ga-1)\om^2(3\om+1)}.
\eeqas
Substituting in all of these identities and simplifying, we find
\beqa
&\frac{\om^{(3)}(y_0)y_0^3}{4-3\ga-3\om}\\
&=\frac{-(4 - 3 \ga)}{(\ga-1) \om^2 (1 + 3 \om)}\\
&\quad\times\Big((3\ga-1)\om^3+(9\ga^2-18\ga+7)\om^2+(\ga-1)(9\ga^2-24\ga+14)\om+(\ga-1)^2(2-\ga)(4-3\ga)\Big)
\eeqa
It is simple to verify that the roots of the cubic in $\om$ on the right hand side are
$$\om=-(\ga-1),\:\frac{(4-3\ga)(\ga-1)}{3\ga-1},\:2-\ga,$$
and so, as $\frac{(4-3\ga)(\ga-1)}{3\ga-1}<\frac{4-3\ga}{3}$ for all $\ga\in(1,\frac{4}{3})$, we easily see that for $\om\in(\frac{4-3\ga}{3},2-\ga)$, the right hand side of this formula is strictly positive. As $4-3\ga-3\om(y_0)<0$, this yields $\om^{(3)}(y_0)<0$, as required.
\end{proof}

%%%%%%%%%%%%%%%%%%%%%%%
%%%%%%%%%%%%%%%%%%%%%%%

\begin{remark}
The arguments of Proposition \ref{prop:S=Y} may be extended also to the isothermal case, $\ga=1$, treated previously in \cite{Guo20}, to show that the obtained Larson-Penston solution 
satisfies the inequality
\[
\om'>0.
\]
This can be seen by following the proof of Lemma \ref{lemma:omega3} with $\ga=1$. It can be seen that it is impossible to have $\om(y_0)\in(\frac13,1)$ and  $\om'(y_0)=\om''(y_0)=0$ simultaneously. Indeed, computing as far as \eqref{eq:om''} and making the necessary substitutions as in the following equation, the fact that $G$ is independent of $\rho$ when $\ga=1$ allows us to solve directly for $\om(y_0)$ and find either $\om(y_0)=\frac13$ or $\om(y_0)=1$. We then follow the proof of Proposition \ref{prop:S=Y} to obtain the monotonicity of $\om$  in the isothermal case $\ga=1$.
\end{remark}

%%%%%%%%%%%%%%%%%%%%%%%%%%%%

The next key result in this section is to show that the LPH-type solution associated to the critical value $\bar y_*$ exists on the whole of $(0,\bar y_*)$ and hence is a global solution of \eqref{eq:EPSS}. This is the content of the following proposition.

%%%%%%%%%%%%%%%%%%%%%%%%%%%%
%%%%%%%%%%%%%%%%%%%%%%%%%%%%

\begin{prop}\label{P:LEFTGLOBAL}
Let $\ga\in(1,\frac43)$. The sonic time and critical time associated to $\bar y_*$ satisfy $s(\bar y_*)=y_c(\bar y_*)=0$.
\end{prop}

%%%%%%%%%%%%%%%%%%%%%%%%%%%%

\begin{proof}
As in \cite[Proposition 4.12]{Guo20}, there are 3 cases. \\
\textit{Case 1:} $y_c(\bar y_*)=0$. Then we are done as, by definition, $s(\bar y_*)\leq y_c(\bar y_*)$. \\
\textit{Case 2:} $y_c(\bar y_*)>s(\bar y_*)\geq 0$. Then by continuity of the solution, we must have $\om(y_c(\bar y_*);\bar y_*)=\frac{4-3\ga}{3}$, and hence $\bar y_*\in Y$, a contradiction to $Y=(\bar y_*,y_F)$. \\
\textit{Case 3:} $y_c(\bar y_*)=s(\bar y_*)>0$. 
Now take a sequence $y_*^n\to \bar y_*$ such that all $y_*^n\in Y$. Then by definition of $Y$, $y_c(y_*^n)>s(y_*^n)$ for all $n\in\N$. We define $$\bar y_c=\limsup(y_c(y_*^n)).$$
Without relabelling, we take a subsequence such that $y_c(y_*^n)\to \bar y_c$. Then from Lemma \ref{lemma:extension} and Proposition \ref{prop:S=Y}, we know that there exist $\eta>0$ and $\tau=\tau(\eta,\bar y_c)>0$ such that for all $n$ sufficiently large
$$G(y;\rho(y;y_*^n),\om(y;y_*^n))\geq\eta>0\quad\text{ for all }y\in[\bar y_c-\tau,y_*-\nu].$$
From Proposition \ref{prop:sonictime}(ii), we therefore find that, possibly shrinking $\tau$, we have $s(\bar y_*),s(y_*^n)<\bar y_c-\tau$ for all $n$. Therefore, using the uniform convergence of Proposition \ref{prop:sonictime}(iii) on the interval $[\bar y_c-\tau,y_*-\nu]$, we find that the limit 
$$\om(\bar y_c;\bar y_*)=\lim_{n\to\infty}\om(y_c(y_*^n),y_*^n)=\frac{4-3\ga}{3},$$
and thus $y_c(\bar y_*)\geq \bar y_c>s(\bar y_*)$, a contradiction to the assumption $y_c(\bar y_*)=s(\bar y_*)$.
\end{proof}

%%%%%%%%%%%%%%%%%%%%%%%
%%%%%%%%%%%%%%%%%%%%%%%

\subsection{Asymptotics at the scaling origin, $y=0$}

%%%%%%%%%%%%%%%%%%%%%%%
%%%%%%%%%%%%%%%%%%%%%%%

It is straightforward to exploit the uniform convergence property of Proposition \ref{prop:sonictime} to obtain the weak monotonicity of $\om(\cdot;\bar y_*)$. However, in order to obtain the strict monotonicity and the correct boundary value at the origin, $y=0$, we must rule out the possibility that $\bar y_*=y_f$.
\begin{lemma}\label{lemma:y*notyf}
Let $\ga\in(1,\frac43)$. The critical sonic  point $\bar y_*$  is not equal to $y_f$. In particular, the global solution $(\rho(\cdot;\bar y_*),\om(\cdot;\bar y_*))$ is not the far-field solution $(\rho_f,\om_f)$, defined in~\eqref{def:far-field}.
\end{lemma}

Before presenting the proof of this lemma, we collect some identities for an important auxiliary function.

\begin{lemma}
Let $\ga\in(1,\frac43)$, $y_*\in[y_f,y_F]$ and let $(\rho,\om)$ be the associated unique LPH-type solution on $(s(y_*),y_*)$. We define a function
\beq
f(y)=\frac{4\pi}{\ga(4-3\ga)}y^2\om\rho^{2-\ga}-\frac{2}{2-\ga}.
\eeq
Then the following identity holds for $f(y)$:
\beqa\label{eq:f'2}
f'(y)=&\,\frac{4\pi}{\ga(4-3\ga)}y\rho^{2-\ga}\bigg(f(y)(\ga-1)\Big(2-\ga+\frac{y^2\om^3}{G}\Big)\\
&+ (2-\ga-\om)\Big(1 -(\ga-1)\frac{4\pi}{\ga(4-3\ga)}y^2\om\rho^{2-\ga} -(\ga-1)\frac{y^2\om(\ga-1)(\frac{2\om}{2-\ga}+1)}{G}\Big)\bigg).
\eeqa
\end{lemma}
\begin{remark}
The principal utility of the function $f$ is in comparing the density of an LPH-type solution to the density of the far-field solution, $\rho_f$. Indeed, by construction (compare~\eqref{def:far-field}),  $$\frac{4\pi}{\ga(4-3\ga)}y^2\om_f\rho_f^{2-\ga}-\frac{2}{2-\ga}\equiv0.$$ Moreover, for $y_*\in(y_f,y_F]$, we have $f(y_*)>0$ by \eqref{ineq:f(y*)}. 
\end{remark}

\begin{proof}
Let $y_*\in[y_f,y_F]$ and let $(\rho,\om)=(\rho(\cdot;y_*),\om(\cdot;y_*))$. 
Direct differentiation yields
\beqa\label{eq:f'}
f'(y)=&\,\frac{4\pi}{\ga(4-3\ga)}\Big(2y\om\rho^{2-\ga}+y(4-3\ga-3\om)\rho^{2-\ga}-\frac{y^3\om h\rho^{2-\ga}}{G}+\frac{(2-\ga)y^3\om h\rho^{2-\ga}}{G}\Big)\\
=&\,\frac{4\pi}{\ga(4-3\ga)}y\rho^{2-\ga}\Big(2-\ga-\om+(\ga-1)\big(-\frac{y^2\om h}{G}-2\big)\Big).
\eeqa
Next, we rearrange the equation for $f'$. We  expand
\begin{align*}
-&\frac{y^2\om h}{G}-2\\
=&\,\frac{\frac{4\pi}{4-3\ga}y^2\om^2\rho-y^2\om\big(2\om^2+(\ga-1)\om+(\ga-1)(2-\ga)\big)}{\ga\rho^{\ga-1}-y^2\om^2}-2\\
=&\,\frac{4\pi}{\ga(4-3\ga)}y^2\om^2\rho^{2-\ga}-2+\frac{y^2\om^2\frac{4\pi}{4-3\ga}y^2\om^2\rho}{\ga\rho^{\ga-1}(\ga\rho^{\ga-1}-y^2\om^2)}-\frac{y^2\om\big(2\om^2+(\ga-1)\om+(\ga-1)(2-\ga)\big)}{\ga\rho^{\ga-1}-y^2\om^2}\\
=&\,(2-\ga)f(y)-(2-\ga-\om)\frac{4\pi}{\ga(4-3\ga)}y^2\om\rho^{2-\ga} \\
&+\frac{y^2\om^3\big(f(y)+\frac{2}{2-\ga}\big)}{\ga\rho^{\ga-1}-y^2\om^2}-\frac{y^2 \om \big(2\om^2+(\ga-1)\om+(\ga-1)(2-\ga)\big)}{\ga\rho^{\ga-1}-y^2\om^2}\\
=&\,(2-\ga)f(y)-(2-\ga-\om)\frac{4\pi}{\ga(4-3\ga)}y^2\om\rho^{2-\ga} + \frac{y^2\om^3f(y)}{G} \\
&+\frac{y^2\om\big(\frac{2(\ga-1)}{2-\ga}\om^2-(\ga-1)\om-(\ga-1)(2-\ga)\big)}{G}.
\end{align*}
Note that
$$\frac{2(\ga-1)}{2-\ga}\om^2-(\ga-1)\om-(\ga-1)(2-\ga)=-(\ga-1)(2-\ga-\om)(\frac{2\om}{2-\ga}+1).$$
Therefore, substituting this into \eqref{eq:f'}, we have
\beqas
f'(y)=&\,\frac{4\pi}{\ga(4-3\ga)}y\rho^{2-\ga}\bigg(f(y)(\ga-1)\Big(2-\ga+\frac{y^2\om^3}{G}\Big)\\
&+ (2-\ga-\om)\Big(1 -(\ga-1)\frac{4\pi}{\ga(4-3\ga)}y^2\om\rho^{2-\ga} -(\ga-1)\frac{y^2\om(\ga-1)(\frac{2\om}{2-\ga}+1)}{G}\Big)\bigg),
\eeqas
that is, \eqref{eq:f'2}.
\end{proof}

\begin{proof}[Proof of Lemma \ref{lemma:y*notyf}]
\textit{Step 1: Setup for a contradiction argument.}\\
 Suppose for a contradiction that $\bar y_*=y_f$, so that $Y=(y_f,y_F)$. We will use the fact that for any $y_*\in(y_f,y_F)$, we have $\om'(\cdot;y_*)\geq0$ on $[y_c(y_*),y_*]$ by Proposition \ref{prop:S=Y}, and so on this interval, $2-\ga-\om(\cdot;y_*)>0$.
Along with \eqref{eq:f'2}, we  also note
\beqa\label{eq:2-g-om}
\big(2&-\ga-\om\big)'
=-\frac{\frac{4-3\ga}{2-\ga}\big(2-\ga-\om\big)}{y}+\frac{\om}{y}\frac{(\ga-1)y^2\big(\frac{2\om}{2-\ga}+1\big)}{G}(2-\ga-\om)-\frac{\om}{y}\frac{\frac{4\pi y^2\rho\om}{4-3\ga}-\frac{2}{2-\ga}\ga\rho^{\ga-1}}{G},
\eeqa
which is a reformulation of~\eqref{eq:2-ga-om}.

\textit{Step 2: Collect initial estimates for $f$ and $2-\ga-\om$ and define the basic set for a continuity argument to propagate the estimates.}\\
  Let $\eps>0$, $\al>0$ and $y_0>0$ be sufficiently small (to be fixed later), then by Proposition \ref{prop:sonictime} there exists $\de>0$ such that if $y_*-y_f<\de$, we have 
\beq\label{ineq:finitial}
|2-\ga-\om(y_0)|+ A|f(y_0)|<\eps,\quad \rho(y_0)>\rho_f(y_0)-\eps>M,\eeq
where $A>\max\{\frac{(2-\ga)^2}{\ga-1},1\}$ is a fixed, $\ga$-dependent constant and $M$ is assumed sufficiently large so that $\rho(y_0)>M$ and $\om(y_0)\in(0,2-\ga)$ implies $\frac{1}{G(y_0)}<\al$. Moreover, by upper semi-continuity of the sonic time $s(y_*)$ from Proposition \ref{prop:sonictime}(i), as $s(y_f)=0$, we may take $|y_*-y_f|<\de$ with $\de$ sufficiently small so that $s(y_*)\leq \frac{y_0}{8}$. Using now the uniform continuity from Proposition \ref{prop:sonictime}(iii) for $y\geq\frac{y_0}{4}$, we may take $\de$ smaller if necessary to ensure $|\om(y;y_*)-(2-\ga)|$ is small enough that $\om(y;y_*)>\frac{4-3\ga}{3}$ for $y\in[\frac{y_0}{4},y_*]$ and hence also $y_c(y_*)<\frac{y_0}{2}$ giving, in total,
$$s(y_*)<y_c(y_*)<\frac{y_0}{2},$$
where the first inequality follows from $y_*\in Y$ (so that $y_c(y_*)>s(y_*)$).

We take $y_0$ small enough (depending only on $\ga$) so that in all of the (finitely many) positive constants $C=C(\ga)$ below depending only on $\ga$, $y_0<C$.

Let the set $F$ be defined as 
$$F=\{y\in(s(y_*),y_0]\,|\,\om(\tilde y)\geq2-\ga-C_0\eps,\:-C_1\eps\leq f(\tilde y)\leq |f(y_0)|\text{ for all }\tilde y\in[y,y_0]\},$$
where $C_1>C_0>1$ depend only on $\ga$ (and are to be chosen later). By taking $C_0>1$ and $C_1>\frac{1}{A}$, we have by \eqref{ineq:finitial} that $y_0\in F$, so that $F$ is clearly non-empty and relatively closed. 

We will assume $\eps>0$ is small enough so that $2-\ga-C_0\eps>\frac34(2-\ga)>\frac{4-3\ga}{3}$. Note that if $y\in F$, then as $y_*\in Y$ and $\om(y)>\frac{4-3\ga}{3}$, we must have $\rho(y)>\rho(y_0)$ by Lemma \ref{lemma:leftinvariants}, and so also $\frac{1}{G(y)}<\al$.

Our goal is to prove that $F=(s(y_*),y_0]$ (by showing that $F$ is relatively open in $(s(y_*),y_0]$). This then gives $\inf_{(s(y_*),y_*)}\om(\cdot;y_*)\geq2-\ga-C_0\eps>\frac{4-3\ga}{3}$, a contradiction to $y_*\in Y$.

\textit{Step 3: Show that $f<0$ is an invariant property as $y$ decreases and partition the set $F$.}\\
 Now for any $\bar y\in F$ such that $0\leq f(\bar y)<\eps$, we use \eqref{eq:f'2} along with the uniform bound on $\om$ and the estimate $G^{-1}<\al$ to see that
\beqa\label{ineq:f'lower}
f'(\bar y)=&\,\frac{4\pi}{\ga(4-3\ga)}\bar y\rho^{2-\ga}\bigg(f(\bar y)(\ga-1)\Big(2-\ga+O(\al \bar y^2)\Big)\\
&+ (2-\ga-\om)\Big(1 -(\ga-1)\frac{4\pi}{\ga(4-3\ga)}y^2\om\rho^{2-\ga} +O(\al \bar y^2)\Big)\bigg)\\
\geq&\frac{4\pi}{\ga(4-3\ga)}\bar y\rho^{2-\ga}(2-\ga-\om)\Big(1 -(\ga-1)\frac{2}{2-\ga} +O(|f(\bar y)|+\al \bar y^2)\Big)>0
\eeqa
as $1 -(\ga-1)\frac{2}{2-\ga}=\frac{4-3\ga}{2-\ga}>0$, $|f(\bar y)|<\eps$ and $0<\bar y\leq y_0$ is small. Thus the region $\{f(y)<0\}$ is an invariant region in $F$.

In particular, we may define a point $y_1$ such that $\inf F\leq y_1\leq y_0$ as follows:
\beq
y_1=\begin{cases}
\inf\{y\in F\,|\, f(y)>0\} &\text{ if }f(y_0)>0,\\
y_0 &\text{ if }f(y_0)\leq 0.
\end{cases}
\eeq If $f(y_0)>0$, we therefore have (by the invariance of $\{f(y)<0\}$) that $f(y)< 0$ for $y\in[\inf F,y_1)$, $f(y)> 0$ for $y\in(y_1,y_0]$. On the other hand, if $f(y_0)\leq0$, then $f(y)<0$ for all $y\in F\setminus\{y_0\}$.

In addition, we conclude that $F$ is not a singleton set as follows: if $f(y_0)\geq 0$, then we have from the inequality just shown for $f'(y_0)$ that there is an interval to the left of $y_0$ such that $f(y)<f(y_0)$ and the other defining inequalities of $F$ follow from simple continuity considerations. If $f(y_0)< 0$, then the upper bound $f(y)<|f(y_0)|$ follows trivially on an open neighbourhood of $y_0$, while the other defining estimates for $F$ likewise follow from simple continuity considerations on an open neighbourhood of $y_0$. This yields in particular that 
$$\inf F<y_0.$$
\textit{Step 4: Obtain a uniform lower bound $f(y)>-C_1\eps$ on $F$.}\\
 We  note the identity
\be\label{E:HILFE}
\frac{\om}{y}\frac{\frac{4\pi y^2\rho\om}{4-3\ga}-\frac{2}{2-\ga}\ga\rho^{\ga-1}}{G}=\frac{\om}{y}f(y)+\frac{\om}{y}f(y)\frac{y^2\om^2}{G},
\ee 
and then use \eqref{eq:2-g-om} along with $G>0$ and $\om'>0$ (as $y_*\in Y=S$ by Proposition \ref{prop:S=Y} gives $\om'>0$ on $[y_c(y_*),y_*]$ which contains $\overline{F}$) to see
\beq
f(y)\geq-\frac{4-3\ga}{2-\ga}\frac{(2-\ga-\om)}{\om}-f(y)\frac{y^2\om^2}{G}\text{ for all }y\in F.
\eeq
Using now that $\om(y)\in(\frac{3}{4}(2-\ga),2-\ga)$ and $G>0$, if $f(y)<0$, then this estimate yields 
\beq
f(y)\geq-\frac{4-3\ga}{2-\ga}\frac{(2-\ga-\om)}{\om},
\eeq
while if $f(y)\geq0$, then this estimate holds trivially (as the right hand side is negative due to $\om(y)<2-\ga$). Thus, we have obtained
\beq\label{ineq:flower}
f(y)\geq-\frac{4-3\ga}{2-\ga}\frac{(2-\ga-\om)}{\om}\geq -C_2C_0\eps\ \text{ for all }y\in F,
\eeq
where $C_2$ depends only on $\ga$ as we have assumed the estimate $\om\geq\frac{3(2-\ga)}{4}$, and $C_1$ was chosen originally so that $C_1>C_2C_0$.

 \textit{Step 5: Obtain the uniform bound $2-\ga-\om(y)<\eps$ on $[y_1,y_0]$.}\\
If $f(y_0)\leq 0$, then, by definition of $y_1$, we have $y_1=y_0$ and the inequality follows trivially. 

Suppose that $f(y_0)>0$. Then $y_1\in[\inf F,y_0)$. We  then have from \eqref{ineq:f'lower} that for all $y\in[y_1,y_0]$, $f'(y)> 0$, and so $0\leq f(y)< f(y_0)<\eps$ for all $y\in[y_1,y_0)\cap F$.
 
 We recall the constant $A>\max\{\frac{(2-\ga)^2}{\ga-1},1\}$ is a fixed, $\ga$-dependent constant and consider the quantity
$$g_A(y)=Af(y)+(2-\ga-\om(y)).$$
Using \eqref{eq:f'2},~\eqref{eq:2-g-om}, and~\eqref{E:HILFE} we get 
\beqa
g_A'(y)=&\,A\frac{f(y)+\frac{2}{2-\ga}}{y\om}\bigg(f(y)(\ga-1)\Big(2-\ga+\frac{y^2\om^3}{G}\Big)\\
&+ (2-\ga-\om)\Big(1 -(\ga-1)\frac{4\pi}{\ga(4-3\ga)}y^2\om\rho^{2-\ga} -(\ga-1)\frac{y^2\om(\ga-1)(\frac{2\om}{2-\ga}+1)}{G}\Big)\bigg)\\
&-\frac{\frac{4-3\ga}{2-\ga}\big(2-\ga-\om\big)}{y}+\frac{\om}{y}\frac{(\ga-1)y^2\big(\frac{2\om}{2-\ga}+1\big)}{G}(2-\ga-\om)-\frac{\om}{y}f(y)\Big(1+\frac{y^2\om^2}{G}\Big).
\eeqa
By writing $\om^{-1}=\frac{1}{2-\ga}+O(|2-\ga-\om|)$, we treat terms that are quadratic in $f(y)$ and $2-\ga-\om(y)$ as higher order and recall $0<\om<2-\ga$, $G^{-1}<\al$ where $\al$ is small to rearrange this as
\beqa\label{eq:gA'}
g_A'(y)= &\,\frac{Af(y)}{y}\Big(\frac{2(\ga-1)}{2-\ga}-\frac{2-\ga}{A}+O\big(|f(y)|+|2-\ga-\om(y)|+y^2\big)\Big)\\
&+\frac{2-\ga-\om}{y}\Big(\frac{2A(4-3\ga)}{(2-\ga)^3}-\frac{4-3\ga}{2-\ga}+O\big(|f(y)|+|2-\ga-\om(y)|+y^2\big)\Big). 
\eeqa
For $y\in[y_1,y_0]\cap F$, as $f(y)\geq0$ and $A>\frac{(2-\ga)^2}{\ga-1}$, this gives us $g_A'(y)\geq0$  (using both $|f(y)|+|2-\ga-\om|\leq C\eps$ and $y_0$ small relative to $\ga$), and hence $g(y)\leq g(y_0)$ on this interval. In particular, we obtain
\beq\label{ineq:2-ga-omonF}
2-\ga-\om(y)<\eps\text{ for all }y\in[y_1,y_0]\cap F,
\eeq
and so clearly $[y_1,y_0]\subset F$ (using also~\eqref{ineq:flower}). 

If $y_1=\inf F$, the strict inequality, along with \eqref{ineq:flower} (recall $C_1>C_2C_0$ by definition), shows that $F$ is also relatively open in $(s(y_*),y_0]$, i.e., $F=(s(y_*),y_0]$, and hence we conclude $y_c(y_*)=s(y_*)$ and $\inf_{(s(y_*),y_0)} \om>\frac{4-3\ga}{3}$, a contradiction to $y_*\in Y$.

\textit{Step 6: Obtain the final remaining estimate $2-\ga-\om(y)<4\eps$ on $[\inf F,y_1]$.}\\
 We now suppose that $y_1>\inf F$ (as we are already done by Step 5 if not) and work with either the case $f(y_0)>0$ or the alternative, $f(y_0)\leq 0$. Then the interval $[\inf F,y_1]\cap F$ is non-empty and non-singleton. 

By definition of $y_1$, for $y\in[\inf F,y_1)$, we trivially have the estimate $f(y)< 0\leq |f(y_0)|$. 

Choosing $\tilde A=\frac{(2-\ga)^2}{2}+a$, where $a>0$ will be taken small depending only on $\ga$, we obtain from \eqref{eq:gA'} that
\beqas
g_{\tilde{A}}'(y)=&\,\frac{\tilde{A}f(y)}{y}\Big(-2+O\big(a+|f(y)|+|2-\ga-\om(y)|+y^2\big)\Big)\\
&+\frac{2-\ga-\om}{y}\Big(\frac{2a(4-3\ga)}{(2-\ga)^3}+O\big(|f(y)|+|2-\ga-\om(y)|+y^2\big)\Big)\geq0,
\eeqas
on $F$, where we have used that $f<0$ on $[\inf F,y_1)$ and $2-\ga-\om>0$.

Thus, for $y\in[\inf F,y_1)$,
\beqas
(2-\ga-\om)(y)\leq&\, g_{\tilde{A}}(y_1)-\big(\frac{(2-\ga)^2}{2}+a\big)f(y)\\
\leq &\,(2-\ga-\om)(y_1)+\big(\frac{(2-\ga)^2}{2}+a\big)\frac{4-3\ga}{2-\ga}\frac{2-\ga-\om(y)}{\om(y)},
\eeqas
where we have used the first bound in~\eqref{ineq:flower}. Noting that the coefficient $\big(\frac{(2-\ga)^2}{2}+a\big)\frac{4-3\ga}{2-\ga}\frac{1}{\om}\leq\frac{3}{4}$ provided $\om>\frac{3(2-\ga)}{4}$ and  $a$ is small, depending only on $\ga$, we absorb the last term on the right onto the left and conclude that 
$$(2-\ga-\om)(y)\leq 4(2-\ga-\om)(y_1)<4\eps,$$
where the last estimate follows from  \eqref{ineq:2-ga-omonF} in the case $y_1<y_0$ and \eqref{ineq:finitial} in the case $y_1=y_0$. So provided $C_0>4$ initially, we obtain that $F$ is open.
Applying again \eqref{ineq:flower}, we obtain the estimate $0>f(y)>-C_2C_0\eps>-C_1\eps$, and hence we again find $F$ is relatively open, leading to a contradiction as before.
\end{proof}

%%%%%%%%%%%%%%%%%%%%%%%
%%%%%%%%%%%%%%%%%%%%%%%

We are now able to give a proof of the strict monotonicity of $\om(\cdot;\bar y_*)$ and the correct boundary value at the origin, $\om(0;\bar y_*)=\frac{4-3\ga}{3}$. These two properties are proved in the following two lemmas.

\begin{lemma}\label{L:OMEGAMONOTONE}
Let $\ga\in(1,\frac43)$. Then the global solution $(\rho(\cdot;\bar y_*),\om(\cdot;\bar y_*))$ satisfies $\om'(y;\bar y_*)>0$ for all $y\in(0,\bar y_*)$.
\end{lemma}

%%%%%%%%%%%%%%%%%%%%%%%

\begin{proof}
For each $y\in(0,\bar y_*)$, by the convergence with respect to $y_*$ of $\om'(y;y_*)$ from Proposition \ref{prop:sonictime}, as $\om'(y,y_*)>0$ for all $y_*\in Y$, we easily obtain $\om'(y;\bar y_*)\geq0$. If we then suppose for a contradiction that $\om'(y;\bar y_*)=0$, $y$ is a local minimum of $\om'$, and hence $\om''(y;\bar y_*)=0$. By Lemma \ref{lemma:y*notyf}, we have that $\bar y_*\neq y_f$, and hence $\om(\bar y_*;\bar y_*)<2-\ga$. By the weak monotonicity, this yields moreover that $\om(y;\bar y_*)<2-\ga$ for all $y\in(0;\bar y_*)$. In addition, from $y_c(\bar y_*)=0$ from Proposition \ref{P:LEFTGLOBAL}, we obtain that, for all $y\in(0,\bar y_*)$, $\om(y;\bar y_*)>\frac{4-3\ga}{3}$ and so we may apply  Lemma \ref{lemma:omega3} to obtain $\om^{(3)}(y;\bar y_*)<0$, a contradiction.
\end{proof}

%%%%%%%%%%%%%%%%%%%%%%%
%%%%%%%%%%%%%%%%%%%%%%%

We therefore obtain that $\om$ is strictly monotone decreasing as we decrease $y$ towards the origin.
%%%%%%%%%%%%%%%%%%%%%%
%%%%%%%%%%%%%%%%%%%%%%

\begin{prop}\label{P:LIMITATZERO}
Let $\ga\in(1,\frac43)$ and consider the global solution $(\rho,\om)=(\rho(\cdot;\bar y_*),\om(\cdot;\bar y_*))$. The relative velocity $\om$ extends continuously up to the origin and satisfies the limit 
$$\om(0;\bar y_*)=\lim_{y\to0}\om(y;\bar y_*)=\frac{4-3\ga}{3}.$$
\end{prop}

%%%%%%%%%%%%%%%%%%%%%%

\begin{proof}
Suppose that $\lim_{y\to0^+}\om(y;\bar y_*)\neq \frac{4-3\ga}{3}$. We will derive a contradiction. Recall first of all that, by construction and Lemma \ref{lemma:y*notyf}, we have $\om(\bar y_*)<2-\ga$, and thus $\om(y;\bar y_*)\in(\frac{4-3\ga}{3},2-\ga)$ for all $y\in(0,\bar y_*)$, where the strict lower bound comes from the fact that $y_c(\bar y_*)=0$, proved in Proposition \ref{P:LEFTGLOBAL}.

Define \beq
\bar\al=\lim_{y\to0^+}\frac{3\om-(4-3\ga)}{\om}.
\eeq
(Note that the limit exists by monotonicity of $\om$ and that $\bar\al>0$.) One easily sees that the function
$$\om\mapsto A(\om)=\frac{3\om-(4-3\ga)}{\om}$$
is monotone increasing on $(\frac{4-3\ga}{3},2-\ga)$ and achieves its maximum value $\al_{\max}=\frac{2}{2-\ga}$ at $\om=2-\ga$. We therefore have the crucial inequality
\beq\label{ineq:alphabar}
\bar\al<\frac{2}{2-\ga}.
\eeq
Now from the inequality $\om'\geq 0$ from Lemma \ref{L:OMEGAMONOTONE}, we derive
\beqas
0\leq&\,\om'=\frac{4-3\ga-3\om}{y}-\frac{y\om h}{G}=
-\frac{A(\om)\om}{y}  - \frac{y\om h}{G}  \leq-\frac{\bar\al\om}{y}-\frac{y\om h}{G},
\eeqas
where we have again used the monotonicity of $A(\om)$ to see $-A(\om(y))\leq-A(\om(0))=-\bar\al$ by the monotonicity of $\om$. Thus
$$\frac{yh}{G}\leq -\frac{\bar\al}{y},$$
and so, using now the equation for $\rho$ from \eqref{eq:rhoom}, we find that
\beqs
\rho'=\frac{y\rho h}{G}\leq-\frac{\bar\al\rho}{y}.
\eeqs
Thus, for $y$ sufficiently small, we must have 
\beq
\rho\geq c_1y^{-\bar\al},\quad\text{for some }c_1>0.
\eeq
Recalling the definitions~\eqref{def:h} and~\eqref{def:G} of $h$ and $G$ respectively, this then yields that, for some possibly different constant $\tilde{c}_1>0$, for $y$ sufficiently small, we must have
$$h\leq-\tilde c_1 y^{-\bar\al},\quad G\geq \tilde c_1 y^{-(\ga-1)\bar\al}.$$
We recall \eqref{eq:h'}:
\beqa
\frac{\dif}{\dif y}h(\rho,\om)=&\,\big(2\om^2-(\ga-1)(2-\ga)+h(\rho,\om)\big)\frac{4-3\ga-3\om}{y\om}-y\frac{h(4\om^2+(\ga-1)\om)}{G(\rho,\om,y)}\\
=&\,h\Big(\frac{4-3\ga-3\om}{y\om}-y\frac{4\om^2+(\ga-1)\om}{G}\Big)+\big(2\om^2-(\ga-1)(2-\ga)\big)\frac{4-3\ga-3\om}{y\om}.
\eeqa
Using the upper bound for $h$ and lower bound for $G$ just obtained, given $\de>0$ (to be chosen later), we may take $y$ sufficiently small so that
\beq
-(1-\de)\frac{\bar\al}{y}h\leq(1-\frac12\de)\frac{4-3\ga-3\om}{y\om}h\leq h'\leq (1+\frac12\de)\frac{4-3\ga-3\om}{y\om}h\leq-(1+\de)\frac{\bar\al}{y}h.
\eeq
This allows us to get the complementary bound
\beq
h\geq -\tilde{c}_2y^{-\bar\al(1+\de)},\text{ and hence }\rho\leq c_2y^{-\bar\al(1+\de)},\quad G\leq \bar{c}_2y^{-\bar\al(\ga-1)(1+\de)}.
\eeq
Thus we may make the estimate, for $y$ sufficiently small,
$$\Big|\frac{y\om h}{G}\Big|\leq C y^{1-\bar\al(1+\de)+(\ga-1)\bar\al}.$$
Recall from \eqref{ineq:alphabar} that, by construction, $\bar\al<\frac{2}{2-\ga}$. We take $\de>0$ such that
$$\de<\frac{2-(2-\ga)\bar\al}{\bar\al}.$$ Then the exponent here is such that $ 1-\bar\al(1+\de)+(\ga-1)\bar\al>-1$. Thus, again taking $y$ sufficiently small once more,
$$\om'(y)=\frac{4-3\ga-3\om}{y}-\frac{y\om h}{G}\leq \frac12\frac{4-3\ga-3\om}{y}<0,$$
a contradiction to the fact that $\om'\geq 0$ for all $y\in(0,\bar y_*)$. Thus $\lim_{y\to0^+}\om(y)=\frac{4-3\ga}{3}$.
\end{proof}

%%%%%%%%%%%%%%%%%%%%%%%%%
%%%%%%%%%%%%%%%%%%%%%%%%%

\begin{lemma}\label{L:DENSITYMONOTONE}
Let $\ga\in(1,\frac43)$ and consider the global solution $(\rho,\om)=(\rho(\cdot;\bar y_*),\om(\cdot;\bar y_*))$. The density $\rho$ remains bounded and monotone as $y\to0$, i.e.~$\rho$ converges monotonically to some $\rho(0)>\frac{1}{6\pi}$. 
\end{lemma}

%%%%%%%%%%%%%%%%%%%%%%%%%
\begin{proof}
The monotonicity of $\rho$ follows from the inequality $\rho'(y;y_*)<0$ for all $y\in[y_c(y_*),y_*]$ for all $y_*\in Y$ (by Lemma \ref{lemma:leftinvariants}) and the strong convergence $\rho'(y;y_*)\to \rho'(y;\bar y_*)$ for all $y\in(0,\bar y_*)$ as $y_*\to \bar y_*$ given by Proposition \ref{prop:sonictime}(iii).

 To show that $\rho$ stays bounded, suppose for a contradiction that it is not. Note that as $\om$ is bounded (away from 0) and convergent as $y\to 0$, in this limit,
$$h(\rho,\om)\sim -\frac{4\pi}{4-3\ga}\rho\om\sim-\frac{4\pi}{3}\rho\text{ in the sense that }\lim_{y\to0^+}\frac{-h(\rho,\om)}{\frac{4\pi}{3}\rho}=1.$$
Moreover, we clearly also have the asymptotic form
$$\lim_{y\to0^+}\frac{\ga\rho^{\ga-1}-y^2\om^2}{\ga\rho^{\ga-1}}=1.$$
So, given $\de>0$ the ODE for $\rho$ in \eqref{eq:rhoom} becomes, for $y$ sufficiently small,
$$-\big(\frac{4\pi}{3\ga}+\de\big)y\rho^{3-\ga}\leq\rho'\leq-\big(\frac{4\pi}{3\ga}-\de\big)y\rho^{3-\ga}.$$
The solution to an ODE of the form
$$f'(x)=-axf(x)^p \quad\text{ is }\quad f(x)=\Big((p-1)\big(c_1+\frac{ax^2}{2}\big)\Big)^{\frac{1}{1-p}}.$$
Thus solving this pair of ordinary differential inequalities lead to exactly two possibilities: either $\rho$ remains bounded up to the origin, a contradiction to the assumption that it is unbounded, or $\rho= \ka y^{-\frac{2}{2-\ga}}\big(1+o(1)\big)$ as $y\to0$. To see this, choose $y\ll1$ and rearrange the differential inequalities to yield
$$(2-\ga)\big(\frac{4\pi}{3\ga}-\de\big)y\leq \big(\rho^{\ga-2}\big)'\leq(2-\ga)\big(\frac{4\pi}{3\ga}+\de\big)y.$$
Thus, for $\tilde y\in(0,y)$, we have first from the lower bound, integrating from $\tilde y$ to $y$,
$$\rho^{\ga-2}(\tilde y)\leq\rho^{\ga-2}(y)-\frac{2-\ga}{2}\big(\frac{4\pi}{3\ga}-\de\big)y^2+\frac{2-\ga}{2}\big(\frac{4\pi}{3\ga}-\de\big)\tilde y^2.$$
Using that $\rho>0$ and sending $\tilde y\to0$ (as $\rho(\tilde y)\to\infty$, we have $\rho(\tilde y)^{\ga-2}\to0$), this easily gives $$\rho(y)^{\ga-2}\geq \frac{2-\ga}{2}\big(\frac{4\pi}{3\ga}-\de\big)y^2.$$
On the other hand, from the upper bound for $(\rho^{\ga-2})'$, we get the inequality
$$\rho(\tilde y)\leq\Big(\rho^{\ga-2}(y)-\frac{2-\ga}{2}\big(\frac{4\pi}{3\ga}+\de\big)y^2+\frac{2-\ga}{2}\big(\frac{4\pi}{3\ga}+\de\big)\tilde y^2\Big)^{-\frac{1}{2-\ga}},$$
and hence $$\rho^{\ga-2}(y)\leq \frac{2-\ga}{2}\big(\frac{4\pi}{3\ga}+\de\big)y^2,$$ else $\rho(\tilde y)$ would be bounded as $\tilde y\to0$.
Combining these inequalities, we have obtained that 
$$\Big(\frac{2-\ga}{2}\big(\frac{4\pi}{3\ga}+\de\big)\Big)^{-\frac{1}{2-\ga}}y^{-\frac{2}{2-\ga}}\leq\rho(y)\leq\Big(\frac{2-\ga}{2}\big(\frac{4\pi}{3\ga}-\de\big)\Big)^{-\frac{1}{2-\ga}}y^{-\frac{2}{2-\ga}},$$
as required. Noting then that
$$\frac{4-3\ga}{3}\frac{4\pi}{3\ga}y\rho^{2-\ga}\geq \frac{c_1}{y}.$$
We return to the ODE for $\om$ from \eqref{eq:rhoom} to obtain that, for $y$ sufficiently small,
$$\om'=\frac{4-3\ga-3\om}{y}-\frac{y\om h}{G}\geq \frac{4-3\ga-3\om}{y}+\frac{4-3\ga}{6}\frac{4\pi}{3\ga} y\rho^{2-\ga}\geq\frac{c_1}{2y},$$
for $y$ sufficiently small, using $\om\to\frac{4-3\ga}{3}$, $\frac{G}{\ga\rho^{\ga-1}}\to1$, $\frac{-h}{\frac{4\pi}{3}\rho}\to1$. But this leads to a contradiction as $\om$ has a finite limit at the origin.

Thus, as $\rho$ is both monotone and bounded, it has a finite limit $\rho(0)=\lim_{y\to0+}\rho(y)$.

To finish the proof, suppose that $\rho(0)\leq\frac{1}{6\pi}$. As we have $h(\rho(y),\om(y))<0$ for all $y\in(0,\bar y_*)$ and $\om(y)\to\frac{4-3\ga}{3}$, we must have $\rho(0)=\frac{1}{6\pi}$. In this case, we may use that $\om\geq \om_F\equiv \frac{4-3\ga}{3}$, $\rho\leq \rho_F\equiv\frac{1}{6\pi}$ to get the following:
\beqa
|(\om-\om_F)(y)|=&\,(\om-\om_F)(\de)-3\int_\de^y\frac{\om-\om_F}{\tilde y}\dif \tilde y-\int_\de^y\Big(\frac{\tilde y \om h(\rho,\om)}{G(\tilde y,\rho,\om)}-\frac{\tilde y \om_F h(\rho_F,\om_F)}{G(\tilde y,\rho_F,\om_F)}\Big)\dif\tilde y\\
&\,\leq |(\om-\om_F)(\de)|+\int_\de^y\Big|\frac{\tilde y \om h(\rho,\om)}{G(\tilde y,\rho,\om)}-\frac{\tilde y \om_F h(\rho_F,\om_F)}{G(\tilde y,\rho_F,\om_F)}\Big|\dif\tilde y,\\
|(\rho_F-\rho)(y)|\leq&\,|(\rho_F-\rho)(\de)|+\int_\de^y\Big|\frac{\tilde y \rho h(\rho,\om)}{G(\tilde y,\rho,\om)}-\frac{\tilde y \rho_F h(\rho_F,\om_F)}{G(\tilde y,\rho_F,\om_F)}\Big|\dif\tilde y.
\eeqa
Sending $\de\to0$, and then applying a simple Gronwall argument using the the Lipschitz continuity of the expression $$(\rho,\om)\mapsto \frac{h(\rho,\om)}{G(y;\rho,\om)},$$
on bounded sets of $(\rho,\om)$ away from the sonic points $y_*$ and $y_F$, we obtain that $\om\equiv\om_F$, $\rho\equiv\rho_F$, and so conclude the contradiction as, by construction, $(\rho(\cdot;\bar y_*),\om(\cdot;\bar y_*))\neq(\rho_F,\om_F)$.
\end{proof}

%%%%%%%%%%%%%%%%%%%%%%%%%%%
%%%%%%%%%%%%%%%%%%%%%%%%%%%

\begin{lemma}\label{L:REGULARITYATZERO}
Let $\ga\in(1,\frac43)$ and consider the global solution $(\rho,\om)=(\rho(\cdot;\bar y_*),\om(\cdot;\bar y_*))$. The derivatives of $\rho$ and $\om$ converge to zero as $y\to0$ and $\rho'(0)=\om'(0)=0$. Moreover, the density is $C^2$ up to the origin.
\end{lemma}

%%%%%%%%%%%%%%%%%%%%%%%%%%%

\begin{proof}
Now for the solution $(\rho,\om)$ (suppressing the dependence on $\bar y_*$),  we may use the facts that $\om'(y)\geq 0$ for $y>0$ and $\om\geq\frac{4-3\ga}{3}$, to find
$$0\leq \om'(y)=\frac{4-3\ga-3\om}{y}-\frac{y\om h}{\ga\rho^{\ga-1}-y^2\om^2}\leq -\frac{y\om h}{\ga\rho^{\ga-1}-y^2\om^2}\to0 \text{ as }y\to0,$$
leading to $\lim_{y\to0+}\om'(y)=0$. In addition, 
$$-3\om'(0)=\lim_{y\to0^+}\frac{4-3\ga-3\om(y)}{y}\to0$$
by the above inequalities.

Similarly,
$$|\rho'(y)|\leq\Big|\frac{y\rho h}{\ga\rho^{\ga-1}-y^2\om^2}\Big|\leq Cy,$$
so $\rho'(0)=\lim_{y\to0^+}\frac{\rho(y)-\rho(0)}{y}=\lim_{\xi\to0^+}\rho'(\xi)=0$ by the mean value theorem. 

Finally,
$$\rho''(0)=\lim_{y\to0^+}\frac{\rho'(y)-\rho'(0)}{y}=\lim_{y\to0^+}\frac{\rho h}{\ga\rho^{\ga-1}-y^2\om^2}=\frac{\rho(0) h(\rho(0),\om(0))}{\ga\rho(0)^{\ga-1}}$$
and one easily checks that this is also the limit of $\rho''(y)$ as $y\to0$ as required.
\end{proof}

%%%%%%%%%%%%%%%%%%%%%%%%%%%
%%%%%%%%%%%%%%%%%%%%%%%%%%%

%%%%%%%%%%%%%%%%%%%%%%%%%

\section{Proof of the main theorem}\label{S:PROOF}

%%%%%%%%%%%%%%%%%%%%%%%%%
%%%%%%%%%%%%%%%%%%%%%%%%%

We now prove Theorem~\ref{T:MAIN}.
Let $\ga\in(1,\frac43)$ be fixed. Consider the local real analytic solution
associated with the sonic point $\bar y_\ast\in(y_f,y_F)$:
\begin{align}\label{E:SOLUTION}
(\rho(\cdot;\bar y_\ast), \ \om(\cdot;\bar y_\ast)).
\end{align}
By Lemma~\ref{lemma:supersonic} the solution extends globally to the right, and by 
Proposition~\ref{P:LEFTGLOBAL} the solution extends to the left to the whole interval $[0,\bar y_\ast]$.
We therefore obtain a global solution, which is real analytic at $(0,\infty)$ and $C^1$ at $y=0$ by Lemma~\ref{L:REGULARITYATZERO}.

By Lemmas~\ref{lemma:rightmonotonicity},~\ref{L:OMEGAMONOTONE} and~\ref{L:DENSITYMONOTONE}, it follows that both $\rho(\cdot;\bar y_\ast)$ and
$\om(\cdot;\bar y_\ast)$ are strictly monotone on $(0,\infty)$: $\om$ is increasing and $\rho$ is decreasing. This proves~\eqref{E:MONOTONICITY}.
We now recall~\eqref{E:OMEGADEF}, which implies $u(y)= y \om(y)-(2-\ga)y$. Since $\lim_{y\to0}\om(y)=\frac{4-3\ga}{3}$ by Proposition~\ref{P:LIMITATZERO} and 
$\lim_{y\to\infty}\om(y)=2-\ga$ by Lemma~\ref{lemma:asymptotics}, the strict monotonicity of $\om$ on $(0,\infty)$ implies the second claim of~\eqref{E:RHOOMBOUNDS}. 
The strict positivity of $\rho$ is obvious.

%%%%%%%%%%%%%%%%%%%%%%%%%
%%%%%%%%%%%%%%%%%%%%%%%%%

\appendix

%%%%%%%%%%%%%%%%%%%%%%%%%
%%%%%%%%%%%%%%%%%%%%%%%%%

\section{Well-posedness away from singular points}\label{A:EU}

%%%%%%%%%%%%%%%%%%%%%%%%%
%%%%%%%%%%%%%%%%%%%%%%%%%

\renewcommand{\theequation}{\thesection.\arabic{equation}}

At several points throughout the paper, we make use of the following straightforward local existence and uniqueness theorem for \eqref{eq:rhoom} provided the solution is away from both the singularities $y=0$ and any sonic points. Although the proof is essentially standard, we include it here to make explicit the dependence of the time of existence on the uniform subsonicity or supersonicity. This is made precise in the following proposition.
\begin{prop}\label{prop:Picard}
Suppose that $y_0>0$ and $(\bar\rho,\bar\om)$ are given such that $\bar\rho>\frac{1}{M}$, $|\bar\rho|+|\bar\om|\leq M$ and $\big|G(y_0,\bar\rho,\bar\om)\big|\geq\eta>0$. Then there exists $\de>0$, depending on $y$, $M$ and $\eta$, such that the flow \eqref{eq:rhoom} has a local, unique solution on the interval $[y_0-\de,y_0+\de]$. Moreover, on $[y_0-\de,y_0+\de]$, we have the estimates
$$\bar\rho\geq\frac{1}{2M},\quad |\bar\rho|+|\bar\om|\leq 2M\text{ and }\big|G(y_0,\bar\rho,\bar\om)\big|\geq\frac{\eta}{4}>0.$$
\end{prop}

\begin{proof}
This follows directly from the usual existence and uniqueness theory for ODEs with a locally Lipschitz right hand side. However, for the convenience of the reader and to emphasise the dependence on $M$, $\eta$ and $y$, we provide a proof.

By the local Lipschitz continuity of the map $y\mapsto G(y;\rho,\om)$ on the set $\{|\rho|+|\om|\leq 2M,\: \rho\geq\frac{1}{2M}\}$, there exists $\de_1>0$ such that $G(y,\bar\rho,\bar\om)\geq\frac{\eta}{2}$ for all $y\in[y_0-\de_1,y_0+\de_1]$. For any $\tilde\eta,\tilde{M},\tilde\de>0$, we define the set
$$\Omega_{\tilde{M},\tilde\eta,\tilde\de}=\big\{(\rho,\om)\in\R^2\,|\,\rho\geq\frac{1}{\tilde{M}},\:|\rho|+|\om|\leq \tilde{M},\:G(y;\rho,\om)\geq\tilde\eta\text{ for all }y\in[y_0-\tilde\de,y_0+\tilde\de]\big\}.$$
Clearly by definition we have $(\bar\rho,\bar\om)\in\Omega_{M,\eta/2,\de_1}$.

For notational convenience, we define two new functions,
\begin{align}
\overline{\mathcal{F}}(y,\rho,\om)=&\,\frac{y\rho h(\rho,\om)}{G(y;\rho,\om)},  \label{E:FDEF}\\
\overline{\mathcal{G}}(y,\rho,\om)=&\,\frac{4-3\ga-3\om}{y}-\frac{y\om h(\rho,\om)}{G(y;\rho,\om)}. \label{E:GDEF}
\end{align}
Then for given constants $M$, $\eta$, there exist constants $N>0$, $L>0$ and $l>0$, depending also on $y_0$, such that
\beqa
&|\overline{\mathcal{F}}(y,\rho_1,\om_1)|+|\overline{\mathcal{G}}(y,\rho_1,\om_1)|\leq N,\\
&|G(y;\rho_1,\om_1)-G(y;\rho_2,\om_2)|\leq l\big(|\rho_1-\rho_2|+|\om_1-\om_2|\big),\\
 &|\overline{\mathcal{F}}(y,\rho_1,\om_1)-\overline{\mathcal{F}}(y,\rho_2,\om_2)|+|\overline{\mathcal{G}}(y,\rho_1,\om_1)-\overline{\mathcal{G}}(y,\rho_2,\om_2)|\leq L\big(|\rho_1-\rho_2|+|\om_1-\om_2|\big)
\eeqa
for all $y\in[y_0-\de_1,y_0+\de_1]$, $(\rho_i,\om_i)\in\Omega_{2M,\eta/4,\de_1}$, $i=1,2$.

We define a Picard operator via 
$$\mathcal{T}[\rho,\om](y)=\begin{pmatrix}
\bar\rho+\int_{y_0}^y\overline{\mathcal{F}}(\tilde y,\rho(\tilde y),\om(\tilde y))\dif \tilde y\\
\bar\om+\int_{y_0}^y\overline{\mathcal{G}}(\tilde y,\rho(\tilde y),\om(\tilde y))\dif\tilde y
\end{pmatrix}.$$
We fix $\de\in(0,\de_1)$ such that 
$$\de l N<\frac{\eta}{4},\quad \de N\leq \frac{1}{2M}<M,\quad \de L\leq\frac12.$$
 Then for any $(\rho,\om)\in C([y_0-\de,y_0+\de];\Omega_{2M,\eta/4,\de})$, we let $(\tilde\rho,\tilde\om)=\mathcal{T}[\rho,\om]$ and see that for any $y\in[y_0-\de,y_0+\de]$ we have
\beqas
|\tilde\rho(y)|+|\tilde\om(y)|\leq&\, |\bar\rho|+|\bar\om|+\Big|\int_{y_0}^y\Big(\big|\overline{\mathcal{F}}(\tilde y,\rho(\tilde y),\om(\tilde y))\big|+\big|\overline{\mathcal{G}}(\tilde y,\rho(\tilde y),\om(\tilde y))\big|\Big)\dif\tilde y\Big|\\
\leq&\,M+N|y-y_0|\leq 2M.
\eeqas
Moreover, estimating $\tilde\rho(y)$, we have
$$\tilde\rho(y)\geq\bar\rho-\Big|\int_{y_0}^y\big|\overline{\mathcal{F}}(\tilde y,\rho(\tilde y),\om(\tilde y))\big|\dif\tilde y\Big|\geq\frac{1}{M}-\de N\geq\frac{1}{2M}.$$
In addition, for all $y\in[y_0-\de,y_0+\de]$ such that $(\tilde\rho,\tilde\om)(y)\in\Omega_{2M,\eta/4,\de}$ (note that this set is non-empty and open by construction of $\mathcal{T}$ and continuity of $G$ away from sonic points), we have that
\beqas
|G(y,&\tilde\rho,\tilde\om)-G(y,\bar\rho,\bar\om)|\\
\leq&\, l\big(|\tilde\rho-\bar\rho|+|\tilde\om-\bar\om|\big)\\
\leq&\, l\int_{y_0}^y\Big(\big|\overline{\mathcal{F}}(\tilde y,\rho(\tilde y),\om(\tilde y))\big|+\big|\overline{\mathcal{G}}(\tilde y,\rho(\tilde y),\om(\tilde y))\big|\Big)\dif\tilde y\\
\leq&\, l \de N<\frac{\eta}{4},
\eeqas
so that as $G(y,\bar\rho,\bar\om)\geq\frac{\eta}{2}$ for every such $y$ (as $\de<\de_1$), a simple continuity argument shows that $(\tilde\rho,\tilde\om)(y)\in\Omega_{2M,\eta/4,\de}$ for all $y\in[y_0-\de,y_0+\de]$. Thus we have shown
$$\mathcal{T}:C([y_0-\de,y_0+\de];\Omega_{2M,\eta/4,\de})\to C([y_0-\de,y_0+\de];\Omega_{2M,\eta/4,\de}).$$
We equip $C\big([y_0-\de,y_0+\de];\Omega_{2M,\eta/4,\de}\big)$ with the norm $\|(\rho,\om)\|_{X}=\|\rho\|_{C^0}+\|\om\|_{C^0}$ and observe that it is a complete metric space. To see $\mathcal{T}$ is a contraction, take $(\rho_1,\om_1),(\rho_2,\om_2)\in C([y_0-\de,y_0+\de];\Omega_{2M,\eta/4,\de})$, write $\mathcal{T}[\rho_j,\om_j]=(\tilde\rho_j,\tilde\om_j)$ for $j=1,2$, and observe
\beqas
&\big|\tilde\rho_1-\tilde\rho_2\big|(y)+\big|\tilde\om_1-\tilde\om_2\big|(y)\\
&\leq\int_{y_0}^y\Big(\big|\overline{\mathcal{F}}(\tilde y,\rho_1(\tilde y),\om_1(\tilde y))-\overline{\mathcal{F}}(\tilde y,\rho_2(\tilde y),\om_2(\tilde y))\big|+\big|\overline{\mathcal{G}}(\tilde y,\rho_1(\tilde y),\om_1(\tilde y))-\overline{\mathcal{G}}(\tilde y,\rho_2(\tilde y),\om_2(\tilde y))\big|\Big)\dif\tilde y\\
&\leq L\int_{y_0}^y\big(|\rho_1-\rho_2|(\tilde y)+|\om_1-\om_2|(\tilde y)\big)\dif\tilde y\\
&\leq\de L\big(\|\rho_1-\rho_2\|_{C^0}+\|\om_1-\om_2\|_{C^0}\big)\\
&\leq\frac{1}{2}\|(\rho_1,\om_1)-(\rho_2,\om_2)\|_X,
\eeqas
as required. Thus there is a fixed point of the operator $\mathcal{T}$, $(\rho,\om)\in C([y_0-\de,y_0+\de];\Omega_{2M,\eta/4,\de})$ satisfying the ODE system and the claimed estimates.
\end{proof}

\section{Combinatorial bootstrap - convergence of the series at the sonic point}\label{A:COMB}

%%%%%%%%%%%%%%%%%%%%%%%%%%
%%%%%%%%%%%%%%%%%%%%%%%%%

\renewcommand{\theequation}{\thesection.\arabic{equation}}
The central outcome of this section is  Lemma~\ref{lem:FGbound}, which establishes that key $N$-dependent growth bounds for the 
coefficients in the formal Taylor expansion~\eqref{E:FORMAL} around the sonic point can be bootstrapped. This is the key ingredient of the induction argument
used in Lemma~\ref{L:INDUCTION}. Our arguments are combinatorial in nature and
we first prove some technical lemmas. In the following, $\lfloor\alpha\rfloor$ is the usual floor function, denoting the greatest integer not bigger than $\alpha$, for any $\alpha\in\mathbb R$.

%%%%%%%%%%%%%%%%%%%
%%%%%%%%%%%%%%%%%%%

\begin{lemma} There exists a universal constant $c>0$ such that for all $N\in \mathbb N$, the following holds
\begin{align}
\sum_{\substack{l+m=N\\ l,m \geq 1}}  \frac{1}{l^{3} m^{3}} \leq  \frac{c}{N^3}, \label{2.51}\\
\sum_{\substack{l+m=N\\ l,m \geq 1}}  \frac{1}{l^{2} m^{2}} \leq  \frac{c}{N^2}, \label{2.50}\\
\sum_{\substack{l+m=N\\ l,m \geq 1}}  \frac{1}{l^{3} m^{2}} \leq  \frac{c}{N^2}, \label{2.54}\\
\sum_{\substack{l+m+n=N\\ l, m, n \geq 1}} \frac{1}{l^{3} m^{3} n^{3}} \leq  \frac{c}{N^3}, \label{2.52}\\
\sum_{\substack{l+m+n=N\\ l, m, n \geq 1}} \frac{1}{l^{3} m^{2} n^{3}} \leq  \frac{c}{N^2}. \label{2.53}
\end{align}
\end{lemma}

%%%%%%%%%%%%%%%%%%%

\begin{proof}
The first bound \eqref{2.51} follows from 
\[
\sum_{\substack{l+m=N\\ l,m \geq 1}}  \frac{1}{l^{3} m^{3}}= \sum_{m=1}^{N-1} \frac{1}{(N-m)^{3} m^{3}}= \sum_{m=1}^{N-1} \frac{1}{N^3} \left( \frac{1}{N-m}+ \frac{1}{m}\right)^3\leq \frac{2}{N^3}\sum_{m=1}^\infty\frac{1}{m^3} \lesssim N^{-3}.
\]
The second bound \eqref{2.50} is entirely analogous. The third bound \eqref{2.54} follows from 
\[
\begin{split}
\sum_{\substack{l+m=N\\ l,m \geq 1}}  \frac{1}{l^{3} m^{2}} &= \sum_{m=1}^{N-1}\frac{1}{m^{2}} \frac{1}{(N-m)^3} = \sum_{m=1}^{N-1}\frac{1}{N^2} \left( \frac{1}{m}+ \frac{1}{N-m}\right)^2\frac{1}{N-m} \\
&\lesssim \frac{1}{N^2} \left( \sum_{m=1}^{\lfloor \frac{N}{2}\rfloor} \frac{1}{m^3}+ \sum_{m= \lfloor\frac{N}{2}\rfloor}^{N-1} \frac{1}{ (N-m)^3}  \right) \lesssim {N^{-2}}.
\end{split}
\]
For \eqref{2.52}, 
\[
\sum_{\substack{l+m+n=N\\ l, m, n \geq 1}} \frac{1}{l^{3} m^{3} n^{3}} \leq \sum_{l=1}^{N-1} \frac{1}{l^3} \sum_{\substack{m+n=N-l \\ m, n \geq 1} } \frac{1}{m^{3} n^{3} }\leq 
\sum_{l =1}^{N-1}  \frac{1}{l^3} \frac{c}{(N-l)^3} \lesssim N^{-3},
\] where we have used \eqref{2.51}  twice. 

For \eqref{2.53}, using   \eqref{2.51} and \eqref{2.54}, we have 
\[
\begin{split}
\sum_{\substack{l+m+n=N\\ l, m, n \geq 1}} \frac{1}{l^{3} m^{2} n^{3}} &\leq \sum_{m=1}^{N-1}\frac{1}{m^{2}}\sum_{\substack{l+n=N-m\\ l, n \geq 1}} \frac{1}{l^{3}  n^{3}}
 \lesssim \sum_{m=1}^{N-1}\frac{1}{m^{2}} \frac{1}{(N-m)^3} 
 \lesssim {N^{-2}}.
\end{split}
\] 
This finishes the proof. 
\end{proof}

%%%%%%%%%%%%%%%%%%%
%%%%%%%%%%%%%%%%%%%

Define the set 
\begin{equation}\label{pi}
\pi(n,m)= \left\{(\lambda_1, \dots,\lambda_n): \lambda_i \in \mathbb Z_{\geq 0}, \sum_{i=1}^n \lambda_i =m, \sum_{i=1}^n i \lambda_i =n \right\}.
\end{equation}
An element of $\pi(n,m)$ encodes the partitions of the first $n$ numbers into $\lambda_i$ classes of cardinality $i$ 
for $i\in \{1,\dots,m\}$. Observe that by necessity $\lambda_j=0$ for any $n-m+2\leq  j\leq n$. With this partition set, the coefficient $P_N$ of Taylor series for $\rho^{\ga-1}=\sum_{N=0}^\infty P_N(y-y_*)^N$ in 
\eqref{eq:seriesPN} can be written as 
\begin{equation}\label{eq:PN}
P_N = 
\begin{cases}
\rho_0^{\ga-1}, & \ \text{ if } \  N=0, \\
\rho_0^{\ga -1} \sum_{m=1}^N \frac{1}{\rho_0^m}\sum_{\pi(N,m)}  \frac{(\ga -1)_m}{\lambda_1 ! \dots \lambda_N !}  \prod_{j=1}^N {\rho_j}^{\lambda_j}  & \ \text{ if } \  N\ge 1,
\end{cases}
\end{equation}
where $(\gamma-1)_m = \prod_{j=1}^m(\ga -j)$. 

To obtain bounds of $P_N$ in terms of the coefficients $\rho_j$, we will make use of the following combinatorial identities and inequalities. For any $\alpha\in \mathbb R$, we let
\[
\binom{\alpha}{j} =\frac{(\alpha)_j}{j!}= \frac{\alpha (\alpha-1)\cdots (\alpha-j+1)}{j!} \ \  \text{for} \ \ j\in \mathbb N, \ \ \text{and} \ \  \binom{\alpha}{0}=1.
\]

\begin{lemma}\label{formula1} Recall the set  $\pi(n,m)$ defined in \eqref{pi}. 
\begin{enumerate}
\item  For each $n\in \mathbb N$, 
\begin{equation}
\sum_{m=1}^n\sum_{\substack{\pi(n,m) } } \frac{(-1)^m m! }{ \lambda_1 ! \dots \lambda_n!} \binom{\frac12}{1}^{\lambda_1} \cdots  \binom{\frac12}{n}^{\lambda_n}  = 2 (n+1)\binom{\frac12}{n+1}
\end{equation} holds. 
\item There exist universal constants $c_1,c_2>0$ such that 
\begin{equation}\label{12n}
 c_1 \frac{1}{n^\frac32}\leq (-1)^{n-1} \binom{\frac12}{n} \leq c_2 \frac{1}{n^\frac32}, \quad n\in \mathbb N.
\end{equation}
\end{enumerate}
\end{lemma}

\begin{proof}
The first statement follows from Lemma 1.5.2 of \cite{KP02}. 

For the second statement, \eqref{12n} is trivial for $n=1$. Let $n\geq2$. Then 
\be\label{3.167}
(-1)^{n-1} \binom{\frac12}{n} = \frac{ \frac{1}{2}\cdot \frac{1}{2} \cdots \frac{2n-3}{2}}{  n!} = \frac{(2n-2)! }{ 2^{2n-1} (n-1)! n! } = \frac{1}{2n-1}\frac{(2n)!}{2^{2n} (n!)^2}.
\ee
To estimate the last fraction, we invoke Stirling's formula $ 
n! \sim \sqrt{2\pi n}\left( \frac{n}{e}\right)^n, \ \ n \gg 1$. We will use the following version with upper and lower bounds valid for all $n$: 
\be
\sqrt{2\pi} n^{n+\frac{1}{2}} e^{-n} \leq n! \leq e n^{n+\frac12}e^{-n}, \quad n\in \mathbb N.
\ee
Then we have
\be
\frac{ \sqrt{2\pi}\sqrt{2} }{e^2n^\frac12}=\frac{\sqrt{2\pi}  (2n)^{2n+\frac12}e^{-2n} }{2^{2n} e^2 (n^{n+\frac{1}{2}} e^{-n})^2 }  \leq \frac{(2n)!}{2^{2n} (n!)^2} \leq \frac{ e (2n)^{2n+\frac12}e^{-2n} }{ 2^{2n}2\pi (n^{n+\frac{1}{2}} e^{-n})^2 }= \frac{e \sqrt{2}}{ 2\pi n^\frac12}.
\ee
Hence, combining this with \eqref{3.167}, we have \eqref{12n}. 
\end{proof}

\begin{lemma}\label{lem:comb} Let $p>0$ be a given positive number. Let $(\lambda_1,\dots, \lambda_l ) \in \pi (l ,m)$ where $1\leq m\leq l $ and $l\geq2$ be given. 
\begin{enumerate}
\item If $1\leq m\leq \lfloor \tfrac{\sqrt l}{\sqrt3}\rfloor$, there exists a constant $c_3=c_3(p)>0$ such that 
\be\label{E:comb1}
 \prod_{n=1}^l \left(\frac{1}{n^{\lambda_n}}\right)^p  \leq \frac{c_3}{l^{p}}.
\ee
\item There exist  $c_4=c_4(p)>0$ and $L_0=L_0(p)>1$ such that, if $L\geq L_0$,  the following holds:  
\be\label{E:comb2}
\frac{1}{L^{m-1}}\prod_{n=1}^l \left(  \frac{1}{n^{\lambda_n}} \right)^p \leq \frac{c_4}{l^{p}}   \text{ for all }  1\leq m \leq l. 
\ee
\item Let $l \geq 3$. Then there exists  $c_5=c_5(p)>0$ such that, if $L\geq L_0$,  the following holds:  
\be\label{E:comb3}
\frac{1}{L^{m-2}}\prod_{n=1}^l \left(  \frac{1}{n^{\lambda_n}} \right)^p \leq \frac{c_5}{l^{p}}   \text{ for all }  2\leq m \leq l. 
\ee
\end{enumerate}
\end{lemma}

\begin{proof} {\it Proof of \eqref{E:comb1}}. Let $ \overline  m:= \lfloor\tfrac{\sqrt l}{\sqrt3}\rfloor$. 
We first claim that there exists at least one $\lambda_j \geq 1$ for $j \geq  \overline  m$. If not, $\lambda_j =0$ for all $j \geq  \overline  m $. Then we would have for $1\leq m\leq \overline  m$
\[
l =\sum_{j=1}^l j \lambda_j = \sum_{j < \overline  m } j \lambda_j < \overline  m  \sum_{j=1}^l \lambda_j = \overline  m m \leq \overline  m^2 \leq \frac{l }{3}
\]
which is a contradiction. 

We are now ready to prove \eqref{E:comb1}. Consider two cases.  

\noindent{\it Case 1.} Suppose there exists exactly one $\lambda_{j_0} \geq 1$ for $j_0\geq \overline  m$. If $\lambda_{j_0} = 1$, 
 we must have $j_0 \geq \frac{l}{2}$,  for otherwise we would have 
\[
l =\sum_{j=1}^l j \lambda_j = \sum_{j < \overline  m } j \lambda_j + j_{0}\lambda_{j_0}  <  \overline  m^2 + \frac{l}{2} \leq \frac{l }{3} +\frac{l}{2} 
\]
which leads to a contradiction. 
Therefore, we have 
\[
 \prod_{n=1}^l \left(\frac{1}{n^{\lambda_n}}\right)^p \leq \left(\frac{1}{j_0}\right)^p \leq  \frac{2^p}{l^p}.
\]
If $\lambda_{j_0}\geq 2$, then $j_0^{\lambda_{j_0}} \geq j_0^2 \geq \overline m^2$, which leads to 
\[
 \prod_{n=1}^l \left(\frac{1}{n^{\lambda_n}}\right)^p \leq \left(\frac{1}{j_0^{\lambda_{j_0}}}\right)^p \leq  \left(\frac{1}{\overline m^{2}}\right)^p = \frac{3^p}{l^p}.
\]

\noindent{\it Case 2.} Suppose there exist at least two $\lambda_{j_1}, \lambda_{j_2} \geq 1$ for $j_1, j_2\geq \overline m$. Then $j_1^{\lambda_{j_1}} j_2^{\lambda_{j_2}} \geq j_1 j_2\geq \overline m^2$, which gives 
\[
 \prod_{n=1}^l \left(\frac{1}{n^{\lambda_n}}\right)^p \leq \left(\frac{1}{j_1^{\lambda_{j_1}}} \frac{1}{j_2^{\lambda_{j_2}}}\right)^p \leq  \left(\frac{1}{\overline m^{2}}\right)^p = \frac{3^p}{l^p}.
\]
This finishes the proof of \eqref{E:comb1}. 

\noindent{\it Proof of \eqref{E:comb2}}.  If $1\leq m\leq \lfloor\tfrac{\sqrt l}{\sqrt3}\rfloor$, then by \eqref{E:comb1}, for all $L>1$, 
\[
\frac{1}{L^{m-1}}\prod_{n=1}^l \left(  \frac{1}{n^{\lambda_n}} \right)^p \leq  \frac{c_3}{l^{p}}.
\]
If $\lfloor\tfrac{\sqrt l}{\sqrt3}\rfloor+1\leq m \leq l$, note  $$
\frac{1}{L^{m-1}}\prod_{n=1}^l \left(  \frac{1}{n^{\lambda_n}} \right)^p \leq \frac{1}{L^{m-1}} \leq 
\frac{1}{L^{\lfloor\frac{\sqrt l}{\sqrt3}\rfloor}}\leq 
\begin{cases} 
1 & \text{ if }l =2,\\
\frac{1}{L^{\frac{\sqrt l}{\sqrt3} -1}} &   \text{ if }l \geq 3.
\end{cases}
$$
Now letting $L_0= e^{2p}$, it is easy to see that $(\frac{\sqrt l}{\sqrt3} -1) \log L - p \log l  + p\log 3\geq 0$ for all $l \geq 3$ and $L\geq L_0$. Hence we obtain 
\[
\frac{1}{L^{m-1}}\prod_{n=1}^l \left(  \frac{1}{n^{\lambda_n}} \right)^p \leq \frac{3^p}{l^p}
\]
for all $l \geq 2$ and $L\geq L_0=e^{2p}$. 

\noindent{\it Proof of \eqref{E:comb3}}. The proof is analogous to \eqref{E:comb2}. We omit the details. 
\end{proof}

Let $M>0$ be a fixed upper bound of $|\rho_0|, \ |\om_0|, \  |\rho_1|, \  |\om_1|$ such that 
\begin{align}\label{E:MBOUND}
|\rho_0^{-1}|, \ |\rho_0|, \ |\om_0|, \  |\rho_1|, \  |\om_1| \leq M.
\end{align}
Note that such an $M$ may be taken to depend only on $\ga$ by continuity of these values as functions of $y_*\in[y_f,y_F]$ and the uniform lower bound on $\rho_0$ given by Lemma \ref{lemma:rho0om0}.

\begin{lemma} Let  $\alpha\in (1,2)$ be given. Assume that 
\begin{align}
|\rho_m| \leq \frac{C^{m-\alpha}}{m^3}, \quad 2\leq m\leq N-1, \label{assumptionR}\\
|\om_m| \leq \frac{C^{m-\alpha}}{m^3}, \quad 2\leq m\leq N-1 \label{assumptionW}
\end{align} for some $C\geq 1$ and $N\geq3$. Then there exists a  constant $D=D(M)>0$ such that 
\begin{align}
|(\om^2)_l| +|(\rho\om)_l | +|(\rho^2)_l| &\le 
\begin{cases}
D & \ \text{ if } \ l=0,1, \\
D \frac{C^{l-\alpha}}{l^3} & \ \text{ if } \ 2\le l \le N-1,
\end{cases}\label{eq:quadratic} \\
|(\om^3)_l| +|(\rho\om^2)_l | + |(\rho^2 \om)_l| & \le 
\begin{cases}
D & \ \text{ if } \ l=0,1, \\
D \frac{C^{l-\alpha}}{l^3} & \ \text{ if } \ 2\le l \le N-1.
\end{cases}\label{eq:cubic}%\\
\end{align} 
\end{lemma}

\begin{proof} We first prove the bounds for $|(\om^2)_l|$, $l\ge0$. 
The bounds $|(\om^2)_0|\le M^2$ and $|(\om^2)_1|\le 2M^2$ are obvious from~\eqref{E:MBOUND}.
Clearly 
\be\label{E:CASETWO}
|(\om^2)_2|\le 2M|\om_2|+M^2 \le 2M \frac{C^{2-\alpha}}{2^3}+M^2\le (2M +2^3 M^2) \frac{C^{2-\alpha}}{2^3},
\ee where we have used $C^{2-\alpha}\geq 1$. 
If $l\ge3$ we then have
\begin{align}
|(\om^2)_l| & \le \sum_{m=0}^l|\om_m||\om_{l-m}| \notag \\
&\le 2|\om_0||\om_l| + 2|\om_1||\om_{l-1}| + \sum_{m=2}^{l-2}|\om_m||\om_{l-m}| \notag\\
& \le 2M \frac{C^{l-\alpha}}{l^3}  + 2M  \frac{C^{l-1-\alpha}}{(l-1)^3}
+ \sum_{m=2}^{l-2} \frac{C^{l-2\alpha}}{m^3(l-m)^3} \notag \\
& \le 2M C^{l-\alpha} \left(\frac1{l^3} + \frac1{(l-1)^3} + \frac{1}{2M}\sum_{m=2}^{l-2} \frac{1}{m^3(l-m)^3} \right) \notag \\
& \le 2M \tilde C \frac{C^{l-\alpha}}{l^3},
\end{align}
for some constant $\tilde C$.  
It is now clear that the estimates for $(\rho\om)_l$ and $(\rho^2)_l$, $l\ge0$ follow in the same way, as the only estimates we have used are~\eqref{E:MBOUND} and the inductive assumptions~\eqref{assumptionR}--\eqref{assumptionR}, which both depend only on the index and are symmetric with respect to $\rho$ and $\om$. The bound \eqref{eq:cubic} can be obtained analogously. 
\end{proof}

%%%%%%%%%%%%%%%%%%%%%%%%%
%%%%%%%%%%%%%%%%%%%%%%%%%

\begin{lemma} Let $\alpha\in (1,2)$ be given. Assume that \eqref{assumptionR} and \eqref{assumptionW} hold for  $N\geq3$ and 
 some large enough $C>1$ satisfying 
 \be\label{largeC}
 C >\frac{L_0}{c_1 \rho_0},
 \ee
 where $c_1$ and $L_0=L_0(\frac32)$ are universal constants in \eqref{12n} and Lemma \ref{lem:comb}. Then there exists a constant $D=D(M, \ga)>0$ such that 
\begin{align}
|P_l  |& \le
\begin{cases}
 D 
& \ \text{ if } \ l=1, \\
 D \left( \frac{C^{l-\alpha}}{l^3}+\frac{C^{l-2}}{l^2} \right) & \ \text{ if } \ 2\le l \le N-1,
\end{cases} \label{faaRl} 
\end{align} 
where we recall~\eqref{eq:PN}.
\end{lemma}

%%%%%%%%%%%%%%%%%%%%%%%%%

\begin{proof} The bound of $P_1$ immediately follows by recalling $P_1= (\ga-1)\rho_0^{\gamma-2}\rho_1$. For $P_2$, observe that 
\[
P_2= \rho_0^{\gamma-1} \left[ \frac{1}{\rho_0} \frac{(\ga-1)}{1!} \rho_2 + \frac{1}{\rho_0^2} \frac{(\ga-1)(\ga-2)}{2!} \rho_1^2  \right]
\]
from which we deduce 
\[
|P_2|\le (\ga-1) \left( M^{2-\ga} \frac{C^{2-\alpha}}{2^3} + \frac{(2-\ga) } {2}M^{5-\ga}  \right)\le 2(\ga-1)M^{5-\ga}\left(  \frac{C^{2-\alpha}}{2^3} +  \frac{1}{2^2} \right).
\]
Now let $l \ge 3$ and split $P_l$ into two parts, $m=1$ and $m\ge 2$:  
\begin{equation}\label{eq:Pell}
P_l = \rho_0^{\ga-1} \frac{1}{\rho_0} \frac{(\ga-1)}{1!} \rho_l +  \rho_0^{\ga -1} \sum_{m=2}^l  \frac{1}{\rho_0^m}\sum_{\pi(l,m)}  \frac{(\ga -1)_m}{\lambda_1 ! \dots \lambda_l !}  \prod_{j=1}^l {\rho_j}^{\lambda_j} =: P_{l,1} + P_{l,2},
\end{equation}
where we note  $\pi(l, 1)=\{(0,\dots,0,1) \}$. By  \eqref{assumptionR}, it is clear that 
\begin{equation}
|P_{l,1}|\le D \frac{C^{l-\alpha}}{l^3}
\end{equation}
for some constant $D>0$ depending only on $M$ and $\gamma$. Next we claim that there exists $D>0$ such that 
\begin{equation}\label{eq:Pell2}
|P_{l,2}| \le D \frac{C^{l-2}}{l^2}.
\end{equation}
To prove the claim, using  \eqref{assumptionR} and  Lemma  \ref{formula1}, we first observe 
\[
\begin{split}
\left| \prod_{n=1}^l \rho_n^{\lambda_n} \right| &\le \left(\tfrac{1}{1^3}\right)^{\lambda_1} \left(\tfrac{C^{2-\alpha}}{2^3}\right)^{\lambda_2}  \dots 
\left(\tfrac{C^{l-\alpha}}{l^3}\right)^{\lambda_l} \\
&= C^{(\alpha-1)\lambda_1 + \sum_{i=1}^l(i\lambda_i -\alpha \lambda_i)}\left[ \prod_{n=1}^l \left( \frac{1}{n^{\frac32}}\right)^{\lambda_n} \right] \left[\prod_{n=1}^l \left( \frac{1}{n^{\lambda_n}}\right)^\frac32 \right] \\
&\leq C^{l - m} c_1^{-m}\left[ \prod_{n=1}^l \left( (-1)^{n-1} \binom{\frac12}{n}\right)^{\lambda_n} \right]
\left[\prod_{n=1}^l \left( \frac{1}{n^{\lambda_n}}\right)^\frac32\right],
\end{split}
\]
where we have used $(\alpha-1)\lambda_1\leq (\alpha-1)m$ in the third line since $\alpha>1$ and $ \lambda_1\leq m$.  Hence, using $|(\ga-1)_m| \le (\ga-1) (m-1)! $ for $1<\ga<\frac43$, we have  
\be\label{3.176}
\begin{split}
&\left|   \frac{(\ga-1)_m}{\lambda_1 ! \dots \lambda_l !}  \prod_{n=1}^l \rho_n^{\lambda_n} \right| \\
&\quad\leq (\ga-1)C^{l-m} c_1^{- m } \frac{1}{m} (-1)^{l}   \frac{(-1)^m m!}{\lambda_1 ! \dots \lambda_l !} \binom{\frac12}{1}^{\lambda_1}\dots 
\binom{\frac12}{l}^{\lambda_l}  \left[\prod_{n=1}^l \left( \frac{1}{n^{\lambda_n}}\right)^\frac32\right].
\end{split}
\ee  
Now recalling $P_{l,2}$ from \eqref{eq:Pell} and using Lemma \ref{formula1} and Lemma \ref{lem:comb} with $p=\frac32$, we have 
\be
\begin{split}
&|P_{l,2} | =\left|  \rho_0^{\ga -1} \sum_{m=2}^l  \frac{1}{\rho_0^m}\sum_{\pi(l,m)}  \frac{(\ga -1)_m}{\lambda_1 ! \dots \lambda_l !}  \prod_{j=1}^l {\rho_j}^{\lambda_j}  \right| \\
&\leq (\ga -1) \rho_0^{\ga -1} \frac{C^l (-1)^l }{2(c_1C\rho_0)^2} \sum_{m=2}^l \sum_{\pi(l,m)} \left[ \frac{1}{(c_1C\rho_0)^{m-2}} \prod_{n=1}^l \left( \frac{1}{n^{\lambda_n}}\right)^\frac32\right] \frac{(-1)^m m!}{\lambda_1 ! \dots \lambda_l !} \binom{\frac12}{1}^{\lambda_1}\dots 
\binom{\frac12}{l}^{\lambda_l} \\
&\leq  (\ga -1) \rho_0^{\ga -1} \frac{C^l }{2(c_1C\rho_0)^2}\frac{c_5}{l^\frac32} (-1)^l 2 (l+1) \binom{\frac12}{l+1} \\
&\leq  (\ga -1) \rho_0^{\ga-3} \tfrac{c_2c_5}{c_1^2}\tfrac{C^{l-2}}{l^\frac32(l+1)^\frac12},
\end{split}
\ee
where $C$ is large enough so that \eqref{largeC} holds. 
This proves \eqref{eq:Pell2} and \eqref{faaRl}. 
\end{proof}

%%%%%%%%%%%%%%%%%%%%%%%%%
%%%%%%%%%%%%%%%%%%%%%%%%%

We are now ready to estimate the source terms $\mathcal F_N$ and $\mathcal G_N$.

\begin{lemma}\label{lem:FGbound} Let $\alpha\in (1,2)$ be given. Then there exists a constant $C_\ast=C_\ast(y_\ast))>0$ such that if $C>C_\ast$ and for any $N\geq3$, the following assumptions hold  
\begin{align}
|\rho_m| \leq \frac{C^{m-\alpha}}{m^3}, \quad 2\leq m\leq N-1, \label{assumptionRm}\\
|\om_m| \leq \frac{C^{m-\alpha}}{m^3}, \quad 2\leq m\leq N-1, \label{assumptionWm}
\end{align}then we have 
\begin{align}
|\mathcal F_N| &\leq \beta  \frac{C^{N-\alpha}}{N^2}\left[ \frac{1}{C^{\alpha-1}}+\frac{1}{C^{2-\alpha}}+  \frac{1}{CN}  \right], \label{FNbound}\\
|\mathcal G_N| &\leq \beta   \frac{C^{N-\alpha}}{N^2}\left[ \frac{1}{C^{\alpha-1}}+\frac{1}{C^{2-\alpha}}+ \frac{1}{CN}  \right], \label{GNbound}
\end{align}
for some constant $\beta=\beta(y_\ast, \gamma))>0$. 
\end{lemma}

\begin{proof} We start with \eqref{FNbound}. Recall $\mathcal F_N=\mathcal{F}_N^{II} - \mathcal{F}_N^I$ where 
\begin{align}
\mathcal{F}_N^I=&-\rho_1 y_*^2\sum_{\substack{j+k=N,\\j,k\neq N}}\om_j\om_k-\rho_1\big(2y_*(\om^2)_{N-1}-(\om^2)_{N-2}\big)\label{eq:FN1}\\
&+\sum_{\substack{k+j=N\\j\neq0,1,N}}(k+1)\rho_{k+1}\big(\ga P_j-y_*^2(\om^2)_j-2y_*(\om^2)_{j-1}-(\om^2)_{j-2}\big)\label{eq:FN2}\\
&+\ga\rho_1 \rho_0^{\ga-1} \sum_{m=2}^N  \frac{1}{\rho_0^m}\sum_{\pi(N,m)}  \frac{(\ga -1)_m}{\lambda_1 ! \dots \lambda_N !}  \prod_{j=1}^N {\rho_j}^{\lambda_j} \label{eq:FN3}
\end{align}
and 
\begin{align}
&\mathcal{F}_N^{II}=\,(\ga-1)(2-\ga)\rho_{N-1}+(\ga-1)\Big(y_*\sum_{\substack{k+j=N\\k\neq0,N}}\rho_k\om_j+(\rho\om)_{N-1}\Big)\label{eq:FN4}\\
&-\frac{4\pi}{4-3\ga}\Big(y_*\sum_{\substack{k+j+l=N\\k,j,l\neq N}}\rho_k\rho_j\om_l+(\rho^2\om)_{N-1}\Big)+2\Big(y_*\sum_{\substack{k+j+l=N\\ k,j,l\neq N}}\rho_k\om_j\om_l+(\rho\om^2)_{N-1}\Big).\label{eq:FN5}
\end{align}
For the first term of \eqref{eq:FN1}, we use \eqref{E:MBOUND}, \eqref{assumptionWm}, and \eqref{2.51} 
\be\label{5.208}
\begin{split}
\Big| \rho_1 y_*^2\sum_{\substack{j+k=N,\\j,k\neq N}}\om_j\om_k\Big| &= \Big|\rho_1 y_*^2\big( 2\om_1\om_{N-1} + \sum_{k=2}^{N-2} \om_{N-k}\om_k \big) \Big|\\
&\lesssim \frac{C^{N-1-\alpha}}{(N-1)^3} + \sum_{k=2}^{N-2} \frac{C^{N-2\alpha}}{ (N-k)^3 k^3}\lesssim \frac{C^{N-1-\alpha}}{N^3}.
\end{split}
\ee
For the last two terms of  \eqref{eq:FN1}, we have from \eqref{eq:quadratic} 
\be
\big| \rho_1\big(2y_*(\om^2)_{N-1}-(\om^2)_{N-2}\big) \big|\lesssim \frac{C^{N-1-\alpha}}{(N-1)^3}
\ee
The first term of \eqref{eq:FN2} can be estimated as follows. By \eqref{assumptionRm}, \eqref{faaRl},  \eqref{2.54} and  \eqref{2.50},  
\be
\begin{split}
\Big| \sum_{\substack{k+j=N\\j\neq0,1,N}}(k+1)\rho_{k+1} \ga P_j\Big|& \lesssim \sum_{\substack{k+j=N\\j\neq0,1,N}}  \frac{C^{k+1-\alpha}}{(k+1)^2} \Big(\frac{C^{j-\alpha}}{j^3} + \frac{C^{j-2}}{j^2}  \Big)\\
&\lesssim \frac{C^{N+1-2\alpha}}{N^2} + \frac{C^{N-1-\alpha}}{N^2}\lesssim  \frac{C^{N+1-2\alpha}}{N^2},
\end{split}
\ee
where we have used $\alpha<2$ at the last step. The rest of  \eqref{eq:FN2} can be bounded by $\frac{C^{N+1-2\alpha}}{N^2}$ similarly by using \eqref{eq:quadratic} in place of  \eqref{faaRl}. 

For \eqref{eq:FN3}, we first note that $\lambda_N=0$ and hence it does not depend on $\rho_N$. The estimation is identical to the estimation of $P_{N,2}$ in \eqref{eq:Pell}. Therefore, as in \eqref{eq:Pell2} we have 
\be
\Big| \ga\rho_1 \rho_0^{\ga-1} \sum_{m=2}^N  \frac{1}{\rho_0^m}\sum_{\pi(N,m)}  \frac{(\ga -1)_m}{\lambda_1 ! \dots \lambda_N !}  \prod_{j=1}^N {\rho_j}^{\lambda_j}  \Big|\lesssim \frac{C^{N-2}}{N^2}.
\ee
For \eqref{eq:FN4}, by  \eqref{assumptionRm},  \eqref{eq:quadratic} and the same argument as in   \eqref{5.208}, we see that 
\be
\Big|(\ga-1)(2-\ga)\rho_{N-1}+(\ga-1)\Big(y_*\sum_{\substack{k+j=N\\k\neq0,N}}\rho_k\om_j+(\rho\om)_{N-1}\Big)\Big| \lesssim \frac{C^{N-1-\alpha}}{N^3}.
\ee
Next we claim 
\be
| \eqref{eq:FN5} |\lesssim  \frac{C^{N-1-\alpha}}{N^3} .
 \ee
It suffices to verify the bound for the first term of \eqref{eq:FN5}, while \eqref{eq:cubic} gives the desired bound for the second and fourth terms. We rewrite the sum as 
\[
\begin{split}
\sum_{\substack{k+j+l =N\\k,j,l \neq N}}\rho_k\rho_j\om_l &= \om_0 \sum_{\substack{k+j =N\\k,j \neq N}}\rho_k\rho_j + \om_1 \sum_{\substack{k+j =N-1}}\rho_k\rho_j + \sum_{l=2}^{N-1} \om_l \sum_{k+j=N-l} \rho_k\rho_j \\
&=\om_0 \Big(2\rho_1\rho_{N-1} + \sum_{j=2}^{N-2}\rho_{N-j}\rho_j  \Big) + 
\om_1 \Big(2\rho_0\rho_{N-1}+2\rho_1\rho_{N-2} + \sum_{j=2}^{N-3}\rho_{N-j-1}\rho_j  \Big)\\
&\quad+  \sum_{l=2}^{N-2} \om_l  \Big( 2\rho_0\rho_{N-l}+2\rho_1\rho_{N-l-1}+ \sum_{j=2}^{N-l-2} \rho_{N-j-l}\rho_j \Big) + 2\rho_0\rho_1 \om_{N-1}.
\end{split}
\]
Using the induction assumptions and \eqref{2.51}, \eqref{2.52}, we have 
\begin{align}
&\Big| \frac{4\pi}{4-3\ga} y_*\sum_{\substack{k+j+l =N\\k,j,l \neq N}}\rho_k\rho_j\om_l \Big|\notag \\
&\lesssim \frac{C^{N-1-\alpha}}{(N-1)^3} + \sum_{j=2}^{N-2} \frac{C^{N-j-\alpha}}{(N-j)^3}\frac{C^{j-\alpha}}{j^3} +\frac{C^{N-1-\alpha}}{(N-1)^3} +\frac{C^{N-2-\alpha}}{(N-2)^3}+\sum_{j=2}^{N-3} \frac{C^{N-j-1-\alpha}}{(N-j-1)^3}\frac{C^{j-\alpha}}{j^3}\notag  \\
&+ \sum_{l=2}^{N-2}\frac{C^{l-\alpha}}{l^3} \Big(\frac{C^{N-l-\alpha}}{(N-l)^3}+\frac{C^{N-l-1-\alpha}}{(N-l-1)^3} +  \sum_{j=2}^{N-l-2} \frac{C^{N-j-l-\alpha}}{(N-j-l)^3}\frac{C^{j-\alpha}}{j^3}  \Big) + \frac{C^{N-1-\alpha}}{(N-1)^3} \notag \\
&\lesssim\frac{C^{N-\alpha-1}}{N^3}
\end{align}
which shows the desired bound. Combining all the bounds above, we obtain \eqref{FNbound}. 

Next, we recall  $\mathcal G_N=\mathcal{G}_N^{II} - \mathcal{G}_N^I $, where 
\begin{align*}
\mathcal{G}_N^I=&-\om_1\big(2y_*(\om^2)_{N-1}-(\om^2)_{N-2}\big)-\om_1 y_*^2\sum_{\substack{j+k=N,\\j,k\neq N}}\om_j\om_k\\
&+\sum_{\substack{k+j=N\\j\neq0,1,N}}(k+1)\om_{k+1}\big(\ga P_j-y_*^2(\om^2)_j-2y_*(\om^2)_{j-1}-(\om^2)_{j-2}\big)\\
&+\ga\om_1\hspace{-4mm}\sum_{\substack{(m_1,\ldots,m_N)\in M_N\\m_N=0}}\hspace{-4mm}\frac{(\ga-1)\cdots(\ga-(m_1+\cdots+m_N))\rho_0^{\ga-1-(m_1+\cdots+m_N)}}{m_1!\cdots m_N!}\prod_{j=1}^N\rho_j^{m_j},
\end{align*}
and 
\beqas
\mathcal{G}_N^{II}=&\,\widetilde{\mathcal{G}}_N^{II}-\frac{4-3\ga-3\om_0}{y_*}y_*^2\sum_{\substack{j+k=N\\j,k\neq N}}\om_j\om_k\\
&+\frac{4-3\ga-3\om_0}{y_*}\ga\hspace{-3.5mm}\sum_{\substack{(m_1,\ldots,m_N)\in M_N\\m_N=0}}\hspace{-3.5mm}\frac{(\ga-1)\cdots(\ga-(m_1+\cdots+m_N))\rho_0^{\ga-1-(m_1+\cdots+m_N)}}{m_1!\cdots m_N!}\prod_{j=1}^N\rho_j^{m_j},
\eeqas
and 
\beqa\label{eq:GN2tilde}
&\widetilde{\mathcal{G}}_N^{II}=\frac{4-3\ga-3\om_0}{y_*}\big(-2y_*(\om^2)_{N-1}-(\om^2)_{N-2}\big)\\
&+\frac{4-3\ga}{y_*}\sum_{\substack{k+j=N\\j\neq N}}\frac{(-1)^k}{y_*^k}\big(\ga P_j-y_*^2(\om^2)_j-2y_*(\om^2)_{j-1}-(\om^2)_{j-2}\big)\\
&-\frac{3}{y_*}\sum_{\substack{k+j+l=N\\l,j\neq N}}\om_l\frac{(-1)^k}{y_*^k}\big(\ga P_j-y_*^2(\om^2)_j-2y_*(\om^2)_{j-1}-(\om^2)_{j-2}\big)-(\ga-1)(2-\ga)\om_{N-1}\\
&-(\ga-1)\Big(y_*\sum_{\substack{k+j=N\\k\neq0,N}}\om_k\om_j+(\om^2)_{N-1}\Big)+\frac{4\pi}{4-3\ga}\Big(y_*\sum_{\substack{k+j+l=N\\k,j,l\neq N}}(\rho_k\om_j\om_l)+(\rho\om^2)_{N-1}\Big)\\
&-2\Big(y_*\sum_{\substack{k+j+l=N\\k,j,l\neq N}}(\om_k\om_j\om_l)+(\om^3)_{N-1}\Big).
\eeqa
Note that the structure of $\mathcal G_N^I$ and $\mathcal G_N^{II}$ is similar structure to the structure of $\mathcal F_N^I$ and $\mathcal F_N^{II}$ except for the second and third lines of \eqref{eq:GN2tilde}.  Hence we focus on the second and third lines of \eqref{eq:GN2tilde}. 

We may take $C>0$ sufficiently large if necessary to ensure 
\be\label{boundy}
\frac{1}{y_\ast^k} \lesssim  \frac{C^{k-2}}{k^2} \text{ for all } k\ge 2.
\ee
 Now for the first term in the second line of \eqref{eq:GN2tilde}, we split indices into $j=0,1$ and $j\ge 2$ and  use  \eqref{faaRl} and \eqref{boundy} to deduce 
\[
\begin{split}
\Big| \sum_{\substack{k+j=N\\j\neq N}}\frac{(-1)^k}{y_*^k} \ga P_j \Big|&=   
\Big|  \frac{(-1)^N}{y_*^N} \ga P_0 +\frac{(-1)^{N-1}}{y_*^{N-1}} \ga P_1  + \sum_{j=2}^{N-1}\frac{(-1)^{N-j}}{y_*^{N-j}} \ga P_j  \Big| \\
&\lesssim \frac{1}{y_\ast^{N}} +  \frac{1}{y_\ast^{N-1}} + \sum_{j=2}^{N-1} \frac{1}{y_\ast^{N-j}} \left(\frac{C^{j-\alpha}}{j^3} + \frac{C^{j-2}}{j^2}\right)\\
&\lesssim \frac{C^{N-2}}{N^2} + \frac{C^{N-1-\alpha}}{(N-1)^3} + \frac{C^{N-3}}{(N-1)^2} .
\end{split}
\]
This yields the desired bound. The remaining terms in the second line can be estimated in the same way by using \eqref{eq:quadratic} in place of \eqref{faaRl}. 

We may proceed analogously for the third line and use \eqref{boundy}. We present the details for the first term in the third line of  \eqref{eq:GN2tilde}. First, we split indices as
\[
\begin{split}
&\Big|\sum_{\substack{k+j+l=N\\l, j\neq N}}\om_l\frac{(-1)^k}{y_*^k} \ga P_j\Big| \\
&\le  
\Big| \ga P_0 \sum_{\substack{k+l=N\\l \neq N}}\om_l\frac{(-1)^k}{y_*^k} \Big| +\Big|  \ga P_1 \sum_{k+l=N-1} \om_l\frac{(-1)^k}{y_*^k} \Big| +\Big|  \sum_{j=2}^{N-1} \ga P_j \sum_{k+l=N-j} \om_l\frac{(-1)^k}{y_*^k}\Big|\\
&=: S_1 +S_2 +S_3.
\end{split}
\]
For $S_1$, using \eqref{assumptionWm}, \eqref{boundy}, and \eqref{2.54}, we have 
\[
\begin{split}
S_1& \lesssim |\om_{N-1}| +| \om_0|\frac{1}{y_\ast^N} + |\om_1|\frac{1}{y_*^{N-1}} +  \sum_{k=2}^{N-2} |\om_{N-k}|\frac{1}{y_*^{k}}\\
&\lesssim \frac{C^{N-1-\alpha}}{(N-1)^3} + \frac{C^{N-2}}{N^2} + \frac{C^{N-3}}{(N-1)^2} + \sum_{k=2}^{N-2} \frac{C^{N-k-\alpha}}{(N-k)^3}\frac{C^{k-2}}{k^2}\\
&\lesssim \frac{C^{N-1-\alpha}}{N^3} + \frac{C^{N-2}}{N^2}.
\end{split}
\]
The estimation of $S_2$ is entirely analogous, while for $S_3$ we split the indices further to deduce 
\[
\begin{split}
S_3&\le \Big|\ga P_{N-1} \Big( \om_1 + \om_0\frac{(-1)}{y_*}\Big) \Big| +\Big| \ga P_{N-2} \om_1\frac{(-1)}{y_*}  \Big|
+\Big|  \sum_{j=2}^{N-2} \ga P_j \sum_{k=2}^{N-j-2}\om_{N-j-k}\frac{(-1)^k}{y_*^k}\Big|\\
&+ \Big| \sum_{j=2}^{N-2} \ga P_j \Big(\om_0 \frac{(-1)^{N-j}}{y_*^{N-j}}+ \om_{N-j} \Big)\Big| +\Big| \sum_{j=2}^{N-3} \ga P_j \Big( \om_1 \frac{(-1)^{N-j-1}}{y_*^{N-j-1}} +\om_{N-j-1}\frac{(-1)}{y_*}  \Big) \Big|   \\
&\lesssim \frac{C^{N-1-\alpha}}{(N-1)^3} + \frac{C^{N-3}}{(N-1)^2} + \sum_{j=2}^{N-2}\left( \frac{C^{j-\alpha}}{j^3} + \frac{C^{j-2}}{j^2} \right)\sum_{k=2}^{N-j-2} \frac{C^{N-j-k-\alpha}}{(N-j-k)^3}\frac{C^{k-2}}{k^2}\\
&+ \sum_{j=2}^{N-2} \left( \frac{C^{j-\alpha}}{j^3} + \frac{C^{j-2}}{j^2} \right) \left(\frac{C^{N-j-2}}{(N-j)^2} +\frac{C^{N-j-\alpha}}{(N-j)^3}  \right)\\
&+ \sum_{j=2}^{N-3} \left( \frac{C^{j-\alpha}}{j^3} + \frac{C^{j-2}}{j^2} \right) \left(\frac{C^{N-j-3}}{(N-j-1)^2} +\frac{C^{N-j-1-\alpha}}{(N-j-1)^3}  \right)\\
&\lesssim \frac{C^{N-1-\alpha}}{N^3} + \frac{C^{N-3}}{N^2}.
\end{split}
\]
Other terms  in the third line of~\eqref{eq:GN2tilde} can be estimated in the same way. This finishes the proof of~\eqref{GNbound}. 
\end{proof}

%%%%%%%%%%%%%%%%%%%%%%%%%%%
%%%%%%%%%%%%%%%%%%%%%%%%%%%

\section{Interval arithmetic}\label{A:IA}

%%%%%%%%%%%%%%%%%%%%%%%%%%%
%%%%%%%%%%%%%%%%%%%%%%%%%%%

\renewcommand{\theequation}{\thesection.\arabic{equation}}

Interval arithmetic is a numerical technique that allows for the rigorous proof of inequalities and estimates through replacing real numbers by closed intervals with end-points representable as floating point numbers. A survey of some of the uses of interval arithmetic in PDE theory may be found in \cite{GS18}. For our purposes, we require only a very basic level of application of this method in order to estimate the signs of somewhat complicated polynomials in two variables over rectangular domains, and so we use the straightforward interval arithmetic packages available in the Julia computing language. 

{\footnotesize{\begin{lstlisting} 
using IntervalArithmetic
using IntervalOptimisation
\end{lstlisting}}}

In this section of the appendix, we give the proofs of Proposition \ref{prop:R1W1}, Lemma \ref{lemma:Qmintervalarithmetic} and inequalities \eqref{ineq:A2}--\eqref{ineq:dNdetA0}. In each proof, we will insert the relevant Julia commands and state the outputs at the relevant point in the proof.

Maximisation or minimisation at ech step is taken either over a fixed interval of $\ga$ or a vector v=(v[1],v[2])$=(\om+\ga,\ga)$. This ensures that the domain of v[1] is a fixed numerical interval (for example, the full range $\om\in[\frac{4-3\ga}{3},2-\ga]$ becomes v[1]$\in[\frac43,2]$). The two principal ranges over which we will work are then defined by
{\footnotesize\begin{lstlisting}
V=IntervalBox((4/3)..2,1..(4/3))
G=1..(4/3)
\end{lstlisting}}
\noindent When defining functions of $\om$ and $\ga$, we use the characters w and g respectively for $\om$ and $\ga$.

%%%%%%%%%%%%%%%%%%%%%%%%
%%%%%%%%%%%%%%%%%%%%%%%%

\subsection{Proofs of $s(\om_0)>0$ and Proposition~\ref{prop:R1W1}}\label{A:TOTHERIGHT}
Before verifying the claimed inequalities on $R_1$ stated in Proposition \ref{prop:R1W1}, we first complete the proof of Lemma \ref{L:BRANCHES} to show that $R_1$ and $R_2$ are well-defined functions of $\om_0$ and $\ga$, that is, that the square root of $s(\om_0)$ in the definitions \eqref{E:R1} and \eqref{E:R2} always yields a real number.
\begin{proof}[Proof of Lemma \ref{L:BRANCHES}, continued]
 Consider the definitions of $R_1$ and $R_2$ stated in \eqref{E:R1} and \eqref{E:R2}. The argument of the square root is $\om_0^3s(\om_0)$, and so to show that these are well-defined functions, it suffices to prove that $s(\om_0,\ga)>0$ for all $\om_0\in[\frac{4-3\ga}{3},2-\ga]$ where we now make explicit the dependence on $\ga$, so that
\beqas
s(\om_0,\ga)=&-4(4-3\ga)(\ga+1)(\ga-1)(2-\ga)+(57 - 114 \ga + 73 \ga^2 - 12 \ga^3)\om_0\\
&-8(14 - 15 \ga + 3 \ga^2)\om_0^2+8(5-3\ga)\om_0^3.\eeqas
We verify with interval arithmetic that $s(\om_0,\ga)>0$ for $\om_0\in[\frac{4-3\ga}{3},2-\ga]$, $\ga\in(1,\frac43)$ as follows:\\
First, we note that when $\ga=1$, $s(\om_0,1)=4\om_0(2\om_0-1)^2$, which is non-negative on the domain. Next, we differentiate with respect to $\ga$ to find
\beqas
\d_\ga s(\om_0,\ga)=&\,-8(6\ga^3-15\ga^2+5\ga+5)-(36\ga^2-146\ga+114)\om_0-8(6\ga-15)\om_0^2-24\om_0^3.
\eeqas
For $\om_0\in[1.42-\ga,2-\ga]$, we find that this is strictly positive by
{\footnotesize\begin{lstlisting}
sg(w,g)=-8*(5+5*g-15*g^2+6*g^3)-(114-146*g+36*g^2)*w-8*(-15+6*g)*w^2-24*w^3
Sg(v)=sg(v[1]-v[2],v[2])
V6=IntervalBox((1.42)..2,1..(4/3))
minimise(Sg,V6,tol=1e-3)
\end{lstlisting}}
\noindent which gives the minimum in $[1.06209, 1.26472]$, hence for $\om_0\geq 1.42-\ga$, $\ga\in[1,\frac{4}{3}]$, $s(\om_0,\ga)>0$. Next, for $\ga\in[1,1.1]$ and $\om_0\in[\frac{4}{3}-\ga,1.42-\ga]$, we check that $s(\om_0,\ga)>0$ by
{\footnotesize\begin{lstlisting}
s(w,g)=-4*(4-3*g)*(g+1)*(g-1)*(2-g)+(57-114*g+73*g^2-12*g^3)*w
          -8*(14-15*g+3*g^2)*w^2+8*(5-3*g)*w^3
S(v)=s(v[1]-v[2],v[2])
V7=IntervalBox((4/3)..(1.42),1..(1.1))
minimise(S,V7,tol=1e-4)
\end{lstlisting}}
\noindent which gives the minimum in $[0.0334093, 0.0431525]$. Finally, for $\ga\in[1.1,\frac43]$, we check first that
\beqas
s(\frac{4-3\ga}{3},\ga)=\frac{1}{27}(5-3\ga)^2(4-3\ga)>0
\eeqas
and then 
\beqas
s_{\om_0}(\om_0,\ga)=&(57 - 114 \ga + 73 \ga^2 - 12 \ga^3)-16(14 - 15 \ga + 3 \ga^2)\om_0+24(5-3\ga)\om_0^2\eeqas
is uniformly positive by
{\footnotesize\begin{lstlisting}
sw(w,g)=(57-114*g+73*g^2-12*g^3)-16*(14-15*g+3*g^2)*w+24*(5-3*g)*w^2
Sw(v)=sw(v[1]-v[2],v[2])
V8=IntervalBox((4/3)..2,(1.1)..(4/3))
minimise(Sw,V8,tol=1e-2)
\end{lstlisting}}
\noindent which puts the minimum in $[0.336312, 2.0698]$, and hence $s(\om_0,\ga)$ is strictly increasing with respect to $\om_0$.
\end{proof}

%%%%%%%%%%%%%%%%%%%%%%%%%
%%%%%%%%%%%%%%%%%%%%%%%%%

\begin{proof}[Proof of Proposition \ref{prop:R1W1}]
We will first show that for $\om_0\in[\frac{4-3\ga}{3},2-\ga]$, we have $R_1<-\frac{1}{2-\ga}$ for all $\ga\in(1,\frac43)$, while for $\om_0\geq\frac{4-3\ga}{3}$ and $\ga\geq\frac{10}{9}$, we have $R_1\leq-\frac{2\ga}{(2-\ga)(\ga+1)}$ with strict inequalities when either $\ga>\frac{10}{9}$ or $\om_0>\frac{4-3\ga}{3}$.

To check the claimed inequalities on $R_1$, we use the following method:\\
\textit{Step 1: We prove $R_1<-\frac{1}{2-\ga}$.}\\
 First, 
\beqas
R_1&+\frac{1}{2-\ga}=\frac{(2-\ga)(9\om_0^2-7\ga\om_0^2-8\om_0^3)+2(\ga+1)\om_0^3}{2(2-\ga)(\ga+1)\om_0^3}\\
& - \frac{2-\ga}{2(2-\ga)\om_0^{3}(\ga+1)}\Big(-4(4-3\ga)(\ga+1)(\ga-1)(2-\ga)\om_0^3+(57 - 114 \ga + 73 \ga^2 - 12 \ga^3)\om_0^4\\
&\hspace{4cm}-8(14 - 15 \ga + 3 \ga^2)\om_0^5+8(5-3\ga)\om_0^6\Big)^{\frac12}.
\eeqas
It is therefore sufficient to check the sign of the numerator. When the terms in the numerator on the first line are negative, as the contribution of the square root is negative, we are clearly done. We claim 
\beqa
&\Big((2-\ga)(9\om_0^2-7\ga\om_0^2-8\om_0^3)+2(\ga+1)\om_0^3\Big)^2\\
&-(2-\ga)^2\Big(-4(4-3\ga)(\ga+1)(\ga-1)(2-\ga)\om_0^3+(57 - 114 \ga + 73 \ga^2 - 12 \ga^3)\om_0^4\\
&\quad-8(14 - 15 \ga + 3 \ga^2)\om_0^5+8(5-3\ga)\om_0^6\Big)<0
\eeqa
for all $\om_0\in[\frac43-\ga,2-\ga]$, $\ga\in(1,\frac43)$. This implies the claimed inequality as, in the remaining case that the first terms are positive, this establishes that the contribution of the square root is strictly larger, and hence the difference is negative. To verify this claim, we first cancel a factor of $\om_0^3$ and consider instead
\beqas
r_1(\om_0):=&\,\om_0\Big((2-\ga)(9-7\ga-8\om_0)+2(\ga+1)\om_0\Big)^2\\
&-(2-\ga)^2\Big(-4(4-3\ga)(\ga+1)(\ga-1)(2-\ga)+(57 - 114 \ga + 73 \ga^2 - 12 \ga^3)\om_0\\
&\qquad-8(14 - 15 \ga + 3 \ga^2)\om_0^2+8(5-3\ga)\om_0^3\Big)\\
&=4(\ga^2-1)\Big((6\ga-9)\om_0^3+(6\ga^2-19\ga+14)\om_0^2+(3\ga^3-18\ga^2+36\ga-24)\om_0\\
&\qquad+(3\ga^4-22\ga^3+60\ga^2-72\ga+32)\Big).
\eeqas
Eliminating the strictly positive factor $4(\ga^2-1)$, we check that at $\om_0=\frac{4-3\ga}{3}$, the remainder satisfies
\beqas
\frac{r_1(\frac{4-3\ga}{3})}{4(\ga^2-1)}=&\,\Big((6\ga-9)\om_0^3+(6\ga^2-19\ga+14)\om_0^2+(3\ga^3-18\ga^2+36\ga-24)\om_0\\
&\qquad+(3\ga^4-22\ga^3+60\ga^2-72\ga+32)\Big)\bigg|_{\om_0=\frac{4-3\ga}{3}}\\
=&\,-\frac{2}{9}(4-3\ga)^2(\ga-1)<0.
\eeqas
We then take a derivative with respect to $\om_0$ to arrive at
\beqas
\d_{\om_0}\Big(\frac{r_1(\om_0)}{4(\ga^2-1)}\Big)=3(6\ga-9)\om_0^2+2(6\ga^2-19\ga+14)\om_0+(3\ga^3-18\ga^2+36\ga-24)<0
\eeqas
for all $\om_0\in[\frac{4-3\ga}{3},2-\ga]$ and $\ga\in[1,\frac43]$ by interval arithmetic:
{\footnotesize\begin{lstlisting}
quad10(w,g)=3*(6*g-9)*w^2+ 2*(6*g^2-19*g+14)*w+3*g^3-18*g^2+36*g-24
q10(v)=quad10(v[1]-v[2],v[2])
maximise(q10,V,tol=1e-2)
\end{lstlisting}}
\noindent with output in the closed interval $[-0.910166, -0.627474]$,
 thus finishing the proof that $R_1<-\frac{1}{2-\ga}$.

\textit{Step 2: Prove $R_1\leq-\frac{2\ga}{(2-\ga)(\ga+1)}$ for $\ga\geq\frac{10}{9}$ with equality only for $\ga=\frac{10}{9}$ and $\om_0=\frac{4-3\ga}{3}$.}\\
 We argue similarly  to Step 1. First, we apply \eqref{E:R1} to find 
 \beqas
R_1&\,+\frac{2\ga}{(2-\ga)(\ga+1)}=\frac{(2-\ga)(9\om_0^2-7\ga\om_0^2-8\om_0^3)+4\ga\om_0^3}{2(2-\ga)(\ga+1)\om_0^3}\\
& - \frac{2-\ga}{2(2-\ga)\om_0^{3}(\ga+1)}\Big(-4(4-3\ga)(\ga+1)(\ga-1)(2-\ga)\om_0^3+(57 - 114 \ga + 73 \ga^2 - 12 \ga^3)\om_0^4\\
&\qquad\qquad-8(14 - 15 \ga + 3 \ga^2)\om_0^5+8(5-3\ga)\om_0^6\Big)^{\frac12}.
\eeqas
We again need only to compare the quantities in the numerator, and so we will prove that
\beqas
&\Big((2-\ga)(9\om_0^2-7\ga\om_0^2-8\om_0^3)+4\ga\om_0^3\Big)^2\\
&-(2-\ga)^2\big(-4(4-3\ga)(\ga+1)(\ga-1)(2-\ga)\om_0^3+(57 - 114 \ga + 73 \ga^2 - 12 \ga^3)\om_0^4\\
&\qquad\quad-8(14 - 15 \ga + 3 \ga^2)\om_0^5+8(5-3\ga)\om_0^6\big)\leq 0
\eeqas
with equality only when both $\ga=\frac{10}{9}$ and $\om_0=\frac{4-3\ga}{3}$.
Simplifying, we find that this expression is equal to $4(\ga-1)(\ga+1)^2\om_0^3r_2(\om_0)$, where
\beqas
r_2(\om_0)=&\,\Big((6\ga^2+8\ga-24)\om_0^3+(6\ga^3-6\ga^2-28\ga+32)\om_0^2\\
&+(3\ga^4-15\ga^3+18\ga^2+12\ga-24)\om_0+3\ga^5-19\ga^4+38\ga^3-12\ga^2-40\ga+32\Big).
\eeqas
Considering only $r_2(\om_0)$ (as the remaining factors are positive), we check
\beqas
r_2(\frac{4-3\ga}{3})=-\frac{2}{27}\ga(4-3\ga)^2(9\ga-10)\leq0
\eeqas
with equality only for $\ga=\frac{10}{9}$. Moreover, differentiating with respect to $\om_0$ yields
\beqas
\d_{\om_0}r_2(\om_0)=&\,3(6\ga^2+8\ga-24)\om_0^2+2(6\ga^3-6\ga^2-28\ga+32)\om_0\\
&+(3\ga^4-15\ga^3+18\ga^2+12\ga-24)<0
\eeqas
for all $\om_0\in[\frac{4-3\ga}{3},2-\ga]$ and $\ga\geq\frac{10}{9}$ (actually all $\ga\in[1,\frac43]$) by interval arithmetic:
{\footnotesize\begin{lstlisting}
quad11(w,g)=3*(6*g^2+8*g-24)*w^2+ 2*(6*g^3-6*g^2-28*g+32)*w
               +3*g^4-15*g^3+18*g^2+12*g-24
q11(v)=quad11(v[1]-v[2],v[2])
maximise(q11,V,tol=1e-2)
\end{lstlisting}}
\noindent with output in $[-2.12454, -1.63927]$, concluding the proof.

\textit{Step 3: Prove $R_1>-\frac{4}{(4-3\ga)(2-\ga)}$.}\\
 The only remaining estimate for $R_1$ is the lower bound, and again we follow the above approach. We first group
\beqa\label{eq:R1loweridentity}
R_1&\,+\frac{4}{(2-\ga)(4-3\ga)}=\frac{(2-\ga)(4-3\ga)(9\om_0^2-7\ga\om_0^2-8\om_0^3)+8(\ga+1)\om_0^3}{2(2-\ga)(4-3\ga)(\ga+1)\om_0^3}\\
& - \frac{(2-\ga)(4-3\ga)}{2(2-\ga)(4-3\ga)(\ga+1)\om_0^{3}}\Big(8(5-3\ga)\om_0^6-8(14 - 15 \ga + 3 \ga^2)\om_0^5\\
&\hspace{28mm}+(57 - 114 \ga + 73 \ga^2 - 12 \ga^3)\om_0^4-4(4-3\ga)(\ga+1)(\ga-1)(2-\ga)\om_0^3\Big)^{\frac12}.
\eeqa
One easily sees that
$$(2-\ga)(4-3\ga)(9\om_0^2-7\ga\om_0^2-8\om_0^3)+8(\ga+1)\om_0^3>0$$
provided $\om_0>\bar\om=\frac{ (2-\ga)(4-3\ga)(9-7\ga)}{8(3\ga^2-11\ga+7)}$. As this value is always less than $\frac{4-3\ga}{3}$ (indeed, $\frac{4-3\ga}{3}-\bar\om=\frac{(4-3\ga)(3\ga^2-19\ga+2)}{24 (3\ga^2-11\ga+7)}>0$), we conclude that the quantity is positive always.

It is therefore sufficient to compare the sizes of the squares of the terms in the numerator of \eqref{eq:R1loweridentity}. We therefore consider
\beqas
&\Big((2-\ga)(4-3\ga)(9\om_0^2-7\ga\om_0^2-8\om_0^3)+8(\ga+1)\om_0^3\Big)^2\\
&-(2-\ga)^2(4-3\ga)^2\Big(-4(4-3\ga)(\ga+1)(\ga-1)(2-\ga)\om_0^3+(57 - 114 \ga + 73 \ga^2 - 12 \ga^3)\om_0^4\\
&\qquad\quad -8(14 - 15 \ga + 3 \ga^2)\om_0^5+8(5-3\ga)\om_0^6\Big)\\
&=4(\ga+1)\om_0^3r_3(\om_0),
\eeqas
where
\beqas
r_3(\om_0)=&\,6 (9\ga^4-60\ga^3+132\ga^2-104\ga+24)\om_0^3+2 (2 - \ga) (4 - 3 \ga) (9\ga^3-42\ga^2+50\ga-14)\om_0^2\\
&\qquad-3 (2 - \ga)^3 (\ga-1) (4 - 3 \ga)^2\om_0+(2 - \ga)^3 (\ga-1) (4 - 3 \ga)^3.
\eeqas
As usual, we evaluate  at $\om_0=\frac{4-3\ga}{3}$ and find
\beqas
r_3(\frac{4-3\ga}{3})=-\frac{4}{9} (4 - 3 \ga)^3 (\ga^2-5\ga+2)>0
\eeqas
as $\ga^2-5\ga+2<0$ for $\ga\in(1,\frac43)$. The derivative with respect to $\om_0$ is then
\beqas
r_3'(\om_0)=&\,18 (9\ga^4-60\ga^3+132\ga^2-104\ga+24)\om_0^2+4 (2 - \ga) (4 - 3 \ga) (9\ga^3-42\ga^2+50\ga-14)\om_0\\
&-3 (2 - \ga)^3 (\ga-1) (4 - 3 \ga)^2.
\eeqas
This is strictly positive as, at $\om_0=\frac{4-3\ga}{3}$, we have
$$r_3'(\frac{4-3\ga}{3})=\frac13(4 - 3 \ga)^2 (27\ga^4-183\ga^3+418\ga^2-348\ga+104)>0$$
as the quartic in $\ga$ is uniformly positive:
{\footnotesize\begin{lstlisting}
g1(g)=104 - 348*g + 418*g^2 - 183*g^3 + 27*g^4
minimise(g1,G)
\end{lstlisting}}
\noindent gives lower bound in the interval $[17.4556, 18.0143]$. Moreover, the further $\om_0$ derivative is
\beqas
r_3''(\om_0)=&\,36 (9\ga^4-60\ga^3+132\ga^2-104\ga+24)\om_0+4 (2 - \ga) (4 - 3 \ga) (9\ga^3-42\ga^2+50\ga-14)\\
\geq&\, (4-3\ga)\Big(12 (9\ga^4-60\ga^3+132\ga^2-104\ga+24)+4 (2 - \ga)  (9\ga^3-42\ga^2+50\ga-14)\Big),
\eeqas
where we have used $\om_0\geq\frac{4-3\ga}{3}$ in the second line and that $9\ga^4-60\ga^3+132\ga^2-104\ga+24>0$ by 
{\footnotesize\begin{lstlisting}
g13(g)=9*g^4-60*g^3+132*g^2-104*g+24
minimise(g13,G)
\end{lstlisting}}
\noindent which gives the minimum in $[0.827639, 1.00486]$. We then further apply interval arithmetic to show the positivity of the last quantity:
{\footnotesize\begin{lstlisting}
g2(g)=12*(9*g^4-60*g^3+132*g^2-104*g+24)+4*(2-g)*(9*g^3-42*g^2+50*g-14)
minimise(g2,G)
\end{lstlisting}}
\noindent gives lower bound in the interval $[21.7206, 24.0462]$ which concludes the proof of the estimates for $R_1$.

\textit{Step 4: Prove the lower bound $W_1\geq 0$ with equality only for $y_*=y_f$.}\\
 The final step is the lower bound for $W_1$. We first rearrange \eqref{eq:WRidentity}  to see
\beq
W_1=(\ga-1)R_1\frac{\om_0^2 R_1+(\om_0+2-\ga)}{2\om_0(R_1+1)-\frac{(\ga-1)(2-\ga)}{\om_0}},
\eeq
where we note that, as $R_1<-\frac{1}{2-\ga}$, the denominator satisfies
\beqas
2\om_0(R_1+1)-\frac{(\ga-1)(2-\ga)}{\om_0}<&\,\frac{1}{\om_0}\Big(-\frac{2(\ga-1)}{2-\ga}\om_0^2-(\ga-1)(2-\ga)\Big)<0.
\eeqas
It is therefore sufficient to verify that $\om_0^2 R_1+(\om_0+2-\ga)\geq0$ with strict equality for $\om_0\in[\frac{4-3\ga}{3},2-\ga)$ (equivalently $y_*\in(y_f,y_F]$). At the end-point $y_*=y_f$, equivalently $\om_0=2-\ga$, a direct computation reveals $R_1=-\frac{2}{2-\ga}$ and $W_1=0$.

To prove the lower bound, we substitute $R_1$ from \eqref{E:R1} and rearrange to find
\beqas
\om_0^2 R_1+(\om_0+2-\ga)=&\,\frac{(9-7\ga) \om_0^2 - 8 \om_0^3-\sqrt{\om_0^3s(\om_0)}+2(\ga+1)\om_0^2+2(\ga+1)(2-\ga)\om_0}{2\om_0(\ga+1)}\\
=&\frac{(11-5\ga) \om_0 - 8 \om_0^2-\sqrt{\om_0s(\om_0)}+2(\ga+1)(2-\ga)}{2(\ga+1)}.
\eeqas
It is a simple exercise to check that the quadratic
$$-8\om_0^2+(11-5\ga)\om_0+2(\ga+1)(2-\ga)>0\text{ for all }\om_0\in[\frac{4-3\ga}{3},2-\ga],$$
and so it suffices to show that
\beq
L(\om_0)=\Big(-8\om_0^2+(11-5\ga)\om_0+2(\ga+1)(2-\ga)\Big)^2-\om_0s(\om_0,\ga)>0
\eeq 
for $\om_0\in[\frac{4-3\ga}{3},2-\ga)$.

We first obtain the lower bound for $\om_0\in[\frac{4-3\ga}{3},1.8-\ga]$ by interval arithmetic:
{\footnotesize\begin{lstlisting}
L(w,g)=((11-5*g)*w-8*w^2+2*(g+1)*(2-g))^2-w*s(w,g)
L1(v)=L(v[1]-v[2],v[2])
V9=IntervalBox((4/3)..1.8,1..(4/3))
minimise(L1,V9,tol=1e-2)
\end{lstlisting}}
\noindent which gives that the minimum lies in $[8.32454, 9.37091]$.

On the remaining region, we recall that $W_1(2-\ga)=0$ and hence $L(2-\ga)=0$. A direct computation shows that
$$L'(\om_0)=4 (\ga+1) (24 \om_0^3+6(3\ga-8)\om_0^2 +2\ga(3\ga-7)\om_0 -(2 - \ga) (-7 - 2 \ga + 3 \ga^2))$$
and further interval arithmetic shows that this is strictly negative for $\om_0\in[1.8-\ga,2-\ga]$ by
{\footnotesize\begin{lstlisting}
DL(w,g)=4*(g+1)*(24*w^3+6*(3*g-8)*w^2 +2*g*(3*g-7)*w-(2-g)*(-7- 2*g+3*g^2))
DL1(v)=DL(v[1]-v[2],v[2])
V10=IntervalBox((1.8)..2,1..(4/3))
maximise(DL1,V10)
\end{lstlisting}}
\noindent which gives that the maximum lies in $[-33.9807, -33.5971]$. Hence $L(\om_0)>0$ for $\om_0\in[\frac{4-3\ga}{3},2-\ga)$ as required.
\end{proof}

%%%%%%%%%%%%%%%%%%%%%%%%
%%%%%%%%%%%%%%%%%%%%%%%%

\subsection{Proof of~\eqref{ineq:A2}--\eqref{ineq:dNdetA0}}\label{A:IADETERMINANTS}

\textit{(i)} The easiest of the inequalities to show is inequality \eqref{ineq:A2} for $A_2$. Indeed, we recall that $R<0$ and $W\geq0$ for all $\ga\in(1,\frac43)$ and $\om_0\in[\frac{4-3\ga}{3},2-\ga]$ and consider the coefficient of the $\om_0^2R$ term:
$$-2(3-\ga)\om_0^2+\om_0(\ga-1)(5\ga-9)-(\ga-1)(2-\ga)(\ga+1)<0,$$
where the inequality comes from interval arithmetic:
{\footnotesize\begin{lstlisting}
quad1(w,g)=-2*(3-g)*w^2+(g-1)*(5*g-9)*w-(g-1)*(2-g)*(g+1)
q1(v)=quad1(v[1]-v[2],v[2])
maximise(q1,V,tol=1e-3)
\end{lstlisting}}
\noindent which gives an upper bound in the range $[-0.445382, -0.442693]$. As the contributions from $W$ and the remainder are both non-negative, we conclude $A_2>0$.

\textit{(ii)} Next, we show the  inequality~\eqref{ineq:dNdetA0}, $4A_2+A_1>0$. This is more complicated than before and requires us to consider the coefficients on separate parts of the domain. We first simplify the expression for this sum as
\beqas
4A_2+A_1=&\,\big(-10(3-\ga)\om_0^2+2(\ga-1)(10\ga-19)\om_0-5(\ga-1)(2-\ga)(\ga+1)\big)\om_0^2R\\
&+\big(40\om_0^2-4(4-3\ga)\om_0-2(\ga-1)\om_0\big)\om_0W+20\om_0^4-(14-10\ga)\om_0^3.
\eeqas
Again, the coefficient of $\om_0^2R$ is negative on the whole region of interest by interval arithmetic:
{\footnotesize\begin{lstlisting}
quad2(w,g)=-10*(3-g)*w^2+2*(g-1)*(10*g-19)*w-5*(g-1)*(2-g)*(g+1)
q2(v)=quad2(v[1]-v[2],v[2])
maximise(q2,V,tol=1e-3)
\end{lstlisting}}
\noindent which gives an upper bound in the range $[-2.22671, -2.21346]$. We therefore focus on the other two coefficients. The coefficient of $\om_0W$ is clearly positive when $\om_0>\frac{7-5\ga}{20}$ and negative otherwise (for $\om_0>0$). One also checks easily that $\frac{7-5\ga}{20}\geq\frac{4-3\ga}{3}$ is equivalent to $\ga\geq\frac{59}{45}$ (and $\frac{59}{45}<\frac43$).

Moreover, the final coefficient is
$$\om_0^3(20\om_0-(14-10\ga))\geq0\text{ if and only if }\om_0\geq\frac{7-5\ga}{10}.$$
Note that $\frac{7-5\ga}{10}\geq\frac{4-3\ga}{3}$ only for $\ga\geq \frac{19}{15}$.

In the region $\om_0\geq\frac{7-5\ga}{10}$, we therefore have $4A_2+A_1>0$, as required (and in particular, this holds for the whole region when $\ga\leq \frac{19}{15}$). For $\ga\in(\frac{19}{15},\frac43)$, we consider $$\om_0\in\Big(\max\{\frac{7\ga-5}{20},\frac{4-3\ga}{3}\},\frac{7\ga-5}{10}\Big),$$ so the coefficient of $\om_0W$ is positive for all $\om_0$ of interest. Recalling also that the coefficient of $\om_0^2R$ is negative and that $R<-\frac{1}{2-\ga}$, we therefore bound $4A_2+A_1$ below by
\beqas
4A_2+A_1\geq&\,-\big(-10(3-\ga)\om_0^2+2(\ga-1)(10\ga-19)\om_0-5(\ga-1)(2-\ga)(\ga+1)\big)\om_0^2\frac{1}{2-\ga}\\
&+20\om_0^4-(14-10\ga)\om_0^3\\
=&\,\frac{\om_0^2}{2-\ga}\Big(\big(10(3-\ga)+20(2-\ga)\big)\om_0^2\\
&+\big(-2(\ga-1)(10\ga-19)-(2-\ga)(14-10\ga)\big)\om_0+5(\ga-1)(2-\ga)(\ga+1)\Big)>0,
\eeqas
where we check the sign of the final quadratic using
{\footnotesize\begin{lstlisting}
quad3(w,g)=10*(7-3*g)*w^2-2*(15*g^2 - 46*g + 33)*w+5*(g-1)*(2-g)*(g+1)
q3(v)=quad3(v[1]-v[2],v[2])
minimise(q3,V,tol=1e-3)
\end{lstlisting}}
\noindent which gives a lower bound in the range $[2.49842, 2.51321]$ (actually for all $\om_0\in(\frac{4-3\ga}{3},2-\ga)$ and $\ga\in(1,\frac43)$.

Finally, for $\ga\in(\frac{59}{45},\frac43)$ and $\om_0\in(\frac{4-3\ga}{3},\frac{7-5\ga}{20})$, we compare $\om_0W$ to $\om_0^2R$ using the formula
\beqas
\frac{\om_0W}{\om_0^2R}=\frac{(\ga-1)\om_0^2R+(\ga-1)(\om_0+2-\ga)}{2\om_0^2R+2\om_0^2-(\ga-1)(2-\ga)}=:B(R,\om_0).
\eeqas
Differentiation of $B(R,\om_0)$ with respect to $R$ reveals that
\beqas
\frac{\d}{\d R}B(R,\om_0)=\frac{(\ga-1)\om_0^2\big(2\om_0^2-2\om_0-(2-\ga)(\ga+1)\big)}{\big(2\om_0^2R+2\om_0^2-(\ga-1)(2-\ga)\big)^2}<0,
\eeqas
for all $\om_0\in(\frac{4-3\ga}{3},2-\ga)$, $\ga\in(1,\frac43)$, so $B$  is a decreasing function with respect to $R$. Hence, recalling again that $R\leq-\frac{1}{2-\ga}$, we have that 
\beqa\label{om0W/om0R}
\frac{\om_0W}{\om_0^2R}\leq&\, B(-\frac{1}{2-\ga},\om_0)=\frac{-(\ga-1)\om_0^2+(2-\ga)(\ga-1)(\om_0+2-\ga)}{-2\om_0^2+2(2-\ga)\om_0^2-(\ga-1)(2-\ga)^2}\\
=&\,\frac{\om_0^2-(2-\ga)(\om_0+2-\ga)}{2\om_0^2+(2-\ga)^2}=-1+\frac{3\om_0^2-(2-\ga)\om_0}{2\om_0^2+(2-\ga)^2}\leq-1
\eeqa
for $\om_0\in(\frac{4-3\ga}{3},\frac{7-5\ga}{20})$ (where we are using that $\frac{7-5\ga}{20}<\frac{2-\ga}{3}$). We therefore use the fact that the coefficient of $\om_0W$ is negative on this region to make the lower bound
\beqa\label{eq:4A2+A1}
4A_2+A_1\geq&\,\big(-10(3-\ga)\om_0^2+2(\ga-1)(10\ga-19)\om_0-5(\ga-1)(2-\ga)(\ga+1)\big)\om_0^2R\\
&-\big(40\om_0^2-4(4-3\ga)\om_0-2(\ga-1)\om_0\big)\om_0^2R+20\om_0^4-(14-10\ga)\om_0^3\\
=&\,\big(-10(7-\ga)\om_0^2+4(5\ga^2-17\ga+13)\om_0-5(\ga-1)(2-\ga)(\ga+1)\big)\om_0^2R\\
&+20\om_0^4-(14-10\ga)\om_0^3.
\eeqa
We may check that the coefficient of $\om_0^2R$ is still negative:
{\footnotesize\begin{lstlisting}
quad4(w,g)=-10*(7-g)*w^2+4*(5*g^2 - 17*g + 13)*w-5*(g-1)*(2-g)*(g+1)
q4(v)=quad4(v[1]-v[2],v[2])
maximise(q4,V,tol=1e-3)
\end{lstlisting}}
\noindent which gives an upper bound in $[-2.56641, -2.55714]$. We therefore bound $4A_2+A_1$ below on this region by taking $R=-\frac{1}{2-\ga}$ in \eqref{eq:4A2+A1}. This leaves us with
\beqas
4A_2+A_1\geq&\,\frac{\om_0^2}{2-\ga}\Big(10(7-\ga)\om_0^2-4(5\ga^2-17\ga+13)\om_0+5(\ga-1)(2-\ga)(\ga+1)\\
&\quad+(2-\ga)\big(20\om_0^2-(14-10\ga)\om_0\big)\Big)
\eeqas
which we check is again positive for $\om_0\in(\frac{4-3\ga}{3},\frac{7-5\ga}{20})$ (in fact it is uniformly positive for all $\om_0\in(\frac43-\ga,2-\ga)$ and $\ga\in(1,\frac43)$ by interval arithmetic):
{\footnotesize\begin{lstlisting}
quad5(w,g)=10*(11-3*g)*w^2-2*(15*g^2-51*g+40)*w+5*(g-1)*(2-g)*(g+1)
q5(v)=quad5(v[1]-v[2],v[2])
minimise(q5,V,tol=1e-3)
\end{lstlisting}}
\noindent  with the minimum in $[2.54864, 2.55882]$, concluding the proof of \textit{(ii)}.

\textit{(iii)} A similar strategy holds again for showing \eqref{ineq:detA20} for the last quantity,
$4A_2+2A_1+A_0$. In fact, grouping terms again, we find
\beqas
4A_2&+2A_1+A_0\\
=&\,\big(-2(4+3\ga)(\ga-1)(2-\ga) + 2(\ga-1)(10\ga-21)\om_0 +2(7\ga-19)\om_0^2\big)\om_0^2R\\
&+\big(-2(\ga-1)(2-\ga)(\ga+1) +(6\ga-10)\om_0 + (28+4\ga)\om_0^2\big)\om_0W\\
&+\om_0 \big((4-3\ga)(\ga-1)(2-\ga)(\ga+1) +(\ga-1)(3\ga^2-9\ga+2)\om_0\\
&\qquad\quad -6(\ga-1)(3-\ga)\om_0^2+6(\ga-1)\om_0^3\big),
\eeqas
  the coefficient of $\om_0^2R$ is always negative again:
  {\footnotesize\begin{lstlisting}
quad6(w,g)=2*(7*g-19)*w^2+2*(g-1)*(10*g-21)*w-2*(4+3*g)*(g-1)*(2-g)
q6(v)=quad6(v[1]-v[2],v[2])
maximise(q6,V,tol=1e-3)
\end{lstlisting}}
\noindent gives an upper bound in the interval $[-2.67263, -2.65602]$. Next, we see that there exists
 $$\om_*(\ga)=\frac{5-3\ga + \sqrt{-87+10\ga+129\ga^2-40\ga^3-8\ga^4}}{4(7+\ga)}$$ such that the coefficient of $\om_0W$ is non-negative for $\om_0\geq \om_*(\ga)$ and negative for $\om_0\in(0,\om_*(\ga))$. 
 
 For $\om_0\in(\max\{\frac{4-3\ga}{3},\om_*(\ga)\},2-\ga)$, we check then that 
 \begin{align*}
 4A_2&+2A_1+A_0\\
 \geq&\,\big(-2(4+3\ga)(\ga-1)(2-\ga) + 2(\ga-1)(10\ga-21)\om_0 +2(7\ga-19)\om_0^2\big)\om_0^2\frac{-1}{2-\ga}\\
&+\om_0 \big((4-3\ga)(\ga-1)(2-\ga)(\ga+1) +(\ga-1)(3\ga^2-9\ga+2)\om_0\\
&\qquad\quad -6(\ga-1)(3-\ga)\om_0^2+6(\ga-1)\om_0^3\big)\\
>&\,\frac{\om_0^2}{2-\ga}\Big(\big(2(4+3\ga)(\ga-1)(2-\ga) - 2(\ga-1)(10\ga-21)\om_0 -2(7\ga-19)\om_0^2\big)\\
&+ (2-\ga)\big((\ga-1)(3\ga^2-9\ga+2) -6(\ga-1)(3-\ga)\om_0+6(\ga-1)\om_0^2\big)\Big)\\
=&\,\frac{\om_0^2}{2-\ga}\Big(-2(3\ga^2-2\ga-13)\om_0^2-2(\ga-1)(3\ga^2-5\ga-3)\om_0\\
&\hspace{11mm}+(\ga-1)(2-\ga)(3\ga^2-3\ga+10)\Big)\\
>&\,0
 \end{align*}
 by using interval arithmetic to estimate the final quadratic by
{\footnotesize\begin{lstlisting}
quad7(w,g)=-2(3*g^2-2*g-13)*w^2-2*(g-1)*(3*g^2-5*g-3)*w
              +(g-1)*(2-g)*(3*g^2-3*g+10)
q7(v)=quad7(v[1]-v[2],v[2])
minimise(q7,V,tol=1e-3)
\end{lstlisting}}
\noindent and obtaining the minimum is in $[2.34889, 2.35904]$.
 
 There is a $\ga_*\approx 1.148$ such that $\om_*(\ga)\leq\frac{4-3\ga}{3}$ if $\ga<\ga_*$ and reverse inequality otherwise. In the former case, we are already done. However, $\om_*(\ga)<\frac{2-\ga}{3}$ for all $\ga\in(1,\frac43)$ by using
 {\footnotesize\begin{lstlisting}
wstar(g)=(5-3*g+sqrt(-87+10*g+129*g^2-40*g^3-8*g^4))/(4*(7+g))
G=1..(4/3)
maximise(g->(wstar(g)-(2-g)/3),G,tol=1e-3)
\end{lstlisting}}
\noindent which gives a maximum in $[-0.0134389, -0.0127184]$. Hence, for $\ga\geq \ga_*$ and $\om_0\in(\frac{4-3\ga}{3},\om_*(\ga))$, we again get $\frac{\om_0W}{\om_0^2R}\leq -1$ by the same argument as that leading to \eqref{om0W/om0R}. Therefore, for $\om_0$ in this region, replacing $\om_0W$ with $-\om_0^2R$ and combining terms, we obtain 
 \begin{align}
 4A_2&+2A_1+A_0\nonumber\\
\geq&\,\big(-2(4\ga+3)(\ga-1)(2-\ga) + 2(\ga-1)(10\ga-21)\om_0 +2(7\ga-19)\om_0^2\big)\om_0^2R\nonumber\\
&-\big(-2(\ga-1)(2-\ga)(\ga+1) +(6\ga-10)\om_0 + (28+4\ga)\om_0^2\big)\om_0^2R\nonumber\\
&+\om_0 \big((4-3\ga)(\ga-1)(2-\ga)(\ga+1) +(\ga-1)(3\ga^2-9\ga+2)\om_0\nonumber\\
&\qquad -6(\ga-1)(3-\ga)\om_0^2+6(\ga-1)\om_0^3\big)\nonumber\\
=&\,\big(2 (4 - 7 \ga^2 + 3 \ga^3) + 2 (26 - 34 \ga + 10 \ga^2) \om_0 + 2 (-33 + 5 \ga) \om_0^2\big)\om_0^2R \\
&+\om_0 \big((\ga-1)(3\ga^2-9\ga+2)\om_0 -6(\ga-1)(3-\ga)\om_0^2+6(\ga-1)\om_0^3\big),\nonumber
 \end{align}
where we have also dropped the first order term in $\om_0$ in the last line. The new coefficient of $R$ is again seen to be negative as
 {\footnotesize\begin{lstlisting}
quad8(w,g)=2(5*g-33)*w^2+2*(10*g^2-34*g+26)*w+2*(3*g^3-7*g^2+4)
q8(v)=quad8(v[1]-v[2],v[2])
maximise(q8,V,tol=1e-3)
\end{lstlisting}}
\noindent gives a maximum in $[-2.62435, -2.59568]$.
Thus it is to obtain a lower bound by using $R<-\frac{1}{2-\ga}$ and factoring out $\om_0^2$ from the remainder. We arrive at the lower bound
\beqas
 4A_2&+2A_1+A_0\\
=&\,\frac{\om_0^2}{2-\ga}\Big(\big(-2 (4 - 7 \ga^2 + 3 \ga^3) -2 (26 - 34 \ga + 10 \ga^2) \om_0 - 2 (-33 + 5 \ga) \om_0^2\big)\\
&+ (2-\ga)\big((\ga-1)(3\ga^2-9\ga+2)\om_0 -6(\ga-1)(3-\ga)\om_0^2+6(\ga-1)\om_0^3\big)\\
=&\,\frac{\om_0^2}{2-\ga}\Big(-2(3\ga^2-4\ga-27)\om_0^2+2(-3\ga^3+8\ga^2+\ga-8)\om_0+3(\ga-1)(2-\ga)(\ga^2-\ga+2)\Big).
\eeqas
We verify that the quadratic in $\om_0$ in parentheses is always positive on $(\frac{4-3\ga}{3},\om_*(\ga))$ for $\ga>\ga_*$ (in fact the sign holds on all $\om_0\in(\frac{4-3\ga}{3},2-\ga)$ and $\ga\in(1,\frac43)$) by the following interval arithmetic:
 {\footnotesize\begin{lstlisting}
quad9(w,g)=(-6*g^2+8*g+54)*w^2+(-6*g^3+16*g^2+2*g-16)*w
            +3*(g-1)*(2-g)*(g^2-g+2)
q9(v)=quad9(v[1]-v[2],v[2])
minimise(q9,V,tol=1e-3)
\end{lstlisting}}
\noindent shows a minimum in the range $[1.5754, 1.58222]$, concluding the proof.

%%%%%%%%%%%%%%%%%%%%%%%%
%%%%%%%%%%%%%%%%%%%%%%%%

\subsection{Proof of Lemma \ref{lemma:Qmintervalarithmetic}}\label{app:Qm}
\textit{Step 1: We prove \eqref{ineq:Qmmax} in the case $m=1$.}\\
Recall
\begin{align*}
Q_1^+(\om)=&\,\big(1-\frac{\om}{2-\ga}\big)\Big(-\frac{3-2\ga}{2-\ga}(4-3\ga)\om^2+(\ga-1)(4-3\ga)(2-\ga)\Big)-\frac{\ga-1}{2-\ga}\om^2-2(\ga-1)\om\\
=&\, (4-3 \ga) (2 - \ga) (\ga-1) -3(2-\ga)(\ga-1) \om + \big(-\frac{\ga-1}{2 - \ga} + \frac{(4 - 3 \ga) (-3 + 2 \ga)}{2 - \ga} \big)\om^2 \\
&+ \frac{(-3 + 2 \ga) (-4 + 3 \ga)}{(2 - \ga)^2}\om^3\\
Q_1^-(\om)=&\,\big(1-\frac{\om}{2-\ga}\big)\Big(-\frac{3-2\ga}{2-\ga}(4-3\ga)\om^2+(\ga-1)(4-3\ga)(2-\ga)\Big)-\frac{\ga-1}{2-\ga}\om^2-2(\ga-1)\om\\
&+\frac{\om^2}{(2-\ga)^2\om}\big((4-3\ga)(2-\ga)-(3-2\ga)\big)\\
=&\,(\ga-1)\Big((4-3\ga)(2-\ga)+(3\ga-6)\om-\frac{6\ga-7}{2-\ga}\om^2+\frac{3(2\ga-3)}{(2-\ga)^2}\om^3\Big)
\end{align*}
Considering first $Q_1^+$, we check that 
\beqa
Q_1(\frac{4-3\ga}{3})=&\,-\frac{(4-3\ga)^2}{27(2-\ga)^2}(9\ga^2-25\ga+18)<0\text{ for all }\ga\in(1,\frac43),\\
Q_1(2-\ga)=&\,-3(2-\ga)(\ga-1)<0\text{ for all }\ga\in(1,\frac43).
\eeqa
We then check using interval arithmetic that \begin{itemize}
\item[(i)] $Q_1^+(\om)<0$ for all $\om\in[\frac43-\ga,2-\ga]$, for $\ga\in[1.02,1.15]$,
\item[(ii)] $Q_1^+(\om)<0$ for all $\om\in[\frac43-\ga,1.8-\ga]$, for $\ga\in[1,1.02]$,
\item[(iii)] $\d_\om Q_1^+(\om)<0$ for all $\om\in[\frac43-\ga,2-\ga]$, for $\ga\in[1.15,\frac43]$,
\item[(iv)] $\d_\om Q_1^+(\om)>0$ for all $\om\in[1.8-\ga,2-\ga]$, for $\ga\in[1,1.02]$,
\end{itemize}
all of which combine to prove that $Q_1^+(\om)<0$ for all $\om\in(\frac43-\ga,2-\ga)$, for $\ga\in(1,\frac43)$.\\
These are checked with the following Julia code (removing line breaks in the definition of functions):
{\footnotesize{\begin{lstlisting} 
Q1plus(w,g)=((2*g-3)*(3*g-4)*w^3)/((2-g)^2)
            +(-((g-1)/(2-g))+((4-3*g)*(2*g-3))/(2-g))*w^2
            +(-2*(g-1)-(g-1)*(4-3*g))*w+(4-3*g)*(2-g)*(g-1)
dQ1plus(w,g)=3*((2*g-3)*(3*g-4)*w^3)/((2-g)^2)
            +2*(-((g-1)/(2-g))+((4-3*g)*(2*g-3))/(2-g))*w^2
            +(-2*(g-1)-(g-1)*(4-3*g))

p1(v)=Q1plus(v[1]-v[2],v[2])
p2(v)=dQ1plus(v[1]-v[2],v[2])

V2=IntervalBox((4/3)..2,(1.02)..(1.15))
V3=IntervalBox((4/3)..(1.8),1..(1.02))
V4=IntervalBox((1.8)..2,1..(1.02))
V5=IntervalBox((4/3)..2,(1.15)..(4/3))
\end{lstlisting}}}
\noindent Property (i) then follows from 
{\footnotesize{\begin{lstlisting} 
maximise(p1,V2,tol=1e-4)
\end{lstlisting}}}
\noindent which gives $\max_{V_2} p_1\in[-0.0178999, -0.0177931]$. Property (ii) follows from
{\footnotesize{\begin{lstlisting} 
maximise(p1,V3,tol=1e-3)
\end{lstlisting}}}
\noindent which gives $\max_{V_3} p_1\in[-0.0638237, -0.0622099]$. Property (iii) follows from 
{\footnotesize{\begin{lstlisting} 
maximise(p2,V5,tol=1e-2)
\end{lstlisting}}}
\noindent giving  $\max_{V_5} p2\in[-0.321421, -0.231604]$. Finally,
{\footnotesize{\begin{lstlisting} 
minimise(p2,V4,tol=1e-3)	
\end{lstlisting}}}
\noindent yields $\min_{V_4}p_2\in[0.178011, 0.190886]$, as required.

To prove the negativity of $Q_1^-$, it is enough to observe that
$$Q_1^-(\frac43-\ga)=-\frac{2(4-3\ga)^2(\ga-1)^2}{9(2-\ga)^2}<0,$$
and, moreover, by interval arithmetic, $\d_\om Q_1^-(\om)<0$ always. To check this last property, we cancel the factor $\ga-1$ to define a function Q1min$=\frac{Q_1^-}{\ga-1}$ and then find the maximum:
{\footnotesize{\begin{lstlisting}
Q1min(w,g)=(4-3*g)*(2-g)+(3*g-6)*w-(6*g-7)*w^2/(2-g)+3*(2*g-3)*w^3/(2-g)^2
dQ1min(w,g)=(3*g-6)-2*(6*g-7)*w/(2-g)+9*(2*g-3)*w^2/((2-g)^2)

p3(v)=Q1min(v[1]-v[2],v[2])
p4(v)=dQ1min(v[1]-v[2],v[2])
V=IntervalBox((4/3)..2,1..(4/3))

maximise(p4,V,tol=1e-3)
\end{lstlisting}}}
\noindent This yields $\max_V p_4\in[-2.003, -1.99999]$, so that, for all $\ga\in(1,\frac43)$, we have $\d_\om Q_1^-<0$.

\textit{Step 2: We prove the estimate \eqref{ineq:Qmmax} for $Q_m$ for $m\in[1,\frac{2\ga}{\ga+1}]$.}\\
 To extend the estimates for $Q_m$ from $m=1$ to $m\in[1,\frac{2\ga}{\ga+1}]$, we proceed as follows.
We first define a new variable $k$ so that $m-1=\frac{\ga-1}{\ga+1}k$, to ensure $k\in[0,1]$ when $m\in[1,\frac{2\ga}{\ga+1}]$. $m$ is then recovered from $k$ by
{\footnotesize\begin{lstlisting}
m(k,g)=((g-1)*k+(g-1))/(g+1)
\end{lstlisting}}
We create two new functions
\begin{align*}
Q_5(\om,m)=&\,Q_m^+(\om)\\
&\,\big(1-\frac{\om}{2-\ga}\big)\Big(-\frac{4-m-2\ga}{2-\ga}(4-3\ga)\om^2+(\ga-1)(4-3\ga)(2-\ga)\Big)\\
&-\frac{2(4-m-2\ga)(m-1)\om^3}{(2-\ga)^2}-\frac{m(\ga-1)}{2-\ga}\om^2-2(\ga-1)\om,\\
Q_6(\om,m)=&\,Q_m^-(\om)\\
=&\,\big(1-\frac{\om}{2-\ga}\big)\Big(-\frac{4-m-2\ga}{2-\ga}(4-3\ga)\om^2+(\ga-1)(4-3\ga)(2-\ga)\Big)\\
&-\frac{2(4-m-2\ga)(m-1)\om^3}{(2-\ga)^2}-\frac{m(\ga-1)}{2-\ga}\om^2-2(\ga-1)\om\\
&+\frac{m\om^2}{(2-\ga)^2}\Big((4-3\ga)(2-\ga)-\om\big(4-3\ga+(\ga-1)(2-m)\big)\Big)\\
=&\,\frac{6\ga^2+\ga(m^2+8m-24)+m^2-16m+24}{(2-\ga)^2}\om^3+\frac{-6\ga^2+\ga(20-7m)+9m-16}{2-\ga}\om^2\\
&-3(2-\ga)(\ga-1)\om+(4-3\ga)(\ga-1)(2-\ga).
\end{align*}
Note that when $m=1$, these are just $Q_1^+$ and $Q_1^-$ from above. We then compute the derivative with respect to $m$ to get
\beqas
\d_mQ_5=&\,\frac{4m+7\ga-14}{(2-\ga)^2}\om^3+\frac{5-4\ga}{2-\ga}\om^2.
\eeqas
It is then straightforward to see that for $m\in[1,\frac{2\ga}{\ga+1}]$, the coefficient of the $\om^3$ term is negative. Note also that $\d_m Q_5<0$ for $\om>\om_*=\frac{(2-\ga)(5-4\ga)}{14-7\ga-4m}$. We check by interval arithmetic that for all $m\in[1,\frac{2\ga}{\ga+1}]$, all $\ga\in(1,\frac43)$, we have $\om_*<\frac43-\ga+0.1$:
{\footnotesize\begin{lstlisting}
wcrit(g,n)=(2-g)*(5-4*g)/(14-7*g-4*n)
wcritdiff(g,n)=wcrit(g,n)-(4/3)-0.1+g
fun(h)=wcritdiff(h[1],m(h[2],h[1]))
Gcrit=IntervalBox(1..(4/3),0..1)
maximise(fun,Gcrit)
\end{lstlisting}}
\noindent The maximum lies in $[-0.154379, -0.153729]$, hence is negative. Thus, for $\om>\frac43-\ga+0.1$, we have $Q_5(\om,m)\leq Q_1^+(\om)<0$. On the other hand, for $\om\in(\frac43-\ga,\frac43-\ga+0.1)$, we have from interval arithmetic that $\d_\om Q_1<-0.29$:
{\footnotesize\begin{lstlisting}
Vcrit=IntervalBox((4/3)..((4/3)+0.1),1..(4/3))
maximise(p2,Vcrit)
\end{lstlisting}}
\noindent The output is in $[-0.304395, -0.297167]$.

 We check that
$$\d^2_{m\om}Q_5=3\frac{4m+7\ga-14}{(2-\ga)^2}\om^2+2\frac{5-4\ga}{2-\ga}\om<0$$
provided $\om>\frac{2}{3}\om_*$. We check that $\frac23 \om_*<\frac43-\ga$ always:
{\footnotesize\begin{lstlisting}
fun2(h)=2*wcrit(h[1],m(h[2],h[1]))/3-(4/3)+h[1]
maximise(fun2,Gcrit)
\end{lstlisting}}
\noindent with output $[-0.0363685, -0.0358193]$. Thus, we retain $\d_\om Q_5<0$ on $\om\in(\frac43-\ga,\frac43-\ga+0.1)$ for all $m$ and $\ga$ in the range we require. Thus, using the fact that, at $\om=\frac43-\ga$, we have
$$Q_5\big|_{\om=\frac43-\ga}=-\frac{(4-3\ga)^2m(26 + 9 \ga^2 - 8 m + \ga (-31 + 6 m))}{27(2-\ga)^2}<0$$
by using
{\footnotesize\begin{lstlisting}
Q5end(g,n)=26+9*g^2-8*n+g*(-31+6*n)
fun3(h)=Q5end(h[1],m(h[2],h[1]))
minimise(fun3,Gcrit,tol=1e-2)
\end{lstlisting}}
\noindent with minimum in the range $[0.483905, 0.681199]$, we conclude $Q_5<0$ for all suitable $\ga$ and $m$.

To handle $Q_6$, we compare it to $Q_1^-$ above. We write
$$\d_\om Q_6=\d_\om(Q_6-Q_1^-)+\d_\om Q_1^-,$$
exploiting the definition of $k$ to introduce factors of $\ga-1$ wherever we find $m-1$. In particular, we have
\beqas
Q_6(\om,m(k))-Q_1^-(\om)=&\,\om^3\frac{\ga-1}{\ga+1}\frac{k}{(2-\ga)^2}\big((\ga-1)k+10\ga-14\big)+\om^2\frac{\ga-1}{\ga+1}k\frac{9-7\ga}{2-\ga}.
\eeqas 
Factoring out $\ga-1$, we differentiate and find
\beqas
\d_\om(Q_6-Q_1^-)=&\,(\ga-1)\om\Big(3\om\frac{k}{(\ga+1)(2-\ga)^2}\big((\ga-1)k+10\ga-14\big)+2\frac{k}{\ga+1}\frac{9-7\ga}{2-\ga}\Big).
\eeqas
Interval arithmetic then yields 
$$\d_\om(Q_6-Q_1^-)+\d_\om Q_1^-<0\text{ for all }(\om,\ga,k)\in[\frac43-\ga,2-\ga]\times[1,\frac43]\times[0,1]$$
by working without the common factor of $(\ga-1)$:
{\footnotesize\begin{lstlisting}
Q6diffw(w,g,k)=3*w^2*(k/(g+1))*((g-1)k+10*g-14)/((2-g)^2)
               +2*w*(k/(g+1))*(9-7*g)/(2-g)
B=IntervalBox((4/3)..2,1..(4/3),0..1)
p6diff(u)=dQ1min(u[1]-u[2],u[2])+Q6diffw(u[1]-u[2],u[2],u[3])
maximise(p6diff,B,tol=1e-2)
\end{lstlisting}}
\noindent giving a maximum in the range $[-2.03123, -1.99999]$.
This establishes inequality \eqref{ineq:Qmmax} for $m\in\big[1,\frac{2\ga}{\ga+1}\big]$.

\textit{Step 3: We extend to cover the full range $m\in\big[1,\frac{2\ga}{\ga+1}+\de\big]$.}\\
 To extend  \eqref{ineq:Qmmax} to $m\in\big[1,\frac{2\ga}{\ga+1}+\de\big]$, we argue directly by continuity with respect to $m$, uniformly with respect to $\om\in[\frac{4-3\ga}{3},2-\ga]$ for each $\ga\in(1,\frac43)$. As $Q_{\frac{2\ga}{\ga+1}}(\om)<0$ for all $\om\in[\frac{4-3\ga}{3},2-\ga]$, for each $\ga\in(1,\frac43)$, we obtain the existence of such a claimed $\de>0$.

\textit{Step 4: Prove \eqref{ineq:Qmmin}.}\\
 To check \eqref{ineq:Qmmin}  rigorously, we follow the following procedure:
Define 
\begin{align*}
Q_3(\om)=&\,Q_{\frac{4}{4-3\ga}}^+(\om)\\
=&\,M(\ga)(4-3\ga)(1-\frac{\om}{2-\ga})\om^2+M(\ga)\om^3\big(-\frac{2}{2-\ga}+\frac{8}{(4-3\ga)(2-\ga)}\big)\\
&-\frac{2(\ga-1)\om^2}{(4-3\ga)(2-\ga)}-2(\ga-1)\om+(4-3\ga)(\ga-1)(2-\ga)\big(1-\frac{\om}{2-\ga}\big)\\
=&\,\om^3\frac{2(3\ga^2-10\ga+6)(9\ga^2-30\ga+16)}{(4-3\ga)^2(2-\ga)^2}+\om^2\frac{2(9\ga^3-42\ga^2+57\ga-23)}{(2-\ga)(4-3\ga)}\\
&-3(2-\ga)(\ga-1)\om+(4-3\ga)(\ga-1)(2-\ga),\\
Q_4(\om)=&\,Q_{\frac{4}{4-3\ga}}^-(\om)\\
=&\,M(\ga)(4-3\ga)(1-\frac{\om}{2-\ga})\om^2+M(\ga)\om^3\big(-\frac{2}{2-\ga}+\frac{8}{(4-3\ga)(2-\ga)}\big)\\
&-\frac{2(\ga-1)\om^2}{(4-3\ga)(2-\ga)}-2(\ga-1)\om+(4-3\ga)(\ga-1)(2-\ga)\big(1-\frac{\om}{2-\ga}\big)\\
&+\om^2\Big(\frac{4(1-\frac{\om}{2-\ga})}{2-\ga}-\om\frac{4(\ga-1)}{(4-3\ga)(2-\ga)}\big(\frac{2}{2-\ga}-\frac{4}{(4-3\ga)(2-\ga)}\big)\Big)\\
=&\,\om^3\frac{6(9\ga^4-60\ga^3+132\ga^2-104\ga+24)}{(4-3\ga)^2(2-\ga)^2}+\om^2\frac{6(3\ga^3-14\ga^2+17\ga-5)}{(2-\ga)(4-3\ga)}\\
&-3(2-\ga)(\ga-1)\om+(4-3\ga)(\ga-1)(2-\ga).
\end{align*}
We need to prove the positivity of both $Q_3$ and $Q_4$. To show the positivity of $Q_3$, we note the following four facts:
\begin{itemize}
\item $Q_3(\frac43-\ga)=\frac{-2(4-3\ga)(21\ga^2-71\ga+42)}{27(2-\ga)^2}>0$ for $\ga\in(1,\frac43)$,
\item $Q_3'(\frac43-\ga)=\frac{27 \ga^4-183 \ga^3+ 402 \ga^2 - 312 \ga  +80}{3(2-\ga)^2}>0$ for $\ga\in(1,\frac43)$,
\item $Q_3^{(2)}(\frac43-\ga)=\frac{4(18\ga^4-120\ga^3+261\ga^2-203\ga+50)}{(4-3\ga)(2-\ga)^2}>0$ for $\ga\in(1,\frac43)$,
\item $Q_3^{(3)}(\om)=\frac{12(3\ga^2-10\ga+6)(9\ga^2-30\ga+16)}{(4-3\ga)^2(2-\ga)^2}>0$ for all $\om\in(\frac43-\ga,2-\ga)$, $\ga\in(1,\frac{4}{3})$.
\end{itemize}
Each of these is proved by interval arithmetic. We scale out the factors of $(4-3\ga)^{-1}$ and $(4-3\ga)^{-2}$ in the second and third derivatives of $Q_3$ before computing to ensure the computations remain bounded.
{\footnotesize\begin{lstlisting}
g5(g)=-(21*g^2-71*g+42)
g6(g)=80-312*g+402*g^2-183*g^3+27*g^4
g7(g)=50-203*g+261*g^2-120*g^3+18*g^4
g8(g)=(6-10*g+3*g^2)*(16-30*g+9*g^2)
G=(4/3)..2
minimise(g5,G)
minimise(g6,G)
minimise(g7,G)
minimise(g8,G)
\end{lstlisting}}
\noindent yielding $\min_G g_5\in[15.2996, 15.3379]$, $\min_G g_6\in[29.4466, 30.2343]$, $\min_G g_7\in[14.8603, 15.7849]$, $\min_G g_8\in[15.9016, 16.0085]$.

Similarly,
\begin{itemize}
\item $Q_4(\frac43-\ga)=\frac{-2(4-3\ga)(3\ga^2-13\ga+6)}{9(2-\ga)^2}>0$ for $\ga\in(1,\frac43)$,
\item $Q_4'(\frac43-\ga)=\frac{9\ga^4-61\ga^3+138\ga^2-112\ga+32}{(2-\ga)^2}>0$ for $\ga\in(1,\frac43)$,
\item $Q_4^{(2)}(\frac43-\ga)=\frac{12(6\ga^4-40\ga^3+87\ga^2-65\ga+14)}{(4-3\ga)(2-\ga)^2}>0$ for $\ga\in(1,\frac43)$,
\item $Q_4^{(3)}(\om)=\frac{36(9\ga^4-60\ga^3+132\ga^2-104\ga+24)}{(4-3\ga)^2(2-\ga)^2}>0$ for all $\om\in(\frac43-\ga,2-\ga)$, $\ga\in(1,\frac{4}{3})$.
\end{itemize}
{\footnotesize\begin{lstlisting}
g9(g)=-(3*g^2-13*g+6)
g10(g)=32-112*g+138*g^2-61*g^3+9*g^4
g11(g)=14-65*g+87*g^2-40*g^3+6*g^4
g12(g)=24-104*g+132*g^2-60*g^3+9*g^4
minimise(g9,G)
minimise(g10,G)
minimise(g11,G)
minimise(g12,G)
\end{lstlisting}}
\noindent yielding $\min_G g_9\in[5.99518, 6.00151]$, $\min_G g_{10}\in[11.5885, 11.8567]$, $\min_G g_{11}\in[5.98055, 6.15133]$, $\min_G g_{12}\in[5.96691, 6.22624]$.

%%%%%%%%%%%%%%%%%%%%%%%%
%%%%%%%%%%%%%%%%%%%%%%%%

%%%%%%%%%%%%%%%%%%%%%%%%%
%%%%%%%%%%%%%%%%%%%%%%%%%

\section{Proof of Proposition~\ref{prop:sonictime}}\label{app:continuity}

%%%%%%%%%%%%%%%%%%%%%%%%%
%%%%%%%%%%%%%%%%%%%%%%%%%

\renewcommand{\theequation}{\thesection.\arabic{equation}}

Before we prove the proposition,
it is convenient to rescale the sonic point to a fixed value so that some of the continuity properties are easier to prove. 
We let
\begin{align}
z: = \frac{y}{y_\ast}, \ \ \rho(y)=r(z), \ \ \omega(y) = w(z).
\end{align}
The system~\eqref{eq:rhoom} takes the form
\begin{align}
r' &=\,\frac{y_\ast^2 zr h(r,w)}{\mathcal G(z;r,w)},\label{E:REQN}\\
w' &=\,\frac{4-3\ga-3w}{z}-\frac{y_\ast^2 zw h(r,w)}{\mathcal G(z;r,w)}, \label{E:OEQN}
\end{align}
where 
\begin{align}
\mathcal G(z;r,w) : = \ga r^{\ga-1}-y_\ast^2z^2w^2.
\end{align}
Moreover, the sonic time $s(y_\ast)$ scales naturally into
\begin{align}
S(y_\ast) : = \frac{s(y_\ast)}{y_\ast},
\end{align}
so that the interval $(S(y_\ast),1)$ comprises all the $z$-values in the interval $(0,1)$ for which the unique LPH-type solution exists and $\mathcal G>0$.
By analogy to~\eqref{E:FDEF}--\eqref{E:GDEF} we introduce the abbreviations
\begin{align}
\mathcal I &: = \frac{ zr h(r,w)}{\mathcal G(z;r,w)}, \ \
\mathcal J  : = \frac{ zw h(r,w)}{\mathcal G(z;r,w)}. \label{E:IJDEF}
\end{align}

\begin{proof}
We work with the formulation~\eqref{E:REQN}--\eqref{E:OEQN} for convenience. From there, it is easy to recover all the statements in the original $(\rho(y),\omega(y))$ variables.

\noindent
{\em Proof of part (i).}
We fix an $y_\ast\in[y_f,y_F]$ and an arbitrary $\mathring z\in (S(y_\ast),1-\nu)$. In the following all generic constants will depend on $\mathring z$ unless specified otherwise.
Since $\mathring z>S(y_\ast)$ there exists an $\eta>0$ such that 
\[
\ga r^{\ga-1} >\eta + y_\ast^2 z^2 w^2, \ \ z\in[\mathring z, 1-\nu).
\]
It follows in particular that 
\begin{align}\label{E:RLB}
r(z) > C_\gamma \eta^{\beta}, \ \ z\in[\mathring z, 1-\nu),
\end{align}
where
\begin{align}\label{E:BETACGAMMADEF}
\beta : = \frac1{\gamma-1}, \ \ C_\gamma : = \gamma^{-\frac1{\gamma-1}}.
\end{align}
Moreover, by Lemma~\ref{lemma:apriori} it is clear that there exists a constant $C = C(\mathring z)$ 
such that for any $\tilde y_\ast\in[y_f,y_F]$
\begin{align}\label{E:APRIORI1}
|r(z;\tilde y_\ast)|\le C, \ \ |w(z;\tilde y_\ast)| \le C, 
\ \ z\in[\mathring z, 1-\nu]\cap (S(\tilde y_\ast), 1-\nu).
\end{align}
Let $0<\delta\ll1$ be a control constant to be fixed later and consider the set of $\tilde y_\ast\in[y_f,y_F]$ such that $|\tilde y_\ast-y_\ast|<\delta$. For 
any such $\tilde y_\ast$ let $(\tilde r(\cdot;\tilde y_\ast),\tilde w(\cdot;\tilde y_\ast))$ be the unique LPH-type solution given by Theorem~\ref{thm:Taylor}. 
Let
\[
Z : = \max\{S(\tilde y_\ast), \mathring z\}. 
\]
and 
define the control function
\begin{align}\label{E:LITTLEGDEF}
g(z): = |r(z)-\tilde r(z)| + |w(z)-\tilde w(z)|, \ \ z\in (Z,1-\tilde\nu],
\end{align}
where $\tilde\nu$ is a $y_\ast$-independent positive constant whose existence follows from the existence of $\nu>0$ in Theorem~\ref{thm:Taylor}.
It is straightforward to check that
\begin{align}\label{E:IMINUSITILDE}
\mathcal I(y_\ast,r,w) - \mathcal I(\tilde y_\ast,\tilde r,\tilde w)
= z  \frac{r h(\tilde{\mathcal G}-\mathcal G) + \mathcal G r (h-\tilde h) + \mathcal G \tilde h (r - \tilde{r})}{\mathcal G \tilde{\mathcal G}},
\end{align}
where we used the shorthand $\tilde{\mathcal G} = \mathcal G(\tilde y_\ast,\tilde r,\tilde w)$ and similarly for $\tilde h$.
Note that 
\begin{align}\label{E:GDIFF}
\tilde{\mathcal G} - \mathcal G = \ga \left(\tilde r^{\gamma-1}-r^{\gamma-1}\right) - z^2\left(\tilde y_\ast^2 \tilde w^2 - y_\ast^2 w^2\right)
\end{align}
and also
\begin{align}
\lv \tilde r^{\gamma-1}-r^{\gamma-1}\rv
& = r^{\gamma-1}\lv\left(1+\frac{\tilde r - r}{r}\right)^{\gamma-1}-1\rv \notag \\
& \le  (\gamma-1)r^{\gamma-1} \sup_{|\theta|\le \frac{|\tilde r - r|}{r}}\lv1 - |\theta|\rv^{\gamma-2} \frac{|\tilde r - r|}{r} \notag \\
& =(\gamma-1) r^{\gamma-2}  \lv1 - \frac{|\tilde r - r|}{r} \rv^{\gamma-2} |\tilde r - r|,
\end{align}
where we have used the mean value theorem in the second line above.
Note that by~\eqref{E:LITTLEGDEF} and~\eqref{E:RLB} $ \frac{|\tilde r - r|}{r} \le \frac{g(z)}{C_\gamma \eta^\beta}$ and therefore since $\gamma-2<0$
\be\label{E:RZDIFF}
\lv \tilde r^{\gamma-1}-r^{\gamma-1}\rv \le C_\gamma^{-(2-\gamma)}\eta^{-(2-\gamma)\beta} \lv1 - \frac{g(z)}{C_\gamma\eta^\beta}\rv^{\gamma-2} g(z).
\ee
Moreover, by~\eqref{E:APRIORI1} it is easy to see that 
\begin{align}
\lv z^2\left(\tilde y_\ast^2 \tilde w^2 - y_\ast^2 w^2\right)\rv \le C\left(|y_\ast-\tilde y_\ast| + |w - \tilde w|\right).
\end{align}
Together with~\eqref{E:GDIFF} and~\eqref{E:RZDIFF} this gives
\begin{align}
\lv \tilde{\mathcal G} - \mathcal G \rv & \le  C\left(1+ \eta^{-(2-\gamma)\beta} \lv1 - \frac{g(z)}{C_\gamma \eta^\beta}\rv^{\gamma-2}\right) g(z) + C |y_\ast-\tilde y_\ast| \notag \\
& =: C K(\eta,g(z)) g(z) + C |y_\ast-\tilde y_\ast|, \label{E:GDIFF2}
\end{align}
where
\begin{align}\label{E:KETADEF}
K(\eta,g(z)) : = 1+ \eta^{-(2-\gamma)\beta} \lv1 - \frac{g(z)}{C_\gamma \eta^\beta}\rv^{\gamma-2}.
\end{align}
A simple consequence of~\eqref{E:GDIFF2} is a lower bound for $\tilde{\mathcal G}$,
\begin{align}
\tilde{\mathcal G} &\ge  \mathcal G - \lv \tilde{\mathcal G} - \mathcal G\rv \notag\\
&\ge \eta - C K(\eta,g(z)) g(z) - C |y_\ast-\tilde y_\ast| \notag \\
& = : \bar\eta(z). \label{E:GLB}
\end{align}
From the definition of $h(r,w)$ and the a priori bounds~\eqref{E:APRIORI1} it is straightforward to obtain the bound
\begin{align}\label{E:HUB}
\lv \tilde h(z) - h(z)\rv \le C g(z).
\end{align}
Using~\eqref{E:APRIORI1},~\eqref{E:GDIFF2},~\eqref{E:GLB}, and~\eqref{E:HUB} in~\eqref{E:IMINUSITILDE} we conclude
\begin{align}
\lv\mathcal I(y_\ast,r,w) - \mathcal I(\tilde y_\ast,\tilde r,\tilde w) \rv
\le \frac{C K(\eta,g(z)) g(z) +  C |y_\ast-\tilde y_\ast| }{\eta \, \bar\eta(z)}. \label{E:IBOUND}
\end{align}
The same proof also yields the bound
\begin{align}
\lv\mathcal J(y_\ast,r,w) - \mathcal J(\tilde y_\ast,\tilde r,\tilde w) \rv
\le \frac{C K(\eta,g(z)) g(z) +  C |y_\ast-\tilde y_\ast| }{\eta\, \bar\eta(z)}, \label{E:JBOUND}
\end{align}
where we recall~\eqref{E:IJDEF}.

Clearly, for $\delta>0$  and $|1-\nu-z|$ sufficiently small, we have from~\eqref{E:GDIFF2} and~\eqref{E:KETADEF} by continuity
\[
\bar\eta(z)>\frac\eta2, \ \ g(z)<\frac{C_\gamma \eta^\beta}{2},
\]
where $\bar\eta(z)$ is defined in~\eqref{E:GLB}. Let  
\begin{align}\label{E:BARZDEF}
\bar Z : = \inf_{Z<z<1-\nu} \left\{\bar\eta(z)>\frac\eta2 \ \text{ and } \ g(z)<\frac{C_\gamma\eta^\beta}{2}\right\},
\end{align}
where $C_\gamma>0$ is defined in~\eqref{E:BETACGAMMADEF}. The bound $g(z)<\frac{C_\gamma \eta^\beta}{2}$ ensures that 
\begin{align}\label{E:KETABOUND}
K(\eta,g(z)) \le 1 + \frac1{2^{2-\gamma}}\eta^{-(2-\gamma)\beta} = : K_\eta, \ \  z\in[\bar Z, 1-\tilde\nu].
\end{align}

Integrating over $[z,1-\tilde\nu]$ it follows from~\eqref{E:REQN}--\eqref{E:OEQN} and the bounds~\eqref{E:IBOUND}--\eqref{E:JBOUND} 
that 
\begin{align}
g(z) &\le g(1-\tilde\nu) + \frac{C}{\eta^2}|y_\ast-\tilde y_\ast| + \frac C{\mathring z} \int_z^{1-\tilde\nu}|w-\tilde w|\,\dif\tau
+ \frac C{\eta^2} \int_z^{1-\tilde\nu} K(\eta,g(\tau))\, g(\tau)\,\dif\tau \notag \\
& \le g(1-\tilde\nu) + \frac{C}{\eta^2}|y_\ast-\tilde y_\ast| 
+ \frac C{\eta^2} \int_z^{1-\tilde\nu} K(\eta,g(\tau))\, g(\tau)\,\dif\tau \notag \\
& \le   g(1-\tilde\nu) + \frac{C}{\eta^2}|y_\ast-\tilde y_\ast| +   \frac {CK_\eta}{\eta^2} \int_z^{1-\tilde\nu} g(\tau)\,\dif\tau,   \ \ z\in[\bar Z, 1-\tilde\nu],
\end{align}
where we recall~\eqref{E:KETADEF} and~\eqref{E:KETABOUND}. We now apply the Gr\"onwall inequality to conclude
\begin{align}\label{E:GUNIF}
 g(z) \le \left(g(1-\tilde\nu) + \frac{C}{\eta^2} |y_\ast-\tilde y_\ast|\right)  e^{ \frac{CK_\eta}{\eta^2}(1-\tilde\nu-z)}, \ \ z\in[\bar Z, 1-\tilde\nu].
\end{align}
We note that for any given $\delta'>0$, there exists a $\delta>0$ such that $g(1-\tilde\nu)<\delta'$ for all $|y_\ast-\tilde y_\ast|<\delta$.
Therefore, for any given $\epsilon>0$ we can choose a $\delta=\delta(\eta,\epsilon)$ 
sufficiently small so that 
for all $|y_\ast-\tilde y_\ast|<\delta$ we have the bound 
\be
g(z)<\epsilon, \ \ \bar Z < z \le 1-\tilde\nu. \notag 
\ee
In particular, with $0<\epsilon\ll1$ chosen sufficiently small we have $g(z)<\frac{\tilde C\eta^\beta}{3}$ on $(\bar Z, 1-\tilde\nu]$ and therefore
$
K(\eta,g(z)) < K_\eta 
$
on $[\bar Z,1-\tilde\nu]$. This in turn implies
\be\label{E:GOODFORPARTTHREE}
\bar\eta (\bar Z) \ge \eta - C K_\eta \epsilon - C\delta >\frac{\eta}{2}
\ee
for $0<\delta\ll1$ sufficiently small. This implies $\bar Z = Z$ and provides a uniform lower bound for $\mathcal G$ on $(Z,1-\tilde\nu]$ thus implying
$S(\tilde y_\ast)<Z$. Therefore $Z=\mathring z$ and since $\mathring z>S(y_\ast)$ is chosen arbitrarily,  this implies the upper semi-continuity.

\noindent
{\em Proof of part (ii).}
By Lemma~\ref{lemma:extension} it is clear that there exists a $\tilde\tau=\tilde\tau(\mathring y, \eta)$ such that $S(y_\ast^n)<\mathring y-\tilde\tau$ for all $n\in\mathbb N$.
We now use the lower bounds~\eqref{E:GLB} and~\eqref{E:GOODFORPARTTHREE} applied to the sequence $\{y^n_\ast\}_{n\in\mathbb N}$ to conclude that 
$S(y_\ast) < \frac{\mathring y}{y_\ast}-\tilde\tau$ for a possibly smaller $\tilde\tau>0$, which again depends only on $\mathring y$ and $\eta$.

\noindent
{\em Proof of part (iii).} By the proof of part (i) it follows that there exists a $\delta>0$ sufficiently small so that $S(\tilde y_\ast)<S(y_\ast)+ \frac12\left(\frac{y_0}{\tilde y_\ast} - S(y_\ast)\right)$
for all $|\tilde y_\ast-y_\ast|<\delta$. The claim now follows from the arguments in part (i) using in particular the uniform-in-$\tilde y_\ast$ upper bound~\eqref{E:GUNIF} for the distance function $g(z)$.
\end{proof}

\end{document}